\documentclass[10pt]{amsart}

\usepackage{amsmath, latexsym, amsfonts, amssymb, amsthm,MnSymbol,stmaryrd}
\usepackage{graphicx,hyperref,dsfont,epsfig,caption,wrapfig,subfig,pifont,tikz}
\usepackage{tabularx}
\PassOptionsToPackage{dvipsnames,svgnames,table}{xcolor}
\usepackage{booktabs}
\usepackage{movie15}
\hypersetup{
    linktoc=page,
    linkcolor=red,          % color of internal links
    citecolor=blue,        % color of links to bibliography
    filecolor=blue,      % color of file links
    urlcolor=cyan,
    colorlinks=true           % color of external links
}

\newcommand{\nocontentsline}[3]{}
\newcommand{\tocless}[2]{\bgroup\let\addcontentsline=\nocontentsline#1{#2}\egroup}
\numberwithin{equation}{section}

\newtheorem{theorem}{Theorem}[section]
\newtheorem{lemma}[theorem]{Lemma}
\newtheorem{proposition}[theorem]{Proposition}
\newtheorem{corollary}[theorem]{Corollary}

\newtheorem{remark}[theorem]{Remark}
\newtheorem{definition}[theorem]{Definition}

\theoremstyle{definition}

\renewcommand{\tilde}{\widetilde}          % wider `tilde'
\DeclareMathSymbol{\leqslant}{\mathalpha}{AMSa}{"36} % nicer `smaller or equal'
\DeclareMathSymbol{\geqslant}{\mathalpha}{AMSa}{"3E} % nicer `larger or equal'
\DeclareMathSymbol{\eset}{\mathalpha}{AMSb}{"3F}     % nicer `emptyset'
\renewcommand{\leq}{\;\leqslant\;}                   % redef. of < or =
\renewcommand{\geq}{\;\geqslant\;}                   % redef. of > or =
\newcommand{\dd}{\text{\rm d}}             % a straight d for differentials

\newcommand{\C}{\mathbb{C}}
\renewcommand{\H}{\mathbb{H}}
\newcommand{\D}{\mathbb{D}}
\newcommand{\R}{\mathbb{R}}
\newcommand{\Z}{\mathbb{Z}}
\newcommand{\N}{\mathbb{N}}
\newcommand{\Q}{\mathbb{Q}}

\newcommand{\E}{\mathds{E}}
\renewcommand{\P}{\mathds{P}}

\newcommand{\cjd}{\rangle}
\newcommand{\cjg}{\langle}

\newcommand{\hf}{\frac{_1}{^2}}

\def\l{\mathbf{l}}
\def\k{\mathbf{k}}

\newcommand{\pl}{\partial}
\newcommand{\bbar}{\overline}
\newcommand{\mc}{\mathcal}
\newcommand{\la}{\lambda}
\newcommand{\til}{\widetilde}

\def\vhi{\boldsymbol{\varphi}}

\usepackage{xparse}
\usepackage{float}
\restylefloat{table}

\def\eps{\epsilon}

\def\T{\mathbb{T}}

\def\bi{\begin{itemize}}
\def\ei{\end{itemize}}

\def\bnum{\begin{enumerate}}
\def\enum{\end{enumerate}}
\def\<#1{\langle #1 \rangle}

\newcommand{\caA}{{\mathcal A}}
\newcommand{\caB}{{\mathcal B}}
\newcommand{\caC}{{\mathcal C}}
\newcommand{\caD}{{\mathcal D}}
\newcommand{\caE}{{\mathcal E}}
\newcommand{\caF}{{\mathcal F}}

\newcommand{\caH}{{\mathcal H}}
\newcommand{\caI}{{\mathcal I}}

\newcommand{\caM}{{\mathcal M}}

\newcommand{\caO}{{\mathcal O}}
\newcommand{\caP}{{\mathcal P}}

\author{Colin Guillarmou}
\address{Universit\'e Paris-Saclay, CNRS,  Laboratoire de math\'ematiques d'Orsay, 91405, Orsay, France.}
\email{colin.guillarmou@math.u-psud.fr}

\author{Antti Kupiainen}
\address{University of Helsinki, Department of Mathematics and Statistics}
\email{antti.kupiainen@helsinki.fi}

\author{R\'emi Rhodes}
\address{Aix Marseille Univ, CNRS, I2M, Marseille, France, and Institut Universitaire de France (IUF)}
\email{remi.rhodes@univ-amu.fr}

\author{Vincent Vargas}
\address{Universit\'e de Gen\`eve, Section de math\'ematiques, UNI DUFOUR, 24 rue du G\'en\'eral Dufour,
CP 64 1211 Geneva 4, Switzerland}
\email{Vincent.Vargas@unige.ch}

\title{Segal's axioms and bootstrap for Liouville Theory}

\date{}
\begin{document}

\begin{abstract} 
 In 1987, Segal gave a functorial definition of Conformal Field Theory (CFT) that was designed to capture the mathematical essence of the Conformal Bootstrap formalism pioneered in physics by  Belavin, Polyakov and Zamolodchikov. In Segal's formulation, the basic objects of CFT, the correlation functions of conformal primary fields, are viewed as functions on the moduli space of Riemann surfaces with marked points which behave naturally under gluing of surfaces. In this paper we give a probabilistic realisation of Segal's axioms in Liouville Conformal Field Theory (LCFT), a CFT that plays a fundamental role in the theory of random surfaces and two-dimensional quantum gravity.  Namely, to a Riemann surface $\Sigma$  
with marked points and boundary given by a union of parameterised circles, we associate a Hilbert-Schmidt operator $\caA_\Sigma$, called the {\it amplitude} of $\Sigma$, which acts on  a tensor product of Hilbert spaces assigned to the boundary circles. We show that this correspondence is functorial: gluing of surfaces along boundary circles maps to a composition of the corresponding operators. Correlation functions of LCFT, constructed probabilistically in earlier works by the authors and F. David, can then be expressed as compositions of the amplitudes of simple building blocks where $\Sigma$ is a sphere with $b\in \{1,2,3\}$ disks removed (hence $b$ boundary circles) and $3-b$ marked points. These amplitudes in turn are shown to be determined by basic objects of LCFT: its {\it spectrum} and its {\it structure constants} determined in earlier works by the authors. As a consequence, we obtain a formula for the correlation functions as   multiple integrals over the spectrum of LCFT, the structure of these  integrals being  related  to a pant decomposition of the surface. The integrand is the square modulus of a function called conformal block: its structure is encoded by the commutation relations of an algebra of operators called the Virasoro algebra  and it depends holomorphically  on the moduli of the surface with marked points. The integration measure involves a product of  structure constants, which have an explicit expression, the so-called DOZZ formula.  Such a holomorphic factorisation of correlation functions has been conjectured in physics since the 80's and we give here its first rigorous derivation for a non-trivial CFT. 
  \end{abstract}

\maketitle
\tableofcontents

%%%%%%%%%%%%%%%%%%%%%%%%%%%%%%%%%%%%%%%%%%%%%%%%%%%%%%%%%%%%%%%%%%%%%%%%%%%%%%%%%%%%%%%%%%%%%%%%%%%%%%%%%%%%%%%%%%%%%%%%%%%%%%%%%%%%%%%%%%%%%%%%%%%%%%%%%%%%%%%%%%%%%%%%%%%%%%%%%%%%%%%%%%%%%%
\section{Introduction}

 Ever since the foundational work by  Belavin, Polyakov and Zamolodchikov \cite{BPZ84} Conformal Field Theory (CFT hereafter) has been a challenge and inspiration to mathematicians. Described by Polyakov \cite{polyakov2008quarksstrings} as ``complex analysis in the quantum domain", it has entered fields of mathematics ranging from the geometric Langlands program (see \cite{TeschnerHitchin} for Liouville CFT which is the topic of this paper) to probability theory through representation theory. The question ``what sort of mathematical object CFT is ?" has been addressed  by devising a set of  axioms it should satisfy and finding examples satisfying these axioms. An axiomatic scheme that grew out of the algebraic structures of CFT is the theory of vertex operator algebras \cite{borcherds,Frenkel:1988xz} and a more physical example are the axioms for conformal bootstrap spelled out in the two dimensional case in \cite{BPZ84}.  

Inspired both by the algebraic structure of CFT and by  the bootstrap picture, G. Segal presented  geometric axioms for CFT that make it particularly attractive for mathematicians \cite{Segal87}. In Segal's approach, CFT is   a {\it functor} from the category whose objects are disjoint unions $\caC=\bigsqcup_i \caC_i$ of parametrised circles $\caC_i$ and morphisms are closed oriented Riemannian surfaces $(\Sigma,g)$ with boundary, to the category whose objects are Hilbert spaces and morphisms are Hilbert-Schmidt operators. For objects, the CFT functor maps   a circle $\caC_i$ to a Hilbert space $\caH$ and  $\caC$ to the tensor product $ \caH^{\otimes_i}$. For morphisms, the CFT functor maps the surface $(\Sigma,g)$, with parametrised boundaries,  to a Hilbert-Schmidt operator $\caA_{\Sigma,g }: \caH^{\otimes_i}\to \caH^{\otimes_j}$, where $i$ is the number of positively oriented boundaries and $j$ the number of negatively oriented boundaries. The CFT functor is then required to behave in a natural way under the operation of gluing surfaces along boundaries which operation maps to a composition of the corresponding operators. 

The motivation for Segal's axioms came from yet another approach to CFT (and to quantum field theory in general) due to Feynman which connects CFT to probability theory, as we now explain.  Let $(\Sigma,g)$ be  a closed oriented Riemannian surface. In this approach, CFT is described in terms of a positive measure on a set  $\caF$  of (generalised)  functions $\phi$ on $\Sigma$.   Expectation (denoted by $\langle \cdot\rangle_{\Sigma,g}$ in what follows) under this measure is formally given as a {\it path integral}  
\begin{align}\label{pathi}
\langle F\rangle_{\Sigma,g}=\int_\caF F(\phi)e^{-S_{\Sigma}(g,\phi)}D\phi
\end{align}
for suitable  observables $F:\caF\to \C$ where $S_{\Sigma} (g,\cdot):\caF\to \R$ is an {\it action functional} (whose expression depends on the metric $g$) and $D\phi$ a formal Lebesgue measure on $\caF$. In this approach, the Hilbert space corresponding to Segal's axioms is a suitable $L^2$ space on the set of functions defined on the (unit) circle $\mathbb{T}$, equipped with some measure, while for a Riemannian surface $(\Sigma,g)$ with boundary, an operator $\caA_{\Sigma,g}$ is then described formally as an integral kernel
\begin{align}\label{ampli}
\caA_{\Sigma,g}(\boldsymbol{\varphi})=\int_{\phi|_\caC=\boldsymbol{\varphi}}e^{-S_\Sigma (g,\phi)}D\phi
\end{align}
where $\boldsymbol{\varphi}=(\varphi_1,\dots,\varphi_b)$ and $\varphi_i$ are  (generalised) functions defined on the boundary circles $\mc{C}_1,\dots,\mc{C}_b$ of $\partial\Sigma$.   We will call these integral kernels {\it amplitudes} in what follows. The physics heuristics behind Segal's axioms come from the following observation:
if $\Sigma=\Sigma_1\cup \Sigma_2$ is cut along a boundary circle $\mc{C}\simeq \mathbb{T}$ into two surfaces $(\Sigma_1,g)$ and $(\Sigma_2,g)$ 
and if the action $S_{\Sigma}(g,\phi)$ is local, then formally
\[\begin{split} 
\mc{A}_{\Sigma,g}(\varphi_1,\dots,\varphi_{b})=&  \int_{\mc{F}(\mc{C})}  \Big(\int_{\phi_1|_{\mc{C}}=\varphi, \phi_1|_{\mc{C}_j}=\varphi_j}e^{-S_{\Sigma_1}(g,\phi_1)}D\phi_1\Big)
\Big(\int_{\phi_2|_\caC=\varphi, \phi_2|_{\mc{C}'_j}=\varphi_j}e^{-S_{\Sigma_2}(g,\phi_2)}D\phi_2\Big) D\varphi\\
=&  \int_{\mc{F}(\mc{C})} \mc{A}_{\Sigma_1,g}(\varphi,\varphi_1,\dots,\varphi_{j_1})\mc{A}_{\Sigma_2,g}(\varphi,\varphi_{j_1+1},\dots,\varphi_{b}) D\varphi
\end{split}\]  
where $\pl \Sigma_1=\mc{C}\cup (\cup_{j=1}^{j_1}\mc{C}_j)$, $\pl \Sigma_2=\mc{C}\cup (\cup_{j=j_1+1}^{b}\mc{C}'_j)$ and $\mc{F}(\mc{C})$ denotes the set of fields on $\mc{C}$. Viewing $\mc{A}_{\Sigma_j,g}$ as integral kernels of operators acting on tensor products of $L^2(\mc{F}(\mathbb{T}))$ (for some measure on $\mc{F}(\mathbb{T})$), the right-hand side is nothing but the integral kernel of the composition of the associated operators.
For a lucid introduction to mathematicians of this point of view to CFT, we refer to \cite{Gawedzki96_CFT}.  In that spirit and from a probabilistic point of view, Segal's axioms are a natural and beautiful generalisation of the notion of semigroup for stochastic Markov processes indexed by the real line, in which case a Hilbert-Schmidt operator $P_{s,t}$ (the transition operator) is attached to any line segment  $[s,t]$ and the Hilbert space is the state space of the Markov process. In a nutshell, Segal's axioms can be seen as a Markov field indexed by Riemann surfaces.

Before we move to the particular case of Liouville CFT, let us quickly recall the basic objects of the conformal bootstrap axiomatics of \cite{BPZ84}  in the probabilistic setup.  One postulates the existence of  random fields $V_{\Delta}$  indexed by $\Delta\in \R$ and defined on $\Sigma$ so that the {\it correlation functions}
$\langle \prod_{j=1}^m V_{\Delta_{j}}(x_j)\rangle_{\Sigma,g}$ 
exist for arbitrary choices of labels $\Delta_j$ and non-coinciding points $x_j\in\Sigma$. The fields $V_{\Delta}$  are called primary conformal fields and are assumed to satisfy the relations 
\begin{align}\label{ax1}
\text{{\bf (Weyl covariance)}}& & \langle \prod_{j=1}^m V_{\Delta_{j}}(x_j)\rangle_{\Sigma,e^\omega g}  =&e^{c S_{\rm L}^0(\Sigma,g,\omega)-\sum_{j=1}^m\Delta_{j}\omega(x_j)}\langle \prod_{j=1}^m V_{\Delta_{j}}(x_j)\rangle_{\Sigma,g}
\\ 
\text{{\bf (Diffeomorphism invariance)}} & & \langle \prod_{j=1}^m V_{\Delta_{j}}(x_j)\rangle_{\Sigma,\psi^\ast g} =&\langle \prod_{j=1}^m V_{\Delta_{j}}(\psi(x_j))\rangle_{\Sigma,g}\label{ax2}
\end{align}
for all smooth $\omega:\Sigma\to \R$ and smooth diffeomorphisms $\psi:\Sigma\to\Sigma$, where the Liouville functional $S_{\rm L}^0(\Sigma,g,\omega)$ is given by 
\begin{equation}\label{SL0}
S_{\rm L}^0(\Sigma,g,\omega):=\frac{1}{96\pi}\int_{\Sigma}(|d\omega|_g^2+2K_g\omega) {\rm dv}_g
\end{equation}
with $K_g$ the scalar curvature and $ {\rm v}_g$ the volume form on $\Sigma$ determined by the metric $g$. Here $c\in\R$ is the {\it central charge} of the CFT and $\Delta$ are  called {\it conformal weights}. Note that \eqref{ax1} and \eqref{ax2} mean that the correlation functions can be viewed as functions   (more precisely  sections of a line bundle)  on  the {\it moduli space} $\caM_{{\bf g},m}$ of Riemann surfaces with $m$ marked points and having the genus of $\Sigma$ (see Subsection \ref{sub:moduli}). A fundamental role is played by the three point correlation function on the sphere $\hat\C=\C\cup\{\infty\}$. In this case there are no moduli and it is determined up to a constant $C(\Delta_1,\Delta_2,\Delta_3)$ depending on the conformal weights. 
This constant is called the {\it structure constant} of the CFT, see Section \ref{sub:dozzintro}.

In physics, Liouville CFT  was introduced  by Polyakov in his path integral formulation of String Theory \cite{Polyakov81} and it served as a motivation for Belavin, Polyakov and Zamolodchikov in their aforementioned  work on CFT \cite{BPZ84}. It plays a fundamental role in the study of random surfaces, quantum cohomology and many other fields of physics and mathematics. It corresponds to taking the particular action, called Liouville action, defined for  $C^1$ maps $\phi:\Sigma\to\R$  by
\begin{equation}\label{introactionL}
S_\Sigma(g,\phi):= \frac{1}{4\pi}\int_{\Sigma}\big(|d\phi|_g^2+QK_g \phi  + 4\pi \mu e^{\gamma \phi  }\big)\,{\rm dv}_g 
\end{equation}
where the parameters of LCFT are  $\mu>0$,  $\gamma\in (0,2)$ and $Q=\frac{\gamma}{2}+ \frac{2}{\gamma}$.  This theory, with central charge $c_{{\rm L}}=1+6Q^2$, has been extensively studied in theoretical physics. The primary fields of LCFT were conjectured to be given by the exponentials $e^{\alpha\phi(x)}$ with $\alpha\in\C$ and conformal weights $\Delta_{\alpha}=\frac{\alpha}{2}(Q-\frac{\alpha}{2})$. To have $\Delta_{\alpha}\in\R$, one then needs $\alpha\in \{Q+i\R\}\cup \R$.  The fields with $\alpha\in \{Q+i\R\}$ were conjectured in  the physics literature \cite{CurtrightThorn82, Braaten_Curtright_Thorn, GERVAIS1984125} to produce the {\it spectrum} of LCFT, a crucial input in the Belavin-Polyakov-Zamolodchikov axiomatics of CFT (see subsections \ref{subsubsec:spectral} and \ref{sub:virasoro}) whereas the correlation functions of the fields with $\alpha\in\R$ are the ingredients of the Kniznik-Polyakov-Zamolodchikov theory of random surfaces \cite{doi:10.1142/S0217732388000982}.  Finally the so-called DOZZ formula for the structure constants of LCFT was proposed in physics by Dorn-Otto \cite{DornOtto94} and Zamolodchikov-Zamolodchikov \cite{Zamolodchikov96} in the nineties, see Appendix \ref{app:dozz}.

The mathematical resolution of LCFT  can be summarised in the following steps: 

\vskip 1mm

\noindent {\bf Step 1}. Give a probabilistic construction of the path integral \eqref{pathi} for the correlation functions and prove \eqref{ax1} and \eqref{ax2}.

\vskip 1mm

\noindent {\bf Step 2}. Find an explicit formula for the %$3$-point functions $G(0,1,\infty, g_0)$, called \emph
{structure constant}. %, on the Riemann sphere.

\vskip 1mm

\noindent  {\bf Step 3}. Find a Hilbert space $\mc{H}$ for which there is a state/field correspondence, and a representation of two commuting Virasoro algebras $({\bf L}_n)_{n\in \Z}, (\til{{\bf L}}_n)_{n\in \Z}$. Decompose the Hilbert space 
$\mc{H}$ in terms of highest weight representations of the Virasoro algebras, by spectral resolution of the 
Hamiltonian ${\bf H}={\bf L}_0+\tilde{{\bf L}}_0$ of the theory.

\vskip 1mm

\noindent {\bf Step 4}. Compute the correlation functions on all surfaces, implementing the conformal bootstrap method. This involves viewing the correlations as pairings of elements in $\mc{H}$ (or tensor powers of $\mc{H}$) via the state field correspondence, decomposing these pairings according to the eigenstates of ${\bf H}$, and then computing the matrix coefficients of these elements on the eigenbasis. This final step is known as the  \emph{modular bootstrap}, a concept initially proposed in physics in \cite{SONODA1988} (see also \cite{10.1007/978-90-481-2810-5_46}). It also allows the rigorous construction of the \emph{conformal blocks}, which are supposed to be holomorphic functions on the Teichm\"uller space.

\vskip 1mm

For Step 1, the path integral construction was carried out by David and the last three authors in \cite{DKRV16} in the case of genus ${\bf g}=0$, or in \cite{DRV16_tori} for   ${\bf g}=1$, and then extended  to arbitrary genus   in \cite{GRVIHES}.  The rigorous definition of the path integral was given  in terms of the Gaussian Free Field (GFF) on $(\Sigma,g)$ which gives rise to
a cylinder measure  $\nu$ on the Sobolev space $H^s(\Sigma)$ of order  $s<0$, see Section \ref{sec:lcft}. The GFF gives a mathematical sense to the formal measure  $e^{-\frac{1}{4\pi }\int_\Sigma |d \phi|_g^2 d{\rm v}_g}D\phi$. The LCFT correlation functions are then given by 
\begin{align}\label{corre}
\langle \prod_{i=1}^m V_{\Delta_{i}}(x_i)\rangle_{\Sigma,g}:=\int_{H^{s}(\Sigma)} \prod_{i=1}^m e^{\alpha_i\phi(x_i)}e^{-\frac{1}{4\pi}\int_{\Sigma}(QK_g \phi  + 4\pi \mu e^{\gamma \phi  })\,{\rm dv}_g}d\nu(\phi)
\end{align}
for $\alpha_i\in (-\infty,Q)$ and $\Delta_i=\frac{\alpha_i}{2}(Q-\frac{\alpha_i}{2})$ under the condition $\sum_{i=1}^m\alpha_i>\chi(\Sigma)Q$ where $\chi(\Sigma)$ is the Euler characteristic of $\Sigma$. The exponentials of the generalised function $\phi$ on the r.h.s. are defined through  limits of  regularized and renormalised expressions. In particular they satisfy the axioms \eqref{ax1}, \eqref{ax2}. Because of the definition \eqref{corre} and the correspondence $\alpha_i\leftrightarrow \Delta_i$, it is customary in LCFT to label the primary fields with the weights $\alpha_i$ instead of the conformal weights $\Delta_i$, i.e. to write $\langle \prod_{i=1}^m V_{\alpha_{i}}(x_i)\rangle_{\Sigma,g}$ for correlation functions and we will adopt this convention from now on. 

Specializing to $\Sigma=\hat\C$ and $m=3$, the 3-point correlation functions in any metric $g$ on $\hat \C$ conformal to the canonical sphere metric can be decomposed as 
\begin{align*}%\nonumber%
 \langle  &V_{\alpha_1}(z_1)  V_{\alpha_2}(z_2) V_{\alpha_3}(z_3)  \rangle_{\hat\C,g}\\
  =&|z_1-z_3|^{2(\Delta_{\alpha_2}-\Delta_{\alpha_1}-\Delta_{\alpha_3})}|z_2-z_3|^{2(\Delta_{\alpha_1}-\Delta_{\alpha_2}-\Delta_{\alpha_3})}|z_1-z_2|^{2(\Delta_{\alpha_3}-\Delta_{\alpha_1}-\Delta_{\alpha_2})}
\Big(\prod_{i=1}^3g(z_i)^{-\Delta_{\alpha_i}}\Big) \\
&\times C(g)C_{\gamma,\mu}^{{\rm DOZZ}} (\alpha_1,\alpha_2,\alpha_3 )
 \end{align*} 
where $C(g)$ is an explicit constant, which only depends on the metric $g$, and the constant $C_{\gamma,\mu}^{{\rm DOZZ}} (\alpha_1,\alpha_2,\alpha_3 )$ is called the structure constant (it depends neither  on $g$ nor on $z_1,z_2,z_3$, see Section \ref{sec:lcft}).
Step 2 has then been solved by the  last three authors in \cite{KRV_DOZZ}, where they proved the so called DOZZ formula for the structure constant (see appendix \ref{app:dozz} for a reminder). 

Step 3 and 4 were established in a particular case, namely for $4$-point correlation on the sphere, in our previous work  \cite{GKRV20_bootstrap}. In particular, the Hilbert space is constructed using the Gaussian Free Field on the unit circle $\T$, the Virasoro operators ${\bf L}_n,\til{\bf L}_n$ are defined using an intertwining method with the free field, and the spectral resolution of ${\bf H}$ is done by means of scattering theory.  The importance of understanding the spectral analysis of the Hamiltonian of LCFT was stressed in physics by Teschner in \cite{Teschner_revisited}.

\begin{figure}
\centering
\begin{tikzpicture}
\node[inner sep=10pt] (F1) at (-5.2,0){\includegraphics[scale=0.6]{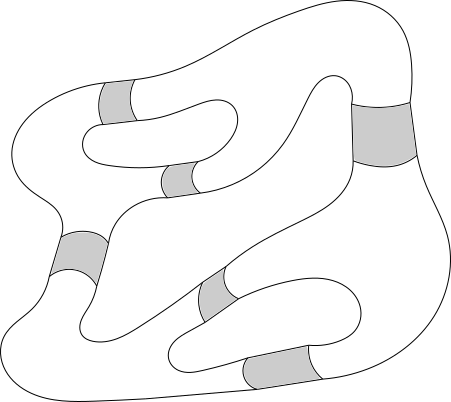}
\hspace{0.5cm}};
\node (F) at (-7.5,0){$\mc{P}_1$};
\node (F) at (-8,-2.5){$\mc{P}_2$};
\node (F) at (-3,-0.7){$\mc{P}_3$};
\node (F) at (-3,2.7){$\mc{P}_4$};
\node (F) at (-1.9,1){$\mathbb{A}_{q_1}$};
\node (F) at (-7,2.3){$\mathbb{A}_{q_2}$};
\node (F) at (-8.3,-0.8){$\mathbb{A}_{q_3}$};
\node (F) at (-5.8,-0.2){$\mathbb{A}_{q_4}$};
\node (F) at (-5.8,-1){$\mathbb{A}_{q_5}$};
\node (F) at (-4.3,-3.4){$\mathbb{A}_{q_6}$};
\end{tikzpicture}
  \caption{The plumbed surfaces $\Sigma_{\bf q}$ with four pairs of pants $\mc{P}_1,\dots,\mc{P}_4$ and six annuli $\mathbb{A}_{q_1},\dots,\mathbb{A}_{q_6}$ of modulus ${\bf q}=(q_1,\dots,q_6)$ between the pairs of pants.}\label{figureplumbing}
\end{figure}

In the present paper, we carry out  Step 4 in full generality and compute all correlation functions on all surfaces in terms of the structure constants and the conformal blocks. We stress that the method developed in \cite{GKRV20_bootstrap} uses crucially the reflection positivity and the symmetry of the Riemann sphere under $z\mapsto 1/\bar{z}$ to realise the conformal bootstrap for the spherical $4$-point function. Even the computation of the $5$-point function cannot directly be obtained by this method. To achieve the full conformal bootstrap, 
we use an approach introduced by Segal, which we implement for the first time in the probabilistic setting of conformal field theory.
The Hilbert space is $\mc{H}=L^2(H^{s}(\T),\mu_0)$ for some measure $\mu_0$, with $H^{s}(\T)$ the Sobolev space of fixed order $s<0$ on the circle $\T$. To each surface with parametrised boundary and   marked points, we construct an operator on $\mc{H}$, called \emph{Segal amplitude}, via its integral kernel defined as a conditional expectation. Then we show that gluing surfaces amounts to composing these amplitudes. This method, based on ``cutting'' the path integral into pieces, is sufficiently rich and flexible to construct the conformal blocks in full generality and to show their holomorphic dependence in the moduli parameters, as well as to perform the modular bootstrap. This approach also produces a rigorous way to establish the state/field correspondence and the so-called Operator Product Expansion, and gives a geometric way to represent the conformal symmetries of the model: indeed, it has allowed us, in subsequent works \cite{BGKRV,BGKR1}, to produce a probabilistic construction with geometric flavour of the Virasoro algebra in $\mc{H}$ using the Segal amplitudes, which will be fundamental in constructing a projective representation of the mapping class group in the space of conformal blocks.

To state our main result on the LCFT correlation functions we use particular holomorphic local coordinates, called {\it plumbing coordinates}, on the moduli space  $\mc{M}_{{\bf g},m}$ of Riemann surfaces  $(\Sigma,g, {\bf x})$ of genus ${\bf g}$ with $m$ marked points ${\bf x}=(x_1,\dots , x_m)$  for $2{\bf g}-2+m>0$ (see Section \ref{sub:plumbing}). These local coordinates are not defined globally on Teichm\"uller or moduli space but are particularly well adapted to the Segal approach to study the holomorphy of the conformal blocks. 
In these coordinates 
${\bf q}=(q_1,\dots,q_{3{\bf g}-3+m})\in \C^{3{\bf g}-3+m}$ with $|q_j|\leq 1$, 
the Riemann surfaces in $\mc{M}_{{\bf g},m}$ can be represented by cutting $\Sigma$ into elementary Riemann surfaces with boundary (called complex building blocks), namely 
pairs of pants, annuli with $1$ marked points and disks with two marked points, and then gluing some annuli $\mathbb{A}_{q_j}:=\{z\in \C\;|\, |z|\in [|q_j|,1]\}$ with a twist angle ${\rm arg}(q_i)$ between the complex building blocks  with the choice of flat metric $|dz|^2/|z|^2$ on the annuli (see Figure \ref{figureplumbing} for a case with $m=0$ and ${\bf g}=3$). 
The surfaces with marked points that we obtain are denoted 
by  $(\Sigma_{\bf q},g_{\bf q},{\bf x})$. Our main result can be outlined as follows:
\begin{theorem}\label{mainth} 
For Riemann surfaces of genus ${\bf g}$ and $m$ marked points, the correlation functions defined probabilistically by \eqref{corre} take  the following 
form  in the plumbing coordinates ${\bf q}=(q_1,\dots,q_{3{\bf g}-3+m})$:
\[
\frac{2^{\frac{3{\bf g}-3+m}{2}}\prod_{j=1}^{3{\bf g}-3+m} |q_j|^{-\frac{1+6Q^2}{12}} }{(2\pi)^{6{\bf g}-6+2m-1}}\int_{\R_+^{3{\bf g}-3+m}}C_{\Sigma,g}(\boldsymbol{\alpha},{\bf p})\boldsymbol{\rho}(\boldsymbol{\alpha},{\bf p})\prod_{j=1}^{3{\bf g}-3+m} |q_j|^{2\Delta_{Q+ip_j}} 
\Big|\tilde{\mc{F}}_{\bf p}(\boldsymbol{\alpha},{\bf q})\Big|^2\dd{\bf p},\]
where the  $\tilde{\caF}_{\bf p}(\boldsymbol{\alpha},{\bf q})$ is the conformal block normalised\footnote{The reason of the tilde notation for the normalised blocks is that in subsequent work \cite{BGKR2} we study the unnormalised blocks, that we denote by 
$\caF_{\bf p}(\boldsymbol{\alpha},{\bf q})$. As functions of $q_j$, they differ essentially by a factor $\prod_{j}\exp(\Delta_{Q+ip_j}\log q_j)$ and $\caF_{\bf p}(\boldsymbol{\alpha},{\bf q})$ needs to be considered on Teichm\"uller space where $\log q_j$ is well defined.} by 
$\tilde{\caF}_{\bf p}(\boldsymbol{\alpha},0)=1$, it is holomorphic in the moduli parameter ${\bf q}$ and defined for almost all ${\bf p}=(p_1,\dots,p_{3{\bf g}-3+m})$. 
The function $\rho({\bf p},\boldsymbol{\alpha})$ is a product of DOZZ structure constants $C^{\rm DOZZ}_{\gamma,\mu}(\tilde\alpha_j,\tilde\alpha_k,\tilde\alpha_\ell)$ with $\tilde\alpha_j,\tilde\alpha_k,\tilde\alpha_\ell$ belonging to the set $\{\alpha_j\}_{j=1}^m\cup \{Q\pm ip_j\}_{j=1}^{3{\bf g}-3+m}$ and the constant $C_{\Sigma,g}(\boldsymbol{\alpha},{\bf p})$ is an explicit constant depending on the choice of metric $g$ but not on ${\bf q}$.   
\end{theorem}
For a precise statement see Theorem \ref{th:fullexpressioncorrel}. 
The special cases of the complex tori or the Riemann sphere are also treated in details in Section \ref{sec:special} using the marked points as moduli parameters. As an example, in Theorem \ref{Torus1pt}, we give the first mathematical proof of the well-known physics formula   for the 1-point correlation function on the flat torus $\T^2_\tau:=\C/(2\pi\Z+2\pi\tau\Z)$ (with ${\rm Im}(\tau)>0)$: 
\[\cjg V_{\alpha_1}(0)\cjd_{\T^2_\tau}=\frac{|q|^{-\frac{1+6Q^2}{12}}}{2e}
\int_0^\infty  |q|^{2\Delta_{Q+ip}} C^{\rm DOZZ}_{\gamma,\mu}(Q+ip,\alpha_1,Q-ip)|\tilde{\mc{F}}_{p}(\alpha_1,q)|^2  \dd p\]
where $q=e^{2i\pi\tau}$. 
The $m$-point function on $\T^2_\tau$, for $x_1=0$ and ${\rm Im}(x_j)<{\rm Im}(x_{j+1})<2\pi {\rm Im}(\tau)$, is also shown (in Theorem \ref{th:kpointtorus}) to be given by the $L^1$-integral
\[\begin{split}
\cjg \prod_{j=1}^mV_{\alpha_j}(x_j)\cjd_{\T_\tau}=\frac{\prod_{j=1}^m|q_j|^{-\frac{1+6Q^2}{12}}}{2^{2m-1}\pi^{m-1}e^m} 
\int_{\R_+^m} \prod_{j=1}^m|q_j|^{2\Delta_{Q+ip_j}}\Big(\prod_{j=1}^m C^{\rm DOZZ}_{\gamma,\mu}(Q+ip_j,\alpha_j,Q-ip_{j-1})\Big)\Big|_{p_0=p_m}|\tilde{\mc{F}}_{\bf p}(\boldsymbol{\alpha},{\bf q})|^2\dd{\bf p}.
\end{split}\]
with $q_j:=e^{i(x_{j+1}-x_j)}$ for $j=1,\dots,m-1$,  $q_m:=qe^{-ix_m}$, ${\bf p}=(p_1,\dots,p_m)$, $\boldsymbol{\alpha}=(\alpha_1,\dots,\alpha_m)\in (0,Q)^m$.

The holomorphic factorisation of the correlation functions of a CFT i.e. the fact that the dependence of the moduli comes through the square of the absolute value of a holomorphic function dates back to string theory in the early 70's. In that setup and for general Riemann surfaces it was explicitly conjectured in physics in 1985 by Belavin and Knizhnik \cite{Belavin_knizhnik} and even considered as a way (if not an axiom) to construct CFTs from the conformal blocks. A related approach was also developed by Friedan and Shenker \cite{FriedanShenker87}. More recently the connection to the plumbing construction has been stressed by Teschner in several physics papers, see e.g. \cite{10.1007/978-90-481-2810-5_46}.\\

The conformal blocks are fundamental holomorphic functions of the moduli parameters that have been introduced in physics 
as building blocks for the correlation functions of CFT. These functions, constructed from the representation theory of Virasoro algebra, are of complex geometric and algebraic nature. Unlike the structure constants  that are model-dependent,  the conformal blocks are universal, in the sense that they only depend on the central charge and the conformal weights.
In our case where the central charge is $c_{\rm L}=1+6Q^2>25$, their construction was mathematically unknown.  A consequence of our work is a rigorous construction for all surfaces and a proof of their holomorphy. 
  They are defined as series in the plumbing  parameters $q_j$
\begin{align}\label{blocks}
\tilde{\caF}_{\bf p}(\boldsymbol{\alpha},{\bf q})=\sum_{(n_1,\dots,n_{L})\in \N^L}q_1^{n_1}\dots q_{L}^{n_L}W_{{\bf p},\boldsymbol{\alpha}}(n_1,\dots,n_L)
\end{align}
where $L:=3{\bf g}-3+m$, the coefficients $W_{{\bf p},\boldsymbol{\alpha}}$  only depend on the commutation relations of the Virasoro algebra and on the conformal weights associated to ${\bf p},\boldsymbol{\alpha}$. The  main difficulty to define the blocks is the convergence of the series because the coefficients are not   explicit or tractable enough   to control  their  growth. 

Our work is related to the approach discussed in physics by Friedan and Shenker \cite{FriedanShenker87} to construct CFTs: roughly speaking the partition function of a CFT is viewed in  \cite{FriedanShenker87}  as the squared norm of a holomorphic section of a projective holomorphic vector bundle (the bundle of conformal blocks) on the moduli space  equipped with a projectively flat Hermitian metric. However, as mentionned in \cite{FriedanShenker87}, for CFTs with central charge $c>1$ (such as Liouville), the space of blocks is infinite dimensional, which complicates the problem significantly. This approach is also discussed in physics by Teschner and Vartanov \cite{Teschner_Vartanov} for Liouville CFT, 
but the mathematical implementation remains conjectural. 
Our normalised conformal blocks depend on the choice of plumbing coordinates (thus on a choice of parametrised circles cutting the surface), and understanding how these conformal blocks change under change of plumbing coordinates is intricate but fundamental in order to show that the conformal blocks can be extended as global holomorphic sections of a line bundle over the Teichm\"uller space.  Using the tools developed here and the construction of Virasoro generators using Segal amplitudes, we have proved with Baverez in a series of papers \cite{BGKRV,BGKR1,BGKR2} the global structure of the (unnormalised) conformal blocks  conjectured in physics by \cite{FriedanShenker87,Teschner_Vartanov}, which is crucial for the projective representation of the mapping class group (see below).

To conclude, we list the key contributions of our paper for Liouville CFT:
\begin{itemize}
\item A probabilistic construction of the Segal amplitudes associated to each surface with parametrised  boundary and marked points, and a proof of Segal axioms in that context 
(conformal and diffeomorphism invariance, gluing of amplitudes) -- in Sections \ref{probamp} and \ref{sec:gluing},
\item A rigorous construction of the normalised conformal blocks  for all surfaces using Segal amplitudes, and a proof of their holomorphy with respect to the moduli parameters (plumbing coordinates) by proving Ward identities for pairs of pants -- in the sections \ref{Section:blocks}, \ref{sec:ward} and \ref{sec:computingamplitudes},
\item A rigorous implementation of the conformal bootstrap for all surfaces with marked points, to express the correlation functions in terms of conformal blocks and structure constants -- in Section \ref{Section:blocks}.
\end{itemize}
Furthermore, we  emphasize several applications based on or using this work: 
\begin{itemize}
\item Based on this work and the subsequent conformal bootstrap for boundary Liouville theory \cite{GRW1,GRW2}, Ghosal, Remy, Sun and Sun \cite{GRSS2} have recently been able to extend the definition of the conformal blocks to other values of the parameters (including $p_i$ complex) for the $4$-point sphere and $1$-point torus, and produce a proof that the transition kernels for conformal blocks are given by Ponsot-Teschner formula.
\item An explicit representation of the Virasoro algebra on $\mc{H}$ as generators of Markov semigroups and as Segal amplitudes of deformed annuli (called Segal semigroup) in joint work with Baverez \cite{BGKRV,BGKR1}.
\item A description of the conformal blocks as global holomorphic sections of a line bundle over Teichm\"uller space that produce a quantization of Teichm\"uller space and a projective unitary representation of the mapping class group in the space of conformal blocks \cite{BGKR2}.
\end{itemize}
 
\subsection{Overview of the proof of Theorem \ref{mainth}}

\begin{figure}[h] 
\centering
 \begin{tikzpicture}[xscale=0.6,yscale=0.6,every node/.style={scale=0.75},scale=0.75]  
 \tikzstyle{noeud}=[minimum width=2cm,text width=5cm,minimum height=0.8cm,rectangle,rounded corners=5pt,draw,fill=yellow!50,text=black,font=\bfseries,text centered,text badly centered]
 \tikzstyle{core}=[minimum width=2cm,text width=4cm,minimum height=0.8cm,rectangle,rounded corners=5pt,draw,fill=lime!50,text=black,font=\bfseries,text centered,text badly centered]
  \tikzstyle{main}=[minimum width=2cm,text width=4cm,minimum height=0.8cm,rectangle,rounded corners=5pt,draw,fill=teal!30,text=black,font=\bfseries,text centered,text badly centered]
    \tikzstyle{exa}=[minimum width=2cm,text width=4cm,minimum height=0.8cm,rectangle,rounded corners=5pt,draw,fill=cyan!30,text=black,font=\bfseries,text centered,text badly centered]
  \tikzstyle{sommet}=[circle,draw,fill=black]
\tikzstyle{pas}=[thick]  
\tikzstyle{fleche}=[->,>= stealth,thick]
\node[noeud] (PBG) at (8,10) {Background  sec. \ref{sec:geometric}: \\ Riemann surfaces\\Cutting and gluing surfaces};
\node[noeud] (PBT) at (21,10) {Background sec. \ref{section:traces}:\\ Partial traces for Hilbert-Schmidt operators};
\node[core] (CA) at (3,7) {Constructing amplitudes, sec. \ref{probamp}};
\node[core] (GA) at (12,7) {Gluing amplitudes, sec. \ref{sec:gluing}};
\node[core] (CSGA) at (3,4) {Semigroup of annuli sec. \ref{sub:hamiltonian}};
\node[core] (CH) at (3,1) { Hamiltonian and spectral resolution, sec. \ref{subsubsec:spectral} };
\node[core] (CCTP) at (12,3) {Decomposing correlation functions into pant amplitudes, Prop. \ref{prop:decompositionAmplitude}};
\node[core] (CDB) at (6,-3) {Defining the conformal blocks sec. \ref{sub:cfb}};
\node[core] (CW) at (21,1) {Ward identities\\ Prop. \ref{propward}};
\node[core] (CCPA) at (16,-3) {Computing the pant amplitudes, Th. \ref{pantDOZZ}};
\node[main] (MTH) at (8,-7) { Conformal bootstrap for Liouville CFT, Theorem \ref{th:fullexpressioncorrel}};
\node[exa] (Ex) at (16,-7) {Examples  sec. \ref{sec:special}\\\textcolor{black}{Torus and sphere} };
%\draw[fleche] (GA) |- (6,6) -| (CSGA);
\draw[fleche] (PBG) -- (8,8.5) -| (CA);
\draw[fleche] (PBG) -- (8,8.5) -| (GA);
\draw[fleche] (CDB) |- (8,-5) -| (MTH);
\draw[fleche] (CA) -- (GA);
\draw[fleche] (CCPA) |- (8,-5) -| (MTH);
\draw[fleche] (MTH) -- (Ex);
\draw[fleche] (CA) -- (CSGA);
\draw[fleche] (CSGA) -- (CH);
\draw[fleche] (GA) -- (CCTP);
\draw[fleche] (CCTP) |- (7,-1) -| (CDB);
\draw[fleche] (CH) |- (7,-1) -| (CDB);
\draw[fleche] (CW) |-  (CCPA);
\draw[fleche] (CW) |-  (CCPA);
\draw[fleche] (PBT) |- (8,8.5) -|  (GA);
\draw[fleche] (GA)  -| (CCPA);
to[jump]
\end{tikzpicture}
\caption{Diagram of the proof of Theorem \ref{th:fullexpressioncorrel}} 
\label{fig:diagram}
\end{figure}

 In this section we give an outline of the probabilistic verification of Segal's axioms and the argument  leading to Theorem \ref{mainth} (or Theorem  \ref{th:fullexpressioncorrel}), summarised in the flowchart Fig. \ref{fig:diagram}. 
To  give a probabilistic definition of the integral kernels \eqref{ampli} that enter the Segal axioms, 
we  consider oriented surfaces $\Sigma$, equipped with a complex structure, with  $m\geq 0$ marked points $x_i$ in the interior of $\Sigma$ and boundary $\partial\Sigma=\bigcup_{i=1}^b \caC_i$ where each $\caC_i$ is parametrised by a real analytic map $\zeta_i:\T\to\caC_i$ from the standard circle $\T$. If the orientation of $\zeta_i(\T)$ agrees with that of $\caC_i$ inherited from $\Sigma$ we call $\caC_i$ outgoing and otherwise incoming. Let $I_{\rm in}$ ($I_{\rm out}$) index the incoming (resp. outgoing) circles.  Furthermore, we equip $\Sigma$ with a smooth Riemannian metric compatible with the complex structure, such that a neighborhood of  $\mc{C}_i$ is isometric to the flat annulus $(\{z\in \C\,|\, |z|\in (\delta,1]\},|dz|^2/|z|^2)$ for some $\delta>0$; such metrics are called \emph{admissible} and behave nicely under sewing of Riemann surfaces. Finally we specify $m$ real numbers $\alpha_i<Q$. Given the data  $(\Sigma, g,{\bf x},\boldsymbol{\alpha},\boldsymbol{\zeta})$ with ${\bf x}=(x_1,\dots,x_m)$,  $\boldsymbol{\alpha}=(\alpha_1,\dots,\alpha_m)$ and $\boldsymbol{\zeta}=(\zeta_1,\dots,\zeta_b)$ we then construct a Hilbert-Schmidt operator  
$$\caA_{\Sigma,g,\bf x,\boldsymbol{\alpha},\boldsymbol{\zeta}}:\bigotimes_{i\in I_{\rm in}}\caH\to \bigotimes_{i\in I_{\rm out}}\caH $$ 
 where the Hilbert space is 
 \begin{align*}
 \caH=L^2(H^{s}(\T),\dd\mu_0)
\end{align*}
and the Sobolev space $H^{s}(\T)$ with $s<0$ is equipped with a cylinder sigma algebra and a Gaussian cylinder measure $\mu_0$ coming from the restriction of the Gaussian Free Field to $\T$ (see the exact definition in Subsection \ref{sub:hilbert}). The operator $\caA_{\Sigma,g,\bf x, \boldsymbol{\alpha},\boldsymbol{\zeta}}$ has an integral kernel, denoted by  $\caA_{\Sigma,g,\bf x,\boldsymbol{\alpha},\boldsymbol{\zeta}}(\vhi_{\rm in},\vhi_{\rm out})$ and called an amplitude, 
 so that for $\psi\in\otimes_{i\in I_{\rm in}}\caH$ we have 
\begin{align*}
(\caA_{\Sigma,g,\bf x,\boldsymbol{\alpha},\boldsymbol{\zeta}}\psi)(\vhi_{\rm out})=\int \caA_{\Sigma,g,\bf x,\boldsymbol{\alpha},\boldsymbol{\zeta}}(\vhi_{\rm in},\vhi_{\rm out})\psi(\vhi_{\rm in})\prod_{i\in I_{\rm in}}\dd\mu_0(\varphi_i).
\end{align*}
The probabilistic definition of these amplitudes is a rigorous version of the path integral formula  \eqref{ampli}:
\begin{align}\label{amplitudeintro}
 \caA_{\Sigma,g,\bf x,\boldsymbol{\alpha},\boldsymbol{\zeta}}(\vhi_{\rm in},\vhi_{\rm out}) :=\caA^0_{\Sigma,g,\boldsymbol{\zeta}}(\vhi_{\rm in},\vhi_{\rm out}) 
 \E\big[ \prod_{j=1}^me^{\alpha_j\phi_g(x_j)}
 e^{-\frac{1}{4\pi}\int_{\Sigma}(QK_g \phi_g  + 4\pi \mu e^{\gamma \phi_g  })\,{\rm dv}_g}\big]
\end{align}
where the expectation $\E$ is over 
  $X_{g,D}$, the Gaussian Free Field on $(\Sigma,g)$ with Dirichlet boundary condition on $\partial\Sigma$, $\phi_g=P\vhi+X_{g,D}$ where $P\vhi$ is the harmonic extension of the boundary fields $(\vhi_{\rm in},\vhi_{\rm out})$ to $\Sigma$ and $\caA^0_{\Sigma,g,\boldsymbol{\zeta}}(\vhi_{\rm in},\vhi_{\rm out}) $ is the free field amplitude defined in terms of the Dirichlet-to-Neumann operator of $\Sigma$ and a $\zeta$-function renormalised determinant of the Dirichlet Laplacian, see Definition \ref{def:amp}. A heuristic explanation for this definition is given in the beginning of Section \ref{sec:amplitudes_Segal}. We show that the amplitudes satisfy versions of the diffeomorphism and Weyl axioms \eqref{ax1} and  \eqref{ax2}, see Proposition \ref{Weyl}.
  
 Let now $(\Sigma_i,g_i,{\bf x}_i,\boldsymbol{\alpha}_i,\boldsymbol{\zeta}_i)$, $i=1,2$ be two surfaces as above. Pick $k$ outgoing  analytically parametrised boundary circles from $\partial\Sigma_1$ and $k$ incoming parametrised circles from $\partial\Sigma_2$. We can then glue the two surfaces by identifying $\pl \Sigma_1$ with $\pl \Sigma_2$ using their parametrisations, and obtain a new Riemann surface  $(\Sigma,g,{\bf x},\boldsymbol{\alpha},\boldsymbol{\zeta})$ with marked points (see  Subsection \ref{Sec:Riemann surfaces with marked points} for details).  Our main result on these amplitudes is the Segal gluing axiom (see Propositions \ref{glueampli} and \ref{selfglueampli}), which can be seen as a generalisation of the Markov property to stochastic processes indexed by surfaces instead of times:
 
  \begin{theorem}\label{gluingthm}
The probabilistic amplitudes satisfy the composition law
 \begin{align*}
C  \caA_{\Sigma_2,g_2,{\bf x}_2,\boldsymbol{\alpha}_2,\boldsymbol{\zeta}_2}\, \caA_{\Sigma_1,g_1,{\bf x}_1,\boldsymbol{\alpha}_1,\boldsymbol{\zeta}_1}= \caA_{\Sigma,g,\bf x,\boldsymbol{\alpha},\boldsymbol{\zeta}}.
\end{align*}
where, for $\partial\Sigma=\emptyset$,  $\caA_{\Sigma,g,\bf x,\boldsymbol{\alpha},\emptyset}$ is defined to be the correlation function \eqref{corre}. The constant $C$ is given by $C= \frac{1}{(\sqrt{2} \pi)^{k}}$ if $\partial\Sigma \not =\emptyset$ and $C=\frac{\sqrt{2}}{(\sqrt{2} \pi)^{k-1}} $ if $\partial\Sigma =\emptyset$.
\end{theorem}
  The probabilistic definition of the Segal amplitude is guided by the analysis of the free field and its Markov property. The proof of the gluing property is also fundamentally a free field property, but since we work in a massless setting with a zero mode, the potential plays an important role for convergence purpose. A treatment of Segal's axioms in a probabilistic context appeared earlier in a work by Pickrell \cite{Pickrell} for the $P(\phi)_2$ QFT. It was done in the massive case and not in a CFT context as originally designed by Segal, where the zero mode problem needs to be handled. Therefore we do not rely on  \cite{Pickrell} and write a self-contained and detailed proof in our massless case, including the analysis of the Dirichlet-to-Neumann maps that come into play.

This gluing property allows us to construct the correlation functions of LCFT  by composition of amplitudes of basic building blocks: (a) annuli with zero or one marked points,  (b) discs with one or two marked points, (c) pairs of pants (i.e. sphere with three discs removed).  The amplitudes of these building blocks give rise to the fundamental data of LCFT: (a) the \emph{Hamiltonian} and the {\it spectrum} of LCFT  (b)  {\it highest weight states} of the Virasoro algebra and (c) the {\it structure constants} of LCFT as we   explain now:
 
 \vskip 2mm
 
 \noindent (a) {\it Hamiltonian and spectrum:}  they are both encoded in the annulus amplitudes with no marked points. Indeed, by Weyl covariance of the amplitude, it suffices to consider the annuli $\mathbb{A}_q=\{z\in \C\,|\, |z|\in [|q|,1]\}$ with $q\in\D$ (with $\D\subset \C$ the unit disk) and boundaries parametrised by $z\in \T\mapsto z$ and  $z\in \T\mapsto qz$ and equipped with the dilation invariant metric $g=|z|^{-2}|dz|^2$. Denoting the corresponding amplitudes by
  $\caA_q$ 
  they form a semigroup under gluing:
 \begin{align*}
\frac{\caA_q}{ \sqrt{2}\pi } \frac{\caA_{q'}}{ \sqrt{2}\pi }= \frac{\caA_{qq'}}{ \sqrt{2}\pi}.
\end{align*} 
We prove in Section \ref{section:semigroup} that this semigroup matches (up to constant) the semigroup studied in \cite{GKRV20_bootstrap} and, writing $q=e^{-t+i\theta}$, it can be written as (see Proposition \ref{prop:annulussimple})
\begin{align*}%\label{}
\caA_q=\sqrt{2}\pi e^{c_{\rm L}\frac{t}{12}}e^{-t\mathbf{H}}e^{i\theta\mathbf{\Pi}}
\end{align*}
where the {\it LCFT Hamiltonian} $\mathbf{H}$ is a positive (unbounded) self-adjoint operator on the Hilbert space $\caH$ and $\mathbf{\Pi}$ is a  self-adjoint operator representing the rotations of $\T$ in $\caH$. The operators ${\bf H}$ and $\mathbf{\Pi}$ commute and, in \cite{GKRV20_bootstrap}, we constructed a joint spectral resolution for them\footnote{Only $\mathbf{H}$ was studied in  detail in \cite{GKRV20_bootstrap}, for $\mathbf{\Pi}$ see Subsection \ref{sec:momenta}.}. The operators $\mathbf{H}$ and $\mathbf{\Pi}$ have a family of generalised eigenfunctions denoted by $\Psi_{\alpha,\nu,\tilde\nu}$, where the indices are as follows: $\alpha\in {\rm Spec}:=Q+i\R_+$ is a {\it spectral parameter} that belongs to the so-called \emph{spectrum} of LCFT,  while $\nu,\tilde \nu$ are two Young diagrams, i.e. finite non-increasing sequences of positive integers $\nu=(\nu(1),\dots,\nu(k))$ where $\nu(i)\geq \nu(i+1)$. We have in particular the relation
\begin{align}\label{eigen}
 \caA_q\Psi_{\alpha,\nu,\tilde\nu}=\sqrt{2}\pi  |q|^{2\Delta_\alpha-\frac{c_{\rm L}}{12}}q^{|\nu|}\bar q^{ |\tilde\nu|}\Psi_{\alpha,\nu,\tilde\nu}.
\end{align}
Here $|\nu|=\sum_i\nu(i)$ is the length of the Young  diagram. Furthermore we have the completeness relation for $u_1,u_2\in \mc{H}$ (the inner product of which is denoted $\langle\cdot,\cdot\cjd_{\mc{H}}$)
\begin{align}\label{completeness}
\cjg u_1,u_2\cjd_{\mc{H}}=\frac{1}{2\pi}\sum_{\nu,\nu',\tilde\nu,\tilde\nu'}\int_{\R_+}\cjg u_1,\Psi_{Q+ip,\nu,\tilde\nu}\cjd_{\mc{H}} \cjg \Psi_{Q+ip,\nu',\tilde\nu'},u_2\cjd_{\mc{H}} F^{-1}_{Q+ip}(\nu,\nu')F^{-1}_{Q+ip}(\tilde\nu,\tilde\nu')\dd p
\end{align}
where the matrix $F^{-1}_\alpha(\nu,\nu')$ is the inverse of a positive definite quadratic form called the Schapovalov form of the Verma module of highest weight $\Delta_{\alpha}$  of the Virasoro algebra with central charge $1+6Q^2$. The Schapovalov coefficient $F_\alpha(\nu,\nu')$ is a Gram-Schmidt coefficient coming from the fact that the generalised eigenstates $\Psi_{\alpha,\nu,\tilde\nu}$ are not orthogonal. This completeness relation can be interpreted as a Plancherel type formula on the Hilbert space $\mc{H}$.

 \vskip 1mm

 \noindent (b)   {\it Highest weight states of the Virasoro algebra:} consider the amplitude   $  \mc{A}_\alpha:=\caA_{\D,g_{\D},0,\alpha,\zeta}$ of the unit disk $\D$ with a marked point at $0$ with weight $\alpha\in (-\infty,Q)$, with $g_\D=e^h|dz|^2$  an admissible metric for some smooth $h:\D\to\R$ and %admissible metric for some $h\in C^\infty$
 $\zeta$ the standard parametrisation of $\partial\D$. It is a measurable function of the boundary field $\varphi\in H^s(\T)$, $s<0$. We show in Section \ref{eigenana} that, up to multiplicative constant, $\mc{A}_\alpha$ is the analytic continuation of $\alpha\mapsto \Psi_{\alpha}$ from $\alpha\in Q+i\R$ to $\alpha\in (-\infty,Q)$; in particular, the disk amplitude  is a generalised eigenfunction of ${\bf H}$, corresponding to a highest weight state. Similarly, we show that the generalised eigenstates $\Psi_{\alpha,\nu,\tilde\nu}$ for $\nu,\tilde\nu$ non empty Young diagrams, dubbed descendant states, can be analytically continued from the spectrum line $\alpha\in {\rm Spec}:=Q+i\R_+$ to a subset of the real line $\alpha<C$ (see Prop. \ref{defprop:desc}) and, in this range of parameters $\alpha$, the descendant $\Psi_{\alpha,\nu,\tilde\nu}$ admits a probabilistic representation in terms of disk amplitudes with a marked point at $0$ with weight $\alpha$ and further insertions of a field $T$ called the Stress-Energy-Tensor (SET), which we expand in more details below.  These SET insertions encode the action of the Virasoro generators on the highest weight states.

\vskip 1mm
 
 \noindent (c)  {\it Structure constants:} consider a sphere $\hat\C$ with $b\in\{1,2,3\}$ analytic disks removed, and with $3-b$ marked points.  The amplitudes of these surfaces are determined by the LCFT {\it structure constants}. 
To explain this connection, consider the case $b=3$ when the surface is a pair of pants $\caP$.  We equip it with an admissible metric $g_{\caP}$ and boundary parametrisations $\boldsymbol{\zeta}=(\zeta_1,\zeta_2,\zeta_3)$ and consider the case where all the boundary circles are incoming so the pant amplitude is a map $\caA_{\caP,g_{\caP},\boldsymbol{\zeta}}:\otimes_{j=1}^3\caH\to\C$. 
In view of the completeness relation \eqref{completeness}, it will be crucial to evaluate it at the state $\otimes_{j=1}^3\Psi_{\alpha_j,\nu_j,\tilde\nu_j}$. The matrix coefficients of $\caA_{\caP,g_{\caP},\boldsymbol{\zeta}}$ on this basis 
produce a function of $({\bf p},\boldsymbol{\nu})$ where ${\bf p}=(p_1,p_2,p_3)$ is defined by $\alpha_j=Q+ip_j\in Q+i\R^+$ 
and $\boldsymbol{\nu}=(\nu_1,\nu_2,\nu_3)$ are Young diagrams, this function is called \emph{block amplitude} 
and is $\ell^2$ summable in the Young diagrams $\boldsymbol{\nu}$ variables and $L^2$ in ${\bf p}\in \R_+^3$ 
by the Plancherel formula \eqref{completeness} if one uses the Schapovalov forms for the pairing. 
The result is the fundamental holomorphic factorisation property:

\begin{theorem}\label{pantresult}
Let  $\caP=(\caP,\boldsymbol{\zeta})$ be a pair of pants with parametrised boundary and $g_\caP$ be an admissible metric. Then for $\alpha_j\in Q+i\R$, j=1,2,3 we have
\begin{align*}%\label{}
\caA_{\caP,g_{\caP},\boldsymbol{\zeta}}(\otimes_{j=1}^3\Psi_{\alpha_j,\nu_j,\tilde\nu_j})=Ze^{\sum_{j=1}^3c_j\Delta_{\alpha_j}}\, 
w_{\caP}(\boldsymbol{\Delta_{\alpha}},
 \boldsymbol{\nu})
 \overline{ w_{\caP}(\boldsymbol{\Delta_{\alpha}},
 \tilde{\boldsymbol{\nu}})}C^{\rm DOZZ}_{\gamma,\mu}(\alpha_1,\alpha_2,\alpha_3)
\end{align*}
where the constants $Z,c_j$ depend on the metric $g_\caP$ and $w_{\caP}(\boldsymbol{\Delta_{\alpha}},
 \boldsymbol{\nu})$ is a polynomial in the conformal weights $\Delta_{\alpha_i}$. Here $\boldsymbol{\nu}=(\nu_1,\nu_2,\nu_3)$ is $3$-plet of Young diagrams and $\boldsymbol{\Delta_{\alpha}}=(\Delta_{\alpha_1},\Delta_{\alpha_2},\Delta_{\alpha_3})$.
\end{theorem}

For a  statement including the other building blocks, see Theorem \ref{pantDOZZ}.
 The proof proceeds by using the analytic continuation of   the states $\Psi_{\alpha,\nu,\tilde\nu}$ to real values of $\alpha$ so as to get a  probabilistic representation for these states. For $\nu=\tilde\nu=\emptyset$, the $\Psi_{\alpha,\emptyset,\emptyset}$ is just the disk amplitude discussed in (b) above and in general (namely when the Young diagrams are non empty) it is given by a disk amplitude with further Stress Energy Tensor (SET) insertions in addition to the vertex operator $e^{\alpha\phi}$ in \eqref{amplitudeintro}. The SET is a random field $T(z)$ formally given by $T(z)=Q\partial^2_z\phi-(\partial_z\phi)^2$ and rigorously defined through regularised expressions. By Segal's gluing axiom, evaluating the pant amplitude at these states amounts to  gluing  such amplitudes to the pant amplitude. So, we end up studying  probabilistic expressions for correlation functions of the type
 \begin{align*}%\label{correset}
\langle \prod_{i=1}^kT(u_i) \prod_{j=1}^l\bar T(v_i) \prod_{i=1}^3 e^{\alpha_i\phi(x_i)}\rangle_{\hat\C,g}.
\end{align*}
The crucial combinatorial input to compute such quantities are the  so called {\it Ward identities}, obtained in  Section \ref{sec:computingamplitudes}. They allow us to express  these correlation functions  in terms of holomorphic derivatives (for $T$) and antiholomorphic derivatives (for $\bar T$) of the correlation function $\langle  \prod_{i=1}^3 e^{\alpha_i\phi(x_i)}\rangle_{\hat\C,g}$, see Propositions \ref{propward} and \ref{wardite}. This holomorphic factorisation leads to the one in Theorem \ref{pantresult}. This specific role of the Ward identities is well established in physics, see for instance \cite{Gawedzki96_CFT,Ribault14}.

Theorem \ref{mainth} now follows by combining the previous arguments. First we need some coordinates on the moduli space $\caM_{\bf g}$ that are adapted to our datas, namely the plumbing construction of surfaces with marked points.

Again, as an example,  consider the case of $m=0$ and genus ${\bf g}\geq 2$. A topological surface of genus ${\bf g}\geq 2$ can be cut  along $3{\bf g}-3$ noncontractible loops to  $2{\bf g}-2$ (topological) pairs of pants. Conversely it is shown in \cite{Hinich-Vaintrob} that the moduli space  $\mc{M}_{\bf g}$  can be covered by gluing admissible pants together with flat annuli $\mathbb{A}_{q_i}$, $i=1,\dots, 3{\bf g}-3$ (see Figure \ref{gluedcylinder}). The variables ${\bf q}\in \D^{3{\bf g}-3}$ provide local coordinates for $\caM_{\bf g}$, see  Section \ref{sub:plumbing}  for details and also for extension to surfaces with marked points. In particular these coordinates cover also  the boundary strata of the compactification of the moduli space.

 Gluing two pant amplitudes  along a boundary circle of each (with Theorem \ref{gluingthm}) and applying the Plancherel identity \eqref{completeness} yields
 \begin{align*}
\caA_{\caP_1,g_{1},\boldsymbol{\zeta}_1}\caA_{\caP_2,g_{2},\boldsymbol{\zeta}_2}=\frac{1}{2\pi}\sum_{\nu,\nu',\tilde\nu,\tilde\nu'}\int_{\R_+} F^{-1}_{Q+ip}(\nu,\nu')F^{-1}_{Q+ip}(\tilde\nu,\tilde\nu')
\caA_{\caP_1,g_{1},\boldsymbol{\zeta}_1}(\cdot, \Psi_{Q+ip,\nu,\tilde\nu})\caA_{\caP_2,g_{2},\boldsymbol{\zeta}_2}(\Psi_{Q+ip,\nu',\tilde\nu'},\cdot)\dd p.
\end{align*}
Incorporating the plumbing annulus $\caA_q$, i.e. considering the amplitude $\caA_{\caP_1,g_{1},\boldsymbol{\zeta}_1}\caA_q\caA_{\caP_2,g_{2},\boldsymbol{\zeta}_2}$, modifies the above relation by powers of $q$ and $\bar q$ by using \eqref{eigen}. Applying these identities to all gluings and using the holomorphic factorisation (i.e. Theorem \ref{pantresult}) for pant amplitudes, one can factorise the resulting sums over Young diagrams $\boldsymbol{\nu},\tilde{\boldsymbol{\nu}}$ as a product of sums over $\boldsymbol{\nu}$ and over $\tilde{\boldsymbol{\nu}}$. The result is the factorised expression in Theorem \ref{mainth} and the expression for the conformal blocks \eqref{blocks}.
The proof for the convergence of the conformal blocks  relies on viewing these series as pairings of Hilbert-Schmidt operators (called block amplitudes) acting on functions defined on the set of Young diagrams: Young diagrams parametrise the eigenstates of the Liouville Hamiltonian, and the block amplitudes are matrix coefficients of Segal amplitudes on the eigenbasis.
The Plancherel formula  \eqref{completeness} for the eigenbasis of the 
Hamiltonian then transfers the Hilbert-Schmidt properties of Segal amplitudes to the block amplitudes, and therefore implies the convergence of the conformal blocks.  

The convergence of the blocks was hitherto unknown, except in the torus $1$-point case   \cite{GosalRemySun20} and the sphere $4$-point case in  \cite{GKRV20_bootstrap}. In extending these two particular cases, our work introduces several significant innovations, both in terms of conceptual understanding and technical implementation. Firstly, the Segal amplitude construction provides a robust and general method for encoding the complex geometric aspects of the problem within a probabilistic framework.  Secondly, to express the block 
amplitudes in terms of algebraic/geometric data depending only on the Virasoro representations 
(conformal weights, central charge), we need to prove a general Ward identity for pairs of pants, 
a task that is considerably more involved than for the $4$-point correlation of the sphere.

\subsection{Related works and future directions}

Segal's axioms have been earlier addressed in a non-conformal setup in \cite{Pickrell} where the gluing axiom was proved for  $P(\phi)_2$ quantum field theories. See also \cite{Dimock2007}  where correlations on $\hat\C$ were discussed for the (conformal) free field. A more recent treatment of Segal's axioms appeared after our work in \cite{Lin} for $P(\phi)_2$.  

As far as conformal bootstrap is concerned, this has been mostly studied mathematically in the setup of  Vertex Operator Algebras (VOA) introduced by Borcherds \cite{borcherds} and Frenkel-Lepowsky-Meurman \cite{Frenkel:1988xz}, see also \cite{Huang_1997} and  \cite{Huang_1997} for more recent developments. Here one aims at algebraic construction of the conformal blocks. The main problems are the question of the convergence of expansions analogous to \eqref{blocks}, their transformation properties under modular group and the problem of constructing CFT by combining holomorphic and antiholomorphic blocks. The case of minimal CFTs is treated in the case of genus $0$ or $1$ Riemann surfaces (see \cite{Huang_blog} and references therein) but the higher genus case is more problematic: in particular we stress that the Moore-Seiberg argument of modular invariance based on the crossing symmetry of the sphere $4$-point correlation function and the modular invariance of the  torus $1$-point correlation function has a gap and still remains a conjecture (see \cite[section 3]{Huang_blog}). While there has been progress in these issues for rational CFTs where the spectrum is countable (see \cite{Bin_Gui} for the discussion of the higher genus case and references to earlier work)  
non-rational CFTs like LCFT have remained a challenge in this setup (the reader may consult \cite{10.1007/978-90-481-2810-5_46} for a lucid introduction to the difficulties of nonrational CFTs).

Finally we want to mention some recent probabilistic work related to LCFT and the topics of this paper. 
In \cite{GosalRemySun20} a probabilistic expression for the torus 1-point conformal block was derived in terms of multiplicative chaos, which opens new perspectives on the probabilistic representation of conformal blocks.
The conformal bootstrap for LCFT on surfaces with boundaries is currrently been developed \cite{Wu,GRW1,GRW2} in collaboration with Wu: in particular, Segal's axioms in that case has recently been proved in \cite{GRW1} as a follow up of our paper. In that case,  there are 4 possible structure constants which were computed in \cite{AngRemySun21_FZZ,Remy20,MR4483018,Ang_Remy_Sun_Zhu}. 

 Recently, there has been a large effort in probability theory  to make sense of Polyakov's path integral formulation of LCFT within the framework of random conformal geometry and the scaling limit of random planar maps: see \cite{LeG13,Mie13,MS15a,MS16a,MS16b,DingDDF19_tightness,DubedatFGPS19_metric,GM20} for the construction of a random metric space describing (at least at the conjectural level) the limit of random planar maps  and \cite{holden_sun_cardy} for exact results on their link with LCFT. A key role was played by the pioneering work \cite{MatingOfTrees} where was devised a set of tools based on couplings between GFF and Schramm-Loewner Evolution (SLE) or variants like Conformal Loop Ensembles (CLE,  see \cite{sheffield_werner_CLE}). Integrability of geometric observables for CLE was conjectured in physics in \cite{PhysRevLett.116.130601} via CFT techniques, stating that the expectation value of  the joint moments of the conformal radii of outermost loops surrounding three points for CLE on the sphere agrees with the imaginary DOZZ formula proposed by \cite{Zamolodchikov_05}. Based on the mating of tree technology \cite{MatingOfTrees} and integrability result for LCFT \cite{KRV_DOZZ}, the breakthrough work  \cite{AngSun21_CLE} establishes this conjecture on rigorous grounds.  It is a major open question to determine other CFT aspects (full set of structure constants, spectrum and Segal's axioms) of CLE. 
 
 In physics, it has been argued that the CFT description of  loop models involves the imaginary LCFT, which can be roughly described as a path integral \eqref{pathi} based on the Liouville action \eqref{introactionL} with imaginary parameter $\gamma$.  A compactified version of this path integral has recently been proposed in \cite{CILT2023} leading to a first instance of non-unitary CFT (i.e. the Hamiltonian is non self-adjoint) with central charge $c<1$ and discrete spectrum. Even more interesting, it gives rise to a  \textit{logarithmic CFT}, topic which has been under active research in physics but remains poorly understood mathematically. Concretely, this means that  the Hamiltonian  should be diagonalisable in Jordan blocks, and that the involved Virasoro representations are indecomposable but not irreducible. This provides a playground to understand the structure of the conformal bootstrap for logarithmic CFT. Another fascinating direction to be understood mathematically is the conformal bootstrap for CFT with extended symmetry. The prototype of CFT with $\mathcal{W}$-symmetry algebra is the Toda CFT, constructed probabilistically in \cite{TODA_2023}, whose CFT structure remains a mystery in spite of recent  progress \cite{cercle2022,cercle2024,cerclehughenin2024}. Another instance are the Wess-Zumino-Witten models with Kac-Moody symmetry algebra, whose probabilistic construction has been achieved in \cite{GKR_WZW} in the case when the target space is the hyperbolic space $\mathbb{H}^3$ and which is another example of non-unitary CFT but now with continuous spectrum.\\

\noindent\textbf{Acknowledgements.} C. Guillarmou acknowledges the support of  European Research Council (ERC)  Consolidator grant  725967 and A. Kupiainen  the support of the ERC Advanced Grant 741487. R. Rhodes is partially supported by the Institut Universitaire de France (IUF). R. Rhodes and V. Vargas acknowledge the support of the ANR-21-CE40-0003. The authors wish to thank  C. Klimcik and especially J. Teschner, whose works have been a deep source of inspiration for us, for many enlightening discussions. We also thank R. Canary, P. Haissinsky, V. Hinich and A. Vaintrob for answering our questions on the plumbing coordinates.

\begin{table}[H]
\caption{Key Notations}
\begin{tabularx}{\textwidth}{@{}XX@{}}
\toprule
  Notations: \\
  $(\Sigma,g)$, ${\rm dv}_g$,  $\dd \ell_g$, ${\bf g}$  & Riemannian surface, volume form on $\Sigma$, $\pl \Sigma$, genus\\
$K_g$, $k_g$ & scalar and geodesic curvature  \\
 $\Delta_g, \Delta_{g,D}$, $G_g$ $G_{g,D}$ & Laplacians (Dirichlet if $\pl\Sigma\not=\emptyset$), Green functions \\
 $X_g$, $X_{g,D}$, $M^{g}_{\gamma}(h,\dd x)$ & GFF and Dirichlet GFF, GMC measure for $h$\\
 $\gamma\in (0,2)$, $Q=\frac{\gamma}{2}+\frac{2}{\gamma}$, $\mu$ & Liouville CFT parameters\\
 $c_{\rm L}=1+6Q^2$, $\Delta_\alpha:= \frac{\alpha}{2}(Q-\frac{\alpha}{2})$ & LCFT central charge and conformal weights\\
 ${\bf x}=(x_1,\dots,x_m)$, $\boldsymbol{\alpha}=(\alpha_1,\dots,\alpha_m)$  & Marked points and attached weights on $\Sigma$ \\
 $(\mc{C}_1,\dots,\mc{C}_{3{\bf g}-3+m})$, ${\bf q}=(q_1,\dots,q_{3{\bf g}-3+m})$, $\mc{G}$ & Cutting curves, plumbing coordinates, gluing graph\\
 $\mc{M}_{{\bf g},m}$, $\mc{T}_{{\bf g},m}$ & Moduli and Teichm\"uller space with marked points \\
$\boldsymbol{\zeta}=(\zeta_1,\dots,\zeta_b)$, $\boldsymbol{\sigma}=(\sigma_1,\dots,\sigma_b)$, $\boldsymbol{\tilde \varphi}=(\tilde\varphi_1,\dots,\tilde\varphi_{b})$& Boundary parametrisations, orientations \& fields \\
$\mc{H}=L^2(H^s(\mathbb{T}),\mu_0)$, $\cjg \cdot,\cdot\cjd_{\mc{H}}$ &  Hilbert space of LCFT, scalar product\\
 $\mathbf{D}_\Sigma$, $\mathbf{D}$, $P\tilde{\boldsymbol{\varphi}}$ & DN map on $\Sigma$, on 
 $\D$, harmonic extension of $\tilde{\boldsymbol{\varphi}}$\\
$\caA_{\Sigma,g}^0$, $\caA_{\Sigma,g,\bf x,\boldsymbol{\alpha},\boldsymbol{\zeta}}$ & Free field and Segal 
amplitudes\\
${\bf p}=(p_1,\dots,p_{3{\bf g}-3+m})$, $\boldsymbol{\nu}=(\nu_1,\dots,\nu_{3{\bf g}-3+m})$ & Family of spectral parameters \& Young diagrams\\
${\bf H}$, $\Psi_{Q+ip,\nu,\tilde{\nu}}$, $H_{Q+ip,\nu,\tilde{\nu}}$ & Hamiltonian, descendant \& normalised eigenstates \\
$\tilde{\mc{F}}_{\bf p}(\boldsymbol{\alpha},\bf{q})$, $\tilde{\mc{B}}_{\mc{P},J,{\bf x},\boldsymbol{\alpha},\boldsymbol{\zeta}}$
 & Normalised conformal block \& block amplitude \\
 $ \omega_{\mc{P}}(\boldsymbol{\Delta_{\alpha}},\boldsymbol{\nu},{\bf z})$ &  Ward coefficients\\
 $C^{\rm DOZZ}_{\gamma,\mu}(\alpha_1,\alpha_2,\alpha_3)$ & DOZZ $3$-point correlation function for LCFT\\
\end{tabularx}
\end{table}

%%%%%%%%%%%%%%%%%%%%%%%%%%%%%%%%%%%%%%%%%%%%%%%%%%%%%%%%%%%%%%%%%%%%%%%%%%%%%%%%%%%%%%%%%%%%%%%%%%%%%%%%%%%%%%%%%%%%%%%%%%%%%%%%%%%%%%%%%%%%%%%%%%%%%%%%%%%%%%%%%%%%%%%%%%%%%%%%%%%%%%%%%%%%%%%%%%%%%%%%
\section{Liouville Conformal Field Theory on closed Riemann surfaces}\label{sec:lcft}
 %%%%%%%%%%%%%%%%%%%%%%%%%%%%%%%%%%%%%%%%%%%%%%%%%%%%%%%%%%%%%%%%%%%%%%%%%%%%%%%%%%%%%%%%%%%%%%%%%%%%%%%%%%%%%%%%%%%%%%%%%%%%%%%%%%%%%%%%%%%%%%%%%%%%%%%%%%%%%%%%%%%%%%%%%%%%%%%%%%%%%%%%%%%%%%%%%%%%%%%%%%%%%%%%%%%%%%%%%%%%%%%%%%%%%%%%%%%%%%%%%%%%%%%%%%%%%%%%%%%%%%%%%%
In this section, we summarise the   construction of LCFT on closed Riemann surfaces carried  out in \cite{GRVIHES}. To begin with, we  gather some preliminary material in Subsection \ref{prelimLCFT}. Following this, we  state the definitions for the path integral (see equation \eqref{def:pathintegralF}) or the correlation functions (see Proposition \ref{limitcorel}). Then we recall their main properties (see Proposition \ref{covconf2}), including diffeomorphism invariance and conformal covariance. Finally, and importantly, we recall the expression of the structure constants in terms of the DOZZ formula in Subsection  \ref{sub:dozzintro}.

 \subsection{Background and notations} \label{prelimLCFT}
%%%%%%%%%%%%%%%%%%%%%%%%%%%%%%%%%%%%%%%%%%%%%%%%%%%%%%

In this paper, we will consider Riemann surfaces $\Sigma$ with or without boundary and equipped with a metric $g$ with associated volume form ${\rm dv}_g$. Integration of functions $f$ will be denoted $\int_{\Sigma} f(x)  {\rm dv}_g(x)$ and $L^2(\Sigma,{\rm dv}_g)$ is the standard space of square integrable functions with respect to ${\rm v}_g$. In this context, we consider the standard non-negative Laplacian $\Delta_{g}=d^*d$ where $d^*$ is the $L^2$ adjoint of the exterior derivative $d$. We will denote $K_g$ the scalar curvature and $k_g$ the geodesic curvature along the boundary.

%%%%%%%%%%%%%%%%%%%%%%%%%%%%%%%%%%%%%%%%%%%%%%%%%%%%%%

\subsubsection{Regularised determinants}
The determinant of the Laplacian ${\det}'(\Delta_g)$ on a compact manifold was introduced by Ray-Singer \cite{Ray-Singer}. We recall its definition and  we refer to  \cite{Ray-Singer,OsgoodPS88} for details and its properties.
We will be interested in two cases. First, for a Riemannian metric $g$ on a connected oriented compact surface $\Sigma$ without boundary, the Laplacian $\Delta_{g}$ has discrete spectrum
${\rm Sp}(\Delta_{g})=(\la_j)_{j\in \N}$ with $\la_0=0$ and $\la_j\to +\infty$ (each $\lambda_j$ is repeated with multiplicity). 
We can define the determinant of $\Delta_{g}$ by 
\[ {\det} '(\Delta_{g})=\exp(-\pl_s\zeta(s)|_{s=0})\]
where $\zeta(s):=\sum_{j=1}^\infty \la_j^{-s}$ is the spectral zeta function of $\Delta_{g}$, which admits a meromorphic continuation from ${\rm Re}(s)\gg 1$ to $s\in \C$ and is holomorphic at $s=0$.

Second,  for a Riemannian metric $g$ on a connected oriented compact surface $\Sigma$ with boundary, the Laplacian $\Delta_{g}$ with Dirichlet boundary conditions has discrete spectrum $(\la_{j,D})_{j \geq 1}$ and the determinant is defined in a similar way as $\det (\Delta_{g,D}): = e^{- \zeta'_{g,D}(0)}$ where $\zeta_{g,D}$   is the spectral zeta function of $\Delta_{g}$ with Dirichlet boundary conditions defined for ${\rm Re}(s)\gg 1$ by $\zeta_{g,D}(s):=\sum_{j=1}^\infty \la_{j,D}^{-s}$. One can prove that $\zeta_{g,D}(s)$ admits a meromorphic extension to $s\in \C$ that is holomorphic at $s=0$.

\subsubsection{Green functions} The Green function $G_g$ on a surface $\Sigma$ without boundary  is defined 
to be the integral kernel of the resolvent operator $R_g:L^2(\Sigma)\to L^2(\Sigma)$ satisfying $\Delta_{g}R_g=2\pi ({\rm Id}-\Pi_0)$, $R_g^*=R_g$ and $R_g1=0$, where $\Pi_0$ is the orthogonal projection in $L^2(\Sigma,{\rm dv}_g)$ on $\ker \Delta_{g}$ (the constants). By integral kernel, we mean that for each $f\in L^2(\Sigma,{\rm dv}_g)$
\[ R_gf(x)=\int_{\Sigma} G_g(x,x')f(x'){\rm dv}_g(x').\] 
The Laplacian $\Delta_{g}$ has an orthonormal basis of real valued eigenfunctions $(e_j)_{j\in \N}$ in $L^2(\Sigma,{\rm dv}_g)$ with associated eigenvalues $\la_j\geq 0$; we set $\la_0=0$ and $\varphi_0=({\rm v}_g(\Sigma))^{-1/2}$.  The Green function then admits the following Mercer's representation in $L^2(\Sigma\times\Sigma, {\rm dv}_g \otimes {\rm dv}_g)$
\begin{equation}
G_g(x,x')=2\pi \sum_{j\geq 1}\frac{1}{\lambda_j}e_j(x)e_j(x').
\end{equation}
Similarly, on a surface with smooth boundary $\Sigma$, we will consider the Green function with Dirichlet boundary conditions $G_{g,D}$ associated to the Laplacian $\Delta_{g}$. In this case, the associated resolvent operator
\[ R_{g,D}f(x)=\int_{\Sigma} G_{g,D}(x,x')f(x'){\rm dv}_g(x')\] 
solves $\Delta_{g}R_{g,D}=2\pi {\rm Id}$.

\subsubsection{Gaussian Free Fields} In the case of a surface $\Sigma$ with no boundary, the Gaussian Free Field (GFF in short) is defined as follows. Let $(a_j)_j$ be a sequence of i.i.d. real Gaussians   $\mc{N}(0,1)$ with mean $0$ and variance $1$, defined on some probability space   $(\Omega,\mc{F},\mathbb{P})$,  and define  the Gaussian Free Field with vanishing mean in the metric $g$ by the random functions
\begin{equation}\label{GFFong}
X_g:= \sqrt{2\pi}\sum_{j\geq 1}a_j\frac{e_j}{\sqrt{\la_j}} 
\end{equation}
 where  the sum converges almost surely in the Sobolev space  $H^{s}(\Sigma)$ for $s<0$ defined by
\[ H^{s}(\Sigma):=\{f=\sum_{j\geq 0}f_j\varphi_j\,|\, \|f\|_{s}^2:=|f_0|^2+\sum_{j\geq 1}\lambda_j^{s}|f_j|^2<+\infty\}.\] 
This Hilbert space is independent of $g$, only its norm depends on a choice of $g$.
The covariance is then the Green function when viewed as a distribution, which we will write with a slight abuse of notation
\[\mathbb{E}[X_g(x)X_g(x')]= \,G_g(x,x').\]
In the case of a surface $\Sigma$ without boundary, we will denote the Liouville field by $\phi_g:=c+X_g$ where $c\in\R$ is a constant that stands for the constant mode of the field.

In the case of a surface with boundary $\Sigma$, the Dirichlet Gaussian free field (with covariance $G_{g,D}$) will be denoted $X_{g,D}$. It is defined similarly to the sum \eqref{GFFong} with the $(e_j)_j$  and $(\lambda_j)_j$ replaced by the normalised eigenfunctions $(e_{j,D})_j$ and ordered eigenvalues $(\lambda_{j,D})_j$ of the Laplacian with Dirichlet boundary conditions, the sum being convergent   almost surely  in the Sobolev space  $H^{s}(\Sigma)$ (for all $s\in (-1/2,0)$) defined by
\[ H^{s}(\Sigma):=\{f=\sum_{j\geq 0}f_j\varphi_{j,D}\,|\, \|f\|_{s}^2:=\sum_{j\geq 0}\lambda_{j,D}^{s}|f_j|^2<+\infty\}.\] 
In this context, we  will always consider the harmonic extension $P \tilde{\boldsymbol{\varphi}}$ of a boundary field $ \tilde{\boldsymbol{\varphi}}$ (the exact definition appears when more relevant in subsection \ref{subDTN} below). In the case of a surface with boundary the Liouville field will be denoted by  
\begin{equation}\label{defliouvillefield}
\phi_g:= X_{g,D}+P \tilde{\boldsymbol{\varphi}}. 
\end{equation}
Hence $\phi_g$ will depend on boundary data in this case.
 
 \subsubsection{$g$-Regularisations} As Gaussian Free Fields are rather badly behaved (they are distributions), we will need to consider their regularisations. We then introduce a regularisation procedure, which we will call \emph{$g$-regularisation}. Let $\Sigma$ be a surface with or without boundary equipped with a Riemannian metric $g$ and associated distance $d_g$. For a random distribution $h$ on $\Sigma$ and  for $\eps>0$ small, we define a regularisation $h_{\eps}$ of $h$ by averaging on geodesic circles of radius $\eps>0$: let $x\in \Sigma$ and let $\mc{C}_g(x,\eps)$ be the geodesic circle of center $x$ and radius $\eps>0$, and let $(f^n_{x,\eps})_{n\in \N} \in C^\infty(\Sigma)$ be a sequence with $||f^n_{x,\eps}||_{L^1}=1$ 
which is given by $f_{x,\eps}^n=\theta^n(d_g(x,\cdot)/\eps)$ where $\theta^n(r)\in C_c^\infty((0,2))$ is non-negative, 
supported near $r=1$ and such that $f^n_{x,\eps}{\rm dv}_g$ 
converges in $\mc{D}'(\Sigma)$ to the uniform probability measure 
$\mu_{x,\eps}$
on $\mc{C}_g(x,\eps)$ as $n\to \infty$ (for $\eps$ small enough, the geodesic circles form a sub-manifold and the restriction of $g$ along this manifold gives rise to a finite measure, which corresponds to the uniform measure after renormalisation so as to have mass $1$). If the pairing $\langle h, f_{x,\eps}^n\rangle$ converges almost surely towards a random variable $h_\epsilon(x)$ that has a modification which is continuous in the parameters $(x,\epsilon)$, we will say that $h$ admits a $g$-regularisation $(h_\epsilon)_\epsilon$. This is the case for the GFF $X_g$ or $X_{g,D}$, see \cite[Lemma 3.2]{GRVIHES}. We will denote $X_{g,\epsilon}$, $X_{g,D,\epsilon}$ their respective $g$-regularisations
and $\phi_{g,\epsilon}$ the $g$-regularisation of the Liouville field.

\subsection{Construction of Liouville Conformal Field Theory} 
 %%%%%%%%%%%%%%%%%%%%%%%%%%%%
 
\subsubsection{Gaussian multiplicative chaos}
For $\gamma\in\R$ and $h$ a random distribution admitting a  $g$-regularisation $(h_\epsilon)_\epsilon$, we define the measure 
\begin{equation}\label{GMCg}
M^{g,\eps}_{\gamma}(h,\dd x):= \eps^{\frac{\gamma^2}{2}}e^{\gamma h_{ \eps}(x)}{\rm dv}_g(x).
\end{equation}
Of particular interest for us is the case when $h=X_g$ or $h=X_{g,D}$ (and consequently $h=\phi_g$ too). In that case, for $\gamma\in (0,2)$,  the random measures above converge  as $\eps\to 0$ in probability and weakly in the space of Radon measures towards   non trivial random measures respectively denoted by $M^g_\gamma(X_g,\dd x)$, $M^g_\gamma(X_{g,D},\dd x)$ and $M^g_\gamma(\phi_g,\dd x)$; this is a standard fact and the reader may consult \cite{Kahane85,rhodes2014_gmcReview,GRVIHES} for further details. Clearly, when $\Sigma$ has no boundary, we have the relation $M^g_\gamma(\phi_g,\dd x)=e^{\gamma c}M^g_\gamma(X_{g},\dd x)$ and, in the case when $\Sigma$ has a  boundary then $M^g_\gamma(\phi_g,\dd x)= e^{\gamma P  \tilde{\boldsymbol{\varphi}} (x)} M^g_\gamma(X_{g,D},\dd x)$.

Also, from  \cite[Lemma 3.2]{GRVIHES}  we recall that there exist $W_g,W_g^D\in C^\infty(\Sigma)$ such that 
\begin{equation}\label{varYg}
\lim_{\eps \to 0}\E[X^2_{g,\eps}(x)]+\log\eps=W_g(x),\quad \lim_{\eps \to 0}\E[X^2_{g,D,\eps}(x)]+\log\eps=W^D_g(x)
\end{equation}
uniformly over the compact subsets of $\Sigma$. Moreover, in the case of the Dirichlet GFF,  one has the following relation if $g'=e^{\omega}g$ for some $\omega \in C^\infty(\Sigma)$ (observe that the Dirichlet GFF is the same for $g$ and $g'$): 
\begin{equation}\label{varYgconformal}
W^D_{g'}(x)= W^D_g(x)+\frac{\omega(x)}{2}
\end{equation}
which leads directly to the scaling relation 
\begin{equation}\label{scalingmeasure}
M^{g'}_\gamma(X_{g',D},\dd x)= e^{\frac{\gamma}{2}Q \omega(x)} M^g_\gamma(X_{g,D},\dd x).
\end{equation}

\subsubsection{LCFT path integral}
  Let  $\Sigma$ be a closed Riemann surface of genus ${\bf g}$, $g$ be a fixed metric on $\Sigma$ and   $\gamma\in(0,2)$, $\mu>0$ and $Q=\frac{\gamma}{2}+\frac{2}{\gamma}$. Recall that the Liouville field is $\phi_g=c+X_g$ with $c\in\R$. For $F:  H^{s}(\Sigma)\to\R$ (with $s<0$) a bounded continuous functional, we set 
\begin{align}\label{def:pathintegralF}
\cjg F(\phi)\cjd_{\Sigma, g}:=& \big(\frac{{\rm v}_{g}(\Sigma)}{{\det}'(\Delta_{g})}\big)^\hf  \int_\R  \E\Big[ F( \phi_g) \exp\Big( -\frac{Q}{4\pi}\int_{\Sigma}K_{g}\phi_g\,{\rm dv}_{g} - \mu   M^g_\gamma(\phi_g,\Sigma)  \Big) \Big]\,\dd c . 
\end{align}
 By \cite[Proposition 4.1]{GRVIHES}, this quantity  defines a measure and moreover the partition function defined as  the total mass of this measure, i.e $\cjg 1 \cjd_ {\Sigma, g}$, is finite iff the genus $ \mathbf{g}\geq 2$.

\subsubsection{Vertex operators} On a surface $\Sigma$ with or without boundary, we introduce the regularised vertex operators, for fixed $\alpha\in\R$ (called {\it weight}) and $x\in \Sigma$, (when $\pl\Sigma\not=\emptyset$, we use the field \eqref{defliouvillefield})
\begin{equation*}
V_{\alpha,g,\eps}(x)=\eps^{\alpha^2/2}  e^{\alpha  \phi_{g,\epsilon}(x) } .
\end{equation*}
Notice that in the case of a surface with boundary and if $g'=e^{\omega}g$, then the relation \eqref{varYgconformal} gives
\begin{equation}\label{scalingvertex}
V_{\alpha,g',\eps}(x)=(1+o(1)) e^{\frac{\alpha^2}{4} w(x)} V_{\alpha,g,\eps}(x)
\end{equation}
when $\eps$ goes to $0$. We will use this fact later.

Next, if the surface $\Sigma$ is closed, the correlation functions are defined by the limit
\begin{equation}\label{defcorrelg}
\cjg \prod_i V_{\alpha_i,g}(x_i) \cjd_ {\Sigma, g}:=\lim_{\eps \to 0} \: \cjg \prod_i V_{\alpha_i,g,\eps}(x_i) \cjd_ {\Sigma, g}
\end{equation}
where we have fixed $m$ distinct points $x_1,\dots,x_m$  on $\Sigma$ with respective associated weights $\alpha_1,\dots,\alpha_m\in\R$. Non triviality of correlation functions are then summarised in the following proposition (see \cite[Prop 4.4]{GRVIHES}):

 \begin{proposition}\label{limitcorel} Let $ x_1,\dots,x_m$ be distinct points on a closed surface $\Sigma$ and $ (\alpha_1,\dots,\alpha_m)\in\R^m$.
The limit \eqref{defcorrelg} exists and is non zero if and only if the weights $(\alpha_1,\dots,\alpha_m)$ obey the Seiberg bounds
 \begin{align}\label{seiberg1}
 & \sum_{i}\alpha_i + 2 Q ({\bf g}-1)>0,\\ 
 &\forall i,\quad \alpha_i<Q\label{seiberg2}.
 \end{align}
 \end{proposition}

The correlation function then obeys the following transformation laws\footnote{The reader may compare \eqref{confan} with the general axiomatic of CFTs exposed in \cite{Gawedzki96_CFT}.} (see \cite[Prop 4.6]{GRVIHES}):
 \begin{proposition}\label{covconf2}{\bf (Conformal anomaly and diffeomorphism invariance)}
Let $g,g'$ be two conformal metrics on a closed surface $\Sigma$ with $g'=e^{\omega}g$ for some $\omega\in C^\infty(\Sigma)$, and let $ x_1,\dots,x_m$ be distinct points on $\Sigma$ and $ (\alpha_1,\dots,\alpha_m)\in\R^m$ obeying the Seiberg bounds. Then we have
\begin{equation}\label{confan} 
\log\frac{\cjg \prod_i V_{\alpha_i,g'}(x_i) \cjd_ {\Sigma, g'}}{\cjg \prod_i V_{\alpha_i,g}(x_i) \cjd_ {\Sigma, g}}= 
\frac{1+6Q^2}{96\pi}\int_{\Sigma}(|d\omega|_g^2+2K_g\omega) {\rm dv}_g-\sum_i\Delta_{\alpha_i}\omega(x_i)
\end{equation}
where the real numbers $\Delta_{\alpha_i}$, called {\it conformal weights}, are defined by the relation for $\alpha\in\R$
 \begin{align}\label{deltaalphadef}
\Delta_\alpha=\frac{\alpha}{2}(Q-\frac{\alpha}{2}).
\end{align} 
Let $\psi:\Sigma\to \Sigma$ be an orientation preserving diffeomorphism. Then  
\[ \cjg \prod_i V_{\alpha_i,\psi^*g}(x_i) \cjd_ {\Sigma, \psi^*g}=\cjg \prod_i V_{\alpha_i,g}(\psi(x_i)) \cjd_ {\Sigma, g}  .\]
\end{proposition}

\subsection{Case of the Riemann sphere and the DOZZ formula}\label{sub:dozzintro}
%%%%%%%%%%%%%%%%%%%%%%%%%%%%%%%%%%%%%%%%%%%%%%%%%%%%%%%%%%%%%%

We identify   the Riemann sphere with the extended complex plane $\hat\C$ by stereographic projection. On the sphere, every metric is (up to diffeomorphism) conformal to the round metric $g_0:=\frac{4}{(1+|z|^2)^2}|dz|^2$. From Proposition \ref{limitcorel}, the $m$-point correlations exist and are non trivial   if and only if the  Seiberg bounds \eqref{seiberg1} and \eqref{seiberg2} are satisfied, here with genus $\mathbf{g}=0$. This implies in particular that $m$ must be greater or equal to $3$. 

Applying the transformation rules of Proposition \ref{covconf2}, one gets   that these correlation functions 
are {\it conformally covariant}. More precisely, if $g=g(z)|dz|^2$ is a conformal metric and if $z_1, \cdots, z_m$ are $m$ distinct points in $\hat \C$  then for a M\"obius map $\psi(z)= \frac{az+b}{cz+d}$ (with $a,b,c,d \in \C$ and $ad-bc=1$) 
 \begin{equation}\label{KPZformula}
\langle \prod_{i=1}^m V_{\alpha_i,g}(\psi(z_i))    \rangle_{\hat\C,g}=  \prod_{i=1}^m \Big(\frac{|\psi'(z_i)|^2g(\psi(z_i))}{g(z_i)}\Big)^{-  \Delta_{\alpha_i}}     \langle \prod_{i=1}^m V_{\alpha_i,g}(z_i)     \rangle_{\hat\C,g}.
\end{equation}

The M\"obius covariance implies in particular that the three point functions ($m=3$) are determined up to a constant, called the {\it structure constant}, which was proven to be given by the DOZZ formula in \cite{KRV_DOZZ}. This means
\begin{align}
\label{3pointDOZZ}
 \langle  &V_{\alpha_1}(z_1)  V_{\alpha_2}(z_2) V_{\alpha_3}(z_3)  \rangle_{\hat\C,g}\\
  =&|z_1-z_3|^{2(\Delta_{\alpha_2}-\Delta_{\alpha_1}-\Delta_{\alpha_3})}|z_2-z_3|^{2(\Delta_{\alpha_1}-\Delta_{\alpha_2}-\Delta_{\alpha_3})}|z_1-z_2|^{2(\Delta_{\alpha_3}-\Delta_{\alpha_1}-\Delta_{\alpha_2})}
\Big(\prod_{i=1}^3g(z_i)^{-\Delta_{\alpha_i}}\Big)\nonumber \\
&\times \frac{1}{2}C_{\gamma,\mu}^{{\rm DOZZ}} (\alpha_1,\alpha_2,\alpha_3 ) \big(\frac{{\rm v}_{g}(\hat\C)}{{\det}'(\Delta_{g})}\big)^\hf e^{-6Q^2 S_{\rm L}^0(\hat\C,g,\omega)}  \nonumber
 \end{align} 
with $\omega:\hat\C\to \R$ defined by $e^\omega g= |z|_+^{-4}|dz|^2$, $S_{\rm L}^0(\hat\C,g,\omega)$ given by \eqref{SL0} and $C_{\gamma,\mu}^{{\rm DOZZ}} (\alpha_1,\alpha_2,\alpha_3 ) $ is an explicit  function whose expression is recalled in Appendix \ref{app:dozz}.

%%%%%%%%%%%%%%%%%%%%%%%%%%%%%%%%%%%%%%%%%%%%%%%%%%%%%%%%%%%%%%%%%%%%%%%%%%%%%%%%%%%%%%%%%%%%%%%%%%%%%%%%%%%%%%%%%

 \section{Geometric decomposition of Riemann surfaces and gluing}\label{sec:geometric}
%%%%%%%%%%%%%%%%%%%%%%%%%%%%%%%%%%%%%%%%%%%%%%%%%%%%%%%%%%%%%%
%%%%%%%%%%%%%%%%%%%%%%%%%%%%%%%%%%%%%%%%%%%%%%%
 
In this section, we start by going over some material about Riemann surfaces with marked points and 
with parametrised boundary. These are the natural geometric objects in Segal's formalism to which we shall attach probabilistically defined amplitudes: the marked points are related to vertex operators and the parametrisations of the boundary by maps $\T\to \pl \Sigma$ allow to glue these surfaces along their boundary and will be used in next sections to attach to each boundary component $\pl_j\Sigma$ a field $\varphi_j$ on the unit circle $\T$, which is the starting point to defined amplitudes. To describe the holomorphic structure of the correlation functions via its conformal blocks as functions of the moduli parameters, we need to use a convenient description of the moduli space and Teichm\"uller space of Riemann surfaces. As we explain below, this is done through a family of holomorphic coordinates $q_1,\dots,q_{2{\bf g}-3+m}$, called \emph{plumbing coordinates}, on the moduli space, that correspond geometrically to glue or cut complex annuli 
$\mathbb{A}_{q_i}=\{z\in \C\,|\, |q_i|\leq |z|\leq 1\}$ between each pair of pants in a fixed pant decomposition, as shown in Figures \ref{figureplumbing} and \ref{gluedcylinder}, where  $|q_i|$ correspond to the modulus of the annuli $\mathbb{A}_{q_i}$ and the argument of $q_i$ correspond to the twisting angle used for gluing the annuli between two pants. We also 
need to introduce a convenient way, namely a multigraph, to encode the topological decomposition of the surface into pairs 
of pants, in order to express the way the Segal amplitudes of pants are paired together, to define appropriately the conformal blocks and to decompose the correlation functions in terms of conformal blocks. Finally, the conformal blocks will be defined in Section \ref{Section:blocks} as functions on the Teichm\"uller space, and we will thus express the conformal blocks in terms of these coordinates: we shall see that gluing complex annuli produce holomorphic families of conformal blocks in the variables $q_i$.

\subsection{Riemann surfaces with marked points}
\label{Sec:Riemann surfaces with marked points} 
\subsubsection{Closed Riemann surfaces with marked points}
A closed Riemann surface $(\Sigma,J)$ is a smooth oriented compact surface $\Sigma$ with no boundaries, equipped with a complex structure $J$ (i.e. $J \in {\rm End}(T\Sigma)$ with $J^2=-{\rm Id}$), or equivalently a set of charts $\omega_j:U_j\to \D\subset \C$ so that $\omega_j\circ \omega_{k}^{-1}:\D\to \D$ is a biholomorphic map, where $\D=\{z\in \C\,| \, |z|<1\}$ is the unit disk. The complex structure $J$ is the canonical one (i.e. $J\pl_x=\pl_y, J\pl_y=-\pl_x$) when viewed in $\D$ via the charts. 
A fixed set of points $z_i\in\Sigma$, $i=1,\dots,m$ on $\Sigma$ are called \emph{marked points} on $\Sigma$.

\subsubsection{Riemann surfaces with analytic boundary.}\label{subsub:boundary} A compact Riemann surface $\Sigma$ with real analytic boundary $\partial\Sigma=\sqcup_{j=1}^{b}\caC_j$ is a 
compact oriented surface with smooth boundary with a family of charts 
 $\omega_j:V_j\to \omega_j(V_j)\subset \C$ for $j=1,\dots,j_0$ where $\cup_j V_j$ is an open covering of $\Sigma$ and $\omega_k \circ \omega_j^{-1}$ are holomorphic maps (where they are defined), and $\omega_j(V_j\cap\pl \Sigma)$ is a real analytic curve if $V_j\cap\pl\Sigma\not=\emptyset$. 
Using the Riemann uniformization theorem, we can moreover assume that for $j\in [1,b]$, $V_j$ are neighborhoods of $\mc{C}_j$ with 
$\omega_j(V_j)=\mathbb{A}_\delta$ where $\mathbb{A}_{\delta}=\{z\in \C\,|\, |z|\in [\delta,1]\}$ for some $\delta\in (0,1)$ with $\omega_j(\mc{C}_j)=\{|z|=1\}$, while for all other $j$, $V_j$ are open sets not intersecting $\pl \Sigma$ satisfying $\omega_j(V_j)=\D\subset \C$. The charts induce a complex structure $J$ as for the closed case, and we shall often write $(\Sigma,J)$ for the Riemann surface with boundary.

Recall that the orientation is given by a choice of non-vanishing $2$-form $w_\Sigma\in C^\infty(\Sigma;\Lambda^2T^*\Sigma)$, and $(\omega_j^{-1})^*w_\Sigma=e^{f_j} dx\wedge dy$ in $\D$ for some function $f_j$ if $z=x+iy$ is the complex variable of $\D\subset \C$. The boundary circles $\mc{C}_j$ all inherit an orientation from the orientation of $\Sigma$, simply by taking 
the $1$-form $\iota_{\mc{C}_j}^*(i_{\nu}\omega_\Sigma)$  where $i_\nu$ is the interior product with a non-vanishing exterior pointing vector field $\nu$ to $\Sigma$ and $\iota_{\mc{C}_j}:\mc{C}_j\to \Sigma$ is the natural inclusion. 
In the chart given by the annulus $\mathbb{A}_\delta$, the orientation is then given by $d\theta$ (i.e. the counterclockwise orientation) on the unit circle parametrised by $(e^{i\theta})_{\theta\in [0,2\pi]}$. 
We shall also add a set of marked points $x_1,\dots,x_m$ in the interior $\Sigma^\circ$ in what follows.

\subsubsection{Parametrised boundaries and gluing.} \label{sub:gluing} 
On a Riemann surface $(\Sigma,J)$ (not necessarily connected) with real analytic boundary, we can choose an analytic parametrisation by using the holomorphic charts $\omega_j$ described above as follows: let  $\T:=\{z\in \C\,|\, |z|=1\}$ be the standard unit circle, then for each $j\in [1,b]$, fix a point $p_j\in \mc{C}_j$  and an orientation $\sigma_j\in\{-1,+1\}$, and define the parametrisation $\zeta_j:\T\to \mc{C}_j$ by 
$\zeta_j(e^{i\theta})=\omega_j^{-1}(e^{-i\sigma_j\theta}\omega_j(p_j))$ (in particular $\zeta_j(1)=p_j$). Observe that the parametrisation of $\mc{C}_j$ is entirely described by $(\sigma_j,p_j)$ and the  complex coordinate chart $\omega_j$ near $\mc{C}_j$.  
We say that the boundary $\mc{C}_j$ is \emph{outgoing} if the orientation $(\zeta_j)_*(d\theta)$ is the orientation of $\mc{C}_j$ induced by that of $\Sigma$ as described above, i.e. if $\sigma_j=-1$, otherwise the parametrised boundary $\mc{C}_j$ is called \emph{incoming} if $\sigma_j=+1$. 
One can also proceed conversely by choosing an analytic parametrisation $\zeta_j:\T\to \mc{C}_j$, which produces a holomorphic chart $\omega_j$ by holomorphically extending $\zeta_j$ in an annular neighborhood of $\T\subset \C$ and taking its inverse. 

If $\mc{C}_j$ is outgoing and $\mc{C}_k$ is incoming, we can define a new Riemann surface $\Sigma_{jk}$ with $(b-2)$-boundary components by gluing/identifying $\mc{C}_j$ with $\mc{C}_k$ as follows: 
we identify $\mc{C}_j$ with $\mc{C}_k$ by setting $\zeta_j(e^{i\theta})=\zeta_k(e^{i\theta})$, in particular $p_j$ gets identified with $p_k$. A neighborhood in $\Sigma_{jk}$ of the identified circle $\mc{C}_j\simeq \mc{C}_k$ is given by $(V_j\cup V_k)/\sim $ (where $\sim$ means the identification $\mc{C}_j\sim\mc{C}_k$) which identifies with the annulus $\mathbb{A}_\delta\cup \mathbb{A}^{-1}_\delta=\{|z|\in [\delta,\delta^{-1}]\}$ by the chart
\[\omega_{jk}: z\in V_j \mapsto \frac{\omega_j(z)}{\omega_j(p_j)}\in \mathbb{A}_\delta, \quad z\in V_k\mapsto \frac{\omega_k(p_k)}{\omega_k(z)}\in \mathbb{A}^{-1}_\delta.\]
Note that 
\begin{equation}\label{gluingonboundary}
\omega_{jk}(\zeta_j(e^{i\theta}))=\frac{\omega_j(\zeta_j(e^{i\theta}))}{\omega_j(p_j)}=
e^{-i\sigma_j\theta}=e^{i\sigma_k\theta}=\frac{\omega_k(p_k)}{\omega_k(\zeta_k(e^{i\theta}))}=\omega_{jk}(\zeta_k(e^{i\theta})).
\end{equation}
The resulting Riemann surface depends on $\zeta_j,\zeta_k$, or equivalently $(p_j,\sigma_j,\omega_j)$, $(p_k,\sigma_k,\omega_k)$. We notice that if $\mc{C}_j$ and $\mc{C}_k$ belong to different connected components, $b_0(\Sigma_{jk})=b_0(\Sigma)-1$  if $b_0$ denotes the $0$-th Betti number; otherwise $b_0(\Sigma_{jk})=b_0(\Sigma)$.

We can also equip $\Sigma$ with a metric $g$ compatible with its complex structure, i.e. 
$g:=\omega_j^*(e^{f_j}|dz|^2)$ on each $V_j$ for some smooth functions $f_j$, and  
it is convenient to consider $g$ so that it induces a smooth metric on the glued Riemann surface $\Sigma_{jk}$ compatible with the complex structure. For this reason we will use a particularly convenient choice: if $\boldsymbol{\zeta}=(\zeta_1,\dots,\zeta_{b})$, we say that a metric $g$ on $\Sigma$ is \emph{admissible} with respect to $(J,\boldsymbol{\zeta})$ if $g=\omega_j^*(|dz|^2/|z|^2)$ on each neighborhood $V_j$ of $\mc{C}_j$. The metric 
$|dz|^2/|z|^2$ is invariant by the map $z\mapsto 1/z$, which implies that $g$ induces a smooth metric, called glued metric, on $\Sigma_{jk}$ that is compatible with the complex structure of $\Sigma_{jk}$. 

We also make the important remark that, starting from a Riemann surface $\Sigma$ with $b\geq 0$ boundary circles and choosing a simple analytic curve $\mc{C}$ in the interior $\Sigma^\circ$ of $\Sigma$
and a chart $\omega:V\mapsto \{|z|\in [\delta,\delta^{-1}]\}\subset \C$ for some $\delta<1$ with $\omega(\mc{C})=\T$, there is a natural Riemann surface $\tilde{\Sigma}$ obtained by compactifying 
 $\Sigma\setminus \mc{C}$ into a Riemann surface with with $b+2$ boundary circles by adding two copies $\mc{C}_1,\mc{C}_2$ of $\mc{C}$ to $\Sigma\setminus \mc{C}$ using the chart $\omega$ on respectively $\omega^{-1}(\{|z|\leq 1\})$ and $\omega^{-1}(\{|z|\geq 1\})$ (one incoming, one outgoing).
If we perform  the gluing procedure of $\mc{C}_1$ with $\mc{C}_2$ just described using the parametrisation of $\mc{C}_1$ by $\zeta_1:=\omega^{-1}|_{\T}$ and $\zeta_2:=\omega^{-1}|_{\T}$, we obviously recover $(\Sigma,J)$.
    
\subsubsection{Admissible surfaces}\label{admissiblesurfaces} A Riemann surface with analytically parametrised boundary $\boldsymbol{\zeta}:\T \to (\pl\Sigma)^b$, marked points $x_j\in \Sigma^\circ$ in the interior and a metric $g$ that is  admissible with respect to $(J,\boldsymbol{\zeta})$ is called an {\it admissible} surface. We denote it by    
 $(\Sigma,g, {\bf x},\boldsymbol{\zeta})$ if ${\bf x}:=(x_1,\dots,x_m)$ are the marked points and $\boldsymbol{\zeta}:=(\zeta_1,\dots,\zeta_b)$ are the parametrisations of the boundary connected components. We shall also denote the outgoing boundaries by $\mc{C}_j^+$ and the incoming ones by $\mc{C}_j^-$.
Notice that the parametrisation $\zeta_j$ is entirely determined by the point $p_j$ and the orientation choice $\sigma_j$ once $g$ has been fixed as above, since it is a parametrisation by arclength with 
respect to the metric $g$.

Finally, if $(\Sigma,g, {\bf x},\boldsymbol{\zeta})$ and $(\Sigma',g', {\bf x}',\boldsymbol{\zeta}')$ are two connected admissible surfaces with at least one incoming boundary circle $\mc{C}_j$ for $\Sigma$ and one outgoing boundary circle $\mc{C}_k'$ for $\Sigma'$, 
then the surface $\Sigma\#_{jk}\Sigma':=(\Sigma\sqcup \Sigma')_{jk}$ equipped with the glued metric denoted $g\# g'$ yields a new admissible surface by simply taking 
$(\Sigma\#_{jk}\Sigma',g\# g',({\bf x},{\bf x}'),(\boldsymbol{\zeta},\boldsymbol{\zeta}'))$.

\subsection{Moduli and Teichm\"uller space of Riemann surfaces with marked points}\label{sub:moduli}
 
\subsubsection{Moduli space} 
The \emph{moduli space} $\mc{M}_{\bf g}$ of genus ${\bf g}\geq 0$ is the set of Riemann surfaces 
with genus ${\bf g}$ up to biholomorphism: more precisely, each diffeomorphism $\phi\in {\rm Diff}^+(\Sigma)$ preserving orientation induces a complex structure $\phi^*J$ by pull-back, and two such complex structures are said equivalent or biholomorphic; $\mc{M}_{\bf g}$ is the set of such equivalence classes. 
If ${\bf g}\geq 2$, it is a non-compact complex orbifold of dimension $3{\bf g}-3$ which can be compactified, and if ${\bf g}=1$ it has complex dimension $1$ and can be represented as the modular surface ${\rm PSL}_2(\Z)\backslash \H^2$. In genus ${\bf g}=0$, the moduli space is a singleton.
 The moduli space $\mc{M}_{{\bf g},m}$ of Riemann surfaces with $m$ marked points $(x_1,\dots,x_m)$  is the set of closed Riemann surfaces $(\Sigma,J)$ of genus ${\bf g}$ up to biholomorphism $\phi$ satisfying $\phi(x_j)=x_j$ for all $j=1,\dots,m$. This is a complex orbifold of complex dimension $3{\bf g}-3+m$ (including when ${\bf g}=0$)   for $3{\bf g}-3+m\geq 0$. There exists a finite cover 
$\widetilde{\mc{M}}_{{\bf g},m}$ of complex dimension $3{\bf g}-3+m$ which is a smooth complex manifold (\cite[Theorem in Section 1.2]{Hinich-Vaintrob}).

\subsubsection{Teichm\"uller space}
The universal cover of $\mc{M}_{\bf g}$ (resp. $\mc{M}_{{\bf g},m}$) is the \emph{Teichm\"uller space} $\mc{T}_{\bf g}$ (resp. $\mc{T}_{{\bf g},m}$), that is the set of complex structures up to diffeomorphisms isotopic to the identity ${\rm Id}$ (resp. isotopic to ${\rm Id}$ relative to 
${\bf x}=\{x_1,\dots,x_m\}$, i.e. isotopies fixing ${\bf x}$). It is an open complex manifold of dimension $3{\bf g}-3+m$. The \emph{mapping class group} ${\rm Mod}_{\bf g}$   is the group ${\rm Mod}_{\bf g}:={\rm Diff}^+(\Sigma)/{\rm Diff}_0(\Sigma)$
where ${\rm Diff}_0(\Sigma)$ is the group of diffeomorphisms isotopic to the identity, and one has 
$\mc{M}_{\bf g}=\mc{T}_{\bf g}/{\rm Mod}_{\bf g}$. Similarly one defines ${\rm Mod}_{{\bf g},m}:={\rm Diff}^+_{\bf x}(\Sigma)/{\rm Diff}_{0,{\bf x}}(\Sigma)$ by adding that the diffeomorphisms are required to fix the marked points $(x_1,\dots,x_m)$, and one has $\mc{M}_{{\bf g},m}=\mc{T}_{{\bf g},m}/{\rm Mod}_{{\bf g},m}$. The manifold $\widetilde{\mc{M}}_{{\bf g},m}$ can be realised as a quotient $\mc{T}_{{\bf g},m}/G_{{\bf g},m}$ for some finite index subgroup $G_{{\bf g},m}$ of ${\rm Mod}_{{\bf g},m}$.\\

A complex structure is also equivalent to a smooth conformal class $[g]=\{e^{\omega}g\,|\, \omega\in C^\infty(M)\}$ of Riemannian metrics (the Hodge star operator $\star_g$ on $1$-forms is dual to $J$ and is conformally invariant), together with an orientation.
The uniformization theorem states that, on $\Sigma$ oriented with genus ${\bf g}$ and $m$ marked points with $2{\bf g}+m-3\geq 0$, for each smooth  Riemannian metric $g$ on $\Sigma$, there is  a unique complete hyperbolic metric $g_0$ (i.e. with scalar curvature $K_{g_0}=-2$) on $\Sigma\setminus {\bf x}$  in the conformal class $[g]$ and a unique $w\in C^\infty(\Sigma\setminus {\bf x})$ such that $ g=e^{w}g_0$.
If $m=0$ the metric $g_0$ and $w$ are smooth on $\Sigma$ while if $m\geq 1$, there is a neighborhood $U_i$ of each $x_i$ so that $(U_i\setminus \{x_i\},g_0)$ is isometric to 
a \emph{hyperbolic cusp} (also called puncture) 
\[ \D_r^*:=\{z\in \C\,|\, |z|\in (0,r)\}, \textrm{ with metric } \frac{|dz|^2}{|z|^2(\log |z|)^2}.\]

Using this fact, the Teichm\"uller space $\mc{T}_{{\bf g},m}$ (and thus $\mc{M}_{{\bf g},m}$) can be represented by a (finite dimensional) family of complete hyperbolic metrics. If $\{g_\tau \,|\, \tau\in \mc{M}_{{\bf g},m}\}$ is such a parametrisation by complete hyperbolic metrics, for each metric $g$ on $\Sigma$, there is $\phi\in {\rm Diff}_{\bf x}^+(\Sigma)$, $\tau\in \mc{M}_{{\bf g},m}$  and $w\in C^\infty(\Sigma)$ such that
\[ \phi^*g=e^{w}g_\tau.\]  

\subsubsection{Decomposition into blocks}\label{geometricblocks}
One can decompose the hyperbolic surface $(\Sigma,J,g_0)$ described above into surfaces with boundaries that we call blocks.
This decomposition can be viewed at three levels:\\ 
1) as a topological decomposition of $\Sigma$ into topological building blocks, i.e. surfaces with boundaries and marked points, obtained by choosing a collection of $3{\bf g}-3+m$ non intersecting free homotopy classes $[c_i]$ on $\Sigma\setminus {\bf x}$, that can be realised by any choice of simple curves $\mc{C}_i$ in the classes $c_i$, in a way that the connected components of $\Sigma\setminus \bigcup_i\mc{C}_i$ belong to the following cases
\begin{itemize}
\item a pair of pants: any surface diffeomorphic to the sphere $\mathbb{S}^2$ with $3$ disjoint disks removed
\item a cylinder with $1$ marked point: any surface diffeomorphic to the sphere $\mathbb{S}^2$ with $2$ disjoint disks and $1$ point removed 
\item a disk with $2$ marked points: any surface diffeomorphic to the sphere $\mathbb{S}^2$ with $1$ disk and $2$ disjoint points removed.
\end{itemize}
2) As a decomposition of $(\Sigma,J)$ into complex building blocks, i.e. Riemann surfaces with real analytic boundaries and marked points, where the complex structures of the topological blocks are induced by $J$ and a choice of real analytic simple curves $\mc{C}_i\subset \Sigma\setminus {\bf x}$ in the classes $[c_i]$, in a way that the connected components of $\Sigma\setminus \cup_i\mc{C}_i$ belong to the following cases
\begin{itemize}
\item a complex pair of pants, i.e. a Riemann surface biholomorphic to the Riemann sphere $\hat{\C}$ with $3$ disjoint disks removed
\item a complex cylinder with $1$ marked point, i.e.  a Riemann surface biholomorphic to the Riemann sphere $\hat{\C}$ with $2$ disjoint disks and $1$ point removed 
\item a complex disk with $2$ marked points, i.e. a Riemann surface biholomorphic to the Riemann sphere $\hat{\C}$ with $1$ disk and $2$ disjoint points removed.
\end{itemize}
3) As a decomposition into geometric building blocks, which are hyperbolic surfaces (with Gauss curvature $-1$) with geodesic boundary, where the hyperbolic metric on the topological blocks are induced by $g_0$ and choosing the curves $\mc{C}_i$ in $[c_i]$ to be the unique geodesics in those classes. The connected components of $\Sigma\setminus \bigcup_i\mc{C}_i$ belong to the following cases
\begin{itemize}
\item a hyperbolic pair of pants, i.e. a hyperbolic surface diffeomorphic to $\mathbb{S}^2$ with $3$ disjoint disks removed, such that the boundary are three simple geodesics
\item a hyperbolic surface with $2$ geodesic boundary circles and $1$ puncture/cusp, diffeomorphic to the sphere $\mathbb{S}^2$ with $2$ disjoint disks and $1$ point removed 
\item a hyperbolic surface with $1$ geodesic boundary circle and $2$ punctures/cusps, diffeomorphic to $\mathbb{S}^2$ with $1$ disk and $2$ disjoint points removed.
\end{itemize}

Of course, the geometric decomposition 3) induces a complex decomposition 2) which itself induces a topological decomposition 1). Conversely, a topological decomposition 1) and a hyperbolic metric on $\Sigma$ induces a geometric decomposition by choosing the curves $\mc{C}_i$ to be the unique geodesics in the homotopy classes $[c_i]$.
 
\subsubsection{Fenchel-Nielsen coordinates} 
Fixing a topological decomposition into building blocks as in 1) and using hyperbolic metrics to parametrise Teichm\"uller space $\mc{T}_{{\bf g},m}$ with $2{\bf g}-2+m>0$, there are global geometric coordinates called \emph{Fenchel-Nielsen coordinates} obtained as follows.
Choose a fixed hyperbolic surface $(\Sigma,g_0)$ and topological decomposition as in 1), this induces a decomposition as in 3) into $2{\bf g}-2+m$ geometric blocks $\mc{P}_j$ using the simple geodesics $\mc{C}_i$ in $[c_i]$ for $g_0$. 
We can  then construct a new family of hyperbolic metrics on $\Sigma$ as follows.
First, choose a distinguished point $p_i$ on each $\mc{C}_i$ and call $p_i(j)$ the corresponding point on $\mc{P}_j$ if $\mc{C}_i\subset \pl \mc{P}_j$. For each $(\ell_1,\dots,\ell_{3{\bf g}-3+m})\in (0,\infty)^{3{\bf g}-3+m}$, there is a unique (up to isometry) hyperbolic metric $g_j$ on each $\mc{P}_j$, with cusps at the marked points of $\mc{P}_j$ and such that each boundary component $\mc{C}_i$ of $\mc{P}_j$ is a geodesic for $g_j$ with length $\ell_i$. 
For each $\theta_j\in \R$, one can glue each 
$(\mc{P}_j,g_j)$ to $(\mc{P}_{j'},g_{j'})$ (possibly with $j=j'$) along the boundary component $\mc{C}_i$, provided $\mc{C}_i\in \pl \mc{P}_j\cap \pl\mc{P}_{j'}$ by identifying $p_i(j)$ to $R_{\theta_i}(p_{i}(j'))$ where $R_{\theta}$ is the translation of $\theta$ along the geodesic $\mc{C}_i$ parametrised by arclength. The collection $(\ell_i,\theta_i)_{i=1,\dots,3{\bf g}-3+m}$ 
produces coordinates on Teichm\"uller space called Fenchel-Nielsen coordinates associated to the family of chosen free homotopy classes represented by $\mc{C}_i$. Notice that a translation of $\theta_i\in 2\pi \Z$ produces different surfaces as elements in $\mc{T}_{{\bf g},m}$ but the same surface as element of $\mc{M}_{{\bf g},m}$.

\subsubsection{A graph representing the decomposition of $\Sigma$ into building blocks.}\label{subsubsec:graphdec}
The decomposition into $2{\bf g}-2+m$ topological blocks $\mc{P}_j$ glued along $3{\bf g}-3+m$ simple curves corresponding to simple free homotopy classes can be represented by a graph, called an \emph{admissible graph} for $\Sigma$, as follows. First we define the general notion of multigraph with phantom edges. 

\begin{definition}\label{defmultigraph}
An oriented multigraph $\mc{G}$ with $j_0\in\N$ vertices is a finite collection of
\begin{itemize}
\item vertices $V=(v_j)_{j=1,\dots,j_0}$ where $v_j=\{j\}\times \{1,\dots,b_j\}$ is a discrete set with $b_j$ elements $v_{jk}=(j,k)$ for  $1\leq k\leq b_j\in \N$,
\item phantom vertices $V^*=(v_j^*)_j\subset \N$,
\item linking edges $E_{\rm link}=(e_i)_i$ with $e_i\in (W\times W)\setminus {\rm diag}(W)$ where $W:=\bigcup_{j=1}^{j_0}v_j$ and ${\rm diag}(W)$ is the diagonal of $W\times W$, 
\item phantom edges $E_{\rm ph}=(e_i^*)_i$ with $e_i^*\in (W\times V^*)\cup (V^*\times W)$.
\end{itemize}
We define ${\rm v}=({\rm v}_1,{\rm v_2}):E_{\rm link}\to W^2$  the projection on the first and second components, so that ${\rm v}(e_i)=(v_{j_1k},v_{j_2k'})$ represents the vertices $v_{j_1}$ and $v_{j_2}$ linked by $e_i$ and a labelling $(k,k')$. 
A multigraph without orientation is defined the same way but defining the edges modulo the involution $(w,w')\mapsto (w',w)$ on $W\times W$. 
A loop is a linking edge of the form $(v_{jk},v_{jk'})$. 
\end{definition}  

To the decomposition into $2{\bf g}-2+m$ topological blocks $\mc{P}_j$ with $b_j\leq 3$ boundary components,  we associate an oriented multigraph $\mc{G}$, that we call an admissible graph, defined as follows: fix a labelling $(\pl_k\mc{P}_j)_{k\leq b_j}$ of the boundary circles of $\mc{P}_j$ and a labelling  $(x_{jk})_{k\leq 3-b_j}$ of the marked points in $\mc{P}_j$, then  
\begin{itemize}
\item each topological block $\mc{P}_j$ is represented by a vertex $v_j$ with 
$b_j=3$, each element $v_{jk}\in v_j$ is associated to the boundary circle $\pl_k\mc{P}_j$ or to a marked point $x_{jk}$ in $\mc{P}_j$,
\item a linking edge $e_i=((j,k),(j',k'))$ is associated to a curve $\mc{C}_i$ if and only if the simple curve $\mc{C}_i$ bounds $\mc{P}_j$ at $\pl_k\mc{P}_j$ 
and $\mc{P}_{j'}$ at $\pl_{k'}\mc{P}_{j'}$, and $\mc{C}_i$ has outgoing orientation in $\mc{P}_j$ and incoming orientation in $\mc{P}_{j'}$
\item a phantom edge $e_i^*=(v_{jk},v^*_i)$ is associated to a marked point $x_i=x_{jk}\in \mc{P}_j$. 
\end{itemize}

 \begin{figure}[h] 
\centering
\subfloat[Surface with genus 3 and 3 marked points]{\includegraphics[width=.4\textwidth]{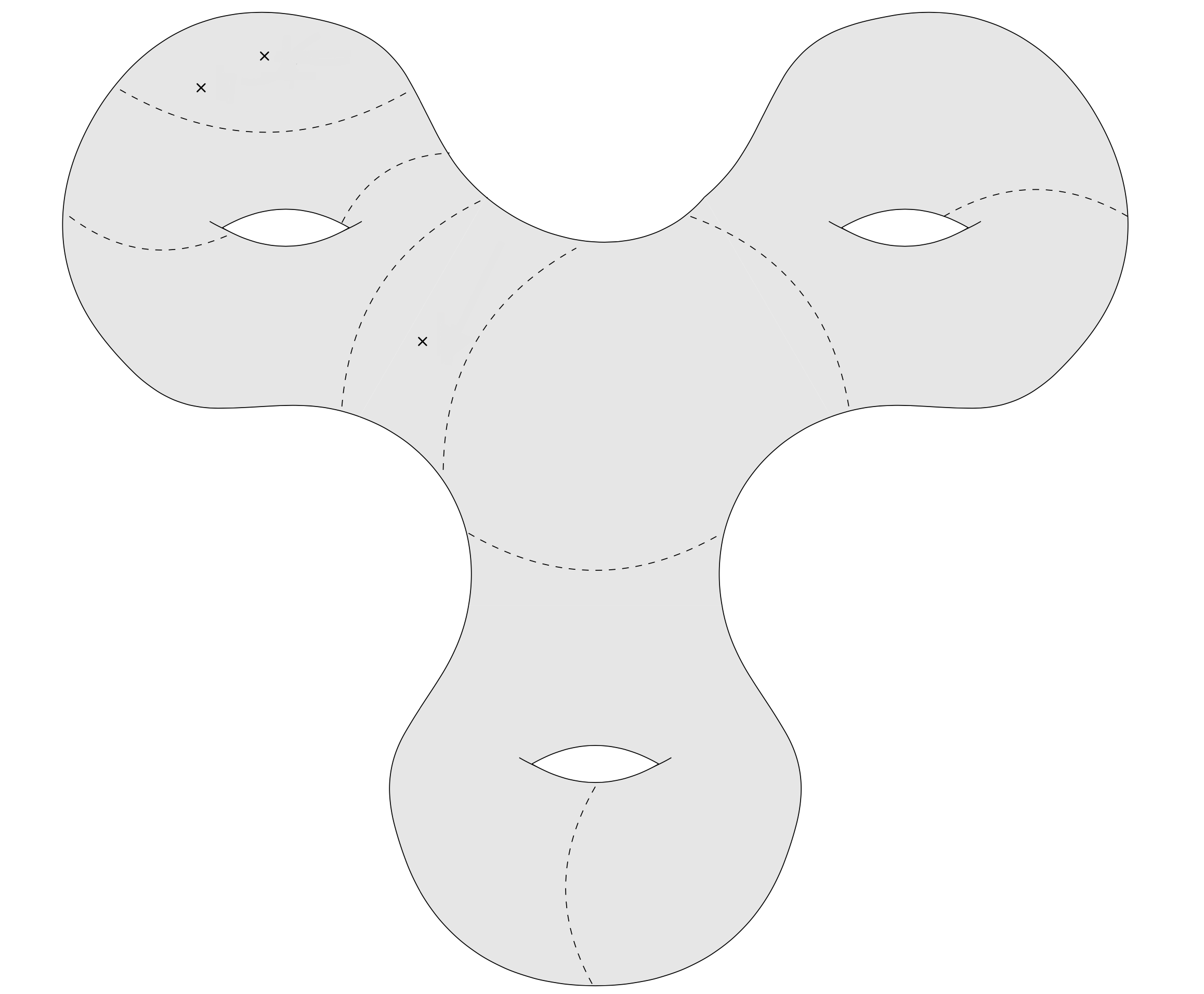} 
\begin{tikzpicture}[remember picture,overlay,xscale=0.75,yscale=0.75] 
\node at (-5.5,5){$x_1$};
\node at (-7,6.5){$x_2$};
\node at (-6.5,6.8){$x_3$};
\end{tikzpicture}
}
\,\hspace{1cm}
\subfloat[Corresponding oriented multigraph (vertices in yellow, phantom vertices or edges in blue) with $e_1=((1,1),(1,2))$, $ e_2=((1,3),(2,1))$, $e_3=((2,2),(3,1)) $, $e_4=((3,2),(3,3))$, $e_5=((2,3),(4,1))$, $e_6=((4,3),(5,1))$, $e_7=((5,2),(6,2))$, $e_8=((5,3),(6,1))$, $e_9=((6,3),(7,1))$.]{
\begin{tikzpicture}[xscale=0.75,yscale=0.75] 
\tikzstyle{sommet}=[circle,draw,fill=yellow,scale=0.7] 
\tikzstyle{phantomV}=[circle,draw,fill=cyan,scale=0.6] 
\node  (P) at (8,0.5){};
\node at (4.5,-0.3){$e_1$};
\node at (1.4,4.6){$e_7$};
\node at (1.4,3.4){$e_8$};
\node at (6.3,4.2){$e_4$};
\node[sommet] (P1) at (4,0.5){ $\mathbf{\mathcal{P}_1}$};
\node[sommet] (P2) at (4,2){ $\mathbf{\mathcal{P}_2}$};
\node[sommet] (P3) at (5,3){ $\mathbf{\mathcal{P}_3}$};
\node[sommet] (P4) at (3,3){ $\mathbf{\mathcal{P}_4}$};
\node[sommet] (P5) at (2,4){ $\mathbf{\mathcal{P}_5}$};
\node[sommet] (P6) at (0.5,4){ $\mathbf{\mathcal{P}_6}$};
\node[sommet] (P7) at (-1,4){ $\mathbf{\mathcal{P}_7}$};
\node[phantomV] (Z2) at (-2.3,2.7){ $x_2$};
\node[phantomV] (Z3) at (-2.3,5){ $x_3$};
\node[phantomV] (Z1) at (4,4.2){ $x_1$};
\draw[->,>=latex] (P1) -- (P2)  node[midway, right]{$e_2$};
\draw[->,>=latex] (P2) -- (P3)  node[midway, below right]{$e_3$};
\draw[->,>=latex] (P2) -- (P4) node[midway, above right]{$e_5$};
\draw[->,>=latex] (P4) -- (P5) node[midway, below left]{$e_6$};
\draw[->,>=latex] (P5) to[bend left] (P6);
\draw[->,>=latex] (P5) to[bend right] (P6) ;
\draw[->,>=latex] (P6) -- (P7) node[midway, below ]{$e_9$};
\draw[->,>=latex,blue] (P7) -- (Z2) node[midway, above left]{$e^*_2$}  ;
\draw[->,>=latex,blue] (P7) -- (Z3) node[midway,above right]{$e^*_3$}  ;
\draw[->,>=latex,blue] (P4) -- (Z1) node[midway,right]{$e^*_1$} ;
\draw[->,>=latex] (P1) to[out=250,in=180] (4,-0.5);
\draw[] (4,-0.5) edge[out=0, in=290] (P1) ;
\draw[->,>=latex] (P3) to[out=30,in=300] (6,4);
\draw[] (6,4) edge[out=120, in=60] (P3);
\end{tikzpicture}
}
\caption{Multigraphs associated to marked Riemann surfaces}
\label{fig:partition}
\end{figure}

The decomposition represented by the graph $\mc{G}$ is purely topological and the choice of $\mc{C}_i$ up to free homotopy gives the same decomposition. If now we fix a conformal class $[g]$ on $\Sigma$, i.e. a complex structure $J$, the choices of the curves $\mc{C}_i$ on $(\Sigma,J)$ induce 
complex structures on the topological building blocks $\mc{P}_j$.

Defining an order on $\bigcup_j v_j$ by setting $(j,k)<(j',k')$ if and only if 
$j<j'$ or $(j=j',k<k')$, we can always choose the orientation of the graph associated to a topological decomposition of $\Sigma$ so that $e_i=(v_{jk},v_{j'k'})$ if $v_{jk}<v_{j'k'}$, and
each phantom edge $e_i$ associated to a marked point $z_i\in \mc{P}_j$ will be oriented by an arrow with foot at a vertex $v_{j}$, that is $e_i^*=(v_{jk},v_i^*)$ for some $k$. Define the 
orientation sign $\sigma_{jk}$ of $v_{jk}$ by 
\[\forall i,\,\,  \sigma_{{\rm v}_2(e_i)}:=+1 ,\quad \sigma_{{\rm v}_1(e_i)}:=-1.
\] 
The orientation of the edges will be relevant below when we consider gluing surfaces with boundary.

\subsection{Decomposition of Riemann surfaces and plumbing coordinates.}  \label{sub:plumbing}
We have seen that one can parametrise Teichm\"uller space by Fenchel-Nielsen coordinates associated to a decomposition of $\Sigma$ into geometric (hyperbolic) building blocks. These coordinates are not holomorphic on Teichm\"uller space, we shall then use another description due 
to Earle-Marden, Kra \cite{marden341,EaMa12,Kra} and Hinich-Vaintrob \cite{Hinich-Vaintrob} of a certain coordinate system on Teichm\"uller space $\mc{T}_{{\bf g},m}$, based on the \emph{plumbing construction}.

\subsubsection{Plumbings}\label{plumbings} A plumbing with parameter $q\in \D$ of two surfaces $\Sigma_1, \Sigma_2$  with two distinguished marked points $x_1,x_2$ is an operation that creates a new surface $\Sigma$ by removing neighborhoods biholomorphic to $\D$ of the marked points and gluing their boundary together in a certain way involving $q$. 
We now give the precise construction, following \cite[Appendix III]{Kra}. 
Consider two Riemann surfaces $\Sigma_1,\Sigma_2$ with distinguished marked points $x_1\in \Sigma_1$ and $x_2\in \Sigma_2$. Let $\mc{D}_1\subset \Sigma_1$ and $\mc{D}_2\subset \Sigma_2$ be two topological disks containing $x_j$ in their interior, and let
\[ \omega_j: \mc{D}_j\to \bbar{\D}\]
be complex coordinates, that are assumed to be analytic up to the boundary with $\omega_j(\pl \mc{D}_j)=\T$ and $\omega_j(x_j)=0$. Let $\mc{D}_j(q)=\{m\in\mc{D}_j \,|\, |\omega_j(m)|< |q|\}$, $\Sigma^*_j:=\Sigma_j\setminus \mc{D}_j(q)$ and consider the Riemann surface $\Sigma$ obtained by 
\[ \Sigma = \Sigma^*_1\cup \Sigma^*_2/ \sim_q\]
where the equivalence relation $\sim_q$ is defined by 
\[ \forall m_1\in \Sigma_1^*, m_2\in \Sigma_2^*,\quad  m_1\sim_q m_2 \iff m_j\in \mc{D}_j\setminus \mc{D}_j(q) \textrm{ and }
\omega_1(m_1)\omega_2(m_2)=q.\]
As explained in \cite[Appendix III]{Kra}, the obtained Riemann surface $\Sigma$ depends only on the choice of $q$, $(\mc{D}_1,\omega_1)$ and $(\mc{D}_2,\omega_2)$.
If $\mathbb{A}_q:=\{z\in \C,|\, |z|\in [|q|,1]\}$ denotes the annulus of modulus $|q|$, the procedure above amounts to cut a disk $\mc{D}_j$ near each $x_j$ and glue $\Sigma_1\setminus \mc{D}_1$ to $\Sigma_2\setminus \mc{D}_2$ by inserting an annulus/cylinder of modulus $|q|$, with a twist ${\rm arg}(q)$, between the boundary components $\pl \mc{D}_j$ of $\Sigma_j\setminus \mc{D}_j$. Alternatively, $\Sigma$ can also be described in terms of the gluing procedure mentionned in Section \ref{sub:gluing} as follows: fix $\delta\in [|q|,1]$ take $\Sigma^{**}_1:=\Sigma_1\setminus \mc{D}_1(\delta)$ and $\Sigma^{**}_2:=\Sigma_2\setminus \mc{D}_2(|q|/\delta)$, take $p_1=\omega_1^{-1}(\delta)$ and $p_2=\omega_2^{-1}(q/\delta)$ and glue $\Sigma^{**}_1$ to $\Sigma^{**}_2$ along $\pl \Sigma^{**}_1$ using the coordinates $\omega_j$ by identifying the point $p_1$ to $p_2$. The result is that an annulus of modulus $|q|$  embedded into $\Sigma$.

\begin{figure}
\centering
\begin{tikzpicture}
 \node[inner sep=10pt] (F1) at (0,0)
    {\includegraphics[width=.6\textwidth]{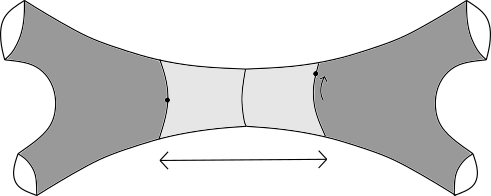}};
\node (F) at (3,0){$\Sigma_1\setminus \mc{D}_1$};
\node (F) at (-3,0){$\Sigma_2\setminus \mc{D}_2$};
\node (F) at (-1,0){$\mathbb{A}_q$};
\node (F) at (-2,0.5){$\pl \mc{D}_2$};
\node (F) at (2,-0.5){$\pl \mc{D}_1$};
\node (F) at (1.8,0.2){$\theta$};
\node (F) at (0,-1.6){$t$};
\end{tikzpicture}
  \caption{The plumbing with parameter $q=e^{-t+i\theta}$, viewed as gluing an annulus ${\mathbb{A}_q}$ to $\pl \mc{D}_1$ and $\pl \mc{D}_2$ with a twist of angle $\theta$. The length for the flat metric $|dz|^2/|z|^2$ of the annulus is $t$.}
\label{gluedcylinder}
\end{figure}

This procedure is a gluing of Riemann surfaces, but this can also be done in the Riemannian setting by choosing a metric compatible with the complex structure on $\Sigma_1\setminus \{x_1\}$ and $\Sigma_2\setminus\{x_2\}$ that descends to $\Sigma$: one can choose any metric $g_j$ on $\Sigma_j$ so that 
\[ \omega_j^*\Big(\frac{|dz|^2}{|z|^2}\Big)=g_j \textrm{ near }\mc{D}_j\setminus \{x_j\}\]
and $g_1$ glues smoothly with $g_2$ in the process.
One can also perform the plumbing procedure described above on a single Riemann surface $\Sigma$ with two distinguished marked point $x_1,x_2$ by choosing the disks $\mc{D}_1$ and $\mc{D}_2$ containing $x_1$ and $x_2$ to be disjoint in $\Sigma$.

\subsubsection{Compactification of moduli space.} There is a compactification $\bbar{\mc{M}}_{{\bf g},m}$ of $\mc{M}_{{\bf g},m}$ due to Deligne and Mumford, we call $\pl \bbar{\mc{M}}_{{\bf g},m}:=\bbar{\mc{M}}_{{\bf g},m}\setminus \mc{M}_{{\bf g},m}$ the boundary of $\mc{M}_{{\bf g},m}$.
The space $\bbar{\mc{M}}_{{\bf g},m}$ has a topological structure compatible with that of $\mc{M}_{{\bf g},m}$ and it is a compact set. 
We follow Wolpert's description \cite[Sections 2.2, 2.3 and 2.4]{10.4310/jdg/1214444322}. The (class of) Riemann surfaces in $\pl \bbar{\mc{M}}_{{\bf g},m}$ are \emph{surfaces with nodes}: a surface with nodes is a complex (singular) surface $\Sigma$ such that each point $p$ has a neighborhood isomorphic either to $\D\subset \C$ 
or, if not, to $\{(z,w)\in \D^2\subset \C^2\,|\, zw=0\}$; in the last case $p$ is called a \emph{node} . Up to passing to the finite cover $\widetilde{\mc{M}}_{{\bf g},m}$ (which also admits a compactification),  a neighborhood of a $(\Sigma,J_0)\in \pl \bbar{\mc{M}}_{{\bf g},m}$ can be described using plumbing coordinates as follows. Let ${\bf n}=\{n_1,\dots,n_k\}$ be the set of nodes of $(\Sigma,J_0)$ and let $\Sigma_0=\Sigma\setminus {\bf n}$, viewed as a Riemann surface $(\Sigma_0,J_0)$ with $m+2k$ punctures (each node produces two punctures) admitting a complete hyperbolic metric, possibly with several connected components. The product of Teichm\"uller spaces of the connected components of $\Sigma_0$ is a complex manifold and let $s=(s_1,\dots,s_d)$ be local complex coordinates on this space near $(\Sigma_0,J_0)$. We denote by $(\Sigma,J(s))$ the corresponding family of Riemann surfaces. This family can be obtained by solving Beltrami equations with Beltrami differentials $\sum_{j=1}^d s_j\nu_j$ where $\nu_j$ are Beltrami differentials supported outside a neighborhood of the $2k$-punctures corresponding to the original nodes.  
Near each node $n_j\in {\bf n}$ there is a neighborhood in $\Sigma_0$ of the form $\mc{D}_{1j}\sqcup \mc{D}_{2j}$ where $\mc{D}_{ij}$ are biholomorphic to the pointed disk $\D^*=\{z\in \C\,|\, 0<|z|<1\}$. Attach to each node $n_j$ a parameter $q_j\in \D$ and perform the plumbing construction explained above for $(\Sigma,J(s))$ using the pairs of disks $\mc{D}_{1j},\mc{D}_{2j}$ and their biholomorphism to $\D^*$ (taking the $\mc{D}_{ij}$ disjoints in $\Sigma\setminus {\bf n}$ and also disjoint from $\cup_\ell {\rm supp}(\nu_\ell)$). For $q=(q_1,\dots,q_{k})$ we then obtain a family 
$(\Sigma,J(s,q))$ of Riemann surfaces, and this family forms a neighborhood $U$ of $(\Sigma,J_0)$ in the compactification of $\widetilde{\mc{M}}_{{\bf g},m}$, and $(s,q)$ are local complex coordinates in $U$. The interior $U\cap \widetilde{\mc{M}}_{{\bf g},m}$ corresponds to $q_j\not=0$ for all $j$.

\subsubsection{Plumbings parameters associated to a graph and a family of  analytic simple curves.}\label{subsubsec:plumbingparameters} 
Let $(\Sigma,J_0)$ be a Riemann surface with $m$ marked points, viewed as an element in $\mc{T}_{{\bf g},m}$. Using the plumbing method described above, for a given admissible graph $\mc{G}$ corresponding to a   decomposition of $\Sigma$ into topological building blocks and for a choice of simple analytic curves $\mc{C}_i$ on $\Sigma$ in the free homotopy classes $[c_i]$ encoded by the graph, we can associate $3{\bf g}-3+m$ complex parameters called plumbing parameters as follows. 

For $\delta\in (0,1)$ close enough to $1$, and each $i$ we can take a biholomorphic map 
$\omega_i:V_i\to \{z\in \C\,|\, |z|\in [\delta,\delta^{-1}]\}$ where $V_i$ is an annular neighborhood of $\mc{C}_i$ so that $\omega_i(\mc{C}_i)=\T$, and denote by $p_i:=\omega_i^{-1}(1)$.
Choosing $\omega_i$ and $\mc{C}_i$ is equivalent to choosing an analytic parametrisation $\zeta_i:\T\to \mc{C}_i$ (here $\zeta_i=\omega_i^{-1}|_{\T}$). Choose the $V_i$ so that $V_i\cap V_{i'}=\emptyset$ if $i\not=i'$.
On a complex building block $\mc{P}_j$, for each $i$ such that $V_i\cap \mc{P}_j\not=\emptyset$ we use the holomorphic chart $\omega_{i,\delta}: z\in V_i\mapsto \delta\omega_i(z)$ if $\omega_i(V_i\cap\mc{P}_j)\subset \D^c$ (or $\omega_{i,\delta}: z\in V_i\mapsto \delta/\omega_i(z)$ if $\omega_i(V_i\cap\mc{P}_j)\subset \bbar{\D})$) 
to map $V_i\cap \mc{P}_j$ to $\mathbb{A}_\delta$ in a way that the curve $\mc{C}_i$ is mapped to the circle $\{|z|=\delta\}$. 
Let $k=k(i)\leq b_j$ be defined by $\pl_{k}\mc{P}_j=\mc{C}_i$, we can then glue to $\mc{P}_j$ the disk $\mc{D}_{jk}(\delta)=\D_\delta$ of radius $\delta$ at each boundary circle $\mc{C}_i=\pl_k\mc{P}_j\subset \mc{P}_j$ by using the chart $\omega_{i,\delta}$, the resulting Riemann surface $\hat{\mc{P}}_j$ is a sphere $\hat{\C}$ with $3$ distinguished marked points given by the union of ${\bf x}\cap \mc{P}_j$ with the centers of the glued disks. When $\mc{P}_j$ is a pair of pants, we have glued $3$ disks, when $\mc{P}_j$ is a cylinder with $1$ marked point we have glued $2$ disks and when $\mc{P}_j$ is a disk with $2$ marked points we have glued $1$ disk. Let $z_1:=-1/2,z_2=1/2,z_3:=i\sqrt{3}/2$ three fixed points in $\hat{\C}$ (we have chosen these so that they belong to $\D$ and $|z_i-z_j|=1$, thus $P(z_1,z_2,z_3)=1$ in \eqref{defp0}, which is convenient for later).
There is a biholomorphic map $\psi_j:\hat{\mc{P}}_j\to \hat{\C}$ so that $\psi_j^{-1}(z_1),\psi_j^{-1}(z_2)$ and $\psi_j^{-1}(z_3)$ are the $3$ marked points of $\hat{\mc{P}}_j$ and $\psi_j(\mc{P}_j)\subset \hat{\C}\setminus \{z_1,z_2,z_3\}$ is a bounded domain with $3$ 
simple analytic curves as boundaries.  Denote by 
$\mc{D}_{jk(i)}:=(V_i\cap \mc{P}_j)\cup \mc{D}_{jk(i)}(\delta)$ the disk in $\hat{\mc{P}}_j$, then the charts $\omega_{i,\delta}$ defined on $\mc{P}_j\cap V_i$ extend holomorphically to $\mc{D}_{jk(i)}$ (by construction of $\hat{\mc{P}}_j$) into maps denoted $\omega_{jk(i)}:\mc{D}_{jk(i)}\to \D$ so that $\omega_{jk(i)}(\pl_{k(i)}\mc{P}_j)=\{|z|=\delta\}$. 
We next use the plumbing construction to the families $\hat{\mc{P}}_j$ using these maps: more precisely by plumbing the disks $\mc{D}_{jk(i)}$ of $\hat{\mc{P}}_j$ with the disk $\mc{D}_{j'k'(i)}$ of $\hat{\mc{P}}_{j'}$ if $\mc{C}_i=\pl_{k(i)}\mc{P}_j=\pl_{k'(i)}\mc{P}_{j'}$ in $\Sigma$ (i.e. the edge $e_i$ of the admissible graph $\mc{G}$ is linking $v_j$ to $v_{j'}$). We perform these plumbings iteratively using a parameter $q_i\in \bbar{\D}$ for the $i$-th plumbing (corresponding to the edge $e_i$). We shall denote ${\bf q}=(q_1,\dots,q_{3{\bf g}-3+m})\in \bbar{\D}^{3{\bf g}-3+m}$ the \emph{plumbing parameter} associated to the graph $\mc{G}$ and the choice or parametrised analytic curves $\mc{C}:=\bigcup_i\mc{C}_i$ (or equivalently the curves $\mc{C}_i$ together with the charts $\omega_i$).
We denote the obtained Riemann surface by 
\[\Phi_{\mc{G},J_0,\boldsymbol{\zeta}}({\bf q})=(\Sigma,J({\bf q}))\in \mc{M}_{{\bf g},m},\]
where $\boldsymbol{\zeta}:=(\zeta_1,\dots,\zeta_{3{\bf g}-3+m})$ is the set of analytic parametrisations of 
the curves $\mc{C}_i$.
We notice that $\Phi_{\mc{G},J_0,\boldsymbol{\zeta}}({\bf q}_\delta)=(\Sigma,J_0)$ if ${\bf q}_\delta:=(\delta^2,\dots,\delta^2)$. The map ${\bf q}\mapsto \Phi_{\mc{G},J_0,\boldsymbol{\zeta}}({\bf q})\in \mc{M}_{{\bf g},m}$ is complex analytic but it is not in general an embedding from $\D^{3{\bf g}-3+m}\setminus\{0\}$, see \cite{Hinich2010}. However, it 
is known to be a biholomophism when restricted to a smaller polydisk $\D_\eps^{3{\bf g}-3+m}$ for $\eps>0$ small enough \cite{Hinich2010}. We also notice that changing $q_j$ by $q_j(t):=e^{2i\pi t}q_j$ for $t\in \R$ generates a one parameter family of surfaces in $\mc{M}_{{\bf g},m}$ whose lift to $\mc{T}_{{\bf g},m}$ contains the orbit of $\Phi_{\mc{G},J_0,\boldsymbol{\zeta}}({\bf q})$ by the cyclic group $\cjg T_j\cjd$ with generator the Dehn twist  
$T_j$ along $\mc{C}_j$, simply by restricting to $t\in \Z$. The coordinate $q_j$ is thus not well-defined (univalued) on $\mc{T}_{{\bf g},m}$ but it  is well-defined on the quotient $\mc{T}_{{\bf g},m}/\cjg T_j\cjd$.\\

In order to cover the whole moduli space with local charts where plumbing parameters are holomorphic coordinates,  we shall use a particular family of curves $\mc{C}_i$ and parametrisations $\zeta_i$, following the description given by Hinich-Vaintrob \cite[Section 5]{Hinich-Vaintrob}. Fix first an admissible graph $\mc{G}$. 
Each Riemann surface $(\Sigma,J_0)\in \mc{M}_{{\bf g},m}$ with $m$ marked points ${\bf x}=\{x_1,\dots,x_m\}$ has a unique complete hyperbolic metric $g_0$ with cusps/punctures at each $x_j$. 
We can decompose $(\Sigma,g_0,J_0)$ into geometric building blocks as explained above by choosing the curves $\mc{C}_i$ to be geodesics in the free homotopy classes $[c_i]$. We fix a point $p_i$ on $\mc{C}_i$ and use the parametrisation $\zeta_i:\T\to \mc{C}_i$ so that $\zeta_i(1)=p_i$ and 
$\zeta_i(e^{i\theta})$ is an arclength parametrisation of the geodesic with respect to $g_0$. It can be checked that $\zeta_j$ extends holomorphically: indeed, by Wolpert \cite{10.4310/jdg/1214444322}, 
if $2\pi/|\log t|$  is the length of $\mc{C}_i$ for some $t\in (0,1)$, there is a holomorphic map 
$\omega_i:V_i\to \{|z|\in [\delta,\delta^{-1}]\}$ for some $\delta\in(t^{1/2},1)$ (with $V_j$ annular neighborhood of $\mc{C}_i$) so that 
\[{\omega_i}_*g_0=F_t(z)|dz|^2/|z|^2, \quad F_t(z)=\frac{\pi^2}{(\log t)^2}\Big(\sin\Big(\pi\frac{\log |zt^{1/2}|}{\log t}\Big)\Big)^{-2}.\] 
Performing the plumbing construction as described above but using these curves $\mc{C}_i$ and the parametrisations $\zeta_i$ by arclength, and choosing the annular neighborhoods $V_i$ of $\mc{C}_i$ small enough so that $V_i\cap V_{i'}=\emptyset$ if $i\not=i'$, we construct a map 
\[ \Phi_{\mc{G},J_0,\boldsymbol{\zeta}}: \D^{3{\bf g}-3+m}\to U_{\mc{G},J_0}\subset \tilde{\mc{M}}_{{\bf g},m},\qquad \Phi_{\mc{G},J_0,\boldsymbol{\zeta}}(q)=(\Sigma,J(q)).\]
It is shown by Hinich-Vaintrob \cite{Hinich-Vaintrob} that $\Phi_{\mc{G},J_0,\boldsymbol{\zeta}}$  is a biholomorphism near $q(J_0):=(\delta^2,\dots,\delta^2)$, 
thus producing a complex chart of $\widetilde{\mc{M}}_{{\bf g},n}$ near $(\Sigma,J_0)$ using the plumbing coordinates.

One can also consider neighborhoods of the boundary of the compactification of $\widetilde{\mc{M}}_{{\bf g},m}$ in the same way: take a surface $(\Sigma,J_0)$ with nodes ${\bf n}=\{n_1,\dots,n_k\}$ and $m$ marked points ${\bf x}=\{x_1,\dots,x_m\}$. As described above for the compactification of moduli space, one can do a plumbing with small parameters $q_i$ at each node $n_i$ in a way that we open the nodes, the resulting Riemann surfaces having marked points ${\bf x}$ but no nodes: after this plumbing, the node $n_i$ is replaced by an annulus $A_i(q_i)$. Take a graph $\mc{G}$ representing a topological decomposition of such surfaces into $2{\bf g}-2+m$ blocks, where $k$ of the free homotopy classes represented by the edges of $\mc{G}$ are given by the non contractible closed simple curves generating $\pi_1(A_i(q_i))$.   
Now take the complete hyperbolic metric $g_0$ on $\Sigma_0=\Sigma\setminus ({\bf n}\cup {\bf x})$ compatible with $J_0$, which has cusps at ${\bf n}\cup {\bf x}$ and use the decomposition of $(\Sigma,g_0)$ into geometric building blocks using the graph. 
For each edge $e_i$ with $i>k$ corresponding to a closed geodesic $\mc{C}_i$ for $g_0$ bounding a geometric building block $\mc{P}_j$ of the decomposition, we can use holomorphic coordinates $\omega_i$ and the induced parametrisations $\zeta_i$ of $\mc{C}_i$ as above and apply the same procedure of adding a pointed disk to $\mc{P}_j$ to obtain a maximal family of $2{\bf g}-2+m$ Riemann spheres $\hat{\mc{P}}_j$ with $3$-marked points. 
We can then perform the plumbing construction at all the marked points corresponding to some edges $\mc{C}_i$ according to the pairing rules given by the graph $\mc{G}$, using plumbing parameters
$(q_1,\dots,q_k,q_{k+1},\dots,q_{3{\bf g}-3+m})$ where the $k$-first parameters are the plumbings at the vertices corresponding to the nodes $(n_i)_{i\leq k}$. This produces a map $\Phi_{\mc{G},J_0,\boldsymbol{\zeta}}$ as above and the surface $(\Sigma,J_0)$ is given by 
$\Phi_{\mc{G},J_0,\boldsymbol{\zeta}}({\bf q}_0)$ with ${\bf q}_0=:(0,\dots,0,\delta^2,\dots,\delta^2)$. Hinich-Vaintrob \cite[Section 5]{Hinich-Vaintrob} prove that 
$\Phi_{\mc{G},J_0,\boldsymbol{\zeta}}$ is a biholomorphism on a neighborhood of ${\bf q}_0$ in ${\D^*}^{3{\bf g}-3+m}$. Since the compactification of $\mc{M}_{{\bf g},m}$ and of its smooth finite cover $\widetilde{\mc{M}}_{{\bf g},m}$ are compact, we can extract a finite number of open sets $(U_i)_{i=1,\dots,I}\subset \widetilde{\mc{M}}_{{\bf g},m}$, ${\bf q}^{(i)}\in \D^{3{\bf g}-3+m}$, $\eps_i>0$, surfaces 
$(\Sigma, J_i)\in U_i$, graphs $\mc{G}_i$ and choice of parametrised analytic curves represented by $\boldsymbol{\zeta}_i$, such that 
\[ \Phi_{\mc{G}_i,J_i,\boldsymbol{\zeta}_i}: \{{\bf q}\in {\D^*}^{3{\bf g}-3+m}\,|\, |{\bf q}-{\bf q}^{(i)}|<\eps_i\}\to U_i\] 
is a biholomorphism.

%%%%%%%%%%%%%%%%%%%%%%%%%%%% %%%%%%%%%%%%%%%%%%%%%%%%%%%%%%%%%%%%%% %%%%%%%%%%%%%%%%%%%%%%%%%%%%%%%%%%%%%% %%%%%%%%%%%%%%%%%%%%%%%%%%%%%%%%%%%%%% %%%%%%%%%%%%%%%%%%%%%%%%%%%%%%%%%%%%%% 
\section{Amplitudes}\label{probamp}
 %%%%%%%%%%%%%%%%%%%%%%%%%%%%%%%%%%%%%% %%%%%%%%%%%%%%%%%%%%%%%%%%%%%%%%%%%%%% %%%%%%%%%%%%%%%%%%%%%%%%%%%%%%%%%%%%%% %%%%%%%%%%%%%%%%%%%%%%%%%%%%%%%%%%%%%% %%%%%%%%%%%%%%%%%%%%%%%%%%%%%%%%%%%%%% %%%%%%%%%%%%%%%%%%%%%%%%%%%%%%%%%%%%%% 
 In this section, we construct and analyse the properties of the Segal amplitudes. 
The initial step is to specify the Hilbert space on which the amplitudes will act as operators. This Hilbert space, described in Section \ref{sub:hilbert}, is basically a $L^2$   space based on the GFF on the unit circle. The Segal amplitudes for Liouville CFT are defined through their integral kernels, whose essential components are expectations of the GMC measure on the surface with respect to the Gaussian free field, conditioned on taking prescribed values  on the boundary circles, each of which identified to the unit circle via the boundary parametrisations. However, there is a non-trivial contribution coming from the free-field: this is related to the fact that the law of the GFF on $\Sigma$ restricted to a parametrised interior simple curve is not the law of the GFF on the circle, but has a Radon-Nikodym derivative expressed in terms of Dirichlet-to-Neumann (DN) maps. In Section \ref{subDTN} we thus study the DN maps and their properties, in order to lay the foundations  for the definition of the amplitudes in 
Section \ref{sec:amplitudes_Segal}. We also give in the beginning of Section \ref{sec:amplitudes_Segal}
a heuristic explanation for the probabilistic definition of the amplitudes.

\subsection{Hilbert space of LCFT}\label{sub:hilbert} The construction of the Hilbert space of LCFT relies on the following real-valued random Fourier series defined on  the unit circle $\T=\{z\in \C\,|\, |z|=1\}$ 
\begin{equation}\label{GFFcircle0}
\forall \theta\in\R,\quad \varphi(\theta)=\sum_{n\not=0}\varphi_ne^{in\theta} 
\end{equation}
with $\varphi_{n}=\frac{1}{2\sqrt{n}}(x_n+iy_n)$ for $n>0$ where $x_n,y_n$ are i.i.d. standard real Gaussians. Convergence holds in the  Sobolev space  $H^{s}(\T)$ with $s<0$, where $H^s(\T)\subset \C^\Z$ is the set of sequences s.t.
\begin{equation}\label{outline:ws}
\|\varphi\|_{H^s(\T)}^2:=\sum_{n\in\Z}|\varphi_n|^2(|n|+1)^{2s} <\infty.
\end{equation}
Such a random series arises naturally when considering the restriction of the whole plane GFF to the unit circle. Also, note that the series $ \varphi$ has no constant mode. The constant mode will play an important role in what follows and this is why we want to single it out:  we will view the random series $\tilde{\varphi}:=c+\varphi$ as the coordinate function of the space $\R\times    \Omega_\T$, where the  probability space 
\begin{align}\label{omegat}
  \Omega_\T=(\R^{2})^{\N^*}
\end{align}
  is equipped with the cylinder sigma-algebra $
 \Sigma_\T=\mathcal{B}^{\otimes \N^*}$ ($\mathcal{B}$  stands for the Borel sigma-algebra on $\R^2$) and the   product measure 
 \begin{align}\label{Pdefin}
 \P_\T:=\bigotimes_{n\geq 1}\frac{1}{2\pi}e^{-\frac{1}{2}(x_n^2+y_n^2)}\dd x_n\dd y_n.
\end{align}
Here $ \P_\T$ is supported on $H^s(\T)$ for any $s<0$ in the sense that $ \P_\T(\varphi\in H^s(\T))=1$. Our Hilbert space, denoted $\mc{H}$, is then  $\mc{H}:=L^2(\R\times \Omega_{\T})$ with underlying measure 
\begin{equation}\label{def_of_mu_0}
\mu_0:=\dd c\otimes  \P_{\T}
\end{equation} 
and Hermitian product denoted by $\langle\cdot,\cdot\cjd_{\mc{H}}$.

\subsection{Dirichlet-to-Neumann map}\label{subDTN}
%%%%%%%%%%%%%%%%%%%%%%%%%%%%%%
Let  $\Sigma$ be a compact Riemann surface   with real analytic boundary $\partial\Sigma=\bigcup_{j=1}^b\caC_j$  consisting of 
 $b$ closed simple curves, which do not intersect each other (here  $b$ could possibly be equal to $0$ in case $\pl\Sigma=\emptyset$) and parametrised with charts as in Subsection \ref{sub:gluing}. The analytic parametrisation of the boundary is denoted by $\boldsymbol{\zeta}=(\zeta_1,\dots,\zeta_b)$ with $\zeta_j:\T \to \caC_j$.
  We consider a metric $g$ on $\Sigma$ so that each boundary component has length $2 \pi$; except when mentionned, $g$ is not assumed to be admissible. We denote by $d\ell_g$ the Riemannian measure on $\pl\Sigma$ induced by $g$.
  
A generic (real valued) field $\tilde\varphi\in H^s(\T)$ (for $s<0$) will be decomposed into its constant mode $c$ and orthogonal part
$$\tilde\varphi=c+\varphi,\quad \varphi(\theta)=\sum_{n\not=0}\varphi_ne^{in\theta}$$
with $(\varphi_n)_{n\not=0}$  its other Fourier coefficients, which will be themselves parametrised by $\varphi_n=\frac{x_n+iy_n}{2\sqrt{n}}$  and $\varphi_{-n}=\frac{x_n-iy_n}{2\sqrt{n}}$  for $n>0$. In what follows, we will consider a family of such fields $\boldsymbol{\tilde \varphi}=(\tilde\varphi_1,\dots,\tilde\varphi_{b})\in (H^{s}(\T))^{b}$, in which case the previous notations referring to the $j$-th field will be augmented with an index $j$, namely, $c_j$, $\varphi_{j,n}$, $x_{j,n}$ or $y_{j,n}$.  By an abuse of notations, we will also denote by $(\cdot,\cdot)_2$ the pairing between $(H^{s}(\T))^{b}$ and $(H^{-s}(\T))^{b}$ (as it is an extension of the $L^2$ pairing of smooth functions).
 
For such a field  $\boldsymbol{\tilde \varphi}=(\tilde\varphi_1,\dots,\tilde\varphi_{b})\in (H^{s}(\T))^{b}$ with $s\in\R$, we will write   $P\tilde{\boldsymbol{\varphi}}$ for  the harmonic extension  of $\tilde{\boldsymbol{\varphi}}$, that is $\Delta_g P\tilde{\boldsymbol{\varphi}}=0$ on $\Sigma\setminus \bigcup_j\mc{C}_j$ with boundary values  $P\tilde{\boldsymbol{\varphi}}|_{\mc{C}_j}=\tilde\varphi_j\circ \zeta_j^{-1} $  for $j=1,\dots,b$.  The boundary value has to be understood in the following weak sense: for all $u\in C^\infty(\T)$
\[\lim_{r\to 1^-}\int_{0}^{2\pi}P\tilde{\boldsymbol{\varphi}}(\zeta_j(re^{i\theta}))\bbar{u(e^{i\theta})}d\theta =2\pi (\tilde\varphi_j,u)_{2}\]
where $(\cdot,\cdot)_2$ stands for the canonical inner product on $L^2(\T)$ and $\T$ is equipped with the probability measure $\dd \theta/(2\pi)$ (recall that each $\mc{C}_j$ has length $2 \pi$).

The definition of our amplitudes will involve   the Dirichlet-to-Neumann operator (DN map for short). Recall that the DN map $\mathbf{D}_\Sigma:C^\infty(\T)^{b}\to C^\infty(\T)^{b}$ is defined as follows: for $\tilde{\boldsymbol{\varphi}}\in C^\infty(\T)^{b}$ 
\[\mathbf{D}_\Sigma\tilde{\boldsymbol{\varphi}}=(-\partial_{\nu } P\tilde{\boldsymbol{\varphi}}|_{\mc{C}_j}\circ\zeta_j)_{j=1,\dots,b} \]
where $\nu$ is the inward unit normal vector fields to $\mc{C}_j$. 
Note that, by the Green formula,
\begin{equation}\label{Greenformula}
\int_{\Sigma} |dP\tilde{\boldsymbol{\varphi}}|_g^2{\rm dv}_g = 2 \pi ( \tilde{\boldsymbol{\varphi}},\mathbf{D}_\Sigma\tilde{\boldsymbol{\varphi}})_2
\end{equation}
 By formula \eqref{Greenformula}, $\mathbf{D}_\Sigma$ is a non-negative symmetric operator with kernel $\ker {\bf D}_{\Sigma}=\R \tilde{1}$ where $\tilde{1}= (1, \dots, 1)$.
 
 We will also consider the following variant of the DN map. We  consider      a compact Riemann surface  $\Sigma$ with real analytic (possibly empty) boundary $\partial\Sigma=\bigcup_{j=1}^b\caC_j$  as before. Here  $b$ could possibly be equal to $0$ in case $\Sigma$ has no boundary and we consider charts $\omega_j:V_j\to \mathbb{A}_{1-\eps_j}:=\{z\in \C\,|\, |z|\in (1-\eps_j,1]\}$ with $\eps_j\in (0,1)$ and $\omega_j(\mc{C}_j)=\{|z|=1\}$ as in Subsection      
 \ref{subsub:boundary}.
We further assume that we are given a collection of $b'$ analytic closed simple non overlapping curves  $\mc{C}':=\sqcup_{j=1}^{b'}\mathcal{C}'_j$ in the interior of $\Sigma$. Each such curve comes equipped with a holomorphic chart   $\omega'_j:V'_j\to \omega'_j(V'_j)\subset \C$ for $j=1,\dots,b'$ where $V'_j$ is an open neighborhood of $\mathcal{C}'_j$ and $\omega'_j(V'_j)=A'_j$, with $A'_j=\{z\in \C\,|\, |z|\in [1-\eps'_j,1+\eps'_j]\}$ with $\eps'_j\in (0,1)$ and $\omega_j'(\mc{C}'_j)=\{|z|=1\}$. We assume that the metric $g$ is such that there exists $f_j\in C^\infty(V_j)$ and $f_j'\in C^\infty(V_j')$ with $f_j|_{\mc{C}_j}=0$ and $f_j'|_{\mc{C}'_j}=0$ so that
\begin{equation} \label{metric_near_bdry} 
\omega_j^*\Big(\frac{|dz|^2}{|z|^2}\Big)=e^{f_j}g, \qquad{\omega'_j}^*\Big(\frac{|dz|^2}{|z|^2}\Big)=e^{f'_j}g.
\end{equation}
In the chart given by the annulus $\mathbb{A}_{1-\eps_j}$, the orientation is then given by $d\theta$ (i.e. the counterclockwise orientation) on the unit circle parametrised by $(e^{i\theta})_{\theta\in [0,2\pi]}$. 
For   a field  $\boldsymbol{\tilde \varphi}=(\tilde\varphi_1,\dots,\tilde\varphi_{b'})\in (H^{s}(\T))^{b'}$ with $s\in\R$, we will write   $P_{\mc{C}'}\tilde{\boldsymbol{\varphi}}$ for  the harmonic extension   $\Delta_g P_{\mc{C}'}\tilde{\boldsymbol{\varphi}}=0$ on $\Sigma\setminus \mc{C}'$ with boundary value $0$ on $\partial\Sigma$ and equal to $\tilde\varphi_j\circ \omega'_j $ on $\mc{C}'_j$ for $j=1,\cdots,b'$. The DN map $\mathbf{D}_{\Sigma,\mc{C}'}:C^\infty(\T)^{b'}\to C^\infty(\T)^{b'}$ associated to $\mathcal{C}'$ is then defined as the jump at $\mc{C}'$ of the harmonic extension:    for $\tilde{\boldsymbol{\varphi}}\in C^\infty(\T)^{b'}$ 
\begin{equation}\label{defDSigmaC}
\mathbf{D}_{\Sigma,{\mc{C}'}}\tilde{\boldsymbol{\varphi}}=-((\partial_{\nu_-} P_{\mc{C}'}\tilde{\boldsymbol{\varphi}})|_{\mc{C}'_j}+(\partial_{\nu_+} P_{\mc{C}'}\tilde{\boldsymbol{\varphi}})|_{\mc{C}'_j})_{j=1,\dots,b'}.
\end{equation}
Here $\partial_{\nu_\pm}$ denote the two inward normal derivatives along $\mc{C}'_j$ in the following sense: viewing $P_{\mc{C}'}\tilde{\boldsymbol{\varphi}}=\sum_{k}u_k$ as a sum of smooth functions $u_k$ supported on the closure $\Sigma_k$ of each connected component of $\Sigma\setminus \bigcup_{j}\mc{C}_j'$, then with $\nu_k$ being the inward pointing normal unit vector field of $\Sigma_k$ at $\mc{\pl} \Sigma_k$, set
\[(\partial_{\nu_-} P_{\mc{C}'}\tilde{\boldsymbol{\varphi}})|_{\mc{C}'_j}+(\partial_{\nu_+} P_{\mc{C}'}\tilde{\boldsymbol{\varphi}})|_{\mc{C}'_j}:=
\sum_{k, \Sigma_k\cap \mc{C}'_j\not=\emptyset}\partial_{\nu_k}u_k|_{\mc{C}'_j}.\]  
We notice the useful fact: $\mathbf{D}_{\Sigma,{\mc{C}'}}$ is invertible and, if $G_{g,D}$ denotes the Green function on $(\Sigma,g)$ with Dirichlet condition at $\pl \Sigma$, the Schwartz kernel of $\mathbf{D}_{\Sigma,{\mc{C}'}}^{-1}$ is (see e.g. the proof of \cite[Theorem 2.1]{Carron})
\begin{equation}\label{DNmapandGreen}
\mathbf{D}_{\Sigma,{\mc{C}'}}^{-1}(y,y')=\frac{1}{2\pi}G_{g,D}(y,y'), \quad y\not=y' \in \mc{C}'.
\end{equation}  
 Let us further introduce   the unbounded operator $\mathbf{D}$ on $L^2(\T)^{b}$ defined as the Friedrichs extension associated to  the quadratic form
 \begin{equation}\label{defmathbfD}
\forall \tilde{\boldsymbol{\varphi}} \in C^\infty(\T;\R)^{b},\quad (\mathbf{D}\tilde{\boldsymbol{\varphi}},\tilde{\boldsymbol{\varphi}}):= 2\sum_{j=1}^{b}\sum_{n>0} n|\varphi_{j,n}|^2=\frac{1}{2}\sum_{j=1}^{b}\sum_{n>0}((x_{j,n})^2+(y_{j,n})^2).
\end{equation}
Finally we consider the operators on $C^\infty(\T)^{b}$ and $C^\infty(\T)^{b'}$
\begin{align}
\label{tildeD}& \widetilde{\mathbf{D}}_{\Sigma}  :=\mathbf{D}_{\Sigma}-\mathbf{D},  
&  \Pi_0(\tilde\varphi_1,\dots,\tilde\varphi_b):=  (( \tilde\varphi_1,1)_2,\dots,( \tilde\varphi_b,1)_2) \\
& \widetilde{\mathbf{D}}_{\Sigma,{\mc{C}'}} :=\mathbf{D}_{\Sigma,{\mc{C}'}}-2\mathbf{D}, 
& \Pi'_0(\tilde\varphi_1,\dots,\tilde\varphi_{b'}):=  (( \tilde\varphi_1,1)_2,\dots,(\tilde\varphi_{b'},1 )_2).\label{defPi_0'}
\end{align}
\begin{lemma}\label{lemmaDSigma-D}
The operators $\widetilde{\mathbf{D}}_{\Sigma}$, $\mathbf{D}_{\Sigma}(\Pi_0+{\bf D})^{-1}-{\rm Id}$, respectively $\widetilde{\mathbf{D}}_{\Sigma,{\mc{C}'}}$,  
 $\mathbf{D}_{\Sigma,\mc{C}'}(2\Pi_0'+2{\bf D})^{-1}-{\rm Id}$, are smoothing operators in the sense that they are operators with smooth Schwartz kernel that are bounded for all $s,s'\in \R$ as maps 
\[(H^{s}(\T))^{b} \to (H^{s'}(\T))^{b}, \textrm{ respectively } (H^{s}(\T))^{b'} \to (H^{s'}(\T))^{b'}.\]
In particular they are trace class on $L^2(\T^b), L^2(\T^{b'})$ and the Fredholm determinant  $\det_{\rm Fr}(\mathbf{D}_{\Sigma,\mc{C}'}(2\Pi_0'+2{\bf D})^{-1})$ is well-defined.
\end{lemma}
\begin{proof} First, notice that, writing $\Sigma_1^\circ,\dots, \Sigma_K^\circ$ for the connected components of 
$\Sigma\setminus \bigcup_{j}\mc{C}'_j$ and $\Sigma_k$ for the closure of $\Sigma_k^\circ$, one has 
\begin{equation}\label{DNC'} 
\mathbf{D}_{\Sigma,{\mc{C}'}}\tilde{\boldsymbol{\varphi}}=\sum_{k=1}^K 
\mathbf{D}_{\Sigma_k}\tilde{\boldsymbol{\varphi}}|_{\mc{C}'}.
\end{equation}
Next, we consider the Poisson operator and Dirichlet-to-Neumann map on the flat unit disk $(\D,|dz|^2)$. It is direct to see that 
\[ P_{\D}\tilde{\varphi}(z)=\varphi_0+\sum_{n>0} \varphi_n z^n+ \varphi_{-n} \bar{z}^n, \quad {\bf D}_{\D}\tilde{\varphi}=\sum_{n>0}
n(\varphi_ne^{in\theta}+\varphi_{-n}e^{-in\theta}), \quad \tilde{\varphi}=\sum_{n\in \Z} \varphi_ne^{in\theta}\]
so that the Poisson kernel $P_{\D}(z,e^{i\theta})=\frac{1}{2\pi}(1+2{\rm Re}(ze^{-i\theta}(1-ze^{-i\theta})^{-1})$ and 
${\bf D}_{\D}=|\pl_\theta|$ is the Fourier multiplier by $|n|$ on the disk.

Let us consider $\widetilde{\mathbf{D}}_{\Sigma}$. 
Let $G_{g,D}$ be the Green function for the Dirichlet Laplacian $\Delta_g$ on $\Sigma$.
The Poisson operator $P:C^\infty(\pl \Sigma)\to C^\infty(\Sigma)$ defined by $\Delta_gP\boldsymbol{\tilde{\varphi}}=0$ with $P\tilde{\boldsymbol{\varphi}}|_{\pl\Sigma}=\boldsymbol{\tilde{\varphi}}\circ \boldsymbol{\zeta}^{-1}$ is given by 
\[P\boldsymbol{\tilde{\varphi}}(x)=\frac{1}{2\pi}\sum_{j=1}^b\int_{\caC_j}\pl_{\nu'}G_{g,D}(x,y')\tilde{\varphi}_j(\zeta_j^{-1}(y'))d\ell_g(y')\]
where $\nu'$  denotes the inward pointing unit normal vector to $\pl \Sigma$. In particular, this expression applies to $(\Sigma,g)=(\D,|dz|^2)$. 
We shall compare the Green function $G_{g,D}$ to the Green function $G_{\D}$ of $\D$ (with Dirichlet condition at $\T$) using charts.
Take $\chi_j,\chi_j',\chi_j''\in C^\infty(\Sigma)$ be equal to $1$ near $\pl \Sigma$ and supported in $V_j$ (i.e. near $\pl \Sigma$), 
and $\chi'_j=1$ on support of $\chi_j$ and $1-\chi_j''=1$ on support of $1-\chi_j$.  
We also denote $G_{g,D}$ for the operator with $G_{g,D}(x,x')$ as Schwartz kernel and more generally we identify operators with their Schwartz kernels.
Since $g=e^{f_j}|dz|^2$ and 
$\Delta_g=e^{-f_j}\Delta_{\D}$ in the charts $\omega_j(V_j)\subset \D$ with $\Delta_{\D}$ the Laplacian on $(\D,|dz|^2)$, we have, using $\Delta_{\D}G_{\D}=2\pi {\rm Id}$ on $L^2(\D)$ and identifying implicitly $V_j$ with $\omega_j(V_j)\subset \D$, that $\Delta_g\int_{V_j}G_{\D}(\omega_j(x),\omega_j(x'))u(x'){\rm dv}_g(x')=2\pi u(x)$ for   $u\in C_c^\infty(V_j)$. 
Using that $\chi_j(\nabla \chi_j')=0=(1-\chi_j)\nabla\chi_j''$ and that the integral kernels $G_{\D}, G_{g,D}$ are smooth outside the diagonal in $\Sigma\times \Sigma$, we get 
\[ \Delta_{g}\Big(\sum_{j=1}^b\chi_j'G_{\D}^j\chi_j+(1-\chi_j'')G_{g,D}(1-\chi_j)\Big)=2\pi({\rm Id}+K)\]
where $G_{\D}^j(x,x'):=G_{\D}(\omega_j(x),\omega_j(x'))$, $K$ is an operator with smooth Schwartz kernel vanishing near $\pl \Sigma\times \Sigma$ and at $\Sigma\times \pl \Sigma$. 
We can then write 
\[ G_{g,D}=(\sum_{j}\chi'_jG_{\D}^j\chi_j+(1-\chi_j'')G_{g,D}(1-\chi_j))-G_{g,D}K.\]
We claim that the operator $G_{g,D}K$ has a Schwartz kernel that is smooth near $\pl \Sigma\times  \Sigma$ (in fact it is smooth on $\Sigma\times \Sigma$ but we do not need it). 
Indeed, the Green function is smooth outside the diagonal of $\Sigma$, $G_{g,D}\in C^\infty(\Sigma\times \Sigma\setminus \{\rm diag\})$ and $K(x'',x')=0$ for $x''$ near $\pl \Sigma$, thus 
\[ (G_{g,D}K)(x,x')=\int_{\Sigma}G_{g,D}(x,x'')K(x'',x')d{\rm v}_{g}(x'')\]
is smooth for $(x,x')$ near $\pl\Sigma\times \Sigma$. We deduce that
\begin{equation}\label{parametrixoperatorP} 
P\boldsymbol{\varphi}(x)=\sum_j \chi_j'(x) \int_{\pl\Sigma}P_{\D}(\omega_j(x),\omega_j(y))\tilde{\varphi}_j(y)d\ell_g(y)-(K'\varphi)(x)
\end{equation}
for some operator $K':C^\infty(\pl \Sigma)\to C^\infty(\Sigma)$
having an integral kernel on $\Sigma\times \pl\Sigma$ that is smooth near $\pl \Sigma\times \pl \Sigma$. 
Taking the normal derivative at $\pl \Sigma$, we deduce that
\[ \mathbf{D}_{\Sigma}={\bf D}+K'_{\pl \Sigma}\]
where $K'_{\pl\Sigma}$ has  a smooth Schwartz kernel $K'|_{\pl \Sigma\times \pl\Sigma}$ on $\pl \Sigma$. The fact that $\mathbf{D}_{\Sigma}(\Pi_0+{\bf D})^{-1}-{\rm Id}$ has smooth Schwartz kernel then follows directly. This also implies that $\widetilde{\mathbf{D}}_{\Sigma,{\mc{C}'}}$ is smoothing by using \eqref{DNC'} and that $\mathbf{D}_{\Sigma,\mc{C}'}(2\Pi_0+2{\bf D})^{-1}-{\rm Id}$ is smoothing. 
\end{proof}

\subsection{Amplitudes}\label{sec:amplitudes_Segal}
%%%%%%%%%%%%%%%%
 Let us first give a brief heuristic explanation about how to define probabilistically the Segal amplitudes. Let $(\Sigma,g)$ be a 
 Riemannian surface  with a parametrised boundary $\zeta:\T\to \pl \Sigma$ (say with $\pl \Sigma$ connected). 
Any $H^1(\Sigma)$ function decomposes as $\phi=\phi_0+P\tilde{\varphi}$ with 
$\phi_0\in H^1_0(\Sigma)$ vanishing at $\pl \Sigma$ and $P\tilde{\varphi}$ the harmonic extension of 
$\tilde{\varphi}=\phi|_{\pl \Sigma} \circ \zeta$. The Dirichlet energy obeys
\[ 
\int_{\Sigma}|d\phi|_g^2{\rm dv}_g=\int_{\Sigma}|d\phi_0|_g^2{\rm dv}_g+\int_{\Sigma}|dP\til{\varphi}|_g^2{\rm dv}_g=\int_{\Sigma}|d\phi_0|_g^2{\rm dv}_g+2 \pi ( \tilde{\varphi},\mathbf{D}_\Sigma\tilde{\varphi})_2.\]
If we now think of the Segal amplitude of $\Sigma$ to be given by the formal conditional path integral \eqref{ampli} and say $V(\phi)$ is the non-linear potential in the action, one can formally write 
\[\begin{split}
\caA_{\Sigma,g}(\til{\varphi})=& \int_{\phi|_{\pl \Sigma}=\til{\varphi}}e^{-\frac{1}{4\pi}\int_{\Sigma}|d\phi|_g^2{\rm dv}_g}e^{-\int_\Sigma V(\phi){\rm dv}_g}D\phi\\
=& e^{-\frac{1}{2}( \tilde{\varphi},\mathbf{D}_\Sigma\tilde{\varphi})_2}\int_{\phi_0|_{\pl \Sigma}=0}e^{-\frac{1}{4\pi}\int_{\Sigma}|d\phi_0|_g^2{\rm dv}_g}e^{-\int_\Sigma V(\phi_0+P\til{\varphi}){\rm dv}_g}D\phi_0 
\end{split}\]
The $e^{-\frac{1}{4\pi}\int_{\Sigma}|d\phi_0|_g^2{\rm dv}_g}D\Phi_0$ Gaussian integral can be defined probabilistically using the Dirichlet GFF on $\Sigma$, where one also needs to input the determinant of the Dirichlet Laplacian as the total mass. 
The term $e^{-\frac{1}{2}( \tilde{\varphi},\mathbf{D}_\Sigma\tilde{\varphi})_2}$ involving the DN map is ill defined if $\til{\varphi}$ has the law of a GFF (up to the constant term) on the unit circle, but what can be given a sense is $e^{-\frac{1}{2}( \tilde{\varphi},(\mathbf{D}_\Sigma-{\bf D})\tilde{\varphi})_2}$
on $H^{s}(\T)$ for any $s<0$ fixed by Lemma \ref{lemmaDSigma-D}, $\mathbf{D}_\Sigma-{\bf D}$. One is tempted to 
rewrite 
\begin{equation}\label{Atilde} 
\begin{split}
\caA_{\Sigma,g}(\til{\varphi})= &\Big(e^{-\frac{1}{2}( \tilde{\varphi},(\mathbf{D}_\Sigma-{\bf D})\tilde{\varphi})_2}\int_{\phi_0|_{\pl \Sigma}=0}e^{-\frac{1}{4\pi}\int_{\Sigma}|d\phi_0|_g^2{\rm dv}_g}e^{-\int_\Sigma V(\phi_0+P\til{\varphi}){\rm dv}_g}D\phi_0\Big)e^{-\frac{1}{2}( \tilde{\varphi},{\bf D}\tilde{\varphi})_2}\\
=:& \, \til{\caA}_{\Sigma,g}(\til{\varphi})e^{-\frac{1}{2}( \tilde{\varphi},{\bf D}\tilde{\varphi})_2}.
\end{split}\end{equation}
A direct computation gives us that the measure $\mu_0$ of \eqref{def_of_mu_0}
represents the Gaussian measure $e^{-( \tilde{\varphi},\mathbf{D}\tilde{\varphi})_2}D\til{\varphi}$ (up to normalisation of the total mass). This means that, formally, the pairing of two amplitudes with respect to the uniform measure $D\tilde{\varphi}$
becomes 
\[ \int \caA_{\Sigma,g}(\til{\varphi})\caA_{\Sigma',g'}(\til{\varphi})D\tilde{\varphi}= \int \til{\caA}_{\Sigma,g}(\til{\varphi})\til{\caA}_{\Sigma',g'}(\til{\varphi})\mu_0(d\tilde{\varphi}).\]
Since $\mu_0$ is rigorously defined, while $D\tilde{\varphi}$ is not, and since $\til{\caA}_{\Sigma,g}(\til{\varphi})$ can be given a probabilistic definition, we shall rather work with $\mu_0$ and define  amplitudes in Definition \ref{def:amp} by the natural probabilistic expression representing the $\til{\caA}_{\Sigma,g}(\til{\varphi})$ formal expression given in \eqref{Atilde}.

Given a Riemannian surface  $(\Sigma,g)$  with an analytic parametrisation $
\boldsymbol{\zeta}$ of the boundary (the metric $g$ is not necessarily taken admissible except when mentioned, e.g. in Theorem \ref{integrcf}), and given marked points ${\bf x}:=(x_1,\dots,x_m)$ in  the interior $\Sigma^\circ$ of the surface  and some weights $\boldsymbol{\alpha}:=(\alpha_1,\dots,\alpha_m)\in\R^m$ associated to ${\bf x}$, the  amplitude $\caA_{\Sigma,g,{\bf x},\boldsymbol{\alpha},\boldsymbol{\zeta}}$ is defined as follows: 

 \begin{definition}[\textbf{Amplitudes}]\label{def:amp}
 
 We suppose that the second Seiberg bound \eqref{seiberg2} holds, i.e. $\alpha_i<Q$, $i=1,\dots,m$.   
 
\noindent {\bf (A)}
Let $\partial\Sigma=\emptyset$. For $F$ continuous  nonnegative function on $H^{s}(\Sigma)$ for some $s<0$ we define 
\begin{equation}\label{defampzerobound} 
\caA_{\Sigma,g,{\bf x},\boldsymbol{\alpha}}(F):=\lim_{\epsilon\to 0}\langle F(\phi_g)\prod_{i=1}^m V_{\alpha_i,g,\epsilon}(x_i) \rangle_{\Sigma,g} .
\end{equation}
using \eqref{def:pathintegralF} with $\phi_g= c+X_{g}$ and \eqref{defcorrelg}.
 If $F=1$ then the amplitude is just the LCFT correlation function  and will be simply denoted by $\caA_{\Sigma,g,{\bf x},\boldsymbol{\alpha}}$.
 \vskip 3mm
 
\noindent {\bf (B)}  If  $\partial\Sigma$ has $b>0$ boundary  components, $\caA_{\Sigma,g,{\bf x},\boldsymbol{\alpha},\boldsymbol{\zeta}}$ is
  a function $(F,\tilde{\boldsymbol{\varphi}})\mapsto \caA_{\Sigma,g,{\bf x},\boldsymbol{\alpha},\boldsymbol{\zeta} }(F,\tilde{\boldsymbol{\varphi}})$ of  the boundary fields $\tilde{\boldsymbol{\varphi}}:=(\tilde\varphi_1,\dots,\tilde\varphi_b)\in (H^{s}(\T))^b$ with $s<0$  and of continuous nonnegative functions $F$ defined on  $H^{s}(\Sigma)$ for $s\in (-1/2,0)$.  
 It is  defined by (recall that   $\phi_g= X_{g,D}+P \boldsymbol{\tilde \varphi}$)
\begin{align}\label{amplitude}
 \caA_{\Sigma,g,{\bf x},\boldsymbol{\alpha},\boldsymbol{\zeta}}(F,\tilde{\boldsymbol{\varphi}}) =\lim_{\eps\to 0}Z_{\Sigma,g}\caA^0_{\Sigma,g}(\tilde{\boldsymbol{\varphi}})
 \E \big[F(\phi_g)\prod_{i=1}^m V_{\alpha_i,g,\epsilon}(x_i)e^{-\frac{Q}{4\pi}\int_\Sigma K_g\phi_g\dd {\rm v}_g-\frac{Q}{2\pi}\int_{\partial\Sigma}k_g\phi_g\dd \ell_g -\mu M_\gamma^g (\phi_g,\Sigma)}\big]
\end{align}
where the expectation $\E$ is over the Dirichlet GFF $X_{g,D}$, $M_\gamma^g (\phi_g,\Sigma)$ is defined as a limit in \eqref{GMCg} and $Z_{\Sigma,g}$ is  a normalisation constant given by\footnote{Note that $k_g=0$ when $g$ is assumed to be admissible.}
 \begin{align}\label{znormal}
 Z_{\Sigma,g}=\det (\Delta_{g,D})^{-\hf}\exp\big(\frac{1}{8\pi}\int_{\partial\Sigma}k_g\,\dd\ell_g\big)
 \end{align}
with $k_g$ the geodesic curvature of $\pl \Sigma$, and $\caA^0_{\Sigma,g}(\tilde{\boldsymbol{\varphi}})$ is the free field amplitude defined as  
\begin{align}\label{amplifree}
\caA^0_{\Sigma,g}(\tilde{\boldsymbol{\varphi}})=e^{-\frac{1}{2}( \tilde{\boldsymbol{\varphi}}, (\mathbf{D}_\Sigma-\mathbf{D})  \tilde{\boldsymbol{\varphi}})_2}.
\end{align}
When $F=1$, we will simply write $ \caA_{\Sigma,g,{\bf x},\boldsymbol{\alpha},\boldsymbol{\zeta}}( \tilde{\boldsymbol{\varphi}})$.
 \end{definition}

Note that the existence of the limit above in the case when $\partial\Sigma\neq\emptyset$ results from the Girsanov argument as in \cite[Section 3]{DKRV16}. The definitions above trivially extend to the situation when $F$ is no more assumed to be nonnegative but with the further requirement that $\caA_{\Sigma,g,{\bf x},\boldsymbol{\alpha}}(|F|)<\infty$ in the case $\partial\Sigma=\emptyset$ and $ \caA_{\Sigma,g,{\bf x},\boldsymbol{\alpha},\boldsymbol{\zeta}}(|F|,\tilde{\boldsymbol{\varphi}})<\infty$ $(\dd c\otimes \P_\T)^{\otimes b}$ almost everywhere in the case $\partial\Sigma\not=\emptyset$. We finally remark that $\mc{A}^0_{\Sigma,g}$ does depend on 
$\boldsymbol{\zeta}$, and its dependence on $g$ is only through the conformal class of $g$.

Let us also stress that there is a subtle difference in the above definition according to which $\Sigma$ has a boundary or not:

\begin{remark}
If  $\partial\Sigma$ has $b>0$ boundary   components then the amplitude \eqref{amplitude} is defined and non trivial, i.e. belongs to $(0,\infty)$ as soon as $\alpha_i<Q$, $i=1,\dots,m$. However, when $\Sigma$ has no boundary and $F=1$ then \eqref{defampzerobound} defines a non trivial quantity if and only if the Seiberg bounds \eqref{seiberg1} and \eqref{seiberg2} hold. If \eqref{seiberg1} does not hold then \eqref{defampzerobound} is infinite; hence, with our convention, the amplitude for a closed surface and for $F=1$ is meaningful only when the Seiberg bound \eqref{seiberg1} holds.
\end{remark}

In what follows, we shall sometimes remove the $\boldsymbol{\zeta}$ index for notational simplicity in the amplitudes when this does not play an important role, keeping in mind that each boundary component $\mc{C}_j$ of $\Sigma$ comes with a parametrisation $\zeta_j$. 
 We gather below the main properties we need for amplitudes: we first establish bounds and then prove the Weyl invariance. For boundary fields $\tilde{\boldsymbol{\varphi}}=(\tilde\varphi_1,\dots,\tilde\varphi_b)\in (H^{s}(\T))^b$, we will denote by $(c_j)_{j\leq b}$ their constant modes, set $ \bar c=\tfrac{1}{b}\sum_{j=1}^b c_j$ and $\boldsymbol{\varphi}=(\varphi_1,\dots,\varphi_b)$ the centered fields (recall that $\tilde\varphi_j=c_j+\varphi_j$). Also, recall that $c_+=c\mathbf{1}_{\{c>0\}}$ and $c_-=c\mathbf{1}_{\{c<0\}}$.  The following result will be useful to control the integrability properties of the amplitude:

 \begin{theorem}\label{integrcf}
Let $(\Sigma,g,{\bf x},\boldsymbol{\zeta})$ be an admissible surface with $\partial\Sigma\not=\emptyset$ having  $b$ boundary components, and $\boldsymbol{\alpha}$ some weights attached to the marked points ${\bf x}$. 
Then, there exists some constant $a>0$ such that for any $R>0$, there exists $C_R>0$ such that  $\mu_0^{\otimes_b}$-almost everywhere in  $ \boldsymbol{\tilde \varphi}\in H^{s}(\T)^b$ with $s<0$
\begin{align*}%\label{}
 \caA_{\Sigma,g,{\bf x},\boldsymbol{\alpha}}( \tilde{\boldsymbol{\varphi}}) \leq C_R  e^{s_0 \bar c_--R \bar c _+  -a  \sum_{j=1}^b(\bar c-c_j)^2}
A(\boldsymbol{\varphi})
\end{align*}
  satisfying $\int A(\boldsymbol{\varphi}) ^2 \,\P_\T^{\otimes_b}(\dd \boldsymbol{\varphi})<\infty$, 
where 
\begin{align*}
s_0:=\sum_i\alpha_i-Q\chi(\Sigma)
\end{align*}
and the Euler characteristic is equal to $\chi(\Sigma)=2-2\mathbf{g}-b$.  
  \end{theorem}
The proof is based on the   following two lemmas. First, we  deal with the free field amplitude.
 
\begin{lemma}\label{lem:amp1}
Let $(\Sigma,g)$ be a smooth Riemannian surface with boundary ($g$ is not necessarily admissible). There exist a constant $a>0$ such that for all $N>0$, there is $k_N>0$ such that
\begin{align*}
\caA^0_{\Sigma,g}(\tilde{\boldsymbol{\varphi}})\leq  \prod_{j=1}^be^{-a( \bar c-c_j)^2}\prod_{n=1}^\infty e^{a_n((x_{j,n})^2+(y_{j,n})^2)} 
\end{align*}
where $a_n\leq (\frac{1}{4}-a)$ for $n\leq k_N$ and $a_n\leq \frac{1}{4n^{N}}$ for $n\geq k_N$.
\end{lemma}

\begin{proof} By Lemma \ref{lemmaDSigma-D},  
we can write 
\[\mathbf{D}_\Sigma=\widetilde{\mathbf{D}}_\Sigma+{\bf D}=(\Pi_0+{\bf D})^{1/2}({\rm Id}+K)(\Pi_0+{\bf D})^{1/2}\] with $K=(\Pi_0+{\bf D})^{-1/2}(\widetilde{\mathbf{D}}_\Sigma-\Pi_0)(\Pi_0+{\bf D})^{-1/2}$ having 
$C^\infty$ Schwartz kernel and self-adjoint on $L^2(\T)^b$ (recall $\Pi_0$ is defined in  \eqref{tildeD}) . Since $\mathbf{D}_\Sigma$ is a non-negative operator with kernel $\ker {\bf D}_{\Sigma}=\R \tilde{1}$,  ${\rm Id}+K$ is  a non-negative Fredholm self-adjoint operator with $\ker ({\rm Id}+K)=\R \tilde{1}$ (recall $\tilde{1}=(1,\dots,1)$). Then there is $a>0$ so that 
\[{\rm Id}+K\geq a \left ( {\rm Id}-\frac{(\tilde{1},\cdot )_2}{\|\tilde{1}\|^2_2}\tilde{1}  \right ) .\]
This implies (in the sense of quadratic forms)
\begin{equation}\label{upperboundtildeD}
{\bf D}-\mathbf{D}_\Sigma\leq -a \left (\Pi_0- \frac{(\tilde{1},\cdot )_2}{\|\tilde{1}\|^2_2}\tilde{1}  \right )+(1-a){\bf D}. 
\end{equation}
Let $\Pi_k:L^2(\T)^b\to {E_k}^b$ be the $L^2$-orthogonal projection where $E_k:={\rm span}\{e^{in\theta}\, |\, |n|\leq k\}$; note that $\Pi_k$ extends continuously to $H^{s}(\T)^b$. 
Using \eqref{upperboundtildeD}, for each $k>0$
\[-\Pi_k \widetilde{\mathbf{D}}_\Sigma \Pi_k \leq -a \left ( \Pi_0-\frac{(\tilde{1},\cdot )_2}{\|\tilde{1}\|^2_2}\tilde{1}  \right )+(1-a){\bf D}\Pi_k \]
and we obtain 
\begin{equation}\label{boundPi_k}
-(\widetilde{\mathbf{D}}_\Sigma \Pi_k\tilde{\boldsymbol{\varphi}},\Pi_k\tilde{\boldsymbol{\varphi}})_2\leq -a\sum_{j=1}^b(\bar{c}-c_j)^2+\frac{1-a}{2}\sum_{j=1}^b\sum_{n=1}^k (x_{j,n})^2+(y_{j,n})^2.
\end{equation} 
Since $\widetilde{\mathbf{D}}_\Sigma$ is a smoothing operator, we can write for all $N\in\N$  
\[ \widetilde{\mathbf{D}}_\Sigma= (\Pi_0+{\bf D})^{-N}(\Pi_0+{\bf D})^{N} \widetilde{\mathbf{D}}_\Sigma (\Pi_0+{\bf D})^{N}(\Pi_0+{\bf D})^{-N}=(\Pi_0+{\bf D})^{-N}K_N(\Pi_0+{\bf D})^{-N}\]
where $K_N$ is a bounded operator on $L^2(\T)^b$ with norm bounded by $C_N$. 
We then get, since $\|(\Pi_0+{\bf D})^{-N}({\rm Id}-\Pi_k)\|_{L^2\to L^2}\leq k^{-N}$ for $k>0$,  that there is $C_N>0$ so that
\begin{align*}
\|K_N (\Pi_0+{\bf D})^{-N} ({\rm Id}-\Pi_k)\tilde{\boldsymbol{\varphi}}\|_{L^2(\T)^b} & = \|K_N (\Pi_0+{\bf D})^{-N/2} ({\rm Id}-\Pi_k) (\Pi_0+{\bf D})^{-N/2} \tilde{\boldsymbol{\varphi}}\|_{L^2}   \\
& \leq  C_N \| (\Pi_0+{\bf D})^{-N/2} ({\rm Id}-\Pi_k) (\Pi_0+{\bf D})^{-N/2} \tilde{\boldsymbol{\varphi}}\|_{L^2}   \\ 
&  \leq C_N \| (\Pi_0+{\bf D})^{-N/2} ({\rm Id}-\Pi_k) (\Pi_0+{\bf D})^{-N/2} ({\rm Id}-\Pi_k) \tilde{\boldsymbol{\varphi}}\|_{L^2} \\
& \leq C_N k^{-N/2} \| (\Pi_0+{\bf D})^{-N/2} ({\rm Id}-\Pi_k) \tilde{\boldsymbol{\varphi}}\|_{L^2} \\ 
&  \leq C_Nk^{-N/2}\|({\rm Id}-\Pi_k)\tilde{\boldsymbol{\varphi}}\|_{H^{-N/2}(\T)^b}
\end{align*}
where the last equality is a consequence of the definition of ${\bf D}$ and the norm $\|.\|_{H^{-N/2}(\T)^b}$. We also have $\|(\Pi_0+{\bf D})^{-N}\Pi_k\tilde{\boldsymbol{\varphi}}\|_{L^2(\T)^b}= \|\Pi_k\tilde{\boldsymbol{\varphi}}\|_{H^{-N}(\T)^b}$, thus since $\widetilde{\mathbf{D}}_\Sigma$ is symmetric on $L^2(\T)^b$, 
\[\begin{split} 
-(\widetilde{\mathbf{D}}_\Sigma \tilde{\boldsymbol{\varphi}},\tilde{\boldsymbol{\varphi}})_2+(\widetilde{\mathbf{D}}_\Sigma \Pi_k\tilde{\boldsymbol{\varphi}},\Pi_k\tilde{\boldsymbol{\varphi}})_2
\leq & C_Nk^{-N}\|(1-\Pi_k)\tilde{\boldsymbol{\varphi}}\|^2_{H^{-N/2}(\T)^b}\\
& +2C_Nk^{-N/2}
\|(1-\Pi_k)\tilde{\boldsymbol{\varphi}}\|_{H^{-N/2}(\T)^b}\|\Pi_k\tilde{\boldsymbol{\varphi}}\|_{H^{-N}(\T)^b}.
\end{split}\]
Combining with \eqref{boundPi_k}, there is $C'_N>0$ and $a>0$ such that for all $k>0$
\[
\begin{split} 
-(\widetilde{\mathbf{D}}_\Sigma \tilde{\boldsymbol{\varphi}},\tilde{\boldsymbol{\varphi}})_2  
\leq & -a\sum_{j=1}^b(\bar{c}-c_j)^2+\frac{1-a+k^{-N/2}}{2}\sum_{j=1}^b\sum_{n=1}^k (x_{j,n})^2+(y_{j,n})^2\\
& +C'_Nk^{-N/2}\sum_{j=1}^b\sum_{n=k+1}^\infty \frac{1}{n^{N/2}}((x_{j,n})^2+(y_{j,n})^2).
\end{split}\]
Finally, take $k>0$ large enough so that $k^{-N/2}<a/2$ and $C_N'k^{-N/2}<\frac{1}{4}$.
This gives the desired estimate.
\end{proof}

The next lemma aims  to bound the expectation  in \eqref{amplitude}: 
\begin{lemma}\label{lem:amp3}
Assume $\partial\Sigma\not=\emptyset$ with $b$ boundary components. For all $R>0$, there exist some constants $C_R,a_R>0$ such that 
\begin{align*} 
  \E \big[ \prod_{i=1}^m V_{\alpha_i,g}(x_i)e^{-\frac{Q}{4\pi}\int_\Sigma K_g X_{g,D}\dd {\rm v}_g-\mu M_\gamma^g(\phi_g,\Sigma)}\big]\leq C_R e^{\sum_i\alpha_i \bar c_--R \bar c_++a_R\sum_{j=1}^b |c_j-\bar c|}B(\boldsymbol{\varphi})
\end{align*}
where $ B(\boldsymbol{\varphi})>0$ and $\int  B(\boldsymbol{\varphi})^p \,\P_\T^{\otimes_b}(\dd \boldsymbol{\varphi})<\infty$ for all $p<\infty$.
\end{lemma}

\begin{proof} The proof relies on an elementary use of    the Girsanov transform applied to the product 
\[e^{-\frac{Q}{4\pi}\int_\Sigma K_g X_{g,D}\dd {\rm v}_g}\prod_{i=1}^m V_{\alpha_i,g}(x_i).\]
This transform has the effect of shifting the law of the field $X_{g,D}$ by the function $u(x):=\sum_{i=1}^m \alpha_iG_{g,D}(x,x_i)-\frac{Q}{4\pi} \int_\Sigma G_{g,D}(x,y)K_g(y)  {\rm dv}_g(y)$, which is smooth on $\Sigma\setminus\{{\bf x}\}$. Therefore, the expectation is bounded by
\begin{equation}\label{eq:boundexp}
e^{v(\mathbf{x})}\prod_{i=1}^m e^{\alpha_i P\tilde{\boldsymbol{\varphi}}(x_i)} \E \big[ \exp\big(-\mu \int_\Sigma e^{\gamma u (x)}M^g_\gamma(\phi_g,\dd x)\big)\big]
\end{equation}
where the function (recall the definition \eqref{varYg})  
\begin{align*}
 v (\mathbf{x}):=&\sum_{i=1}^m\frac{\alpha_i^2}{2}W^D_g(x_i)+\sum_{i\not =i'}\alpha_i\alpha_{i'}G_{g,D}(x_i,x_{i'}) +\frac{Q^2}{32\pi^2}\iint_{\Sigma^2}G_{g,D}(x,x')K_g(x)K_g(x'){\rm dv}_g(x){\rm dv}_g(x')\\
 &-\frac{Q}{4\pi}\sum_{i=1}^m\alpha_i \int_\Sigma G_{g,D}(x_i,x)K_g(x) {\rm dv}_g(x) 
 \end{align*}
comes from the variance of the Girsanov transform. In this expression the first term comes from the renormalisation of vertex operators.

  Then we focus on the product involving  harmonic extensions and we write  $P\tilde{\boldsymbol{\varphi}}=\bar c+f+P\boldsymbol{\varphi}$ where $f$ is the harmonic extension on $\Sigma$ of the boundary fields $c_j-\bar c$ on the boundary curve $\mathcal{C}_j$. By the maximum principle, $|f(x)|\leq \sum_{j=1}^b |c_j-\bar c|$. Therefore 
 \[\prod_{i=1}^me^{\alpha_i P \tilde{\boldsymbol{\varphi}}(x_i)} \leq   e^{(\sum_i\alpha_i)\bar c+(\sum_{i}|\alpha_i|)\sum |c_j-\bar c|)}e^{(\sum_{i=1}^m\alpha_i)\sup_{i=1,\dots,m}P \boldsymbol{\varphi}(x_i)}.\]

When $\bar{c}\leq 0$, we simply bound the exponential term in the expectation in \eqref{eq:boundexp} by $1$. When $\bar{c}\geq 0$, choose now any non empty closed ball $B$, which  contains no $z_i$ and such that $B\cap\partial \Sigma=\emptyset$. We have $  \inf_{x\in B} u (x)>-\infty$ so that, using the inequality $e^{-x}\leq C_Rx^{-R}$ for $x>0$, the expectation appearing in \eqref{eq:boundexp} is bounded by 
\[C_Re^{ -\gamma R \inf_{x\in B} P \boldsymbol{\varphi}(x)-\gamma R\bar c+\gamma R\sum |c_j-\bar c|}\E[ M^{g}_\gamma(X_{g,D},B)^{-R}],\] 
where we have treated the constant modes similarly as before.

We complete the proof by applying the H\"older inequality,  by recalling  that GMC measures possess negative moments of all orders \cite[Th. 2.12]{rhodes2014_gmcReview} and by using that the random variables $e^{   \inf_{x\in K} P \boldsymbol{\varphi}(x)}$ or $e^{  \sup_{x\in K} P \boldsymbol{\varphi}(x)}$ possess moments of all orders for any compact subset $K$ of $\Sigma^\circ$.
\end{proof}   
  
\begin{proof}[Proof of Theorem \ref{integrcf}.]
To complete the proof of our claim, there is only one more thing to analyze: the contribution of the harmonic extension in the integral with curvature in the expectation  \eqref{amplitude}. For this, and again,  we write  $P\tilde{\boldsymbol{\varphi}}=\bar c+f+P\boldsymbol{\varphi}$ where $f$ is the harmonic extension on $\Sigma$ of the boundary fields $c_j-\bar c$ on $\mathcal{C}_j$. Then 
\begin{align*}%\label{}
\frac{Q}{4\pi}\int_\Sigma K_g  P\tilde{\boldsymbol{\varphi}} dv_g= R_1+R_2+R_3  
\end{align*}
with
\[\begin{gathered}
R_1:=\frac{Q}{4\pi}\int_\Sigma K_g \bar c \, {\rm dv}_g=Q\chi(\Sigma) \bar c ,\qquad 
R_2:=\frac{Q}{4\pi}\int_\Sigma K_g f \, {\rm dv}_g\geq -\sum_j  d|c_j-  \bar c|,\\ 
R_3:=\frac{Q}{4\pi}\int_\Sigma K_g P\boldsymbol{\varphi} \, {\rm dv}_g\geq -  \sum_{n\neq 0}d_n\sum_{j=1}^b |\varphi_{j,n}|
\end{gathered}\]
for some constant $d>0$ and some sequence $d_n=\mc{O}(n^{-\infty})$ with fast decay as $n\to\infty$. For $R_1$ above, we used Gauss-Bonnet. For $R_2$, we use the maximum principle to bound $f$ and for $R_3$ we use that 
\[\int_\Sigma K_g P\boldsymbol{\varphi} {\rm dv}_g=\int_{\pl \Sigma} \boldsymbol{\varphi} P^*K_g {\rm d\ell}_g\]
where $P^*$ is the adjoint operator to $P$ and that $P^*$ maps $C^\infty(\Sigma)$ to $C^\infty(\pl \Sigma)$ by Lemma \ref{lem:adjointP} in the Appendix (here $K_g\in C^\infty(\Sigma)$). We complete the proof by gathering these estimates and Lemmas \ref{lem:amp1}  and \ref{lem:amp3}, by noticing that for each $R>0$ there is $C_R>0$ so that for all $c_j$ 
\[-\frac{a}{2}\sum_j|c_j-\bar{c}|^2+(a_R+d)\sum_j|c_j-\bar{c}|\leq C_R.\qedhere\]
 \end{proof}
 
 We complete this section by studying how amplitudes react to conformal changes of metrics and to reparametrisation by diffeomorphisms:  
 \begin{proposition}\label{Weyl} The amplitudes obey the following transformation rules:\\
{\bf 1) Weyl covariance:}  Let   $\omega\in C^\infty(\Sigma)$ with in addition $\omega|_{\pl \Sigma}=0$ if $\partial\Sigma\not=\emptyset$. Then for $F$  measurable  nonnegative function on $H^{-s}(\Sigma)$  (recall \eqref{SL0})
\begin{align*}
\caA_{\Sigma,e^\omega g,{\bf x},\boldsymbol{\alpha}}(F,\tilde{\boldsymbol{\varphi}})= \caA_{\Sigma,g,{\bf x},\boldsymbol{\alpha}}\big(F(\cdot-\frac{Q}{2}\omega ),\tilde{\boldsymbol{\varphi}}\big)\exp\Big((1+6Q^2)S_{\rm L}^0(\Sigma,g,\omega)-\sum_i\Delta_{\alpha_i}\omega(x_i) \Big).
\end{align*}
{\bf 2) Diffeomorphism invariance:} Let $\psi:\Sigma\to\Sigma'$ be an  orientation preserving diffeomorphism and $F$  measurable  nonnegative function on $H^{-s}(\Sigma)$. Setting $F_\psi(\phi):=F(\phi\circ\psi)$,  we have
 \begin{align*} 
\caA_{\Sigma,  \psi^*g,{\bf x},\boldsymbol{\alpha},\boldsymbol{\zeta}}(F,\tilde{\boldsymbol{\varphi}})= \caA_{\Sigma', g,\psi({\bf x}),\boldsymbol{\alpha},\psi\circ\boldsymbol{\zeta}}\big(F_\psi ,\tilde{\boldsymbol{\varphi}}\big).
\end{align*}
\end{proposition}

\begin{proof}
We prove the case $\partial \Sigma\not=\emptyset$ as  $\partial \Sigma=\emptyset$  is proved in \cite[Proposition 4.6]{GRVIHES}. Let us first prove 1). We have to deal with a change of conformal metrics $g'=e^{\omega} g$ in the definition \eqref{amplitude}.    
Recall from \eqref{scalingmeasure} that $M_\gamma^{g'}(X_{g',D},\dd x)=e^{\frac{\gamma Q}{2}\omega(x)}M_\gamma^{g}(X_{g,D},\dd x)$ and  from \eqref{scalingvertex} that $V_{\alpha_i,g'}(x_i)=e^{\frac{\alpha_i^2}{4}\omega(x_i)}V_{\alpha_i,g}(x_i)$ (the last identity making sense when inserted inside expectation values). 
Then, because of the relation $k_{g'}=e^{-\omega/2}(k_g-\partial_\nu\omega/2)$ with $\nu$ the unit inward normal at $\partial\Sigma$, the boundary term  in  \eqref{amplitude} becomes
\begin{equation}\label{boundgeod}
\frac{Q}{2\pi}\int_{\partial\Sigma}k_{g'}P\tilde{\boldsymbol{\varphi}}\dd \ell_{g'}=\frac{Q}{2\pi}\int_{\partial\Sigma}(k_{g}-\partial_\nu\omega/2)P\tilde{\boldsymbol{\varphi}}\dd \ell_{g}. 
\end{equation}
Also, from the relation for curvatures $K_{g'}=e^{-\omega}(K_g +\Delta_g\omega)$ we deduce that the term involving the curvature in  \eqref{amplitude} reads (recall that $X_{g',D}=X_{g,D}$)
\begin{align*}
\frac{Q}{4\pi}\int_\Sigma K_{g'} (X_{g',D} +P\tilde{\boldsymbol{\varphi}})\dd {\rm v}_{g'}=& \frac{Q}{4\pi}\int_\Sigma  K_g (X_{g,D}+P\tilde{\boldsymbol{\varphi}})\dd {\rm v}_g+ \frac{Q}{4\pi}\int_\Sigma  \Delta_g\omega (X_{g,D}+P\tilde{\boldsymbol{\varphi}})\dd {\rm v}_g.
\end{align*}
In the last integral above, the contribution of the harmonic extension is treated  by the Green identity to produce
$$ \frac{Q}{4\pi}\int_\Sigma  \Delta_g\omega  P\tilde{\boldsymbol{\varphi}} \dd {\rm v}_g=\frac{Q}{4\pi}\int_{\partial \Sigma}\partial_\nu \omega P\tilde{\boldsymbol{\varphi}}\,\dd \ell_{g},$$
which cancels out with the corresponding term in \eqref{boundgeod}.
Therefore the expectation in  \eqref{amplitude} expressed in the metric $g'$ writes
$$  \E \big[F(\phi_g)e^{ -\frac{Q}{4\pi}\int_\Sigma  \Delta_g\omega  X_{g,D} \dd {\rm v}_g}\prod_{i=1}^m V_{\alpha_i,g'}(x_i)e^{-\frac{Q}{4\pi}\int_\Sigma K_g (X_{g,D}+P\tilde{\boldsymbol{\varphi}})\dd {\rm v}_g-\frac{Q}{2\pi}\int_{\partial\Sigma}k_{g}P\tilde{\boldsymbol{\varphi}}\dd \ell_{g}-\mu M_\gamma^{g'}(X_{g',D},\Sigma)}\big].$$
Now we apply the Girsanov transform to the first exponential term. The variance of this transform is given by 
$$ \frac{Q^2}{16\pi^2}\int_{\Sigma^2}  \Delta_g\omega (x) G_D(x,y)   \Delta_g\omega (y)   {\rm dv}_g(x) {\rm dv}_g(y)= \frac{Q^2}{8\pi}\int_{\Sigma}|d\omega|_g^2{\rm dv}_g.$$
It has the effect of shifting the mean of the GFF $X_{g,D}$, i.e. $X_{g,D}$ becomes $X_{g,D}-\frac{Q}{2}\omega$. Therefore the expectation becomes
\begin{align*} 
\exp\Big(&\frac{ 6Q^2}{96\pi}\int_{\Sigma}(|d\omega|_g^2+2K_g\omega) {\rm dv}_g-\sum_i\Delta_{\alpha_i}\omega(z_i)\Big)\\
&\times \E \big[F(\phi_g-\frac{Q}{2}\omega) \prod_{i=1}^n V_{\alpha_i,g}(x_i)e^{-\frac{Q}{4\pi}\int_\Sigma K_g (X_{g,D}+P\tilde{\boldsymbol{\varphi}})\dd {\rm v}_g-\frac{Q}{2\pi}\int_{\partial\Sigma}k_{g}P\tilde{\boldsymbol{\varphi}}\dd \ell_{g}-\mu M_\gamma^{g,D}(\Sigma)}\big],
\end{align*}
which is almost our claim up to the term $1+$ in front of the $6Q^2$. This term comes from the Polyakov formula for the regularised determinant of Laplacian \cite[section 1]{OsgoodPS88}: indeed Polyakov's formula implies that 
\begin{align*}
 Z_{\Sigma,  g'}=&\det (\Delta_{g',D})^{-\hf}\exp\big(\frac{1}{8\pi}\int_{\partial\Sigma}k_{g'}\,\dd\ell_{g'}\big)\\
=&  \det (\Delta_{g,D})^{-\hf}\exp\Big(\frac{1 }{96\pi}\int_{\Sigma}(|d\omega|_g^2+2K_g\omega) {\rm dv}_g  +\frac{1}{16\pi}\int_{\partial\Sigma}\partial_\nu\omega\,\dd\ell_g+\frac{1}{8\pi}\int_{\partial\Sigma}k_{g'}\,\dd\ell_{g'}\Big).
\end{align*}
Here we have used $\omega|_{\pl \Sigma}=0$. We conclude by using again the relation $k_{g'}=e^{-\omega/2}(k_g-\partial_\nu\omega/2)$. We also notice that $\mc{A}^0_{\Sigma,g}(\tilde{\boldsymbol{\varphi}})=\mc{A}^0_{\Sigma,g'}(\tilde{\boldsymbol{\varphi}})$ since the quadratic form 
$({\bf D}_\Sigma \tilde{\boldsymbol{\varphi}},\tilde{\boldsymbol{\varphi}})_{2}=\int_{\Sigma}|\nabla^gP\tilde{\boldsymbol{\varphi}}|_g^2{\rm dv}_g$ 
depends only on the conformal class of $g$ for $\tilde{\boldsymbol{\varphi}}$ smooth (and using the density of $C^\infty(\pl \Sigma)\subset 
H^{s}(\pl\Sigma)$). 

For 2), it easily follows from the fact that $\phi_g=\phi_{\psi_*g}\circ\psi$.
 \end{proof}
 
 %%%%%%%%%%%%%%%%%%%%%%%%%%%%%%%%%%%%%% %%%%%%%%%%%%%%%%%%%%%%%%%%%%%%%%%%%%%% %%%%%%%%%%%%%%%%%%%%%%%%%%%%%%%%%%%%%% %%%%%%%%%%%%%%%%%%%%%%%%%%%%%%%%%%%%%% %%%%%%%%%%%%%%%%%%%%%%%%%%%%%%%%%%%%%% %%%%%%%%%%%%%%%%%%%%%%%%%%%%%%%%%%%%%%
 \section{Gluing of surfaces and amplitudes}\label{sec:gluing}
 %%%%%%%%%%%%%%%%%%%%%%%%%%%%%%%%%%%%%% %%%%%%%%%%%%%%%%%%%%%%%%%%%%%%%%%%%%%% %%%%%%%%%%%%%%%%%%%%%%%%%%%%%%%%%%%%%% %%%%%%%%%%%%%%%%%%%%%%%%%%%%%%%%%%%%%% %%%%%%%%%%%%%%%%%%%%%%%%%%%%%%%%%%%%%% %%%%%%%%%%%%%%%%%%%%%%%%%%%%%%%%%%%%%%
 
In this section, we prove the gluing properties of the Segal amplitudes, in the case of two different surfaces (Proposition \ref{glueampli}) and then in the case of self-gluing (Proposition \ref{selfglueampli}). In some sense, the main part of the proof is the gluing/composition of the free field amplitudes $\mc{A}^0_{\Sigma,g}$ and, for this, the core argument is Lemma \ref{density1} which is proved by using finite dimensional approximations and the properties of zeta regularised determinants of elliptic operators together with Fredholm determinants.\\  
 
 Let $(\Sigma_1,g_2, {\bf x}_1,\boldsymbol{\zeta}_1)$ and $(\Sigma_2,g_2, {\bf x}_2,\boldsymbol{\zeta}_2)$ be two admissible surfaces with $\partial\Sigma_i\neq\emptyset$ for $i=1,2$. Let  $\caC_{1,j}\subset\partial\Sigma_1$, $j\in\mathcal{J} :=  \{1,\dots,k\}$ with $k\leq \min(b_1, b_2)$,  be outgoing boundary components and $\caC_{2,j}\subset\partial\Sigma_2$, $j\in\mathcal{J}$,  be incoming boundary components. Then we can glue the two surfaces $(\Sigma_i,g_i)$, $i=1,2$, as described in Subsection \ref{sub:gluing} to form an admissible surface  $(\Sigma_1\#_{\mathcal{J}}\Sigma_2,g_1\# g_2,({\bf x}_1,{\bf x}_2),(\boldsymbol{\zeta}_1,\boldsymbol{\zeta}_2))$
 with $b_1+b_2-2k$ boundary circles by gluing the circle  $\caC_{1,j}$  in $\Sigma_1$ to the  circle  $\caC_{2,j}$  in $\Sigma_2$ so that the   glued circle has an annular neighborhood in $\Sigma_1\#_{\mathcal{J}}\Sigma_2$ with metric $g_1\# g_2=|z|^{-2}|dz|^2$ in local chart.    We consider marked points $\mathbf{x}_1=\{x_{11},\dots,x_{1m_1}\}$ on $\Sigma_1$ with respective weights $\boldsymbol{\alpha}_1=(\alpha_{11},\dots,\alpha_{1m_1})$ and   $\mathbf{x}_2=\{x_{21},\dots,x_{2m_2}\}$ on $\Sigma_2$ with respective weights $\boldsymbol{\alpha}_2=(\alpha_{21},\dots,\alpha_{2m_2})$. Set $\mathbf{x}:=(\mathbf{x}_1,\mathbf{x}_2)$, $\boldsymbol{\alpha}:=(\boldsymbol{\alpha}_1,\boldsymbol{\alpha}_2)$ and denote by $\boldsymbol{\zeta}$ the collection of parametrisations of the boundaries $\mc{C}_{1,j}$ and $\mc{C}_{2,j}$ where  the index $j$ runs over all possible indices with $j>k$. 
The surface $\Sigma$ thus has an analytic boundary made up of the curves $\caC_{1,j} $ for $j=k+1,\dots,b_1$ and the curves $\caC_{2,j}\subset $ for $j=k+1,\dots,b_2 $. We denote by $\mc{C}_j=\mc{C}_{1j}=\mc{C}_{2j}$ the glued curves on $\Sigma$ for $j=1,\dots,k$.
 Boundary conditions on $\Sigma$ will thus be written as couples $(\tilde{\boldsymbol{\varphi}}_1,\tilde{\boldsymbol{\varphi}}_2)\in   (H^{s}(\T))^{b_1-k}\times (H^{s}(\T))^{b_2-k}$. Similarly, for $i=1,2$, the surface $\Sigma_i$ had analytic boundary made up of the curves  $\caC_{i,j} $ for $j=1,\dots,k$ and  $\caC_{i,j} $ for $j=k+1,\dots,b_i$ and boundary conditions on $\Sigma_i$ will thus be written as couples $(\tilde{\boldsymbol{\varphi}},\tilde{\boldsymbol{\varphi}}_i)\in   (H^{s}(\T))^{k}\times (H^{s}(\T))^{b_i-k}$, $s<0$.

 The corresponding amplitudes compose as follows:
\begin{proposition}\label{glueampli}
Let $F_1,F_2$ be measurable  nonnegative functions respectively  on $H^{s}(\Sigma_1)$ and  $H^{s}(\Sigma_2)$ for $s<0$ and let us denote by $F_1\otimes F_2$ the functional on $H^{s}(\Sigma_1\#_{\mathcal{J}}\Sigma_2,g_1\# g_2)$ defined by  \[F_1\otimes F_2(\phi_{g_1\# g_2}):=F_1(\phi_{g_1\# g_2|\Sigma_1})F_2(\phi_{g_1\# g_2|\Sigma_2}).\] Then
 \begin{align*}%\label{}
&\caA_{\Sigma_1\#_{\mathcal{J}}\Sigma_2,g_1\# g_2, {\bf x},\boldsymbol{\alpha}}
(F_1\otimes F_2,\tilde{\boldsymbol{\varphi}}_1,\tilde{\boldsymbol{\varphi}}_2)\\
&\quad =C\int \caA_{\Sigma_1,g_1,{\bf x}_1,\boldsymbol{\alpha}_1}(F_1,\tilde{\boldsymbol{\varphi}} ,\tilde{\boldsymbol{\varphi}}_1)\caA_{\Sigma_2,g_2,{\bf x}_2,\boldsymbol{\alpha}_2}(F_2,\tilde{\boldsymbol{\varphi}},\tilde{\boldsymbol{\varphi}}_2) \dd\mu_0^{\otimes_k} (\tilde{\boldsymbol{\varphi}}).
\end{align*}
 where   $C= \frac{1}{(\sqrt{2} \pi)^{k}}$ if $\partial\Sigma \not =\emptyset$ and $C=\frac{\sqrt{2}}{(\sqrt{2} \pi)^{k-1}} $ if $\partial\Sigma =\emptyset$.
  \end{proposition}

One can also consider the case of self gluing. 
 Let $(\Sigma,g, {\bf x},\boldsymbol{\zeta})$ be an admissible surface with $b$ boundary components such that the boundary contains an  outgoing boundary component  $\caC_j \subset\partial\Sigma$ and an incoming boundary component $\caC_k\subset\partial\Sigma$. Then we can glue the two boundary components as described in Subsection \ref{sub:gluing} to produce the surface $(\Sigma_{jk},g, {\bf x},\boldsymbol{\zeta})$. In this context, we will write the boundary component for $\Sigma$ as  $(\tilde\varphi_j,\tilde\varphi_k  ,\tilde{\boldsymbol{\varphi}'}) \in   (H^{s}(\T)) \times (H^{s}(\T)) \times (H^{s}(\T))^{b-2}$ where $\tilde\varphi_j$ corresponds to  $\caC_j$, $\tilde\varphi_k$ corresponds to  $\caC_k$ and $\tilde{\boldsymbol{\varphi}}'$ corresponds to the boundary components of $\Sigma_{jk}$. Now, we have

\begin{proposition}\label{selfglueampli}
Let  $F$ be a measurable  nonnegative function on $H^{s}(\Sigma)$ for $s<0$
 \[
\caA_{\Sigma_{jk},g, {\bf x},\boldsymbol{\alpha}}
(F,\tilde{\boldsymbol{\varphi}}')=C \int \caA_{\Sigma,g,{\bf x},\boldsymbol{\alpha}}(F,\tilde\varphi , \tilde\varphi, \tilde{\boldsymbol{\varphi}}') \dd\mu_0(\tilde{\varphi}).
\]
where $C= \frac{1}{\sqrt{2} \pi}$ if $\partial\Sigma \not =\emptyset$ and $C= \sqrt{2}  $ if $\partial\Sigma =\emptyset$.
 \end{proposition}

\noindent 
\emph{Proof of propositions \ref{glueampli} and \ref{selfglueampli}.}  We split the proof in two cases depending on whether the resulting surface $\Sigma_1\#_{\mathcal{J}}\Sigma_2$ has a non trivial boundary (case 1) or not (case 2). We deal first with case 1 and, since the proof of propositions \ref{glueampli} and \ref{selfglueampli} are the same up to notational changes, we only deal with Proposition \ref{glueampli}.
For notational simplicity, we will write $\Sigma$ for $\Sigma_1\#_{\mathcal{J}}\Sigma_2$, $g$ for $g_1\# g_2$ and $m=m_1+m_2$.

\medskip
$\bullet$ Assume first $\partial\Sigma \not= \emptyset$. Let us denote   the union of glued boundary components by $\mathcal{C}:=\bigcup_{j=1}^k\mc{C}_j$  and define
\begin{equation}
\mathcal{B}_{\Sigma,g}(F_1\otimes F_2,\tilde{\boldsymbol{\varphi}}_1,\tilde{\boldsymbol{\varphi}}_2):=\E \big[F_1\otimes F_2(\phi_g) \prod_{i=1}^{m} V_{\alpha_i,g}(x_i)e^{-\frac{Q}{4\pi}\int_\Sigma K_g \phi_g\dd {\rm v}_g-\mu M^g_\gamma(\phi_g,\Sigma)}\big]
\end{equation}
where $\phi_g=X_{g,D}+P(\tilde{\boldsymbol{\varphi}}_1,\tilde{\boldsymbol{\varphi}}_2)$, $X_{g,D}$ is the Dirichlet GFF on $\Sigma$ and expectation $\E$ is taken over this Dirichlet GFF, $P(\tilde{\boldsymbol{\varphi}}_1,\tilde{\boldsymbol{\varphi}}_2)$ stands for the harmonic extension to $\Sigma$ of the boundary fields $\tilde{\boldsymbol{\varphi}}_1,\tilde{\boldsymbol{\varphi}}_2$, which stand respectively for the  boundary conditions on the remaining (i.e. unglued) components of $\partial \Sigma_1$ and  $\partial \Sigma_2$, namely
\[\Delta_g P(\tilde{\boldsymbol{\varphi}}_1,\tilde{\boldsymbol{\varphi}}_2)=0\quad  \text{on }\Sigma ,\quad\quad \text{ and for }j>k\quad P(\tilde{\boldsymbol{\varphi}}_1,\tilde{\boldsymbol{\varphi}}_2)_{|\mc{C}_{1,j}}= \tilde{\varphi}_{1,j}\circ \zeta_{1,j}^{-1}  ,\quad\quad P(\tilde{\boldsymbol{\varphi}}_1,\tilde{\boldsymbol{\varphi}}_2)_{|\mc{C}_{2,j}}= \tilde{\varphi}_{2,j}\circ\zeta_{2,j}^{-1}  .\]

Therefore
$$\caA_{\Sigma,g, {\bf x},\boldsymbol{\alpha}}(F_1\otimes F_2,\tilde{\boldsymbol{\varphi}}_1,\tilde{\boldsymbol{\varphi}}_2)=Z_{\Sigma,g}\caA^0_{\Sigma,g}(\tilde{\boldsymbol{\varphi}}_1,\tilde{\boldsymbol{\varphi}}_2)\mathcal{B}_{\Sigma,g}(F_1\otimes F_2,\tilde{\boldsymbol{\varphi}}_1,\tilde{\boldsymbol{\varphi}}_2).$$
Let now $X_1:=X_{g_1,D}$ and $X_2:=X_{g_2,D}$ be two  independent Dirichlet GFF respectively on $\Sigma_1$ and $\Sigma_2$. Then we have the following decomposition in law (see Proposition \ref{decompGFF})
\begin{align*}%\label{}
X_{g,D}\stackrel{{\rm law}}=X_1+X_2+P{\bf Y}
\end{align*}
where ${\bf Y}$ is the restriction of $X_{g,D}$ to the  glued boundary components $\mathcal{C}$ expressed in parametrised coordinates, i.e.   ${\bf Y} =( X_{g,D|_{{\caC}_1}}\circ \zeta_{1,1},\dots,X_{g,D|_{{\caC}_k}}\circ \zeta_{1,k})$, and  $P{\bf Y}$ is its harmonic extension to $ \Sigma $ vanishing on $\partial \Sigma $, which is non empty. We stress that, since  $X_{g,D}$ is only a distribution,  making sense of ${\bf Y}$ is not completely obvious but, using the parametrisation, this can be made in the same way as making sense of the restriction of the GFF to a circle: since this is a standard argument, we do not elaborate more on this point. Finally we denote by $h_\mathcal{C}$ the restriction of  the harmonic function $P(\tilde{\boldsymbol{\varphi}}_1,\tilde{\boldsymbol{\varphi}}_2)$ to $\mathcal{C}$ in parametrised coordinates
\[ h_{\mc{C}}:=(P(\tilde{\boldsymbol{\varphi}}_1,\tilde{\boldsymbol{\varphi}}_2)_{|\mc{C}_{1}}\circ\zeta_{1,1},\dots,P(\tilde{\boldsymbol{\varphi}}_1,\tilde{\boldsymbol{\varphi}}_2)_{|\mc{C}_{k}}\circ \zeta_{1,k}).\]
Observe now the trivial fact that, on $\Sigma_i$ ($i=1,2$), the function $P(\tilde{\boldsymbol{\varphi}}_1,\tilde{\boldsymbol{\varphi}}_2)+P{\bf Y}$ is harmonic with boundary values (expressed in parametrised coordinates  on $\Sigma_i$) $ \tilde\varphi_{i,j}$ on $\mathcal{C}_{ij}$ for $j>k$ and $({\bf Y}+h_{\mathcal{C}})_j$ on $\mc{C}_{j}$ for $j=1,\dots,k$.  We observe that $h_{\mc{C}}$ is a random function with values in $(H^N(\T))^k$ for all $N>0$. 
Thus we get
\begin{equation}
\mathcal{B}_{\Sigma,g}(F_1\otimes F_2,\tilde{\boldsymbol{\varphi}}_1,\tilde{\boldsymbol{\varphi}}_2)= \int \mathcal{B}_ {\Sigma_1,g_1}(F_1,\tilde{\boldsymbol{\varphi}} +h_{\mc{C}},\tilde{\boldsymbol{\varphi}}_1 )\mathcal{B}_ {\Sigma_2,g_2}(F_2,\tilde{\boldsymbol{\varphi}} +h_{\mc{C}},\tilde{\boldsymbol{\varphi}}_2 )\dd \P_{{\bf Y}} (\tilde{\boldsymbol{\varphi}})
\end{equation}
with $\P_{{\bf Y}}$ the law of $ {\bf Y} $ and, for $i=1,2$,
\begin{equation}
\mathcal{B}_{\Sigma_i,g_i}(F_i,\tilde{\boldsymbol{\varphi}} ,\tilde{\boldsymbol{\varphi}}_i):=\E \big[F_i( \phi_i)\prod_{q=1}^{m_i } V_{\alpha_{iq},g_i}(x_{iq})e^{-\frac{Q}{4\pi}\int_{\Sigma_i} K_{g_i} \phi_i\dd {\rm v}_{g_i}-\mu M_\gamma(\phi_i,\Sigma_i)}\big]
\end{equation}
where $\E$ is taken with respect to the Dirichlet GFF  $X_i$  on $\Sigma_i$, $\phi_i=X_i+P(\tilde{\boldsymbol{\varphi}} ,\tilde{\boldsymbol{\varphi}}_i)$ and $P(\tilde{\boldsymbol{\varphi}} ,\tilde{\boldsymbol{\varphi}}_i)$ stands for the harmonic extension on $\Sigma_i$ of the boundary fields $\tilde{\boldsymbol{\varphi}} ,\tilde{\boldsymbol{\varphi}}_i$ respectively on $\mathcal{C}$ and $\partial\Sigma_i\setminus \mathcal{C}$. Notice that the orientation of incoming boundary curves $\mc{C}_{2,j}$ for $j\leq k$ on $\Sigma_2$ is opposite to their orientation as curves drawn on $\Sigma$   due to \eqref{gluingonboundary}, which is used to get  the relation above. Writing   $T_* \P_{\boldsymbol{Y}} $ for the pushforward of the measure $ \P_{\boldsymbol{Y}} $ by the map $T:\tilde{\boldsymbol{\varphi}}\mapsto \tilde{\boldsymbol{\varphi}}+h_{\mc{C}}$ and then writing the $\mathcal{B}$'s in terms of the amplitudes on $\Sigma_1$ and $\Sigma_2$, we obtain
\begin{equation}\label{comput1}
\begin{split}
\caA_{\Sigma,g, {\bf x},\boldsymbol{\alpha}}(F_1\otimes  F_2,\tilde{\boldsymbol{\varphi}}_1,\tilde{\boldsymbol{\varphi}}_2)=& \frac{Z_{\Sigma,g}}{Z_{\Sigma_1,g_1}Z_{\Sigma_2,g_2}}\caA^0_{\Sigma,g}(\tilde{\boldsymbol{\varphi}}_1,\tilde{\boldsymbol{\varphi}}_2) \\
& \times \int \frac{\mathcal{A}_{\Sigma_1,g_1,{\bf x}_1,\boldsymbol{\alpha}_1}(F_1,\tilde{\boldsymbol{\varphi}}  ,\tilde{\boldsymbol{\varphi}}_1 )}{\caA^0_{\Sigma_1,g_1}(\tilde{\boldsymbol{\varphi}} ,\tilde{\boldsymbol{\varphi}}_1)} \frac{\mathcal{A}_{\Sigma_2,g_2,{\bf x}_2,\boldsymbol{\alpha}_2}(F_2,\tilde{\boldsymbol{\varphi}}  ,\tilde{\boldsymbol{\varphi}}_2 )}{\caA^0_{\Sigma_2,g_2}(\tilde{\boldsymbol{\varphi}} ,\tilde{\boldsymbol{\varphi}}_2)} \,  \dd T_*\P_{\boldsymbol{Y}} (\tilde{\boldsymbol{\varphi}}) .
\end{split}\end{equation}

The proof of the case  $\partial\Sigma  \not= \emptyset$ is completed with the two following lemmas:

\begin{lemma}\label{density1}
The measure $T_*\P_{\boldsymbol{Y}} $  satisfies
\[T_*\P_{\boldsymbol{Y}} = \pi^{-k/2}  {\rm det}_{\rm Fr}(\mathbf{D}_{\Sigma,{\mc{C}}}(2\mathbf{D}_0)^{-1})^{1/2} \frac{\caA^0_{\Sigma_1,g_1}(\tilde{\boldsymbol{\varphi}} ,\tilde{\boldsymbol{\varphi}}_1)\caA^0_{\Sigma_2,g_2}(\tilde{\boldsymbol{\varphi}} ,\tilde{\boldsymbol{\varphi}}_2)}{\caA^0_{\Sigma,g}(\tilde{\boldsymbol{\varphi}}_1,\tilde{\boldsymbol{\varphi}}_2)}  \dd\mu_0^{\otimes_k} (\tilde{\boldsymbol{\varphi}})\]
with $\mathbf{D}_{\Sigma,\mc{C}}$ the Dirichlet-to-Neumann operator on $L^2(\mc{C})=L^2(\mathbb{T})^k$ defined as in \eqref{defDSigmaC}, $\mathbf{D}_0:=\Pi_0+\mathbf{D}$  on $L^2(\T)^k$ defined using \eqref{defmathbfD} and \eqref{tildeD}, and the Fredholm determinant  $\det_{\rm Fr}(\mathbf{D}_{\Sigma,{\mc{C}}}(2\mathbf{D}_0)^{-1})$ is well defined by Lemma \ref{lemmaDSigma-D}. 
\end{lemma}
\begin{proof}

First we get rid of the shift with the help of Cameron-Martin. Let us denote by $G:H^s(\T)^k\to H^{s+1}(\T)^k$ the operator acting by $(Gf)_j(e^{i\theta}):=\sum_{j'=1}^k\int_0^{2\pi}\E[{\bf Y}_j(e^{i\theta}){\bf Y}_{j'}(e^{i\theta'})]f_{j'}(e^{i\theta'})\,\dd \theta'$, which is nothing but the Green operator restricted to $\mc{C}$ in local coordinates. By Cameron-Martin, and using from \eqref{DNmapandGreen} that $G\mathbf{D}_{\Sigma,{\mc{C}}} =2\pi {\rm Id}$ on $H^s(\T)^k$ 
\begin{equation}\label{CameronMartin}
T_*\P_{{\bf Y}}=\exp\big(({\bf Y},{\bf D}_{\Sigma,{\mc{C}}} h_{\mc{C}})_2-\frac{1}{2}(h_{\mc{C}},{\bf D}_{\Sigma,{\mc{C}}} h_{\mc{C}})_2\big) \P_{\boldsymbol{Y}}.
\end{equation}
Now we claim for measurable bounded functions $F$
\begin{equation}\label{densityDN}
\int F( \tilde{\boldsymbol{\varphi}})\P_{\boldsymbol{Y}}(\dd \tilde{\boldsymbol{\varphi}})=\frac{1}{ \pi^{k/2}\det(\mathbf{D}_{\Sigma,{\mc{C}}}(2\mathbf{D}_0)^{-1})^{-1/2} }\int F(    \tilde{\boldsymbol{\varphi}})\exp(-\frac{1}{2}(\tilde{\boldsymbol{\varphi}},\widetilde{\mathbf{D}}_{\Sigma,{\mc{C}}} \tilde{\boldsymbol{\varphi}})_2)\dd\mu_0^{\otimes_k} (\tilde{\boldsymbol{\varphi}}).
\end{equation}
Indeed, consider the projection $\Pi_N:(H^s(\T))^k\to (H^s(\T))^k$ where the Fourier series of each component is truncated to its first $N$   Fourier components, and let $\mc{H}_N=\Pi_N(H^s(\T))^k$. Then we set $ \mathbf{D}_{\Sigma,{\mc{C}}}^N:=\Pi_N\circ  \mathbf{D}_{\Sigma,{\mc{C}}}\circ \Pi_N$, $ \mathbf{D}_0^N:=\Pi_N\circ  \mathbf{D}_0\circ \Pi_N$, $ \mathbf{D}^N:=\Pi_N\circ  \mathbf{D}\circ \Pi_N$, $\widetilde{\mathbf{D}}_{\Sigma,{\mc{C}}}^N:=\Pi_N\circ (\mathbf{D}_{\Sigma,{\mc{C}}}-2\mathbf{D})\circ \Pi_N$ and
\[\dd\P^N(\tilde{\boldsymbol{\varphi}}):=\frac{1}{\pi^{k/2}\det_{\mc{H}_N}(\mathbf{D}_{\Sigma,{\mc{C}}}^N(2\mathbf{D}^N_0)^{-1})^{-1/2} }\exp\Big(-\frac{1}{2}(\tilde{\boldsymbol{\varphi}},\widetilde{\mathbf{D}}^N_{\Sigma,{\mc{C}}} \tilde{\boldsymbol{\varphi}})_2\Big)\dd\mu_0^{\otimes_k} (\tilde{\boldsymbol{\varphi}}).\]
Let  $\mathbf{c}:=(c_1,\dots,c_k)$ and $|\mathbf{c}|^2=\sum_{i=1}^kc_i^2$ its Euclidean norm. First notice that  $(\tilde{\boldsymbol{\varphi}}, \mathbf{D}_{\Sigma,{\mc{C}}} \tilde{\boldsymbol{\varphi}})_2\geq a |\mathbf{c}|^2+a(\tilde{\boldsymbol{\varphi}},2 \mathbf{D}  \tilde{\boldsymbol{\varphi}})_2$ for some $a>0$ (by
 Lemma \ref{lemmaDSigma-D}, $(2{\bf D}_0)^{-1/2}\mathbf{D}_{\Sigma,{\mc{C}}}(2{\bf D}_0)^{-1/2}$ is a positive Fredholm self-adjoint  bounded operator  on $L^2$ of the form $1+K$ for $K$ compact, thus $(2{\bf D}_0)^{-1/2}\mathbf{D}_{\Sigma,{\mc{C}}}(2{\bf D}_0)^{-1/2}\geq a\, {\rm Id}$ for some $a>0$), 
 we deduce that $(\tilde{\boldsymbol{\varphi}},\widetilde{\mathbf{D}}^N_{\Sigma,{\mc{C}}} \tilde{\boldsymbol{\varphi}})_2\geq  a |\mathbf{c}|^2+(a-1) (\Pi_N\tilde{\boldsymbol{\varphi}},2 \mathbf{D}  \Pi_N\tilde{\boldsymbol{\varphi}})_2$, from which we deduce that $\P^N$ is a finite measure. Observe now that computing its total mass boils down to computing a finite dimensional Gaussian integral, which can be computed to show  that $\P^N$ is a probability measure.   
Indeed, this follows from the two following facts. First the probability measure  $\P_\T ^{\otimes_k}\circ \Pi_N^{-1} $ is given by
\[\int F(\Pi_N\boldsymbol \varphi)\P_\T ^{\otimes_k}  (\dd \boldsymbol \varphi) =\int F(\Pi_N\boldsymbol \phi) e^{-\frac{1}{2}(\Pi_N \boldsymbol \phi,2\mathbf{D} \boldsymbol \Pi_N \boldsymbol \phi)}\frac{\dd \Pi_N\boldsymbol \phi}{(2\pi)^{Nk}\det_{\mc{H}_N}'( 2\mathbf{D}^N)^{-1/2} }\]
with $\boldsymbol \phi:=(\phi_1,\dots,\phi_k)$, $\phi_j(\theta)=\sum_{n>0}\big(u^j_n\sqrt{2}\cos(n\theta)-v^j_n\sqrt{2}\sin(n\theta)\big)$ for $j=1,\dots,k$, $\dd \Pi_N\boldsymbol \phi=\prod_{j=1}^k\prod_{n=1}^N \dd u_n^j\dd v_n^j$. Furthermore $\det_{\mc{H}_N}'( 2\mathbf{D}^N)^{-1/2} =2^{k/2}\det_{\mc{H}_N}( 2\mathbf{D}_0^N)^{-1/2}$.
Second, we have the following relation for   nonnegative symmetric matrices $A,B$ with $B$ positive definite and $F$ continuous bounded
$$\int_{\R^p}F(x)e^{-\frac{1}{2}\langle (B-A) x,x\rangle}e^{-\frac{1}{2}\langle A x,x\rangle}\,\frac{\dd x}{(2\pi)^{p/2}\det(A_0)^{-1/2}}=\det(BA_0^{-1})^{-1/2}\int_{\R^p}F(x)e^{-\frac{1}{2}\langle B x,x\rangle} \,\frac{\dd x}{(2\pi)^{p/2}\det(B)^{-1/2}},$$
where $A_0$ is the symmetric positive definite matrix of the operator on $\R^p$ that is equal to identity on the kernel of $A$ and equal to $A$ on the orthogonal of the kernel. Those two relations also show that
$$\int F(\Pi_N \tilde{\boldsymbol{\varphi}})\P^N(\dd \tilde{\boldsymbol{\varphi}})=\frac{1}{(2\pi)^{(N+\frac{1}{2})k}\det_{\mc{H}_N}(\mathbf{D}_{\Sigma,{\mc{C}}}^N )^{-1/2} }\int F(\mathbf{c}+\Pi_N \boldsymbol{\phi})\exp(-\frac{1}{2}(\mathbf{c}+\boldsymbol{ \phi}, \mathbf{D}^N_{\Sigma,{\mc{C}}} (\mathbf{c}+\boldsymbol{ \phi}))_2)   \dd \mathbf{c}\dd \Pi_N\boldsymbol \phi.$$
In particular, for $f$   a trigonometric polynomial
$$\int e^{ ( f, \tilde{\boldsymbol{\varphi}})_2}\P^N(\dd \tilde{\boldsymbol{\varphi}})=e^{\frac{1}{2}(f, (\mathbf{D}_{\Sigma,{\mc{C}}}^N)^{-1}f)_2}\to e^{\frac{1}{4\pi}(f,Gf)_2}=\E[ e^{ (f, \boldsymbol{Y})_2}], \quad \text{as }N\to\infty$$
where we have used the fact that $(\mathbf{D}_{\Sigma,{\mc{C}}}^N|_{\mc{H}_N})^{-1}\Pi_N\to G/(2\pi)$ when $N\to \infty$, as bounded operators on $H^s(\T)^k$ for all $s\in \R$. On the other hand, we claim that the left-hand side above converges as $N\to\infty$ towards
\begin{equation}\label{limitconvdominee}
\frac{1}{ \pi^{k/2}\det_{\rm Fr}(\mathbf{D}_{\Sigma,{\mc{C}}}(2\mathbf{D}_0)^{-1})^{-1/2} }\int e^{ ( f, \tilde{\boldsymbol{\varphi}})_2}\exp(-\frac{1}{2}(\tilde{\boldsymbol{\varphi}},\widetilde{\mathbf{D}}_{\Sigma,{\mc{C}}} \tilde{\boldsymbol{\varphi}})_2)\dd\mu_0^{\otimes_k} (\tilde{\boldsymbol{\varphi}}).
\end{equation}
Indeed, since $\widetilde{\mathbf{D}}_{\Sigma,{\mc{C}}}$  is a smoothing operator (by Lemma \ref{lemmaDSigma-D}),  we have the convergence $(\tilde{\boldsymbol{\varphi}},\widetilde{\mathbf{D}}^N_{\Sigma,{\mc{C}}} \tilde{\boldsymbol{\varphi}})_2\to (\tilde{\boldsymbol{\varphi}},\widetilde{\mathbf{D}}_{\Sigma,{\mc{C}}} \tilde{\boldsymbol{\varphi}})_2$  almost everywhere (wrt  $(\dd c\otimes \P_\T)^{\otimes_k}$). Moreover  $\det_{\mc{H}_N}(\mathbf{D}_{\Sigma,{\mc{C}}}^N(2\mathbf{D}^N_0)^{-1})\to \det(\mathbf{D}_{\Sigma,{\mc{C}}}(2\mathbf{D}_0)^{-1})$ as $N\to\infty$: indeed, since $\Pi_N$ commutes with ${\bf D}_0$, we have 
 $\mathbf{D}_{\Sigma,{\mc{C}}}^N(2\mathbf{D}^N_0|_{\mc{H}_N})^{-1}\Pi_N=\Pi_N\mathbf{D}_{\Sigma,{\mc{C}}}(2\mathbf{D}_0)^{-1}\Pi_N$ and since $ \mathbf{D}_{\Sigma,{\mc{C}}}(2\mathbf{D}_0)^{-1}-{\rm Id}$ is smoothing, we have 
 $\Pi_N(\mathbf{D}_{\Sigma,{\mc{C}}}(2\mathbf{D}_0)^{-1}-{\rm Id})\Pi_N\to \mathbf{D}_{\Sigma,{\mc{C}}}(2\mathbf{D}_0)^{-1}-{\rm Id}$ in the space of trace class operators on $L^2(\T)^k$; finally the map $K\mapsto \det_{\rm Fr}({\rm Id}+K)$ is continuous on the space of trace class operators. We can then apply the dominated convergence theorem to obtain \eqref{limitconvdominee} using the following estimate: since for all $N_1>0$ and $\epsilon>0$, there is $N_0>0$ such that   for all $N>N_0$,  $\|\widetilde{\mathbf{D}}^N_{\Sigma,{\mc{C}}}-\widetilde{\mathbf{D}}_{\Sigma,{\mc{C}}}^{N_0}\|_{H^{-N_1}\to H^{N_1}}\leq \epsilon$ (again from Lemma \ref{lemmaDSigma-D}), there is $a>0$ such that for any $\epsilon>0$ and $N_1>0$, there is $N_0>0$ such that for all $N>N_0$ and $\tilde{\boldsymbol{\varphi}}\in H^{s}(\T)^k$ for $s<0$,
 \begin{equation}\label{lowerboundDNmaptilde}
 (\tilde{\boldsymbol{\varphi}},\widetilde{\mathbf{D}}^N_{\Sigma,{\mc{C}}} \tilde{\boldsymbol{\varphi}})_2\geq a|\mathbf{c}|^2+(a-1)(\Pi_{N_0}\tilde{\boldsymbol{\varphi}},2\mathbf{D}\Pi_{N_0}\tilde{\boldsymbol{\varphi}})_2-\epsilon\|\tilde{\boldsymbol{\varphi}}\|_{H^{-N_1}}^2.
 \end{equation}
This completes the proof of our claim \eqref{densityDN}.

Finally we compute the ratio of free field amplitudes: observe that for $\tilde{\boldsymbol{\varphi}}_i\in H^1(\T)^{b_i-k}$ and $\tilde{\boldsymbol{\varphi}} \in H^1(\T)^k$
\begin{align*}
 \frac{\caA^0_{\Sigma,g}(\tilde{\boldsymbol{\varphi}}_1,\tilde{\boldsymbol{\varphi}}_2)}{\caA^0_{\Sigma_1,g_1}(\tilde{\boldsymbol{\varphi}} ,\tilde{\boldsymbol{\varphi}}_1)\caA^0_{\Sigma_2,g_2}(\tilde{\boldsymbol{\varphi}} ,\tilde{\boldsymbol{\varphi}}_2)}  =&e^{-\frac{1}{2}(\tilde{\boldsymbol{\varphi}},2\mathbf{D}\tilde{\boldsymbol{\varphi}})_2+\frac{1}{2}((\tilde{\boldsymbol{\varphi}},\tilde{\boldsymbol{\varphi}}_1),\mathbf{D}_{\Sigma_1} (\tilde{\boldsymbol{\varphi}},\tilde{\boldsymbol{\varphi}}_1))_2+\frac{1}{2}((\tilde{\boldsymbol{\varphi}},\tilde{\boldsymbol{\varphi}}_2),\mathbf{D}_{\Sigma_2} (\tilde{\boldsymbol{\varphi}},\tilde{\boldsymbol{\varphi}}_2))_2-\frac{1}{2}((\tilde{\boldsymbol{\varphi}}_1,\tilde{\boldsymbol{\varphi}}_2),\mathbf{D}_{\Sigma} (\tilde{\boldsymbol{\varphi}}_1,\tilde{\boldsymbol{\varphi}}_2))_2} 
 \\ =& e^{-\frac{1}{2}(\tilde{\boldsymbol{\varphi}},2\mathbf{D}\tilde{\boldsymbol{\varphi}})_2-A_1-A_2+A_3+A_4}
 \end{align*}
with \begin{align*}
A_1&= ((0,\tilde{\boldsymbol{\varphi}}_1),\mathbf{D}_{\Sigma_1} (h_{\mc{C}},0))_2+((0,\tilde{\boldsymbol{\varphi}}_2),\mathbf{D}_{\Sigma_2} (h_{\mc{C}},0))_2\\
A_2&=\frac{1}{2}((h_{\mc{C}},0),\mathbf{D}_{\Sigma_1} (h_{\mc{C}},0))_2+\frac{1}{2}((h_{\mc{C}},0),\mathbf{D}_{\Sigma_2} (h_{\mc{C}},0))_2\\
A_3&=\frac{1}{2}((\tilde{\boldsymbol{\varphi}},0),\mathbf{D}_{\Sigma_1} (\tilde{\boldsymbol{\varphi}},0))_2+\frac{1}{2}((\tilde{\boldsymbol{\varphi}},0),\mathbf{D}_{\Sigma_2} (\tilde{\boldsymbol{\varphi}},0))_2\\
A_4&= ((0,\tilde{\boldsymbol{\varphi}}_1),\mathbf{D}_{\Sigma_1} (\tilde{\boldsymbol{\varphi}},0))_2+((0,\tilde{\boldsymbol{\varphi}}_2),\mathbf{D}_{\Sigma_2} (\tilde{\boldsymbol{\varphi}},0))_2,
 \end{align*}
 relation that we have obtained by using the relation 
\begin{equation}\label{zc1}
((\tilde{\boldsymbol{\varphi}}_1,\tilde{\boldsymbol{\varphi}}_2),\mathbf{D}_{\Sigma} (\tilde{\boldsymbol{\varphi}}_1,\tilde{\boldsymbol{\varphi}}_2))_2=((h_{\mc{C}},\tilde{\boldsymbol{\varphi}}_1),\mathbf{D}_{\Sigma_1} (h_{\mc{C}},\tilde{\boldsymbol{\varphi}}_1))_2+((h_{\mc{C}},\tilde{\boldsymbol{\varphi}}_2),\mathbf{D}_{\Sigma_2} (h_{\mc{C}},\tilde{\boldsymbol{\varphi}}_2))_2
\end{equation}
 in the first line above and then expanded all terms to get the second line. Let us mention here that \eqref{zc1} follows from the fact that the harmonic extension $P(\tilde{\boldsymbol{\varphi}}_1,\tilde{\boldsymbol{\varphi}}_2) $ on $\Sigma$ is smooth on $\mathcal{C}$ so that 
 \begin{equation}\label{zc0}
 -\partial_{\nu_-} P(\tilde{\boldsymbol{\varphi}}_1,\tilde{\boldsymbol{\varphi}}_2) _{|\mc{C}}-\partial_{\nu_+} P(\tilde{\boldsymbol{\varphi}}_1,\tilde{\boldsymbol{\varphi}}_2) _{|\mc{C}} =0.
 \end{equation}
  Now we stress that, by construction of the DN map $\mathbf{D}_{\Sigma,{\mc{C}}} $, we have for  all $\tilde{\boldsymbol{\varphi}}\in H^1(\T)^k$
\begin{equation}\label{zc2}
((\tilde{\boldsymbol{\varphi}},0),\mathbf{D}_{\Sigma_1} (\tilde{\boldsymbol{\varphi}},0))_2+((\tilde{\boldsymbol{\varphi}},0),\mathbf{D}_{\Sigma_2}(\tilde{\boldsymbol{\varphi}},0))_2=(\tilde{\boldsymbol{\varphi}},\mathbf{D}_{\Sigma,{\mc{C}}}\tilde{\boldsymbol{\varphi}})_2. 
\end{equation}
 Therefore $A_2=\frac{1}{2}(h_{\mc{C}},\mathbf{D}_{\Sigma,{\mc{C}}} h_{\mc{C}})_2$ and $A_3=\frac{1}{2}(\tilde{\boldsymbol{\varphi}},\mathbf{D}_{\Sigma,{\mc{C}}} \tilde{\boldsymbol{\varphi}})_2$. For $A_1$, we insert $(0,\tilde{\boldsymbol{\varphi}}_1)=(h_{\mc{C}},\tilde{\boldsymbol{\varphi}}_1)-(h_{\mc{C}},0)$, and similarly for $(0,\tilde{\boldsymbol{\varphi}}_2)$, and using that ${\bf D}_{\Sigma_i}^*={\bf D}_{\Sigma_i}$ we obtain 
\[((h_{\mc{C}},\tilde{\boldsymbol{\varphi}}_1),\mathbf{D}_{\Sigma_1} (h_{\mc{C}},0))_2+((h_{\mc{C}},\tilde{\boldsymbol{\varphi}}_2),\mathbf{D}_{\Sigma_2} (h_{\mc{C}},0))_2=-(\partial_{\nu_-} P(\tilde{\boldsymbol{\varphi}}_1,\tilde{\boldsymbol{\varphi}}_2) _{|\mc{C}}+\partial_{\nu_+} P(\tilde{\boldsymbol{\varphi}}_1,\tilde{\boldsymbol{\varphi}}_2) _{|\mc{C}},h_{\mc{C}})_2=0\]
Thus $A_1=-(h_{\mc{C}},\mathbf{D}_{\Sigma,{\mc{C}}} h_{\mc{C}})_2$ using \eqref{zc2}. The same trick applied to $A_4$ gives $A_4=-(\tilde{\boldsymbol{\varphi}},\mathbf{D}_{\Sigma,{\mc{C}}}h_{\mc{C}})_2$. Combining everything, we deduce
\[\frac{\caA^0_{\Sigma,g}(\tilde{\boldsymbol{\varphi}}_1,\tilde{\boldsymbol{\varphi}}_2)}{\caA^0_{\Sigma_1,g_1}(\tilde{\boldsymbol{\varphi}} ,\tilde{\boldsymbol{\varphi}}_1)\caA^0_{\Sigma_2,g_2}(\tilde{\boldsymbol{\varphi}} ,\tilde{\boldsymbol{\varphi}}_2)}  = \exp\Big(\frac{1}{2}(\tilde{\boldsymbol{\varphi}},\widetilde{{\bf D}}_{\Sigma,{\mc{C}}} \tilde{\boldsymbol{\varphi}})_2\Big)\exp\Big(-(\tilde{\boldsymbol{\varphi}},{\bf D}_{\Sigma,{\mc{C}}} h_{\mc{C}})_2+\frac{1}{2}(h_{\mc{C}},{\bf D}_{\Sigma,{\mc{C}}} h_{\mc{C}})_2\Big),\] 
which proves the lemma by using also  \eqref{densityDN} and \eqref{CameronMartin}.
\end{proof}

\begin{lemma}\label{density2}
Assume $\partial\Sigma\not=\emptyset$. Then
\[ Z_{\Sigma_1,g_1}Z_{\Sigma_2,g_2}  =  (2\pi)^{k/2}Z_{\Sigma,g}  {\rm det}_{\rm Fr}(\mathbf{D}_{\Sigma,{\mc{C}}}(2\mathbf{D}_0)^{-1})^{1/2} \]
\end{lemma}
\begin{proof}
It is a direct computation to check that the spectral zeta function of  $2\mathbf{D}_0$ is $k(2^{-s}+2^{1-s} \zeta (s))$ where $\zeta$ is the standard Riemann zeta function (recall ${\bf D}$ is the Fourier multiplier by $|n|$ on each circle). Since $\zeta(0)=-\frac{1}{2}$ and $\zeta'(0)= -\frac{1}{2} \log 2 \pi$ we get
\begin{equation}\label{detD0pik}
{\rm det} (2\mathbf{D}_0)=(2\pi)^k.
\end{equation}
Theorem B in \cite{BurgheleaFK92} says that $\det(\Delta_{\Sigma,D})=\det(\Delta_{\Sigma_1,D})\det(\Delta_{\Sigma_2,D})\det({\bf D}_{\Sigma,\mc{C}})$, where the determinant of ${\bf D}_{\Sigma,\mc{C}}$ is well-defined using the meromorphic continuation of its spectral zeta functions (see \cite{BurgheleaFK92,kontsevich1994}), since it is a positive self-adjoint elliptic pseudo-differential operator of order $1$ on a compact manifold. 
Finally \cite[Proposition 6.4]{kontsevich1994} states that $\det({\bf D}_{\Sigma,\mc{C}})=\det_{\rm Fr}({\bf D}_{\Sigma,\mc{C}}(2{\bf D}_0)^{-1})\det(2{\bf D}_0)$.
\end{proof}

Combining \eqref{comput1} with Lemmas \ref{density1} and \ref{density2} concludes the case when $\partial\Sigma\not=\emptyset$.

\medskip
$\bullet$ Assume now $\partial\Sigma  = \emptyset$: Observe first that if suffices to consider the case when $k=1$. Indeed, the case $k\geq 2$ can be decomposed into first gluing the other $k-1$ boundaries and then gluing the 
 last boundary circle. Using Proposition \ref{glueampli} and \ref{selfglueampli} in the case $\pl \Sigma\not=\emptyset$, it remains only to consider the last gluing circle, which means we are in the case $k=1$. 
 We only  consider the case where we glue two disjoint surfaces $\Sigma_1,\Sigma_2$ (Proposition \ref{glueampli}) as the other case is the same up to notational changes.  
  
So we have now $k=1$ and we denote the glued boundary component by  $\mathcal{C}:= \mc{C}_1$.   By definition, the amplitude for $\Sigma$ is  
\begin{equation}\label{amp0}
\caA_{\Sigma,g,{\bf x},\boldsymbol{\alpha}}(F_1\otimes F_2)=\langle F_1\otimes F_2(\phi_g)\prod_{i=1}^m V_{\alpha_i,g}(x_i) \rangle_{\Sigma,g} .
\end{equation}
Let now $X_1$ and $X_2$ be two  independent Dirichlet GFF respectively on $\Sigma_1$ and $\Sigma_2$. We assume that they are both defined on $\Sigma$ by setting $X_i=0$ outside of $\Sigma_i$. Then we have the following decomposition in law (see Proposition \ref{decompGFF})
\begin{align*}%\label{}
X_g\stackrel{\rm law}=X_1+X_2+P{\rm X}-c_g
\end{align*}
where ${\rm X}$ is the restriction of the GFF $X_g$ to the  glued boundary component  $\mathcal{C}$ expressed in parametrised coordinates, i.e. ${\rm X} = X_{g|_{{\caC}_1}}\circ \zeta_{1,1} $,   $P{\rm X}$ is its harmonic extension to $ \Sigma $ and $c_g:=\frac{1}{{\rm v}_g(\Sigma)}\int_\Sigma (X_1+X_2+P{\rm X})\,\dd {\rm v}_g$. Therefore, plugging this relation into the amplitude  \eqref{amp0},  and then shifting the $c$-integral by $c_g$, we get      
\begin{align*}
 \caA_{\Sigma,g,{\bf x},\boldsymbol{\alpha}}(F_1\otimes F_2)=&\frac{({\det}'(\Delta_{g}))^{-\frac{1}{2}}}{{\rm v}_{g}(\Sigma)^{-\frac{1}{2}}Z_{\Sigma_1,g_1}Z_{\Sigma_2,g_2}}\int \frac{\caA_{\Sigma_1,g_1,{\bf x}_1,\boldsymbol{\alpha}_1, \zeta_1}(F_1,c +  \varphi)\caA_{\Sigma_2,g_2,{\bf x}_2,\boldsymbol{\alpha}_2,\zeta_2} (F_2,c +   \varphi)}{\caA^0_{\Sigma_1,g_1}(c +   \varphi)\caA^0_{\Sigma_2,g_2}(c +    \varphi ) }  \dd c \,  \dd \P_{{\rm X}}    (\varphi).
\end{align*}
 Now we make a further shift in the $c$-variable in the expression above to subtract the mean $m_{\mc{C}}({\rm X}):=\frac{1}{2\pi}\int_0^{2\pi}{\rm X}(e^{i\theta})\,\dd \theta$ to the field ${\rm X}$. As a consequence we can replace the law $ \P_{{\rm X}} $ of ${\rm X}$ in the above expression by the law $  \P_{{\rm X}-m_{\mc{C}}({\rm X})}$ of the recentered field ${\rm X}-m_{\mc{C}}({\rm X})$.
 
Now we  claim, for measurable bounded functions $F$
\begin{equation}\label{densityDNbis}
\int F( \boldsymbol{ \varphi}) \dd\P_{{\bf X}-m_{\mc{C}}({\bf X})}(\boldsymbol{ \varphi})=\frac{ \sqrt{2}}{  \det(\mathbf{D}_{\Sigma,{\mc{C}},0}(2\mathbf{D}_0)^{-1})^{-1/2} }\int F(\boldsymbol{ \varphi})\exp\Big(-\frac{1}{2}(\boldsymbol{ \varphi},\widetilde{\mathbf{D}}_{\Sigma,{\mc{C}}} \boldsymbol{ \varphi})_2\Big)   \dd \P_\T   ( \boldsymbol{ \varphi}).
\end{equation}
 with $\mathbf{D}_{\Sigma,{\mc{C}},0}=\mathbf{D}_{\Sigma,\mc{C}}+\Pi_0'$, recall \eqref{defPi_0'}.
The proof of this claim follows the same lines as the proof of \eqref{densityDN}, using that  $\mathbf{D}_{\Sigma,{\mc{C}}}G=2\pi {\rm Id}$ on the space $\{f\in H^s(\T) \, |\, \int_0^{2\pi}f(e^{i\theta})  \dd \theta=0 \}$ and the following estimate, proved exactly as \eqref{lowerboundDNmaptilde}: there is $a>0$ such that for any $\epsilon>0$, there is $N_0>0$ such that 
$$(\boldsymbol{ \varphi},\widetilde{\mathbf{D}}_{\Sigma,{\mc{C}}} \boldsymbol{ \varphi})_2\geq  (a-1)(\pi_{N_0}\boldsymbol{ \varphi},2\mathbf{D}\pi_{N_0}\boldsymbol{ \varphi})_2-\epsilon(\boldsymbol{ \varphi},2\mathbf{D}\boldsymbol{ \varphi})_2.$$

This achieves the proof by observing that $\exp(-\frac{1}{2}(\boldsymbol{ \varphi},\widetilde{\mathbf{D}}_{\Sigma,{\mc{C}}} \boldsymbol{ \varphi})_2)=\caA^0_{\Sigma_1,g_1}(c + \boldsymbol{   \varphi})\caA^0_{\Sigma_2,g_2}(c + \boldsymbol{  \varphi} ) $ (whatever the value of $c$, since $\mathbf{D}_{\Sigma,\mc{C}}1={\bf D}1=0$) and by using the determinant formula
\begin{equation}\label{detformula}
\frac{({\det}'(\Delta_{g})/{\rm v}_{g}(\Sigma))^{-1/2}}{Z_{\Sigma_1,g_1}Z_{\Sigma_2,g_2}} =   \det(\mathbf{D}_{\Sigma,{\mc{C}},0}(2\mathbf{D}_0)^{-1})^{-1/2},
\end{equation}
that follows from Theorem $B^*$ in  \cite{BurgheleaFK92}, the identity $\det'({\bf D}_{\Sigma,\mc{C}})=\det_{\rm Fr}({\bf D}_{\Sigma,\mc{C},0}(2{\bf D}_0)^{-1})\det(2{\bf D}_0)$ using \cite[Proposition 6.4]{kontsevich1994}, and $\det(2{\bf D}_0)=2\pi$ by \eqref{detD0pik}.
\qed

%

%%%%%%%%%%%%%%%%%%%%%%%%%%%%%%%%%%%%%%%%%%%%%%%%%%%%%%%%%%%%%%%%%%%%%%%%%%%%%%%%%%%%%%%%%%%%%%%%%%%%%%%%%%%%%%%%%%%%%%%%%%%%%%%%%%%%%%%%%%%%%%%%%%%%%%%%%%%%%%%%%%%%%%%%%%%%%%%%%%%%%%%%%%%%%%%%%%%%%%%%%%%%%%%%%%%%%%%%%%%%%%%%%%%%%%%%%%%%%%%%%%%%%%%%%%%%%%%%%%%%%%%%%%
\section{Semigroup of annuli}\label{section:semigroup}
%%%%%%%%%%%%%%%%%%%%%%%%%%%%%%%%%%%%%%%%%%%%%%%%%%%%%%%%%%%%%%%%%%%%%%%%%%%%%%%%%%%%%%%%%%%%%%%%%%%%%%%%%%%%%%%%%%%%%%%%%%%%%%%%%%%%%%%%%%%%%%%%%%%%%%%%%%%%%%%%%%%%%%%%%%%%%%%%%%%%%%%%%%%%%%%%%%%%%%%%%%%%%%%%%%%%%%%%%%%%%%%%%%%%%%%%%%%%%%%%%%%%%%%%%%%%%%%%%%%%%%%%%%

In this section, we connect the Segal amplitudes of annuli to the semigroup $e^{-t \mathbf{H}}$ constructed in \cite{GKRV20_bootstrap}, where $\mathbf{H}$ is an unbounded operator on the Hilbert space $\mathcal{H}$, called Hamiltonian of Liouville CFT. This is done in Subsection \ref{sub:hamiltonian}. \footnote{In the follow-up papers \cite{BGKRV,BGKR1}, we uncover a more general picture where we show how the annuli amplitudes encode the whole symmetry algebra of the Liouville CFT, thus providing a concrete representation of the Virasoro algebra acting on $\mathcal{H}$. In this picture, the Hamiltonian $\mathbf{H}$    turns out to be the generator of dilations. }
Next, in Subsection \ref{subsubsec:spectral}, we recall the construction from  \cite{GKRV20_bootstrap} of a family  $(\Psi_{\alpha,\k,\l})$ indexed by $\alpha\in\C$ and $\k,\l$ finite sequences of nonnegative integers. This family has two important properties: it is analytic in $\alpha$ and, when restricted to $\alpha\in Q+i\R$, it provides a  spectral resolution for $\mathbf{H}$. Yet this family is not suited for computations, and we will rearrange this basis in Subsection \ref{sub:virasoro} to obtain a new family $\Psi_{\alpha,\nu,\tilde{\nu}}$  indexed by $\alpha\in\C$ and $\nu$, $\tilde\nu$ Young diagrams. This family is still analytic in $\alpha$, provides a spectral resolution (see \eqref{changeofbasis} and \eqref{defHzarb}), and furthermore will be crucial for  understanding the eigenstates near the spectrum line $\alpha\in Q+i\R$ using a probabilistic representation, which  we will obtain for $\alpha\in\R$  in Section \ref{sec:computingamplitudes}. The analyticity will be used to transfer the  information gained on the real line to the spectrum line. A key element, developed in this section, to obtain later the probabilistic representation is the intertwining property   \eqref{intert:desc}: it allows us to transfer properties of the eigenstates of   the GFF  theory to the Liouville theory. And in Section \ref{sec:computingamplitudes},  we will show that the eigenstates of the GFF have a probabilistic representation involving the stress-energy tensor, which we will push via the intertwining to a probabilistic representation for the Liouville eigenstates $\Psi_{\alpha,\nu,\tilde{\nu}}$. In other words, this section serves to state the spectral resolution (Plancherel formula) and to prepare the argument for the probabilistic description of the eigenstates  $\Psi_{\alpha,\nu,\tilde{\nu}}$.

\subsection{Hamiltonian and semigroup of LCFT}\label{sub:hamiltonian}
%%%%%%%%%%%%%%%%%%%%%%%%%%%

In \cite[Section 3 and 5]{GKRV20_bootstrap}, a contraction semigroup $S(t):L^2(\R\times\Omega_{\T})\to L^2(\R\times\Omega_{\T})$ is defined 
by the following expression: for $F\in L^2(\R\times\Omega_{\T})$ 
\begin{equation}\label{FKgeneral}
S(t)F(c,\varphi)=e^{-\frac{Q^2t}{2}}\E \big[ F(c+\phi_t) e^{-\mu e^{\gamma c}\int_{e^{-t}<|z|<1} |z|^{-\gamma Q}M_\gamma (\phi,\dd z)}\big]
\end{equation}
where  $\phi_t: \theta\mapsto \phi(e^{-t+i\theta})$ with $\phi:=P_{\D}\varphi+X_{\D,D}$,  $X_{\D,D}$ is the GFF with Dirichlet condition on the unit disk 
$(\D,|dz|^2)$ and $P_{\D}\varphi$ is the harmonic extension in $\D$ of the random Fourier series $\varphi\in H^{s}(\T)$ defined in \eqref{GFFcircle0}, and the expectation is with respect to $X_{\D, D}$.
This semigroup can be written under the form 
\[S(t)=e^{-t{\bf H}}\] 
where  $\mathbf{H}$ is an unbounded self-adjoint operator on $L^2(\R\times\Omega_{\T})$, which has the form 
\[ \mathbf{H}={\bf H}_0+\mu e^{\gamma c}V(\varphi) \textrm{ with }{\bf H}_0:=-\frac{1}{2}\pl_c^2+\frac{1}{2}Q^2+{\bf P}\]
where ${\bf P}, V$ are symmetric unbounded non-negative operators on $L^2(\Omega_{\T})$ and ${\bf P}$ has discrete spectrum given by $\N$. Moreover ${\bf H}_0$ is the generator of $S(t)$ when we set $\mu=0$ in \eqref{FKgeneral}. By \cite[Lemma 6.5]{GKRV20_bootstrap}, for all $\beta\in \R$ the operator $e^{-t{\bf H}}$ is also bounded on the weighted spaces 
\begin{equation}\label{boundednesspropagator}
e^{-t{\bf H}}: e^{\beta c_-}L^2(\R\times \Omega_\T,\dd \mu_0)\to e^{\beta c_-}L^2(\R\times \Omega_\T,\dd \mu_0).
\end{equation}
 
 Following \cite[subsection 3.3]{GKRV20_bootstrap}, one can also consider the generator of rotations $\mathbf{\Pi}$, defined as the generator of the strongly continuous unitary group on $L^2(\R\times \Omega_{\T})$ (for $\vartheta\in\R$)
\begin{align}\label{pi:prob}
 e^{i\vartheta \mathbf{\Pi}}F(\tilde{\varphi})=F(\tilde{\varphi}(\cdot+\vartheta)).
\end{align}
 
For $q\in \D$ with $|q|=e^{-t}$ for $t>0$, consider the flat annulus
\begin{equation}\label{defaq}
\mathbb{A}_q:=\{ z\in \C\,|\, |z|\in [e^{-t},1]\}, \quad g_{\mathbb{A}}=|dz|^2/|z|^2
\end{equation}
with boundary parametrisation $\boldsymbol{\zeta}_q=(\zeta_1,\zeta_2)$ given by $\zeta_2(e^{i\theta})=e^{i\theta}$ and $\zeta_1(e^{i\theta})=qe^{i\theta}$. For $q'\in \D$, denote by $q\mathbb{A}_{q'}:=\{ z\in \C\,|\, |z|\in [|qq'|,|q|]\}$ with metric $g_{\mathbb{A}}$ and boundary parametrisation $\boldsymbol{\zeta}_{q,q'}'=(\zeta'_1,\zeta'_2)$ given by 
$\zeta'_2(e^{i\theta})=q'\zeta_2(e^{i\theta})$ and 
$\zeta'_1(e^{i\theta})=q'\zeta_1(e^{i\theta})$. Observe that one can glue $\mathbb{A}_{q}$ with $q\mathbb{A}_{q'}$ along the boundary circle $\{|z|=|q|\}$, the result becomes $\mathbb{A}_{q}\# (q\mathbb{A}_{q'})=\mathbb{A}_{qq'}$. 
Since $(q\mathbb{A}_{q'},g_{\mathbb{A}})$ is isometric to $(\mathbb{A}_{q'},g_{\mathbb{A}})$, we can use Proposition \ref{Weyl} to express its amplitude as $\mc{A}_{\mathbb{A}_{q'},g_{\mathbb{A}},\boldsymbol{\zeta}_{q'}}$, 
 and using \ref{glueampli}, one obtains that the amplitudes of these annuli satisfy 
\[\mc{A}_{\mathbb{A}_{qq'},g_{\mathbb{A}},\boldsymbol{\zeta}_{qq'}}(\tilde\varphi_1,\tilde\varphi_2)=\frac{1}{\sqrt{2}\pi}\int \mc{A}_{\mathbb{A}_q,g_{\mathbb{A}},\boldsymbol{\zeta}_q}(\tilde\varphi_1,\tilde\varphi)
\mc{A}_{\mathbb{A}_{q'},g_{\mathbb{A}},\boldsymbol{\zeta}_{q'}}(\tilde\varphi,\tilde\varphi_2)d\mu_{0}(\tilde\varphi).\]
When viewing  $(\sqrt{2}\pi)^{-1/2}\mc{A}_{\mathbb{A}_{q},g_{\mathbb{A}},\boldsymbol{\zeta}_q}$ as an integral kernel of an operator, 
this operator behaves as a semigroup. We shall now show  its semigroup property and relate it to
 $e^{-t{\bf H}}$ and $e^{i\vartheta \mathbf{\Pi}}$.
\begin{proposition} \label{prop:annulussimple}
For $q\in \D$, the annuli amplitudes in the metric $g_\mathbb{A}$ satisfy for $F,F'\in L^2(\R\times\Omega_\T)$  
\begin{align}
\int F(\tilde\varphi)\mathcal{A}_{\mathbb{A}_{|q|},g_\mathbb{A},\boldsymbol{\zeta}_{|q|}}(\tilde\varphi ,\tilde\varphi') F'(\tilde\varphi') \dd\mu_0(\tilde\varphi)\dd\mu_0(\tilde\varphi') = \sqrt{2}\pi   |q|^{-\frac{\mathbf{c}_L}{12}}\langle e^{\log |q|\mathbf{H}}F,\bar F'\cjd_{\mc{H}} \label{ident1annulus}\\
\int F(\tilde\varphi)\mathcal{A}_{\mathbb{A}_{q},g_\mathbb{A},\boldsymbol{\zeta}_q}(\tilde\varphi,\tilde\varphi') F'(\tilde\varphi') \dd\mu_0(\tilde\varphi)\dd\mu_0(\tilde\varphi') = \sqrt{2}\pi |q|^{-\frac{\mathbf{c}_L}{12}}\langle  e^{i \arg(q)\boldsymbol{\Pi}} e^{\log |q|\mathbf{H}}F,\bar F'\cjd_{\mc{H}}\label{ident2annulus}
\end{align}
with $\mathbf{c}_L=1+6Q^2$ the central charge. 
\end{proposition}

\begin{proof} The identity \eqref{ident2annulus} is a direct consequence of \eqref{ident1annulus} and the definition of $e^{i \arg(q)\boldsymbol{\Pi}}$. 
To prove  \eqref{ident1annulus} we first rewrite the propagator $e^{-t{\bf H}}$ in the following Feynmann-Kac form:
 for $t>0$ and any $f\in L^2(\R\times\Omega_\T)$ we claim that
\begin{equation}\label{emoinstH}
e^{-t\mathbf{H}}f(\tilde{\varphi})=\int _{\R\times \Omega_\T}K_0(t,\tilde{\varphi}, \tilde{\varphi}')H_t(\tilde{\varphi}, \tilde{\varphi}')f(\tilde{\varphi}')\,\dd \mu_0(\tilde{\varphi}')
\end{equation}
with $H_t$  defined by
 \[H_t(\tilde{\varphi}, \tilde{\varphi}')= \E\Big[\exp\Big(-\mu e^{\gamma c'}\int_{e^{-t}<|z|<1} |z|^{-\gamma Q}M_\gamma (\phi,\dd z)\Big)\, \Big|\,\phi_0= \varphi,\phi_t=c'-c+\varphi'\Big]\]
 and $K_0(t)$ the integral kernel of $e^{-t{\bf H}_0}$, i.e. 
\[e^{-t\mathbf{H}_0} f(\tilde{\varphi})=\int _{\R\times \Omega_\T}K_0(t,\tilde{\varphi}, \tilde{\varphi}')f(\tilde{\varphi}')\,\dd\mu_0(\tilde{\varphi}').\]
Indeed, recall that $\phi_t(\theta)=\phi(e^{-t+i\theta})$ with $\phi:=P_{\D}\varphi+X_{\D,D}$, therefore $\phi_0= \varphi$.   By the Feynman-Kac representation \eqref{FKgeneral}, for $f\in L^2(\R\times\Omega_{\T},\dd\mu_0)$ we have
\begin{align*}
e^{-t\mathbf{H} }f(\tilde{\varphi})=&e^{-\frac{Q^2t}{2}}\E\big[ f(c+\phi_t) e^{-\mu e^{\gamma c}\int_{e^{-t}<|z|<1} |z|^{-\gamma Q}M_\gamma (\phi,\dd z)}\big]\\
=&e^{-\frac{Q^2t}{2}}\E\big[ f(c+\phi_t) H_t(c+\phi_0,c+\phi_t)\big]
\end{align*}
where the expectation is taken with respect to $X_{\D,D}$. This gives \eqref{emoinstH}. 

Then, notice that the expectation in the definition \eqref{amplitude} of the amplitudes coincides with the function $H_{-\log|q|}(\tilde{\boldsymbol{\varphi}})$ (here we use the relation \eqref{scalingmeasure} and the fact that the curvature $K_{g_{\mathbb{A}}}=0$).
We want to compute the free-field amplitude $\mc{A}^0_{\mathbb{A}_{|q|},g_{\mathbb{A}},\boldsymbol{\zeta}}(\boldsymbol{\tilde{\varphi}})$.
Let $P\tilde{\boldsymbol{\varphi}}$ be the harmonic extension on $\mathbb{A}_{|q|}$ that satisfies $P\tilde{\boldsymbol{\varphi}}(e^{i\theta})= \tilde\varphi_{2}(e^{i\theta})$ and $P\tilde{\boldsymbol{\varphi}}(|q|e^{i\theta})= \tilde\varphi_{1}(e^{i\theta})$. Then 
$$P\tilde{\boldsymbol{\varphi}}(z)= c_2+\frac{c_1-c_2}{\log|q|}\log|z|+\sum_{n\geq 1}(a_nz^n+\bar a_n\bar z^n+b_n z^{-n}+\bar b_n\bar z^{-n})$$
with for $n\geq 1$
$$a_n=\frac{\varphi^1_n-\varphi^2_n|q|^{-n}}{|q|^{n}-|q|^{-n}},\quad  b_n=\frac{\bar \varphi^2_n|q|^n-\bar \varphi^1_n }{|q|^{n}-|q|^{-n}}.$$
With this expression, some algebra leads to the conclusion that $ \mathcal{A}^0_{\mathbb{A}_{|q|},g_\mathbb{A},\boldsymbol{\zeta}_{|q|}}(\tilde{\boldsymbol{\varphi}})$ is given by (with $e^{-t}=|q|$)
\begin{equation}\label{algebrafreefield}
 \mathcal{A}^0_{\mathbb{A}_{|q|},g_\mathbb{A},\boldsymbol{\zeta}_{|q|}}(\tilde\varphi,\tilde{\varphi}')=\exp\Big(-\frac{(c-c')^2}{2t}-\sum_n\frac{(x_n'-e^{-nt}x_n)^2}{2(1-e^{-2tn})}-\frac{x'^2_n}{2}+\frac{(y'_n-e^{-nt}y_n)^2}{2(1-e^{-2tn})}-\frac{y'^2_n}{2}\Big).
\end{equation}
Since the dynamics on the $c$-variable for the free propagator $e^{-t{\bf H}_0}$ is a Brownian motion, independent of the dynamics on $\varphi$, which is given by independent Ornstein-Uhlenbeck processes of the Fourier components of $\varphi$ (see Proposition 4.1 in \cite{GKRV20_bootstrap}), 
 $e^{-t{\bf H}_0}$ has integral kernel given by 
\begin{align}\label{kernel0}
&\mc{K}_0(t,\tilde{\varphi}', \tilde{\varphi})\\
&:=\frac{e^{-\frac{Q^2}{2}t}}{\sqrt{2\pi t}}\Big(\prod_{n\geq 1}(1-e^{-2tn})^{-1}\Big) \exp\Big(-\frac{(c-c')^2}{2t}-\sum_n\frac{(x_n-e^{-nt}x'_n)^2}{2(1-e^{-2tn})}-\frac{x^2_n}{2}+\frac{(y_n-e^{-nt}y'_n)^2}{2(1-e^{-2tn})}-\frac{y^2_n}{2}\Big)\nonumber,
\end{align}
which implies, with \eqref{algebrafreefield}
\[\mathcal{A}^0_{\mathbb{A}_{|q|},g_\mathbb{A},\boldsymbol{\zeta}_{|q|}}(\tilde\varphi,\tilde{\varphi}')=\sqrt{-2\pi\log|q|} |q|^{-\frac{6Q^2}{12}}\mc{K}_0(-\log|q|,\tilde\varphi,\tilde{\varphi}')\prod_{n\geq 1}(1-|q|^{2n}).\]
Therefore, from \eqref{emoinstH}, we deduce
\begin{align*}
\int  F(\tilde\varphi)\mathcal{A}_{\mathbb{A}_{|q|},g_\mathbb{A},\boldsymbol{\zeta}_{|q|}}(\tilde{\varphi},\tilde{\varphi}') F'(\tilde\varphi') \dd\mu_0(\tilde{\varphi},\tilde{\varphi}') 
=Z_{\mathbb{A}_q,g_{\mathbb{A}}}\sqrt{-2\pi\log|q|} |q|^{-\frac{6Q^2}{12}}\prod_{n\geq 1}(1-|q|^{2n}) \langle  e^{\log |q|\mathbf{H}}F,\bar F'\cjd_{\mc{H}}.
\end{align*}
Now we conclude using the fact that $Z_{\mathbb{A}_q,g_{\mathbb{A}}}= 2^{-1/2}(-2\pi/\log|q|)^{+1/2}|q|^{-\frac{1}{12}} \prod_{n\geq 1}(1-|q|^{2n})^{-1} $, see \cite{Wei}. 
\end{proof}

\subsection{Spectral resolution of ${\bf H}$}\label{subsubsec:spectral}

One of  the main mathematical inputs of \cite{GKRV20_bootstrap} is the construction, via scattering theory, of a complete family of generalised eigenstates for the Hamiltonian $\mathbf{H}$, which we describe now.
Let $\mc{N}$ be the set of sequences of positive integers ${\bf k}=(k_1,k_2,\dots)\in \N^{\N_+}$ 
 such that $k_n=0$ for all $n$ large enough. We denote $|\k|=\sum_n nk_n$. Note that $\mc{N}$ is countable. For $ \k,\l\in\mc{N}$, we set $\la_{\k\l}=|\k|+|\l|$. 
We recall the notations $c_-=c\mathbf{1}_{c<0}$ and  $c_+=c\mathbf{1}_{c>0}$. For $\beta_-,\beta_+\in\R$ and $p\geq 1$ we introduce the weighted $L^r$-spaces $e^{-\beta_- c_- -\beta_+c_+}L^r(\R\times \Omega_{\T})$ as the space of functions with finite $\|\cdot\|_{\beta,r}$-norm, where
\begin{equation}
\|F\|_{\beta,r}^r:=\int e^{r\beta_- c_-+r\beta_+c_+}\E[|F|^r]\,\dd c.
\end{equation}
We recall the result:\footnote{The last inequality follows directly from the spectral theoreom and the fact that $\Psi_{Q+ip,\k,\l}$ is a spectral resolution of the self-adjoint operator ${\bf H}$.} 
\begin{proposition}{\cite[Proposition 6.26]{GKRV20_bootstrap}}\label{holomorphicpsi}
There is a family $\Psi_{\alpha,\k,\l}\in e^{-\beta c_-}L^2(\R\times \Omega_{\T})$ indexed  by $\alpha\in\C $ and $ \k,\l\in\mc{N}$, where $\beta\in\R$ can be arbitrarily chosen so that $\beta>Q-{\rm Re}(\alpha)$, satisfying ${\bf H}\Psi_{\alpha,\k,\l}=(2\Delta_\alpha+|\k|+|\l|) \Psi_{\alpha,\k,\l}$ and analytic in the variable $\alpha$ in the region  
\begin{equation}\label{regionWl}
W_{\ell}:= \Big\{ \alpha \in \C \setminus \mc{D}_0\, |\,  
{\rm Re}(\alpha)\leq Q, {\rm Re}\sqrt{(Q-\alpha)^2-2\ell}>{\rm Re}(Q-\alpha)-\gamma/2\Big\}
\end{equation}
where $\mc{D}_{0}:=\bigcup_{j\geq 0}\{Q\pm i\sqrt{2j}\}$ and $\ell:=|\k|+|\l|$. This family is continuous in $\alpha\in W_{\ell}\cup  \mc{D}_0$.
In particular
\[\Psi_{Q+ip,\k,\l}\in \bigcap_{\eps>0} e^{-\eps c_-}L^2(\R\times \Omega_\T), \quad p\in\R_+, \k\in\mc{N},\l\in\mc{N},\]  
and  for all $u,v\in L^2(\R\times \Omega_\T)$
\[\cjg u, v\cjd_{L^2}=\frac{1}{2\pi}\sum_{\k,\l\in\mc{N}}\int_0^\infty 
\cjg u,\Psi_{Q+ip,\k,\l}\cjd_{L^2} \cjg\Psi_{Q+ip,\k,\l},v\cjd_{L^2} \dd p.\]
Moreover, for $N\in \N, L\in \R^+$ and $u\in L^2(\R\times \Omega_{\mathbb{T}})$
\[ \frac{1}{2\pi}\sum_{\substack{\k,\l\in\mc{N},\\ |\k|+|\l|\leq N}}\int_0^L 
\cjg u,\Psi_{Q+ip,\k,\l}\cjd_{L^2} \cjg\Psi_{Q+ip,\k,\l},v\cjd_{L^2} \dd p\leq \|u\|^2_{\mc{H}}.\]
\end{proposition}
 
\subsection{Construction of the Virasoro descendant states}\label{sub:virasoro}
%%%%%%%%%%%%%%%%%%%%%%%%
The basis $(\Psi_{Q+ip,\k,\l})_{p,\k,\l}$ diagonalizes the Hamiltonian ${\bf H}$ but it is not very convenient to work with, as it is not related to the symmetries of our model.
There is a convenient change of basis of generalised eigenfunctions, which is better related to the Virasoro algebra and that we shall recall now from the paper \cite{GKRV20_bootstrap}.  

\subsubsection{The free theory} First, we recall how to construct this basis in the case when $\mu=0$, i.e. for the free theory. Here we follow \cite[section 4.4]{GKRV20_bootstrap}.  Let us denote by $\mathcal{S}$ the set of smooth functions depending on finitely many coordinates, i.e. of the form
$F(x_1,y_1, \dots,x_n,y_n)$  with $n\geq 1$ and $F\in C^\infty((\R^2)^n)$, with at most polynomial growth at infinity for $F$ and its derivatives. Obviously $\mathcal{S}$ is dense in $L^2(  \Omega_{\T})$. Let  
\begin{equation}\label{smoothexpgrowth}
 \mathcal{C}_\infty:=\mathrm{Span}\{ \psi(c)F\,|\,\psi\in C^\infty(\R)\text{ and }F\in\mathcal{S} \}.
 \end{equation}
We use the  complex coordinates \eqref{GFFcircle0}, i.e. we denote for $n>0$
\begin{align*}
\partial_n:=\frac{\partial}{\partial\varphi_{n}}= \sqrt{n} (\partial_{x_n}-i \partial_{y_n}) \quad \text{ and }\quad \partial_{-n}:=\frac{\partial}{\partial\varphi_{-n}}= \sqrt{n} (\partial_{x_n}+i \partial_{y_n}).
\end{align*}
Then we introduce on $ \mathcal{C}_\infty$  the following operators for $n>0$: 
\[
 \mathbf{A}_n= \tfrac{i}{2}\partial_{n},\ \ \  \mathbf{A}_{-n}=\tfrac{i}{2}(\partial_{-n}-2n\varphi_{n}) \ \ \ 
\widetilde{\mathbf{A}}_n= \tfrac{i}{2}\partial_{-n},\ \ \ \widetilde{\mathbf{A}}_{-n}=\tfrac{i}{2}(\partial_{n}-2n\varphi_{-n})\ \ \ 
\mathbf{A}_0=\widetilde{\mathbf{A}}_0=\tfrac{i}{2}(\partial_c+Q)
\]
and  the {\it normal ordered product} on $ \mathcal{C}_\infty$ by
 $:\!\mathbf{A}_n\mathbf{A}_m\!\!:\,=\mathbf{A}_n\mathbf{A}_m$ if $m>0$ and $\mathbf{A}_m\mathbf{A}_n$ if $n>0$ (i.e. annihilation operators are on the right). The free Virasoro generators are then   defined for all $n \in \Z$ by
\begin{align}
\mathbf{L}_n^0:=-i(n+1)Q\mathbf{A}_n+\sum_{m\in\Z}:\mathbf{A}_{n-m}\mathbf{A}_m: \label{virassoro}\\
\widetilde{\mathbf{L}}_n^0:=-i(n+1)Q\widetilde{\mathbf{A}}_n+\sum_{m\in\Z}:\widetilde{\mathbf{A}}_{n-m}\widetilde{\mathbf{A}}_m:\,\,.\label{virassorotilde}
\end{align}
They map  $ \mathcal{C}_\infty$ into itself. 
These operators are used to construct the descendant states of the free theory. More precisely (see \cite[subsection 4.4]{GKRV20_bootstrap}), for $\alpha\in \C$, we define 
\begin{align}\label{psialphadef}
\Psi^0_\alpha(c,\varphi):=e^{(\alpha-Q)c}\in \mathcal{C}_\infty.
\end{align}
For $\alpha\in \C$, these are generalised eigenstates of ${\bf H}^0$: they never belong to $L^2(\R \times \Omega_\T)$ but rather to some weighted spaces $e^{\beta |c|}L^2(\R\times \Omega_\T)$ for $\beta>|{\rm Re}(\alpha)-Q|$, hence their name ``generalised eigenstates". We have
\begin{equation}\label{L0psialpha}
\mathbf{L}_0^0\Psi^0_\alpha=\widetilde{\mathbf{L}}_0^0\Psi^0_\alpha=\Delta_{\alpha}
\Psi^0_\alpha,  \qquad 
\mathbf{L}_n^0\Psi^0_\alpha=\widetilde{\mathbf{L}}_n^0\Psi^0_\alpha=0,\ \ \ n>0,
\end{equation}
where  $\Delta_\alpha$ is the conformal weight \eqref{deltaalphadef}; $\Psi^0_\alpha$ is called {\it highest weight state} with highest weight $\Delta_{\alpha}$.

Next,   let $\mc{T}$ be the set of Young diagrams, i.e. the set of  sequences $\nu$ of integers   with the further requirements that $\nu(k)\geq \nu(k+1)$ and $\nu(k)=0$ for $k$ large enough. We will include $\{0\}$ as an element of $\mc{T}$. 
For a Young diagram $\nu$ we denote its length as $|\nu|=\sum_k\nu(k)$ and its size as $s(\nu)=\max\{k\,|\,\nu(k)\not=0\}$. Given two Young diagrams $\nu= (\nu(i))_{i \in [1,k]}$, $\tilde{\nu}= (\tilde{\nu}(i))_{i \in [1,j]}$ with size $k$ and $j$, we define the operators
\begin{equation*}
\mathbf{L}_{-\nu}^0=\mathbf{L}_{-\nu(k)}^0 \cdots \, \mathbf{L}_{-\nu(1)}^0, \quad \quad \quad   \tilde{\mathbf{L}}_{-\tilde \nu}^0=\tilde{\mathbf{L}}_{-\tilde\nu(j)}^0 \cdots\, \tilde{\mathbf{L}}_{-\tilde\nu(1)}^0
\end{equation*}
and define
\begin{align}\label{psibasis}
\Psi^0_{\alpha,\nu, \tilde\nu}=\mathbf{L}_{-\nu}^0\tilde{\mathbf{L}}_{-\tilde\nu}^0  \: \Psi^0_\alpha,
\end{align}
with the convention that $\Psi^0_{\alpha,\emptyset,\emptyset}:=\Psi^0_{\alpha}$ (with $\emptyset$ the empty Young diagram).
The vectors $\Psi^0_{\alpha,\nu, \tilde\nu}$ are called the \emph{descendant} states of $\Psi^0_\alpha$.  They satisfy the following properties (see \cite[Prop 4.9]{GKRV20_bootstrap}):

\begin{proposition}\label{prop:mainvir0}  
1) For each pair of Young diagrams $\nu,\tilde{\nu}\in \mathcal{T}$,  $\Psi^0_{\alpha,\nu, \tilde\nu}$ can be written as   
\begin{align}\label{psibasis1}
\Psi^0_{\alpha,\nu, \tilde\nu}=\mathcal{Q}_{\alpha,\nu,\tilde \nu}\Psi^0_\alpha
\end{align}
where $\mathcal{Q}_{\alpha,\nu,\tilde\nu}  $   is a polynomial in the coefficients $(\varphi_n)_n$ and an eigenfunction of $\mathbf{P}$ with eigenvalue $|\nu|+|\tilde\nu|$.\\
2)  For all $\alpha \in \C$
\begin{equation*}
\mathbf{L}_0^0\Psi^0_{\alpha,\nu, \tilde\nu} = (\Delta_{\alpha}%\tfrac{Q^2}{4}+ \tfrac{P^2}{4}
+|\nu|)\Psi^0_{\alpha,\nu ,\tilde\nu},\ \ \  \tilde{\mathbf{L}}_0^0\Psi^0_{\alpha,\nu, \tilde\nu} = (\Delta_{\alpha}+|\tilde\nu|)\Psi^0_{\alpha,\nu ,\tilde\nu}
\end{equation*}
and thus, since $\mathbf{H}^0=\mathbf{L}_0^0+\tilde{\mathbf{L}}_0^0$,
\begin{equation*}
\mathbf{H}^0\Psi^0_{\alpha,\nu,\tilde{\nu}} = (2 \Delta_\alpha+|\nu|+|\tilde{\nu}| )\Psi^0_{\alpha,\nu,\tilde{\nu}}  .
\end{equation*}
3) The inner products of the descendant states obey
\begin{equation}\label{scapo}  
 \langle \mathcal{Q}_{2Q-\bar\alpha,\nu,\tilde\nu} | \mathcal{Q}_{\alpha,\nu',\tilde\nu'} \rangle_{L^2(\Omega_\T)}=\delta_{|\nu| ,|\nu'|}\delta_{|\tilde\nu| ,|\tilde\nu'|}F_{\alpha}(\nu,\nu')F_{\alpha}(\tilde\nu,\tilde\nu')
\end{equation} 
where each coefficient   $F_{\alpha}(\nu,\nu')$  is a  polynomial in  ${\alpha}$, called the {\it Schapovalov form}.  The functions $(\mathcal{Q}_{\alpha,\nu,\tilde \nu})_{\nu,\tilde\nu\in\mathcal{T}}$ are linearly independent for 
\[\alpha\nin \{{\alpha_{r,s}} \mid \, r,s\in \N^\ast ,rs \leq \max(|\nu|, |\tilde \nu |)\}\quad \quad\text{with }\quad {\alpha_{r,s}}=Q-r\frac{\gamma}{2}-s\frac{2}{\gamma}.\]
 \end{proposition}

\subsubsection{Construction of descendant fields} 
Now we can focus on the construction of the descendant states in the Liouville theory. For  Young diagrams $\nu$, $\tilde\nu$, we consider the subset of $\C$ 
\begin{equation}\label{Inunu}
\mathcal{I}_{ \nu,\tilde\nu}:=\{\alpha\in\C \,|\,\exists \beta=\beta(\alpha)\in \R, \,\beta>{\rm Re} (Q-\alpha)\, \text{ and }\,{\rm Re} \big((Q-\alpha)^2\big)-2(|\nu|+|\tilde\nu|)>(\beta-\gamma/2)^2\}.
\end{equation}
Now we gather the contents of \cite[Prop 6.9]{GKRV20_bootstrap} (together with the notational warning in the proof of \cite[Prop 7.2]{GKRV20_bootstrap}) and of \cite[Prop 6.26]{GKRV20_bootstrap}:

\begin{proposition}\label{defprop:desc}
Fix  $\nu$, $\tilde\nu$ Young diagrams and $\ell=|\nu|+|\tilde\nu|$.  With the notation  \eqref{regionWl}, there is a holomorphic family of functions
\[\alpha\in W_\ell\subset\C\mapsto \Psi_{\alpha,\nu,\tilde{\nu}}\in e^{-\beta c_-}L^2(\R\times \Omega_{\T})\]
with  $\beta>Q-{\rm Re}(\alpha)$ satisfying ${\bf H}\Psi_{\alpha,\nu,\tilde{\nu}}=(2\Delta_\alpha+\ell) \Psi_{\alpha,\nu,\tilde{\nu}}$. This family is characterized by the intertwining property: for any $\chi\in C^\infty(\R)$   equal to $1$ near $-\infty$ and supported in $\R^-$
  \begin{equation}\label{intert:desc}
\Psi_{\alpha,\nu,\tilde{\nu}}= \lim_{t\to \infty}e^{t(2\Delta_\alpha+\ell)} e^{-t{\bf H}}( \chi(c)\Psi^0_{\alpha,\nu,\tilde{\nu}})\in e^{-\beta c_-}L^2(\R\times \Omega_{\T})
\end{equation}
 for any $\alpha\in \mathcal{I}_{ \nu,\tilde\nu}$ and any $\beta=\beta(\alpha)$ fulfilling the definition of $ \mathcal{I}_{ \nu,\tilde\nu}$. Furthermore, if $\alpha\in\R$ satisfies $\alpha<Q-\gamma$, we can take $\chi=1$ in the above statement.
\end{proposition}

Notice that each $\Psi_{\alpha,\nu,\tilde{\nu}}$ is a linear combination $\{\Psi_{\alpha,\k,\l}\}_{|\k|+|\l|=|\nu|+|\tilde\nu|}$, see \cite[eq. (7.1)]{GKRV20_bootstrap}.

%%%%%%%%%%

 \subsection{Spectral resolution for the descendant states}
 %%%%%%%%%%%%%%%%%%%%%%%%%%%%%%
The set $\mc{T}$ of Young diagrams  is   countable and can be partitioned as 
\begin{equation}\label{decofT} 
\mc{T}= \bigcup_{n=0}^\infty \mc{T}_n,\quad \mc{T}_n:=\{\nu \in \mc{T}\,| \, |\nu|=n\} \textrm{ if }n>0, \quad \mc{T}_0=\{0\}.
\end{equation}
Set $d_{n}:=|\mc{T}_n|$ its cardinality, we let $\cjg\cdot,\cdot\cjd_{d_n}$ be the canonical Hermitian product on $\C^{d_n}$ and we define 
\begin{equation}\label{defofHT} 
H_\mc{T}:=\bigoplus_{n=0}^\infty \C^{d_n} \textrm{ with Hermitian product }\cjg v,v'\cjd_{\mc{T}}:= \sum_{n=0}^\infty \cjg v_n,v'_n\cjd_{d_n}
\end{equation}
where $v=\sum_n v_n$ with $v_n\in \C^{d_n}$ and $v'=\sum_n v'_n$. 
The Hilbert space 
$(H_\mc{T},\cjg \cdot,\cdot\cjd_{\mc{T}})$ is isomorphic to $\ell^2(\mc{T})$, i.e. the space $L^2(\mc{T},\mu_{\mc{T}})$ where the measure $\mu_{\mc{T}}$ is the counting measure on the set $\mc{T}$, namely
\[  \int_{\mc{T}} f(\nu)d\mu_{\mc{T}}= \sum_{\nu\in \mc{T}}f(\nu), \quad    \forall f:\mc{T}\to \R^+.\]

We will write ${\bf F}_{\alpha,n}$ for the linear endomorphism of  $(\C^{d_n},\cjg \cdot,\cdot\cjd_{d_n})$ whose matrix coefficients are the Schapovalov elements $F_{\alpha}(\nu,\nu')$ for $\nu,\nu'\in \mc{T}_n$, with the convention that ${\bf F}_{\alpha,0}=1$ (see Prop \ref{prop:mainvir0}). Notice that for $p\in\R$, ${\bf F}_{Q+ip,n}$ is positive definite and the map $\alpha\in\C \mapsto F_\alpha(\nu,\nu')$ is analytic for $\nu,\nu'\in \mc{T}_n$.
If we denote ${\bf F}_{Q+ip}:=\sum_{n=0}^\infty {\bf F}_{Q+ip,n}$ acting on $\mc{H}_{\mc{T}}$, one has for each $u_1,u_2\in e^{\delta c_-}L^2(\R_+\times \Omega_\T)$ with $\delta>0$ (see \cite[Lemma 8.1]{GKRV20_bootstrap})  
\begin{equation}\label{changeofbasis} 
\sum_{|\nu| + |\tilde \nu| \leq N} \int_0^L   \cjg u_1, H_{Q+ip,\nu,\tilde{\nu}} \cjd_{\mc{H}} \cjg  \, H_{Q+ip,\nu,\tilde{\nu}} , u_2\cjd_{\mc{H}}   {\rm d}p= \frac{1}{2\pi}\sum_{|\k|+ |\l| \leq N}  \int_0^L  \cjg u_1 , \Psi_{Q+ip,\k,\l} \cjd_{\mc{H}}  \cjg  \Psi_{Q+ip,\k,\l} , u_2\cjd_{\mc{H}}    {\rm d}p
\end{equation}
where, if $F^{-1/2}_{Q+ip}(\nu,\nu')$ 
are the matrix coefficients of the matrix ${\bf F}_{Q+ip,n}^{-1/2}$ for $|\nu|=|\nu'|=n$, one has set
\begin{equation}\label{defHzarb}
H_{Q+ip,\nu,\tilde{\nu}}: =  \sum_{|\nu'|=|\nu|, |\tilde{\nu}'|= |\tilde{\nu}|}  
F_{Q+ip}^{-1/2}(\nu, \nu')F_{Q+ip}^{-1/2}(\tilde{\nu}, \tilde{\nu}') 
\Psi_{Q+ip,\nu',\tilde{\nu}'}.
\end{equation}

 For $n,\tilde{n}\in \N$ and $p\in \R_+$, let us then define the vector spaces 
\begin{equation}\label{VPn}
\mc{V}_{p,n,\tilde{n}}:={\rm span}\{\Psi_{Q+ip,\nu,\tilde{\nu}}\, |\, \nu\in \mc{T}_n,\tilde{\nu}\in \mc{T}_{\tilde{n}} \}, \quad
\mc{V}_{p,n}:={\rm span}\{\Psi_{Q+ip,\nu,0}\, |\, \nu\in \mc{T}_n\}
\end{equation} 
and the operators
\begin{align} 
&\Pi_{\mc{V}_{p,n,\tilde{n}}}: e^{\delta  c_-}L^2(\R \times\Omega_\T)\to \C^{d_n\times d_{\tilde{n}}}\subset \mc{H}_{\mc{T}}\otimes \mc{H}_{\mc{T}}, &  (\Pi_{\mc{V}_{p,n,\tilde{n}}}u)_{\nu,\tilde{\nu}}:= 
\frac{1}{\sqrt{2\pi}}\cjg u,H_{Q+ip,\nu,\tilde{\nu}}\cjd_{\mc{H}}, \label{pairwithH}\\
 & \Pi_{\mc{V}_{p,n}}:=\Pi_{\mc{V}_{p,n,0}} : e^{\delta  c_-}L^2(\R \times\Omega_\T)\to \C^{d_n}\subset \mc{H}_{\mc{T}}, 
& (\Pi_{\mc{V}_{p,n}}u)_\nu=\frac{1}{\sqrt{2\pi}}\sum_{\nu' \in \mc{T}_n} {\bf F}_{Q+ip,n}^{-1/2}(\nu,\nu')\cjg u,\Psi_{Q+ip,\nu',0}\cjd_{\mc{H}}.
\end{align} 
By \eqref{changeofbasis} and Proposition \ref{holomorphicpsi}, we see that for each $L>0,(n,\tilde{n})\in \N^2$ and $u\in e^{\delta  c_-}L^2(\R\times \Omega_\T)$
\[ \begin{gathered}
\int_0^L \|\Pi_{\mc{V}_{p,n}}u\|_{\mc{H}_{\mc{T}}}^2\dd p\leq \int_0^L \|\Pi_{\mc{V}_{p,n,\tilde{n}}}u\|_{\otimes^2\mc{H}_{\mc{T}}}^2\dd p= \frac{1}{2\pi}\int_0^L \sum_{|\k|+ |\l| \leq n+\tilde{n}}  |\cjg u , \Psi_{Q+ip,\k,\l} \cjd_{\mc{H}}|^2  \dd p \leq \|u\|_{\mc{H}}^2,\\ 
 \int_0^L \sum_{n=0}^N\|\Pi_{\mc{V}_{p,n}}u\|_{\mc{H}_{\mc{T}}}^2\dd p\leq \lim_{(N,L)\to \infty}\int_0^L\sum_{n+\tilde{n}\leq N}\|\Pi_{\mc{V}_{p,n,\tilde{n}}}u\|_{\otimes^2\mc{H}_{\mc{T}}}^2\dd p=\|u\|^2_{\mc{H}}.\end{gathered}\]
This also implies, using a polarisation argument, the  
\begin{corollary}\label{boundednessofPi_V}
Let $\delta >0$, then for all $u, u_1,u_2\in e^{\delta  c_-}L^2(\R\times \Omega_\T)$ 
\[ \int_0^\infty \sum_{n,\tilde{n}\in\N^2}
\cjg\Pi_{\mc{V}_{p,n,\tilde{n}}}u_1,\Pi_{\mc{V}_{p,n,\tilde{n}}}u_2\cjd_{\otimes^2\mc{H}_{\mc{T}}}\dd p=\cjg u_1,u_2\cjd_{\mc{H}} ,\quad \int_0^\infty \sum_{n\in \N}\|\Pi_{\mc{V}_{p,n}}u\|_{\mc{H}_{\mc{T}}}^2\dd p\leq \|u\|^2_{\mc{H}}\]
and the integrals/sums converge in norm. 
The operator  $\Pi_{\otimes^2\mc{V}}: e^{\delta  c_-}L^2(\R\times \Omega_\T)\to L^2(\R_+;\otimes^2\mc{H}_{\mc{T}})$ defined by 
\[
 \forall f\in L^2(\R_+;\otimes^2\mc{H}_{\mc{T}})\, \quad 
\int_{0}^\infty \cjg (\Pi_{\otimes^2\mc{V}}u)(p),f(p)\cjd_{\mc{H}_{\mc{T}}}\dd p:= \int_0^\infty \sum_{n,\tilde{n}\in\N^2}
\cjg\Pi_{\mc{V}_{p,n,\tilde{n}}}u,f(p)\cjd_{\otimes^2\mc{H}_{\mc{T}}}\dd p
\]
is bounded and extends as an isometry $L^2(\R\times \Omega_\T)\to L^2(\R_+;\otimes^2\mc{H}_{\mc{T}})$.
The operator $\Pi_{\mc{V}}: e^{\delta  c_-}L^2(\R\times \Omega_\T)\to L^2(\R_+;\mc{H}_{\mc{T}})$ defined by 
\[
 \forall f\in L^2(\R_+;\mc{H}_{\mc{T}})\, \quad 
\int_{0}^\infty \cjg (\Pi_{\mc{V}}u)(p),f(p)\cjd_{\mc{H}_{\mc{T}}}\dd p:= \int_0^\infty \sum_{n\in\N}
\cjg\Pi_{\mc{V}_{p,n}}u,f(p)\cjd_{\mc{H}_{\mc{T}}}\dd p
\]
is bounded with norm $\leq 1$ and extends continuously to $L^2(\R\times \Omega_\T)$. 
\end{corollary}
When $u\in L^2(\R\times \Omega_\T)$ is real valued, we set
$(\Pi_{\otimes^2\mc{V}}u)(-p,\nu,\tilde{\nu}):=\frac{1}{\sqrt{2\pi}}\cjg u,\bbar{H_{Q+ip,\nu,\tilde{\nu}}}\cjd_{\mc{H}}$ for $p>0,(\nu,\tilde{\nu})\in \mc{T}^2$ and, by Corollary \ref{boundednessofPi_V}, one has for $u_i\in L^2$ real-valued
\begin{equation}\label{rewriting}
\int_{\R_+\times \Omega_\T} u_1u_2 \, \dd \mu_0=
\int_0^\infty\int_{\mc{T}^2}(\Pi_{\otimes^2\mc{V}} u_1)(p,\nu,\tilde{\nu})(\Pi_{\otimes^2\mc{V}}u_2)(-p,\nu,\tilde{\nu})\, \dd \mu_{\mc{T}}(\nu)\dd \mu_{\mc{T}}(\tilde\nu)
\dd p.
\end{equation}

\subsection{Action of $\Pi$ on descendants.}\label{sec:momenta}

\begin{lemma}\label{continuityPi}
The group $( e^{i\vartheta \mathbf{\Pi}})_{\vartheta\in\R}$ maps $e^{-\beta_- c_-- \beta_+ c_+}L^2(\R\times \Omega_\T)$ into itself for any $\beta_-,\beta_+\in\R$ and commutes with the semigroup $(e^{-t\mathbf{H} } )_{t>0}$ on $e^{-\beta_- c_-}L^2(\R\times \Omega_\T)$.
\end{lemma}

\begin{proof}
The fact that $( e^{i\vartheta\mathbf{\Pi}})_{\vartheta\in\R}$ maps $e^{-\beta c_-}L^2(\R\times \Omega_\T)$ into itself results from the representation \eqref{pi:prob} and the fact that a standard Gaussian measure on $\R^2$ is invariant under rotations. The fact that it commutes with $(e^{-t\mathbf{H} } )_{t>0}$ is by construction, see  \cite[subsection 3.3]{GKRV20_bootstrap} (dilations commute with rotations).
\end{proof}

\begin{proposition}\label{propPi}
For $\beta>Q-{\rm Re}(\alpha)$ and $\nu,\tilde{\nu}\in \mc{T}$, we have $  e^{i\vartheta \mathbf{\Pi}}\Psi_{\alpha,\nu,\tilde{\nu}}=e^{i (|\nu|-|\tilde{\nu}|) \vartheta}\Psi_{\alpha,\nu,\tilde{\nu}}$ in $e^{-\beta c_-}L^2(\R\times \Omega_\T)$.
\end{proposition}

 \begin{proof}
 We use Proposition \ref{defprop:desc}: for $\chi\in C^\infty(\R)$ equal to $1$ near $-\infty$ and supported in $\R^-$, and for $\alpha\in \mathcal{I}_{\nu,\tilde\nu}$
\[ \Psi_{\alpha,\nu,\tilde{\nu}}=\lim_{t\to \infty}e^{t(2\Delta_\alpha+\ell)} e^{-t{\bf H}}( \chi(c)\Psi^0_{\alpha,\nu,\tilde{\nu}})\]
where $\Psi^0_{\alpha,\nu,\tilde{\nu}}=\mathcal{Q}_{\alpha,\nu,\tilde{\nu}}e^{(\alpha-Q)c}$ and $\mathcal{Q}_{\alpha,\nu,\tilde{\nu}}\in L^2(\Omega_\T)$ is an eigenfunction of ${\bf P}$ with associated eigenvalue $\ell:=|\nu|+|\tilde{\nu}|$ (see Prop \ref{prop:mainvir0}).
We will show in Lemma \ref{PionEVP} below that for $\alpha\in\C$, $\nu,\tilde{\nu}\in\mathcal{T}$ and $\vartheta\in\R$,  $e^{i\vartheta \mathbf{\Pi}}\mathcal{Q}_{\alpha,\nu,\tilde{\nu}}= e^{i(|\nu|-|\tilde{\nu}|) \vartheta} \mathcal{Q}_{\alpha,\nu,\tilde{\nu}} $.

Therefore,  for $s>0$, using the continuity of $e^{-s{\bf H}}:e^{-\beta c_-}L^2(\R\times \Omega_\T)\to e^{-\beta c_-}L^2(\R\times \Omega_\T)$ of \eqref{boundednesspropagator} and $e^{i\vartheta \mathbf{\Pi}}: e^{-\beta c_-}L^2(\R\times \Omega_\T)\to  e^{-\beta c_-}L^2(\R\times \Omega_\T)$, and since $e^{i\vartheta \mathbf{\Pi}}$ and $ e^{-t{\bf H}}$ commute, we have for $\alpha\in \mathcal{I}_{\nu,\tilde\nu}$
\[\begin{split} 
e^{i\vartheta \mathbf{\Pi}}\Psi_{\alpha,\nu,\tilde{\nu}}=& 
\lim_{t\to \infty}e^{t(2\Delta_\alpha+\ell)} e^{i\vartheta \mathbf{\Pi}} e^{-t{\bf H}}(e^{(\alpha-Q)c}\chi(c)\mathcal{Q}_{\alpha,\nu,\tilde{\nu}})\\
=& 
\lim_{t\to \infty}e^{t(2\Delta_\alpha+\ell)}e^{-t{\bf H}} e^{i\vartheta \mathbf{\Pi}}  (e^{(\alpha-Q)c}\chi(c)\mathcal{Q}_{\alpha,\nu,\tilde{\nu}})\\
=& 
\lim_{t\to \infty}e^{t(2\Delta_\alpha+\ell)}e^{-t{\bf H}} e^{i(|\nu|-|\tilde{\nu}|) \vartheta}  (e^{(\alpha-Q)c}\chi(c)\mathcal{Q}_{\alpha,\nu,\tilde{\nu}})\\
= &  e^{i(|\nu|-|\tilde{\nu}|)   \vartheta} \Psi_{\alpha,\nu,\tilde{\nu}}.
\end{split}\]
Now this identity $e^{i\vartheta \mathbf{\Pi}}\Psi_{\alpha,\nu,\tilde{\nu}}= e^{i(|\nu|-|\tilde{\nu}|) \vartheta} \Psi_{\alpha,\nu,\tilde{\nu}}$ extends analytically in the set $W_\ell$ of \eqref{regionWl} as elements in the weighted space $e^{-\beta c_-}L^2(\R\times \Omega_\T)$ provided $\beta>Q-{\rm Re}(\alpha)$. 
 \end{proof}

\begin{lemma}\label{PionEVP}
The following identity holds for $\alpha\in\C$, $\nu,\tilde{\nu}\in\mathcal{T}$ and $\vartheta\in\R$
$$e^{i\vartheta \mathbf{\Pi}}\mathcal{Q}_{\alpha,\nu,\tilde{\nu}}= e^{i(|\nu|-|\tilde{\nu}|) \vartheta} \mathcal{Q}_{\alpha,\nu,\tilde{\nu}} .$$
\end{lemma}

\begin{proof} From  \cite[proof of Prop 4.9]{GKRV20_bootstrap}, $\mathcal{Q}_{\alpha,\nu,\tilde{\nu}}$ is a linear combination of the polynomials $\pi_{\k\l}$ with $|\k|= |\nu|,|\l|= |\tilde{\nu}|$ defined in \cite[(4.16)]{GKRV20_bootstrap}.   Therefore, it suffices to prove the result for $\pi_{\k\l}$. Then, it is enough to observe (this can easily be done recursively) that each $\pi_{\k\l}$ is a product over $n$ of polynomials of the form $P_n(\varphi)=\sum_{\mathcal{I}(k_n,l_n)} a_{p,q}\varphi_n^p\varphi_{-n}^q$ with $\mathcal{I}(k_n,l_n):=\{(p,q)\,|\, p\leq k_n,q\leq l_n,p-q=k_n-l_n\}$,   in such a way that  $$P_n(\varphi(\cdot+\vartheta))=\sum_{\mathcal{I}(k_n,l_n)} a_{p,q}(e^{in\vartheta}\varphi_n)^p (e^{-in\vartheta}\varphi_{-n})^q=e^{i\vartheta n(k_n-l_n)}P_n(\varphi(\cdot)).$$
Our claim then follows by considering the product of such quantities with the constraint that   $|\k|-|\l|= |\nu|- |\tilde{\nu}|$.\end{proof}

\subsection{Reverting orientation}\label{sub:revert}
%%%%%%%%%%%%%%%%%%%%%%%%
Let $o:\T\to\T$ be the map defined by $o(e^{i\theta})=e^{-i\theta}$. It will serve to analyze the effect of reverting orientation on the boundary of the amplitudes. We introduce the following operators
\begin{align}\label{defCandO}
\forall F\in e^{-\beta c_-}L^2(\R\times\Omega_\T), & & \mathbf{O}F(\varphi)=F(\varphi\circ o), & & \mathbf{C}F(\varphi)=\bbar F(\varphi ) 
\end{align}
defined on the weighted $L^2$-spaces with $\beta\in\R$. The operator $\mathbf{C}$ is nothing but complex conjugation. Recall that the semigroup $(e^{-t\mathbf{H}})_{t\geq 0}$ extends continuously to $e^{-\beta c_-}L^2(\R\times\Omega_\T)$ for $\beta\geq 0$ (see \cite[Lemma 6.5]{GKRV20_bootstrap}). We claim the following:
\begin{lemma}\label{commut:semi}
The following commutation relations hold on $e^{-\beta c_-}L^2(\R\times\Omega_\T)$ for $\beta\geq 0$
$$e^{-t\mathbf{H}} \mathbf{O}= \mathbf{O}e^{-t\mathbf{H}}, \quad \quad e^{-t\mathbf{H}} \mathbf{C}= \mathbf{C}e^{-t\mathbf{H}} .$$
\end{lemma}
\begin{proof}
The second part of our claim related to $\mathbf{C}$ is obvious from the Feynman-Kac formula  \eqref{FKgeneral} and the density of $ L^2(\R\times\Omega_\T)$ in $e^{-\beta c_-}L^2(\R\times\Omega_\T)$ for $\beta\geq 0$.  For the first part of our claim, we use again  the Feynman-Kac formula  \eqref{FKgeneral} together with the two following remarks. First it is obvious to check that $P_\D(\varphi\circ o)(z)=P_\D(\varphi)(\bar z)$ for $|z|<1$ (recall that $P\varphi$ is the harmonic extension of $\varphi$ in $\D$). Second, we have the equality in law $X_{\D,D}(\bar z)=X_{\D,D}(z)$ and so that the change of variables $z\mapsto \bar z$ in the integral appearing in the expectation in  \eqref{FKgeneral}  produces $e^{-t\mathbf{H}} \mathbf{O}= \mathbf{O}e^{-t\mathbf{H}}$.
\end{proof}

The purpose of what follows is to understand the action of the operators $\mathbf{O}$ and $\mathbf{C}$  on the free  descendant   states, and then to deduce how they act on Liouville descendant   states. For this, we first gather in the following lemma  the  commutation relation between the operators $\mathbf{O}$, $\mathbf{C}$ and the free Virasoro generators:

\begin{lemma}\label{commut:vira}
The operators $\mathbf{O}$ and $\mathbf{C}$ map  $ \mathcal{C}_\infty$ into itself. On $ \mathcal{C}_\infty$, we have the relations for all $n \in \Z$
\begin{align*}
\mathbf{O} \mathbf{L}_n^0=\widetilde{\mathbf{L}}_n^0\mathbf{O}, & & \mathbf{O}\widetilde{\mathbf{L}}_n^0= \mathbf{L}_n^0\mathbf{O}\\
\mathbf{C} \mathbf{L}_n^0=\widetilde{\mathbf{L}}_n^0\mathbf{C}, & & \mathbf{C}\widetilde{\mathbf{L}}_n^0= \mathbf{L}_n^0\mathbf{C} .
\end{align*}
\end{lemma}

\begin{proof}
The fact that $\mathbf{O}$ and $\mathbf{C}$ map $ \mathcal{C}_\infty$ into itself is obvious. Then it is straightforward to check that, on $ \mathcal{C}_\infty$, we have the relations
\begin{align*}
\mathbf{O} \mathbf{A}_n=\widetilde{\mathbf{A}}_n\mathbf{O},\quad\quad \mathbf{O}\widetilde{\mathbf{A}}_n = \mathbf{A}_n \mathbf{O}, \\
\mathbf{C} \mathbf{A}_n=-\widetilde{\mathbf{A}}_n\mathbf{C},\quad\quad \mathbf{C}\widetilde{\mathbf{A}}_n =- \mathbf{A}_n \mathbf{C},
\end{align*}
 from which our claim follows (also, recall that on $ \mathcal{C}_\infty$ only finitely many terms contribute to the sums \eqref{virassoro} and \eqref{virassorotilde}).
\end{proof}

Now we study the effect of the operators $\mathbf{O}$ and $\mathbf{C}$ on the states in the spectrum. We claim 
\begin{proposition}\label{revert}
For $p\in\R^*$ and $\nu,\tilde\nu\in\mathcal{T}$ Young diagrams, we have the relation
$$\mathbf{O} \mathbf{C}\Psi_{Q+ip,\nu,\tilde\nu}= \mathbf{C}\mathbf{O}\Psi_{Q+ip,\nu,\tilde\nu}=\Psi_{Q-ip,\nu,\tilde\nu}. $$
\end{proposition}

\begin{proof}
First note that $\mathbf{O}$ and $\mathbf{C}$ commute on $e^{-\beta c_-}L^2(\R\times\Omega_\T)$ for any $\beta>0$ and $\Psi_{Q+ip,\nu,\tilde\nu}\in e^{-\epsilon c_-}L^2(\R\times\Omega_\T)$ for some $\epsilon>0$  (see Proposition \ref{defprop:desc}). Therefore we only need to prove the r.h.s. of our claim. We will prove this relation for $\Psi_{\alpha,\nu,\tilde\nu}$ and $\alpha$ in a complex neighborhood of the real line, and then argue by analytic continuation.
From  Proposition \ref{defprop:desc}, if $\chi\in C^\infty(\R)$ is equal to $1$ near $-\infty$ and supported in $\R^-$, then for $\alpha\in \mathcal{I}_{\nu,\tilde\nu}$  (which is stable under complex conjugation),
\[ \Psi_{\alpha,\nu,\tilde{\nu}}=\lim_{t\to \infty}e^{t(2\Delta_\alpha+|\nu|+|\tilde{\nu}|)} e^{-t{\bf H}}( \chi(c)\Psi^0_{\alpha,\nu,\tilde{\nu}}).\]
By continuity of the semigroup and using the commutation relations of Lemmas \ref{commut:semi} and \ref{commut:vira}, we deduce
\begin{align*}
 \mathbf{C}\mathbf{O}\Psi_{\alpha,\nu,\tilde{\nu}}=
 &\lim_{t\to \infty}e^{t(2\Delta_{\bar{\alpha}}+|\nu|+|\tilde{\nu}|)} e^{-t{\bf H}}( \chi(c) \mathbf{L}_{-\nu}^0\tilde{\mathbf{L}}_{-\tilde\nu}^0 \mathbf{C}\mathbf{O}\Psi^0_{\alpha })
 \\
 =&\lim_{t\to \infty}e^{t(2\Delta_{\bar{\alpha}}+|\nu|+|\tilde{\nu}|)} e^{-t{\bf H}}( \chi(c) \mathbf{L}_{-\nu}^0\tilde{\mathbf{L}}_{-\tilde\nu}^0 \Psi^0_{\bar \alpha })\\
  =&\lim_{t\to \infty}e^{t(2\Delta_{\bar{\alpha}}+|\nu|+|\tilde{\nu}|)} e^{-t{\bf H}}( \chi(c)  \Psi^0_{\bar \alpha,\nu,\tilde\nu })\\
  =&\Psi_{\bar\alpha,\nu,\tilde{\nu}}
 \end{align*}
for $\alpha\in \mathcal{I}_{\nu,\tilde{\nu}}$, where we used the obvious observations $ \mathbf{C}\mathbf{O}\Psi^0_{\alpha }=\Psi^0_{\bar \alpha }$ and $\overline{\Delta_{\alpha}}= \Delta_{\bar\alpha}$. The identity $ \mathbf{C}\mathbf{O}\Psi_{\alpha,\nu,\tilde{\nu}}=\Psi_{\bar\alpha,\nu,\tilde{\nu}}$ is an equality  of anti-holomorphic maps valid over $\mathcal{I}_{\nu,\tilde\nu}$ and thus extends to the connected component of $W_{|\nu|+|\nu'|}$ containing $\mathcal{I}_{\nu,\tilde\nu}$, hence the spectrum line.
\end{proof}

%%%%%%%%%%%%%%%%%%%%%%%%%%%%%%%%%%%%%%%%%%%%%%%%%%%%%%%%%%%%%%%%%%%%%%%%%%%%%%%%%%%%%%%%%%%%%%%%%%%%%%%%%%%%%%%%%%%%%%%%%%%%%%%%%%%%%%%%%%%%%%%%%%%%%%%%%%%%%%%%%%%%%%%%%%%%%%%%%%%%%%%%%%%%%%%%%%%%%%%%%%%%%%%%%%%%%%%%%%%%%%%%%%%%%%%%%%%%%%%%%%%%%%%%%%%%%%%%%%%%%%%%%%
 \section{Traces, Hilbert-Schmidt operators and decomposition 
 of traces using a spectral resolution}\label{section:traces}
%%%%%%%%%%%%%%%%%%%%%%%%%%%%%%%%%%%%%%%%%%%%%%%%%%%%%%%%%%%%%%%%%%%%%%%%%%%%%%%%%%%%%%%%%%%%%%%%%%%%%%%%%%%%%%%%%%%%%%%%%%%%%%%%%%%%%%%%%%%%%%%%%%%%%%%%%%%%%%%%%%%%%%%%%%%%%%%%%%%%%%%%%%%%%%%%%%%%%%%%%%%%%%%%%%%%%%%%%%%%%%%%%%%%%%%%%%%%%%%%%%%%%%%%%%%%%%%%%%%%%%%%%%

In this section, we review some basic material on trace class and Hilbert-Schmidt operators on a Hilbert space,
 and we define a notion of gluing amplitudes associated to an admissible oriented multi-graph. We use a general formalism 
 in order to use it for both the Segal amplitudes and the construction of conformal blocks, which will be obtained by gluing $L^2$
 amplitudes defined on the space $\mc{T}$ of Young diagrams.

\subsection{Traces, integral kernels and Hilbert-Schmidt norms} 
We recall in this section basic facts on Hilbert-Schmidt operators and traces, we refer to \cite[Chapter 4]{GGK_book_2000}.
Let $(M_1,\mu_1)$ and $(M_2,\mu_2)$ be two separable measured spaces and let $H_i:=L^2(M_i,\mu_i)$.
Let $ \mathbf{K}: H_1\to H_2$ be a bounded operator. We say that $ \mathbf{K}$ is Hilbert-Schmidt if $\mathbf{K}^*\mathbf{K}:H_1\to H_1$ is compact and the eigenvalues $(\lambda_j^2)_{j\in\N}$ of $\mathbf{K}^*\mathbf{K}$ (with $\la_j\geq 0$) are in $\ell^1(\N)$. Denote by $|{\bf K}|:=\sqrt{\mathbf{K}^*\mathbf{K}}$.
 If $H_1=H_2$, we say that ${\bf K}:H_1\to H_1$ is trace class if $\sum_j\la_j<\infty$, and its trace is then defined by the converging series 
\[{\rm Tr}({\bf K})=\sum_{j}\cjg {\bf K}e_j,e_j\cjd_{H_1}\]
where $(e_j)_j$ is an orthonormal basis of $H_1$.
The space $\mc{L}_2(H_1,H_2)$ of Hilbert-Schmidt operators (which will be denoted $\mc{L}_2(H_1)$ when $H_1=H_2$) is a Hilbert space and the space $\mc{L}_1(H_1)$ of trace class operators is a Banach space, when equipped with the norms
\[\|\mathbf{K}\|_{{\rm HS}}^2:=\sum_{j} \lambda_j^2={\rm Tr}({\bf K}^*{\bf K})=\sum_{j}\|\mathbf{K}e_j\|^2_{L^2}, \quad \|\mathbf{K}\|_{{\rm Tr}}:=\sum_{j=1}^\infty \la_j={\rm Tr}(|{\bf K}|).\] 
Notice that, if $(e'_i)_i$ is an orthonormal basis of $H_2$, then  $\|\mathbf{K}\|_{{\rm HS}}^2=\sum_{i,j}|\cjg {\bf K}e_j,e_i'\cjd|^2$, thus ${\bf K}^*\in \mc{L}_2(H_2,H_1)$ 
and $\|\mathbf{K}\|_{{\rm HS}}=\|\mathbf{K}^*\|_{{\rm HS}}$. We also have $\mc{L}_1(H_1)\subset \mc{L}_2(H_1)\subset \mc{L}(H_1)$ (with  $\mc{L}(H_1)$ the space of continuous operators from $H_1$ into itself) and 
\begin{equation}\label{orderschatten}
\|\,  |{\bf K}|\, \|_{\mc{L}(H_1)}\leq \|{\bf K}\|_{\rm HS}\leq \|{\bf K}\|_{\rm Tr}.
\end{equation}
The trace of $K$ is controlled by its trace class norm: 
\begin{equation}\label{boundtrace1}
|{\rm Tr}({\bf K})| \leq \|{\bf K}\|_{{\rm Tr}}.
\end{equation}

\begin{lemma}\label{lem:boundstraces}
For any family ${\bf K}_\ell\in \mc{L}_2(H_{\ell+1},H_{\ell})$ with $H_\ell=L^2(M_\ell,\mu_\ell)$ for $\ell=1,\dots,L$
\begin{align} 
& \|{\bf K}_1{\bf K}_2\|_{{\rm HS}}\leq \|\,|{\bf K}_1|\,\|_{\mc{L}(H_2)}\|{\bf K}_2\|_{\rm HS} \label{HSproduct}\\
& \|{\bf K}_1\dots {\bf K}_L\|_{{\rm HS}}\leq \|{\bf K}_1\|_{{\rm HS}}\dots \|{\bf K}_L\|_{{\rm HS}},\label{iterativetraces}\\
& \|{\bf K}_1\dots {\bf K}_L\|_{{\rm Tr}}\leq \|{\bf K}_1\|_{{\rm HS}}\dots \|{\bf K}_L\|_{{\rm HS}} \quad \textrm{ if }M_{L+1}=M_1.\label{iterativetraces2}
\end{align}
\end{lemma}
\begin{proof}
The first identity can be obtained by 
Cauchy-Schwarz, using $(e_j')$ an orthonomal eigenbasis for ${\bf K}_2{\bf K}_2^*$ on $H_2$ and the cyclicity of the trace 
\begin{equation}\begin{split}\label{boundtrace2}
 \|{\bf K}_1{\bf K}_2\|_{{\rm HS}}^2& ={\rm Tr}({\bf K}_2^*{\bf K}_1^*{\bf K}_1{\bf K}_2)={\rm Tr}({\bf K}_2{\bf K}_2^*{\bf K}_1^*{\bf K}_1)=\sum_{j}\cjg {\bf K}^*_1{\bf K}_1e'_j, {\bf K}_2{\bf K}_2^*e'_j\cjd_{H_2} \\
 & \leq \|\,|{\bf K_1}|\,\|^2_{\mc{L}(H_2)}\sum_j \|{\bf K}_2{\bf K}_2^*e'_j\|_{H_2}= 
 \|\,|{\bf K}_1|\,\|^2_{\mc{L}(H_2)}\|{\bf K}_2\|^2_{{\rm HS}}.
\end{split}\end{equation}
This also implies \eqref{iterativetraces} by \eqref{orderschatten} and an iterative argument. The last bound \eqref{iterativetraces2} follows from the bound $\|{\bf K}_1{\bf K}_2\|_{\rm Tr}\leq \|{\bf K}_1\|_{\rm HS}\|{\bf  K}_2\|_{\rm HS}$ if $H_3=H_1$ (\cite[Lemma 7.2]{GGK_book_2000}). 
\end{proof}
Now assume $ \mathbf{K}$ is a compact operator of the form 
\[  \mathbf{K}u(x)=\int_{M_2} K(x,y)u(y)\dd\mu_2(y)\]
for some measurable integral kernel $K$ on $M_1\times M_2$ (we use unbold notation for kernels).
 If $(e_j)_j$ is an orthonormal basis of eigenfunctions of $ \mathbf{K}^* \mathbf{K}$ associated to the eigenvalue $\la_j^2$, ${\bf K}$ is Hilbert-Schmidt if and only if $K\in L^2(M_1\times M_2, \mu_1\otimes \mu_2)$ and the Hilbert-Schmidt norm is 
\begin{equation}\label{HSnorm=L^2}
 \begin{split}
\|\mathbf{K}\|_{{\rm HS}}^2=\sum_{j}\|\mathbf{K}e_j\|^2_{L^2}
=\sum_{j,k}\Big|\int_{M_1\times M_2}K(x,y)e_j(y)\bar{e}_k(x)\dd\mu_1(x)\dd\mu_2(y)\Big|^2
=\|K(\cdot,\cdot)\|^2_{L^2(M_1\times M_2)} .\end{split} 
\end{equation}
Conversely, any Hilbert-Schmidt operator ${\bf K}:L^2(M_1,\mu_1)\to L^2(M_2,\mu_2)$ has an integral kernel $K\in L^2(M_2\times M_1)$ given by $K(x,y)=\sum_{i,j}\cjg K,e_i'\otimes e_j\cjd_{L^2}e_i'(x)e_j(y)$.
Hilbert-Schmidt operators are convenient to work with since they can be identified as $L^2$ functions through their integral kernel.

\begin{lemma}\label{compositionHS}
Assume $\mathbf{K}=\mathbf{K}_2\mathbf{K}_1$ is a composition of two Hilbert-Schmidt operators $\mathbf{K}_1:L^2(M_1)\to L^2(M_2)$ and $\mathbf{K}_2:L^2(M_2)\to L^2(M_3)$ with integral kernels $K_1(y,z)$ and $K_2(x,y)$. Then $\mathbf{K}$ is Hilbert-Schmidt with integral kernel $K(x,z)=\int_{M_2}K_2(x,y)K_1(y,z)\dd\mu_2(y)$ that belongs to $L^2(M_3\times M_1)$. If $(M_3,\mu_3)=(M_1,\mu_1)$ it is also trace class on $L^2(M_1)$.
Furthermore, the restriction of its kernel to the diagonal  $K(x,x)$ makes sense for a.e. $x\in M_1$, is $L^1(M_1,\mu_1)$ and we have 
\[{\rm Tr}(\mathbf{K})=\int_{M_1}K(x,x)\dd\mu_1(x).\]
\end{lemma}
\begin{proof}
First notice that $\mathbf{K}$ has integral kernel  given,  $\mu_3\otimes \mu_1$-almost everywhere in $(x,z)\in M_3\times M_1$, by
\begin{equation}\label{Kxz} 
K(x,z)=\int_{M_2}K_2(x,y)K_1(y,z)\dd\mu_2(y)
\end{equation}
and we see by Cauchy-Schwarz that 
\begin{equation}\label{CauchySchwarzkernel} 
|K(x,z)|\leq \|K_2(x,\cdot)\|_{L^2}\|K_1(\cdot,z)\|_{L^2}.
\end{equation}
If $M_3=M_1$ we infer from \eqref{CauchySchwarzkernel} that there is a set $A\subset M_1$ of full measure   where $K_2(x,\cdot)$ and $K_1(\cdot,z)$ are defined as $L^2(M_1)$ for $x,z\in A$, hence $K(x,z)$ makes sense for $(x,z)\in A^2$. Consequently  $K(x,x)$ makes sense for a.e. $x\in M_1$ if defined as the integral \eqref{Kxz} with $x=z$ belonging to $A$.
We have by Cauchy-Schwarz
\[ \int_{M_1} |K(x,x)|\dd\mu_1(x)\leq \int_{M_1} \|K_1(\cdot,x)\|_{L^2}\|K_2(x,\cdot)\|_{L^2}\dd\mu_1(x)\leq 
\|K_1\|_{L^2(M_2\times M_1)}\|K_2\|_{L^2(M_1\times M_2)}.
\]
so that $x\mapsto K(x,x)\in L^1(M_1)$, and 
\[ \int K(x,x)\dd\mu_1(x)=\int_{M_1\times M_2} K_2(x,y)K_1(y,x)\dd\mu_1(x)\dd\mu_2(y)=\cjg K_1,K^*_2\cjd_{L^2(M_2\times M_1)}\]
with $K_2^*$ the integral kernel of ${\bf K}_2^*$.
We have that $\mathbf{K}=\mathbf{K}_2\mathbf{K}_1$ is trace class, as composition of two Hilbert-Schmidt operators, and its trace is defined by ${\rm Tr}(\mathbf{K})= \sum_{j}\cjg \mathbf{K} e_j,e_j\cjd_{L^2}$ for some orthonormal basis $(e_j)_j$  of eigenfunctions of $ \mathbf{K}^* \mathbf{K}$ (and we denote by $(f_k)_k$  an orthonormal basis of $L^2(M_2)$), which is equal to 
\[\begin{split} 
{\rm Tr}(\mathbf{K})=& \sum_{j}\cjg \mathbf{K} e_j,e_j\cjd_{L^2}=\sum_{j}\cjg \mathbf{K}_1e_j,\mathbf{K}_2^*e_j\cjd_{L^2}=
\sum_{j,k}\cjg \mathbf{K}_1e_j,f_k\cjd_{L^2} \cjg f_k,\mathbf{K}_2^*e_j\cjd_{L^2}\\
=&   \int_{M_1\times M_2} K_2(x,y)K_1(y,x)\dd\mu_1(x)\dd\mu_2(y)= \int_{M_1}K(x,x)\dd\mu_1(x).\qedhere
\end{split}\]
\end{proof}

Let us draw a parallel with \eqref{iterativetraces} using integral kernels. 
If $K_1\in L^2(M_2\times M_1), K_2\in L^2(M_1\times M_2)$ are viewed as integral kernels of Hilbert-Schmidt operators, we can integrate \eqref{CauchySchwarzkernel} in $x,z$ and get
\[ \|{\bf K}_2{\bf K}_1\|_{{\rm HS}}^2 =\int_{M_1\times M_1} \Big(\int_{M_2}|K_2(x,y)K_1(y,z)|\dd\mu_1(y)\Big)^2\dd\mu_2(x)\dd\mu_2(z)\leq \|K_1\|_{L^2}^2\|K_2\|_{L^2}^2=\|{\bf K}_1\|^2_{{\rm HS}}\|{\bf K}_2\|_{{\rm HS}}^2.\]  

\subsection{Amplitudes and gluings of amplitudes}\label{Subsec:amplitudeandgluing}
As before $(M,\mu)$ is a separable measured space and $H=L^2(M,\mu)$; we will also identify $L^2(M^\ell;\otimes^\ell \mu)$ with the tensor product $\otimes^\ell H$. 
We use the notation $\llbracket 1,n\rrbracket:=[1,n]\cap \N$ and for an $k$-uplet 
$x=(x_1,\dots,x_k)\in M^k$ and $p:\llbracket 1,\ell\rrbracket\to \llbracket 1,k\rrbracket$ an injective map, we write $\mathbf{x}_p=(x_{p(1)},\dots,x_{p(\ell)})$. 
We shall write $\mu^\ell$ for $\otimes^\ell \mu$ in what follows.
\begin{definition}
For $k\in\N^*$, we call $k$-amplitude a function $K\in L^2(M^k,\mu^k)$.
\end{definition}

Let us note that a $k$-amplitude $K$ can also be seen  as the integral kernel of a Hilbert-Schmidt operator from $L^2(M^{k_1};  \mu^{k_1})$ to $L^2(M^{k_2};   \mu^{k_2})$ with $k_1+k_2=k$.
The notion of amplitude described here is general but will be mostly applied to the LCFT amplitudes associated to Riemann surfaces with boundary introduced in the previous section and to the construction of conformal blocks in Section \ref{Section:blocks}. Now we describe the notion of  gluing of amplitudes.  
 
If $K_1$ and $K_2$ are amplitudes, we define a gluing of amplitudes by pairing/integrating out some of their coordinates together, as we now explain. 
  A choice of $\ell\leq k$ copies of $M$ in $M^k$ can be represented by an injective  map 
$p:\llbracket 1,\ell\rrbracket\to \llbracket 1,k\rrbracket$ or equivalently by a permutation $\sigma$ of $\llbracket 1,k\rrbracket$ satisfying $\sigma(j)=p(j)$ for $j\in\llbracket 1,\ell\rrbracket$ with $\sigma$ being increasing on 
$[\ell+1,k]$. 
For an amplitude $K\in L^2(M^k)$, we denote 
\[K_\sigma(x):=K(x_{\sigma^{-1}(1)},\dots,x_{\sigma^{-1}(k)}).\]
\begin{definition}
Let $K\in L^2(M^{k})$ and $K'\in L^2(M^{k'})$ be two amplitudes. Let $p:\llbracket 1,\ell\rrbracket\to \llbracket 1,k\rrbracket$ and $p':\llbracket 1,\ell\rrbracket\to \llbracket 1,k'\rrbracket$ be two injective maps, and we say that the $p(i)$-th copy of $M$ in $M^k$ is paired with the $p'(i)$-th copy of $M$ in $M^{k'}$ for each $i\in \llbracket 1,\ell\rrbracket$. Let $\sigma,\sigma'$ be the two associated permutations of $\llbracket 1,k\rrbracket$ and $\llbracket 1,k'\rrbracket$. The gluing of the amplitudes $K$ and $K'$ by the pairing $\mc{G}:=(p,p')$ is the 
amplitude $K\circ_{\mc{G}} K'\in L^2(M^{k+k'-2\ell})$ given by 
\[ \forall (x,x')\in M^{k-\ell}\times  M^{k'-\ell},\quad 
K\circ_{\mc{G}} K'(x,x'):=\int_{M^\ell}K_{\sigma}(y,x)K'_{\sigma'}(y,x')\dd\mu^\ell(y).\]
By Lemma \ref{compositionHS}, this is an amplitude satisfying
\[\|K\circ_{\mc{G}} K'\|_{L^2(M^{k+k'-2\ell})}\leq \|K\|_{L^2(M^k)}\|K'\|_{L^2(M^{k'})}.\]
\end{definition}

The iteration of several gluings $\mc{G}_1,\dots,\mc{G}_N$ of amplitudes can be described using a multigraph $\mc{G}$ (in the sense of Section \ref{subsubsec:graphdec}) as follows.
Let $K_1\in L^2(M^{k_1}),\dots,K_N\in L^2(M^{k_N})$ be $N$ amplitudes and we want to describe the multiple gluing
\[(K_1\circ\dots \circ K_N)_{\mc{G}}:=((K_1\circ_{\mc{G}_1} K_2)\circ_{\mc{G}_2}\dots )\circ_{\mc{G}_N}K_N.\] 
To each amplitude 
$K_j$ we associate a vertex $v_j=\{v_{j1},\dots,v_{jb_j}\}$ with $b_j$ elements.
The gluing of the amplitude $K_{1}$ to $K_{2}$ is described above by pairing 
$\ell$ variables $(x_{p(1)},\dots,x_{p(\ell)})$ 
in $K_{1}$ to $\ell$ variables $(x_{p'(1)},\dots,x_{p'(\ell)})$ in $K_{2}$, which we represent 
as linking edges $e_{p(i)}=(v_{1p(i)},v_{2,p'(i)})$.
The graph with $v_1,v_2$ as vertices and $(e_{p(i)})_i$ as edges represents the gluing $\mc{G}_1=(p,p')$, we shall just then denote this multigraph by $\mc{G}_1$.
 The iteration of this process allows to associate a multigraph $\mc{G}$ by linking the $N$ vertices $v_1,\dots,v_N$  through the choices of $L$ edges linking some of the vertices $(v_j)_{j=1,\dots, N}$ with $L\geq N$. This multigraph does not have loops, i.e. edges of the form $e_i=(v_{jk},v_{jk'})$ for some $j,k,k'$.
The set of vertices is denoted by $V$, the set of edges is denoted by $E$, and let
$W=\bigcup_j v_j$ and $W_{\mc{G}}:=W\setminus \bigcup_{i=1}^L ({\rm v}_1(e_i)\cup {\rm v}_2(e_i))$ (recall Section \ref{subsubsec:graphdec} for the definition of ${\rm v}_j(e_i)$).  
The resulting expression for the composition can be written as follows:  
write 
${\bf x}_j=(x_{j1},\dots,x_{jb_j})\in M^{b_j}$ for each $j$, and ${\bf x}=(x_1,\dots,x_L)\in M^L$
where $x_i:=x_{{\rm v}_1(e_i)}$, then 
 \begin{equation}\label{compositiongraph}
 (K_1\circ\dots \circ K_N)_{\mc{G}}=\int_{M^L}
 \big\cjg \delta_{\mc{G}},\prod_{j=1}^N K_j({\bf x}_j)\big\cjd \, \dd\mu^L({\bf x})
\end{equation}
where $\delta_{\mc{G}}$ is the Dirac measure on $M^{\sum_jb_j}$ defined by 
\begin{equation}\label{deltamcG}
\delta_{\mc{G}}({\bf x}_1,\dots,{\bf x}_N):=\prod_{i=1}^L\delta_{x_{{\rm v}_1(e_i)}=x_{{\rm v}_2(e_i)}}.
\end{equation} 
A priori $\cjg\delta_{\mc{G}},F\cjd$ is well-defined (in $[0,\infty]$) for $F\in C^0(M^{\sum_jb_j};\R_+)$ if one assumes that $M$ is a topological space, but we note by Lemma \ref{compositionHS} that the integral $\int \cjg \delta_{\mc{G}},\prod_{j=1}^NK_j({\bf x}_j)\cjd d\mu^L({\bf x})$ makes sense more generally in our setting since the amplitudes $K_j$ are $L^2(M^{b_j},\mu^{b_j})$: the result is a function of the variables 
$x_{jk}$ for $(j,k)\in W\setminus W_{\mc{G}}$, i.e. a function in $L^2(M^{d})$ with $d:=\sum_j b_j-2L$.
Moreover one has
\begin{equation}\label{L^2boundcomposition} 
\|(K_1\circ\dots \circ K_N)_{\mc{G}}\|_{L^2(M^d)}\leq \prod_{j=1}^N \|K_j\|_{L^2(M^{b_j})}.
\end{equation}

We conclude   with an important remark. 
\begin{remark}\label{remark:selfgluing}
It is possible to choose a multigraph 
$\mc{G}$ representing the gluing of amplitudes so that certain edges $e_i$ are loops, 
i.e. $e_i=(v_{jk},v_{jk'})$ for some $j,k,k'$. In that case, one would need to 
perform the following type of integrals
\begin{equation}\label{selfgluingK}
\int_{M}K_j(\dots,y,\dots,y,\dots)\dd\mu(y)
\end{equation} 
where the integrated entries are the $k$-th and $k'$-th ones. 
By the previous discussion, we observe that if $K_j\in L^2(M^{b_j})$ can be written under the form 
\begin{align*}
& K_j(x_1,\dots,x_{k-1},y,x_{k+1},\dots,x_{k'-1},y',x_{k'+1},\dots,x_b)=\\
&  \int_{M'} K'_j(x_1,\dots,x_{k-1},y,x_{k+1},\dots,x_{k'-1},x_{k'+1},\dots,x_b,z) K''_j(y',z)\dd\mu'(z)
\end{align*}
for some measured space $(M',\mu')$, some $K'_j\in 
L^2(M^{b_j}\times M',\mu^{b_j}\otimes \mu')$ and $K''_j\in L^2(M\times M',\mu\otimes \mu')$, then 
\eqref{selfgluingK} is defined a.e. and belongs to $L^2(M^{b_j-2})$ (again by Lemma \ref{compositionHS}). This means that loops in $\mc{G}$ and the associated self-gluing of amplitudes are not problematic as long as the amplitude corresponding to a vertex with a loop can be decomposed as a gluing of two $L^2$-amplitudes. In some sense this can be viewed as adding a vertex to $\mc{G}$ that cuts the loop. 
\end{remark}

\subsection{Connection with probabilistic amplitudes}
%%%%%%%%%%%%%%%%%%%%%%%%%%%%%%
Finally we conclude this section with an important consequence of Proposition \ref{glueampli}, which connects the probabilistic definition of amplitudes defined in Section \ref{probamp} with the material of the present section. Here we take the separable measured space $(M,\mu):=(\R\times \Omega_\T,\mu_0)$ (recall $\mu_0=\dd c\otimes\P_\T$) and the Hilbert space $H=\mc{H}=L^2(\R\times \Omega_\T,\mu_0)=L^2(H^{s}(\mathbb{T}),\mu_0)$ for $s<0$. 

 \begin{lemma}\label{Lem:splitanamplitudeinL2}
 Let  $(\Sigma,g, {\bf x},\boldsymbol{\zeta})$ be an admissible surface with $b\geq 1$ boundary components and $\boldsymbol{\alpha}=(\alpha_1,\dots,\alpha_m)$ some weights. Assume that the weights satisfy the condition $\sum_i \alpha_i-  Q \chi(\Sigma)>0$. Then the amplitude  $\caA_{\Sigma ,g,{\bf x},\boldsymbol{\alpha} }\in \otimes^b \mc{H}=L^2((\R\times \Omega_{\mathbb{T}})^b,\mu_0^b)$ is a $b$-amplitude and the integral kernel of a Hilbert-Schmidt operator $\otimes^b\mc{H}\to \C$. If $j \leq b$, there are $K_1\in L^2((\R\times \Omega_{\mathbb{T}})^{b},\mu_0^{b})$ and  $K_2\in L^2((\R\times \Omega_{\mathbb{T}})^2,\mu_0^2)$ so that 
 \[ \caA_{\Sigma ,g,{\bf x},\boldsymbol{\alpha} }(\tilde\varphi_1,\dots,\tilde\varphi_b)=\int_{\R\times \Omega_\T} 
 K_1(\tilde\varphi_1,\dots,\tilde\varphi_{j-1},\tilde\varphi_{j+1},\dots,\tilde\varphi_b,\varphi) K_2(\tilde\varphi_{j},\tilde\varphi)d\mu_0(\tilde\varphi).\]
  \end{lemma}
 \begin{proof}
The first statement follows directly from Theorem \ref{integrcf} and the definition of $b$-amplitude. Let us then prove the last statement and we can assume that $j=1$ without loss of generality. 
Let us write 
$\tilde{\boldsymbol{\varphi}}=( \tilde\varphi _1,\tilde{\boldsymbol{\varphi}}_2)\in H^s(\T)\times (H^s(\T))^{b-1}$ for the boundary fields of the amplitude  the amplitude  $\caA_{\Sigma ,g,{\bf x},\boldsymbol{\alpha} }$ (i.e. we single out the boundary field $\tilde\varphi _1$ attached to the first boundary component and gather all other boundary fields in  $\tilde{\boldsymbol{\varphi}}_2$).
Since $g$ is admissible, it is of the form $|z|^{-2}|dz|^2$ in local coordinates over a neighborhood of the first boundary component $\pl_1\Sigma$. By choosing $q\in (0,1)$ close enough to $1$, we can cut a small annulus $\mathbb{A}_q$ around the first boundary component so that the metric $g$ is still of the form $|z|^{-2}|dz|^2$ over this annulus, $\mathbb{A}_q$ contains no marked points and $\Sigma= \mathbb{A}_q \#_1\mathcal{1}  \Sigma_q$, where $\Sigma_q$ is the surface $\Sigma$ with the annulus removed. Let us call $\mc{C}$ the curve splitting $\Sigma$ into the two components $\mathbb{A}_q$ and $ \Sigma_q$. We can decide arbitrarily that $\mc{C}$ is incoming on $ \Sigma_q$ and outgoing on $\mathbb{A}_q$. Then Proposition \ref{glueampli} asserts that
\[\caA_{\Sigma ,g,{\bf x},\boldsymbol{\alpha} }(\tilde\varphi _1,\tilde{\boldsymbol{\varphi}}_2)=\frac{1}{\sqrt{2}\pi}\int \caA_{\Sigma_q ,g,{\bf x},\boldsymbol{\alpha} }(\tilde\varphi ,\tilde{\boldsymbol{\varphi}}_2) \caA_{\mathbb{A}_q ,g_{\mathbb{A}}}(\tilde\varphi _1,\tilde\varphi )  \dd \mu_0(\tilde\varphi).\]
Let us  now set $s_0:=\sum_i \alpha_i-  Q \chi(\Sigma)>0$ and fix $\epsilon >0$ such that $\epsilon<s_0/b$. Write $c$ for the zero mode of the boundary field $\tilde\varphi  $. Let us rewrite the product of amplitudes in the above integral as $\int K_1(\tilde\varphi ,\tilde{\boldsymbol{\varphi}}_2)K_2(\tilde\varphi _1,\tilde\varphi )d\mu_0(\tilde\varphi)$ with 
\[K_2(\tilde\varphi _1,\tilde\varphi) := e^{\epsilon c_-} \caA_{\mathbb{A}_q ,g_{\mathbb{A}}}(\tilde\varphi _1,\tilde\varphi ),\quad K_1(\tilde\varphi ,\tilde{\boldsymbol{\varphi}}_2)= e^{-\epsilon c_-}\caA_{\Sigma_q ,g,{\bf x},\boldsymbol{\alpha} }(\tilde\varphi ,\tilde{\boldsymbol{\varphi}}_2).\]
 By  Theorem \ref{integrcf}, we have that $K_1,K_2$ are $L^2$-kernels, from which we deduce the result.
 \end{proof}

 \subsection{Bounds on correlations}
 As a direct corollary of Propositions \ref{glueampli}, \ref{selfglueampli}, Lemma \ref{Lem:splitanamplitudeinL2} and the estimate \ref{L^2boundcomposition}, we obtain bounds on $m$-point correlation functions of surfaces of any genus in terms of correlation functions of genus $0$ (with $4$ insertions), genus $1$ (with $2$ insertions) or genus $2$ (with no insertions) surfaces.  
 \begin{corollary}
Let $(\Sigma,g,{\bf x},\boldsymbol{\alpha})$ be an admissible Riemann surface with $m$ marked points ${\bf x}=(x_1,\dots,x_m)$ with weights $\boldsymbol{\alpha}=(\alpha_1,\dots,\alpha_m)$ 
and let $(\mc{P}_j,g,{\bf x}_j,\boldsymbol{\alpha}_j)$ be $N:=2{\bf g}-2+m$ building blocks, i.e. pairs of pants, annuli with $1$ insertion or disks with $2$ insertions, where ${\bf x}_j={\bf x}\cap \mc{P}_j$ and let $b_j$ be the number of boundary components of the block $\mc{P}_j$. Let $\Sigma_j=\mc{P}_j\# \mc{P}_j$ be the double of $\mc{P}_j$ obtained by gluing each $\pl_{k}\mc{P}_j$ with itself and let ${\bf x}^2_j:={\bf x}^{\ell}_j\cup {\bf x}^{r}_j$ be 
the obtained marked points on the left/right copy of $\mc{P}_j$ with ${\bf x}^{\ell}_j=(x^{\ell}_{ji})_{i\leq 3-b_j}\subset {\bf x}$ and 
${\bf x}^{r}_j=(x^{r}_{ji})_{i\leq 3-b_j}$, and denote similarly $\boldsymbol{\alpha}^{\ell}_j=(\alpha^{\ell}_{ji})_{i\leq 3-b_j}$,  $\boldsymbol{\alpha}^{r}_j=(\alpha^{r}_{ji})_{i\leq 3-b_j}$. If $\sum_{i\leq 3-b_j} \alpha_{ji}-  Q \chi(\mc{P}_j)>0$ for all $j=1,\dots,N$, 
then we have the bound
\[ \cjg \prod_{i=1}^m V_{\alpha_i}(x_i)\cjd_{\Sigma} \leq \prod_{j=1}^{N}\Big\cjg \prod_{i\leq 3-b_j} V_{\alpha^{\ell}_{ji}}(x^{\ell}_{ji})V_{\alpha^{r}_{ji}}(x^{r}_{ji})\Big\cjd_{\Sigma_j}^{1/2}.\]
 \end{corollary}
In the case of the Riemann sphere $\mathbb{S}^2$, this gives bounds of the $m$-points correlations in terms of $4$-points 
correlation functions and $2$-points correlation function on the torus ${\T^2}$ by splitting the sphere into two half spheres with $2$ insertions and $m-4$ annuli with $1$ insertion.  In the case of the torus $\T^2$, this gives bounds on the $m$-point correlation functions in terms of $2$-point correlation functions on $\T^2$. In the case of a genus ${\bf g}$ surface $\Sigma$ with no insertions, the partition function can be bounded by products of partition functions of genus $2$ surfaces $\Sigma_j$ obtained by doubling each pair of pants $\mc{P}_j$ constituting $\Sigma$.

  %%%%%%%%%%%%%%%%%%%%%%%%%%%%%%%%%%%%%%%%%%%%%%%%%%%%%%%%%%%%%%%%%%%%%%%%%%%%%%%%%%%%%%%%%%%%%%%%%%%%%%%%%%%%%%%%%%%%%%%%%%%%%%%%%%%%%%%%%%%%  %%%%%%%%%%%%%%%%%%%%%%%%%%%%%%%%%%%%%%%%%%%%%%%%%%%%%%%%%%%%%%%%%%%%%%%%%%%%%%%%%%%%%%%%%%%%%%%%%%%%%%%%%%%%%%%%%%%%%%%%%%%%%%%%%%%%%%%%%%%%
\section{Conformal blocks and the proof of the conformal bootstrap}\label{Section:blocks}
    %%%%%%%%%%%%%%%%%%%%%%%%%%%%%%%%%%%%%%%%%%%%%%%%%%%%%%%%%%%%%%%%%%%%%%%%%%%%%%%%%%%%%%%%%%%%%%%%%%%%%%%%%%%%%%%%%%%%%%%%%%%%%%%%%%%%%%%%%%%%
 
In this section, we prove the main result of the paper, namely Theorem \ref{th:fullexpressioncorrel}. 
To do so, first the decomposition of the Riemannian surface $(\Sigma,g)$ into complex building blocks $\mc{P}_j$ is combined with Section \ref{section:traces} to decompose the correlation functions as pairings of (Hilbert-Schmidt) 
 Segal amplitudes $\mc{A}_{\mc{P}_j,g,{\bf x}_j,\boldsymbol{\alpha}_j,\boldsymbol{\zeta}_j}$ 
 in tensor powers of the Hilbert space $\mc{H}$.  Next, we use the spectral resolution of Section \ref{section:semigroup} to rewrite 
 these pairings in terms of Young-diagrams. We shall see in Proposition \ref{prop:decompositionAmplitude} that this becomes a pairing of amplitudes 
 $\widehat{\mc{A}}_{\mc{P}_j,g,{\bf x}_j,\boldsymbol{\alpha}_j,\boldsymbol{\zeta}_j}$
 of Hilbert-Schmidt operators acting on tensor powers of $L^2(\R^+\times \mc{T}^2,dp\otimes \mu_{\mc{T}}^2)$.   
 These amplitudes are  obtained from the Segal amplitudes by considering their matrix coefficients on the eigenfunctions $\Psi_{Q+ip,\nu,\tilde{\nu}}$ of the Hamiltonian ${\bf H}$, or more precisely the orthonormal elements $H_{Q+ip,\nu,\tilde{\nu}}$ defined in \eqref{defHzarb}. We will then use a factorisation property proved in Theorem \ref{pantDOZZ} 
 of these amplitudes, which is a consequence of Ward identity (in Section \ref{sec:ward}): this factorisation roughly says that the amplitude $\widehat{\mc{A}}_{\mc{P}_j,g,{\bf x}_j,\boldsymbol{\alpha}_j,\boldsymbol{\zeta}_j}$ can be represented as  tensor products of an  amplitude on $L^2(\R^+\times \mc{T},dp\otimes \mu_{\mc{T}})$ with its complex conjugate. 
 Up to normalising by the DOZZ $3$-point function and an appropriate metric constant, this amplitude $\tilde{\mc{B}}_{\mc{P}_j,J,{\bf x}_j,\boldsymbol{\alpha}_j,\boldsymbol{\zeta}_j}$ is what we call \emph{normalised conformal block amplitude} of  the complex building blocks $\mc{P}_j$ (see Lemma \ref{lem:blockamplitudes}): it contains essentially the matrix coefficients of the Segal amplitude $\mc{A}_{\mc{P}_j,g,{\bf x}_j,\boldsymbol{\alpha}_j,\boldsymbol{\zeta}_j}$ on "half" the eigenfunctions of ${\bf H}$, namely on the family $(H_{Q+ip,\nu,0})_{p,\nu}$ of Virasoro descendants of the 
 "left copy" of the Virasoro algebra. 
It only depends on the complex structure $J$ of $\mc{P}_j$ but not on the metric $g$, moreover it depends holomorphically on the plumbing parameters ${\bf q}=(q_1,\dots,q_{3{\bf g}-3+m})$ when we consider families of surfaces obtained by the plumping procedure. The pairing of these  normalised conformal block amplitudes produces the normalised conformal blocks, which are holomorphic functions of ${\bf q}$ on the open set of the moduli where these coordinates are well-defined. By gathering everything, we obtain in Theorem \ref{th:fullexpressioncorrel}  the conformal bootstrap factorisation formula for the correlation functions.
 
 \emph{Notation:} At this point, to simplify the reading of this section, let us make the following comments about the notations we shall use for indices indexing functions and parameters related to the graph $\mc{G}$, i.e. the topological decomposition of $\Sigma$: if $u$ stands for a generic parameter or function attached to a simple curve $\mc{C}$ embedded in a surface (possibly with boundary), we will denote by $u_i$ the associated function/parameter attached to the curve $\mc{C}_i\subset \Sigma$, by 
${\bf u}=(u_1,\dots,u_L)$ the whole family attached to the curves 
$\mc{C}_1, \dots,\mc{C}_L$, by $u_{jk}$ the fonction/parameter attached to the boundary curve $\pl_k\mc{P}_j$ and by ${\bf u}_j=(u_{j1},\dots, u_{jb_j})$ the family attached to the whole boundary $\pl \mc{P}_j=\pl_1\mc{P}_j\sqcup\dots \sqcup \pl_{b_j}\mc{P}_j$. Below, the role
of $u$ will be played by the parametrisation $\zeta:\T\to \mc{C}$ of the curve, the random field $\tilde{\varphi}\in H^{s}(\T)$ on the curve (using $\zeta$ to view it as a field on $\mc{C}$), the spectral parameter $p$ of ${\bf H}$ acting on $L^2(H^{s}(\T))$ (cf Section \ref{subsubsec:spectral}), the Young diagrams $\nu,\tilde{\nu}$ associated to the spectral decomposition of ${\bf H}$ on $L^2(H^{s}(\T))$, the moduli parameter $q$. The edges $e_i=({\rm v}_1(e_1),{\rm v}_2(e_i))=((j,k),(j',k'))=(v_{jk},v_{j'k'})$   represent  the identification of $\pl_{k}\mc{P}_j$ with $\pl_{k'}\mc{P}_{j'}$, and this will induce an identification $u_{jk}=u_{j'k'}$, i.e. $u_{{\rm v}_1(e_i)}=u_{{\rm v}_2(e_i)}=u_i$. Note that these notations were also used in Section \ref{Subsec:amplitudeandgluing}, where the amplitudes played the role of $\mc{P}_j$ (the vertices) and the variables played the role of $\mc{C}_i$ (the edges).

\subsection{Step 1: Cutting correlations into pairings of Segal amplitudes and 
projections on eigenstates}\label{decompositionLCFT}
 
Let $(\Sigma,J)$ be a genus ${\bf g}$ Riemann surface with $m$ marked points ${\bf x}=\{x_1,\dots,x_m\}$ with respective real weights $\boldsymbol{\alpha}=(\alpha_1,\dots,\alpha_m)$, which we all assume to be strictly less than $Q$. As discussed in Section \ref{Sec:Riemann surfaces with marked points}, $\Sigma$ can be decomposed into a maximal collection of $N:=2{\bf g}-2+m$ complex building blocks $(\mc{P}_j)_{j=1,\dots,N}$ that are pairs of pants, cylinders with $1$ marked point and disks with $2$ marked points. Denote by $b_j$ the number of 
connected components of $\pl \mc{P}_j$.
The set $\bigcup_{j}\pl \mc{P}_j$ consists of $L:=3{\bf g}-3+m$ analytic simple curves $\mc{C}_i$ embedded in $\Sigma$. This decomposition is described by a multigraph $\mc{G}$ where each vertex $v_j$ is associated to $\mc{P}_j$ and each edge $e_i$ is associated to $\mc{C}_i$. Each curve $\mc{C}_i$ can be viewed as a boundary component of two building blocks, $\pl_{k}\mc{P}_j$ and $\pl_{k'}\mc{P}_{j'}$, with  
${\rm v}_1(e_i)=(j,k)=v_{jk}$ and ${\rm v}_2(e_i)=(j',k')=v_{j'k'}$ (recall the maps ${\rm v}_1,{\rm v}_2$ are defined in Definition \ref{defmultigraph}). \\

Let ${\bf x}_j:={\bf x}\cap \mc{P}_j$ and $\boldsymbol{\alpha}_j=(\alpha_{jk})_{k=1,\dots,3-b_j}$ be the collection of weights associated to the marked points in ${\bf x}_j$. We will make the following crucial assumption
\begin{equation}\label{ass:star}
\forall j,\quad  \sum_{k=1}^{3-b_j}\alpha_{jk}-(2-b_j)Q>0.
\end{equation}
Note that the Euler characteristics  of the block $\mc{P}_j$ is $2-b_j$. In view of Lemma \ref{Lem:splitanamplitudeinL2}, this assumption ensures that the amplitude associated to each $\mc{P}_j$ is Hilbert-Schmidt. 
We stress that this is not a technical condition but rather an intrinsic condition: indeed, if this condition fails to hold, further discrete terms in the bootstrap formulae appear (this fact is discussed in physics in \cite{Zamolodchikov96}).

A choice of analytic parametrisations $\zeta_i: \T\to \mc{C}_i$ for all $i$ induces an analytic parametrisation $\zeta_{jk}:\T\to \pl_k\mc{P}_j$ of the boundary circles of $\pl\mc{P}_j$, 
simply given by $\zeta_{jk(i)}:=\zeta_i$ if $k(i)\leq b_j$ is defined by $\mc{C}_i=\pl_{k(i)}\mc{P}_j$. We can make this choice of $\zeta_i$ so that ${\rm v_1}(e_i)<{\rm v}_{2}(e_i)$, i.e. $\zeta_{{\rm v}_1(e_i)}$ is outgoing while $\zeta_{{\rm v}_2(e_i)}$ is incoming, following the convention of Section \ref{subsubsec:graphdec}. Recall that $\sigma_{jk}=1$ (resp. $=-1$) when $v_{jk}={\rm v}_2(e_i)$ for some $i$, i.e. it is incoming  (resp. when $v_{jk}={\rm v}_1(e_i)$, i.e. it is outgoing).

The map $\zeta_i^{-1}$ extends holomorphically in an annular neighborhood $V_i$ of $\mc{C}_i$ into a holomorphic chart $\omega_i:V_i\to \{|z|\in (\delta,\delta^{-1})\}$ for some $\delta<1$. 
Recall that a metric $g$ on $\Sigma$ that is compatible with the complex structure $J$ and such that there is a complex coordinate $z$ near each $\mc{C}_i$ in which $g=|dz|^2/|z|^2$ and $\mc{C}_i=\{|z|=1\}$ is called  admissible with respect to  $(J,\boldsymbol{\zeta})$ where $\boldsymbol{\zeta}=(\zeta_1,\dots,\zeta_L)$.  Note that such a $g$ always exists.\\

Denoting $\boldsymbol{\zeta}_j=(\zeta_{j1},\dots,\zeta_{jb_j}):\T\to \pl_1\mc{P}_j\times\dots\times \pl_{b_j}\mc{P}_j$, 
this allows to define the LCFT amplitude
 $\mc{A}_{\mc{P}_j,g,{\bf x}_j,\boldsymbol{\alpha}_j,\boldsymbol{\zeta}_j}$ of the complex block $\mc{P}_j$. 
Let us set $M:=H^{s}(\T)$ for $s<0$ fixed, and view it as a measure space using the measure $\mu_0=\dd c\otimes \mathbb{P}_\T$. 
By Proposition \ref{glueampli},  the Liouville amplitude is given by 
\begin{equation}\label{decompositionasproduct}
\begin{split} 
\mc{A}_{\Sigma,g,{\bf x},{\bf \alpha}}&=C_{\rm gluing} (\mc{A}_{\mc{P}_1,g,{\bf x}_1,\boldsymbol{\alpha}_1,\boldsymbol{\zeta}_1}\circ\dots\circ \mc{A}_{\mc{P}_N,g,{\bf x}_N,\boldsymbol{\alpha}_N,\boldsymbol{\zeta}_N})_{\mc{G}}\\ 
& =C_{\rm gluing}\int_{M^L}
 \big\cjg \delta_{\mc{G}}(\tilde{\boldsymbol{\varphi}}_1,\dots,\tilde{\boldsymbol{\varphi}}_N),\prod_{j=1}^N\mc{A}_{\mc{P}_j,g,{\bf x}_j,\boldsymbol{\alpha}_j,\boldsymbol{\zeta}_j}(\tilde{\boldsymbol{\varphi}}_j)\big\cjd \, \dd \mu_0^L(\tilde{\boldsymbol{\varphi}})
 \end{split}\end{equation} 
where we have used the notation $\tilde{\boldsymbol{\varphi}}_j=
(\tilde\varphi_{j1},\dots,\tilde\varphi_{jb_j})\in M^{b_j}$ and $\tilde{\boldsymbol{\varphi}}:=(\tilde\varphi_1,\dots,\tilde\varphi_{L})\in M^L$ with $\tilde{\varphi}_{i}:=\tilde{\varphi}_{{\rm v}_1(e_i)}$, and we recall that $\delta_{\mc{G}}$ is defined by \eqref{deltamcG} and the gluing constant is given by  
\begin{equation}\label{Cgluing}
C_{\rm gluing}=\frac{\sqrt{2}}{(\sqrt{2}\pi)^{L-1}}.
\end{equation}
We shall decompose this integral using the spectral decomposition of ${\bf H}$ acting on $L^2(M,\mu_0)$ as described in Corollary \ref{boundednessofPi_V} and \eqref{rewriting}. 

\begin{proposition}\label{prop:decompositionAmplitude}
The following holds true 
\[\begin{split} 
\mc{A}_{\Sigma,g,{\bf x},{\bf \alpha}}& =\frac{C_{\rm gluing}}{(2\pi)^L}\int_{\R_+^L}\int_{\mc{T}^{2L}}\big\cjg \delta_{\mc{G}}(({\bf p}_j,\boldsymbol{\nu}_j,\tilde{\boldsymbol{\nu}}_j)_{j=1}^N),\prod_{j=1}^N
\widehat{\mc{A}}_{\mc{P}_j,g,{\bf x}_j,\boldsymbol{\alpha}_j,\boldsymbol{\zeta}_j}({\bf p}_j,\boldsymbol{\nu}_j,\tilde{\boldsymbol{\nu}}_j)\big\cjd \,\dd{\bf p}\,\dd \mu_{\mc{T}}(\boldsymbol{\nu})
\,\dd\mu_{\mc{T}}(\tilde{\boldsymbol{\nu}})\\
&=\frac{C_{\rm gluing}}{(2\pi)^L} (\widehat{\mc{A}}_{\mc{P}_1,g,{\bf x}_1,\boldsymbol{\alpha}_1,\boldsymbol{\zeta}_1}\circ\dots\circ \widehat{\mc{A}}_{\mc{P}_N,g,{\bf x}_N,\boldsymbol{\alpha}_N,\boldsymbol{\zeta}_N})_{\mc{G}}
\end{split}\]
with the notation 
\begin{equation}\label{notationboldsymbol}
\begin{gathered}
{\bf p}_j=(p_{j1},\dots,p_{jb_j})\in \R_+^{b_j}, \quad \boldsymbol{\nu}_j=(\nu_{j1},\dots,\nu_{jb_j})\in \mc{T}^{b_j} , \quad \tilde{\boldsymbol{\nu}}_j=(\tilde{\nu}_{j1},\dots,\tilde{\nu}_{jb_j})\in \mc{T}^{b_j},\\
 {\bf p}=(p_1\dots,p_L)\in \R_+^L, \quad \boldsymbol{\nu}=(\nu_{1},\dots,\nu_{L})\in \mc{T}^{L},\quad   
\tilde{\boldsymbol{\nu}}=(\tilde{\nu}_{1},\dots,\tilde{\nu}_{L})\in \mc{T}^L,
\end{gathered}
\end{equation}
where $p_i:=p_{{\rm v}_1(e_i)}$, $\nu_i:=\nu_{{\rm v}_1(e_i)}$, $\tilde{\nu}_i:=\tilde{\nu}_{{\rm v}_1(e_i)}$. Here 
 $\widehat{\mc{A}}_{\mc{P}_j,g,{\bf x}_j,\boldsymbol{\alpha}_j,\boldsymbol{\zeta}_j}\in L^2((\R_+\times \mc{T}^2)^{b_j},\dd{\bf p}_j\otimes \dd\mu^{\otimes b_j}_{\mc{T}^2})$ are defined, using the notation of \eqref{defHzarb}  and \eqref{defCandO}, by   
\begin{equation}\label{defhatA}
\widehat{\mc{A}}_{\mc{P}_j,g,{\bf x}_j,\boldsymbol{\alpha}_j,\boldsymbol{\zeta}_j}({\bf p}_j,\boldsymbol{\nu}_j,\tilde{\boldsymbol{\nu}}_j):=
\big\cjg \mc{A}_{\mc{P}_j,g,{\bf x}_j,\boldsymbol{\alpha}_j,\boldsymbol{\zeta}_j},\bigotimes_{k=1}^{b_j}{\bf C}^{(\sigma_{jk}+1)/2}H_{Q+ip_{jk},\nu_{jk},\tilde{\nu}_{jk}}\big\cjd_{\mc{H}}
\end{equation}
and are considered as $L^2$-amplitudes on the measured space $(\hat{M},\hat{\mu}):=(\R_+\times \mc{T}^2,\dd p\otimes \mu_{\mc{T}^2})$ in the variable $(p,\nu,\tilde{\nu})$. 
\end{proposition}
\begin{proof} When the graph has no loop, the result is a direct consequence of Corollary \ref{boundednessofPi_V} and \eqref{rewriting} using that the amplitudes are real valued. 
In the case $\mc{G}$ has loops, we can use Remark 
\ref{remark:selfgluing}: it suffices to use Lemma \ref{Lem:splitanamplitudeinL2} to get that   $\mc{A}_{\mc{P}_j,g,{\bf x}_j,\boldsymbol{\alpha}_j,\boldsymbol{\zeta}_j}$ is a product of Hilbert-Schmidt operators to make the argument   reduce  to the previous case with no loops and we are done. 
\end{proof}

\noindent\textbf{Example in genus 2:} Let us give a concrete example to fix the ideas and simplify the reading. Take $\Sigma$ a genus $2$ surface with no marked point, and $3$ simple curves $\mc{C}_1,\mc{C}_2,\mc{C}_3$ as on figure \ref{surfacegenre2}, with two pants $\mc{P}_1,\mc{P}_2$, where the gluing rules are given by the oriented multigraph $\mc{G}$ drawn in figure \ref{graphegenre2}: $e_1=(v_{11},v_{21})$, $e_2=(v_{22},v_{23})$,  $e_3=(v_{12},v_{13})$.  The expression \eqref{decompositionasproduct} becomes
\[\mc{A}_{\Sigma,g}=\frac{1}{\sqrt{2}\pi^{2}}\int_{M^3}\mc{A}_{\mc{P}_1,g}(\tilde{\varphi}_1,\tilde{\varphi}_3,\tilde{\varphi}_3)\mc{A}_{\mc{P}_2,g}(\tilde{\varphi}_1,\tilde{\varphi}_2,\tilde{\varphi}_2)\, \dd\mu_0^{3}(\tilde{\varphi}_1,\tilde{\varphi}_2,\tilde{\varphi}_3) \]
and Proposition \ref{prop:decompositionAmplitude} becomes, with $\boldsymbol{\nu}=(\nu_1,\nu_2,\nu_3)\in \mc{T}^3$ and $\boldsymbol{\tilde{\nu}}=(\tilde{\nu}_1,\tilde{\nu}_2,\tilde{\nu}_3)\in \mc{T}^3$
\[\mc{A}_{\Sigma,g}=\frac{2^{\frac{3}{2}}}{(2\pi)^{5}}\sum_{\boldsymbol{\nu},\boldsymbol{\tilde{\nu}}\in\mc{T}^3}\int_{\R_+^3}\widehat{\mc{A}}_{\mc{P}_1,g}(p_1,p_3,p_3,\nu_1,\nu_3,\nu_3,\tilde{\nu}_1,\tilde{\nu}_3,\tilde{\nu}_3)\widehat{\mc{A}}_{\mc{P}_2,g}(p_1,p_2,p_2,\nu_1,\nu_2,\nu_2,\tilde{\nu}_1,\tilde{\nu}_2,\tilde{\nu}_2)\dd p_1\dd p_2\dd p_3\]
where 
\begin{align*}
&\widehat{\mc{A}}_{\mc{P}_1,g}(p_1,p_2,p_3,\nu_1,\nu_2,\nu_3,\tilde{\nu}_1,\tilde{\nu}_2,\tilde{\nu}_3)=\\
& \int_{M^3}\mc{A}_{\mc{P}_1,g}(\tilde{\varphi}_1,\tilde{\varphi}_2,\tilde{\varphi}_3)\bbar{H_{Q+ip_1,\nu_1,\tilde{\nu}_1}(\tilde{\varphi}_1)}\bbar{H_{Q+ip_2,\nu_2,\tilde{\nu}_2}(\tilde{\varphi}_2)}H_{Q+ip_3,\nu_3,\tilde{\nu}_3}(\tilde{\varphi}_3)\dd\mu_0^{3}(\tilde{\varphi}_1,\tilde{\varphi}_2,\tilde{\varphi}_3),
\end{align*}
\begin{align*}
&\widehat{\mc{A}}_{\mc{P}_2,g}(p_1,p_2,p_3,\nu_1,\nu_2,\nu_3,\tilde{\nu}_1,\tilde{\nu}_2,\tilde{\nu}_3)=\\
& \int_{M^3}\mc{A}_{\mc{P}_2,g}(\tilde{\varphi}_1,\tilde{\varphi}_2,\tilde{\varphi}_3)H_{Q+ip_1,\nu_1,\tilde{\nu}_1}(\tilde{\varphi}_1)\bbar{H_{Q+ip_2,\nu_2,\tilde{\nu}_2}(\tilde{\varphi}_2)}H_{Q+ip_3,\nu_3,\tilde{\nu}_3}(\tilde{\varphi}_3)\dd\mu_0^{3}(\tilde{\varphi}_1,\tilde{\varphi}_2,\tilde{\varphi}_3).
\end{align*}

\begin{figure}
\centering
\begin{tikzpicture}
\node[inner sep=10pt] (F1) at (-5.2,0){\includegraphics[width=.4\textwidth]{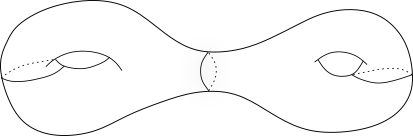}
\hspace{0.5cm}
\includegraphics[width=.5\textwidth]{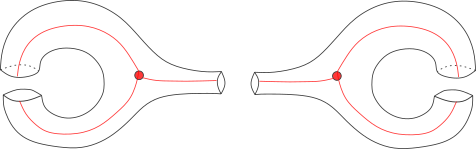}};
\node (F) at (-10.5,-0.2){$\mc{P}_1$};
\node (F) at (-8.3,-0.2){$\mc{P}_2$};
\node (F) at (-9.3,-0.2){$\mc{C}_1$};
\node (F) at (-6.8,0){$\mc{C}_2$};
\node (F) at (-12.5,0){$\mc{C}_3$};
\node (F) at (-4.5,0.2){$\pl_2\mc{P}_1$};
\node (F) at (-4.5,-0.4){$\pl_3\mc{P}_1$};
\node (F) at (-2.2,-0.6){$\pl_1\mc{P}_1$};
\node (F) at (-1,-0.6){$\pl_1\mc{P}_2$};
\node (F) at (1.4,0.2){$\pl_2\mc{P}_2$};
\node (F) at (1.4,-0.4){$\pl_3\mc{P}_2$};
\end{tikzpicture}
  \caption{Surface of genus $2$ with no marked points decomposed into $2$ pairs of pants}\label{surfacegenre2}
\end{figure}
\begin{figure}
\centering
\begin{tikzpicture}
\node[inner sep=10pt] (F1) at (0,0){\includegraphics[width=.4\textwidth]{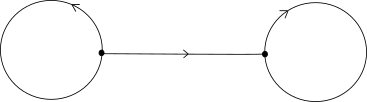}};
\node (F) at (-1.2,-0.3){$v_1$};
\node (F) at (1.3,-0.3){$v_2$};
\node (F) at (0,0.3){$e_1$};
\node (F) at (2.3,0.65){$e_2$};
\node (F) at (-2.3,0.65){$e_3$};
\end{tikzpicture}
\caption{Graph $\mc{G}$ representing the decomposition. The vertex $v_j=\{v_{j1},v_{j2},v_{j3}\}$ is associated to $\mc{P}_j$, with $v_{jk}$  associated to $\pl_{k}\mc{P}_j$. The edges are $e_1=(v_{11},v_{21})$, $e_2=(v_{22},v_{23})$,  $e_3=(v_{12},v_{13})$ and correspond to $\mc{C}_1,\mc{C}_2$, $\mc{C}_3$ respectively, with orientation.}
\label{graphegenre2}
\end{figure}

\subsection{Step 2: Introduction of normalised conformal block amplitudes}\label{sub:cfb}

In this section we introduce the normalised conformal blocks amplitudes for complex building blocks. As we shall see, the conformal blocks fit in the language of amplitudes explained in Section \ref{section:traces} and will be the basic components from which one constructs the conformal blocks. 

Consider the complex building blocks $(\mc{P}_j,J_j:=J|_{\mc{P}_j})$ of the decomposition of $(\Sigma,J)$ using the graph $\mc{G}$ as before, with cutting curves $\mc{C}_i$ and their analytic parametrisations $\zeta_i:\T\to \mc{C}_i$.
We still use the analytic parametrisation $\zeta_{jk}:\T\to \pl_k\mc{P}_j$ of the boundary circles as defined in Section \ref{decompositionLCFT}.
As explained in Section \ref{subsubsec:plumbingparameters}, these parametrisations allow us to glue a disk 
$\mc{D}_{jk}(\delta)=\D_{\delta}$ to $\mc{P}_j$ at $\pl_k\mc{P}_j$ for each $k\leq b_j$, where $\delta<1$ is a parameter chosen in Section \ref{subsubsec:plumbingparameters}: the obtained Riemann surface $(\hat{\mc{P}}_j,\hat{J}_j)$ is a sphere with $3$ marked points given by the union of the points in ${\bf x}_j$ and the centers $x_{jk}$ of the glued discs $\mc{D}_{jk}(\delta)$. We define 
${\bf x}'_j:=(x_{jk})_{k=1,\dots,b_j}$.  
There is a biholomorphism $\psi_j:\hat{\mc{P}}_j\to \hat{\C}$  with $\psi_j^{-1}(-1/2)$, $\psi_j^{-1}(1/2)$, $\psi_j^{-1}(i\sqrt{3}/2)$ being the $3$ marked points of $\hat{\mc{P}}_j$: $\psi_j(x_{j1})=-1/2$, $\psi_j(x_{j2})=1/2$ and $\psi_j(x_{j3})=i\sqrt{3}/2$.  
Then $\Sigma_j:=\psi_j(\mc{P}_j)\subset \hat{\C}$ is a domain with $b_j$ analytic boundary curves 
$\psi_j(\pl_k\mc{P}_j)$ for $k=1,\dots,b_j$. 
If $g$ is an admissible metric for $(J,\boldsymbol{\zeta})$, the metric $\hat{g}_j:={\psi_j}_*g$ on $\Sigma_j$ can be written as $\hat{g}_j=e^{h_j}g_0$ for some $h_j\in C^\infty(\mathbb{S}^2)$, where $g_0$ is the round sphere metric on $\mathbb{S}^2=\hat{\C}$. The next Lemma defines the notion of normalised block amplitude and is motivated by the holomorphic factorisation property in Theorem \ref{pantDOZZ} that we shall prove later.
\begin{lemma}[\textbf{Normalised conformal block amplitude}]\label{lem:blockamplitudes}
Let $\mc{P}_j$ be the complex blocks of the decomposition of $(\Sigma,J)$ represented by a multigraph $\mc{G}$ and some curves $\mc{C}_i$ with analytic parametrisations $\zeta_i:\T\to \mc{C}_i$. Let 
$\boldsymbol{\zeta}_j:\T\to \pl_1\mc{P}_j\times \dots\times \pl_{b_j}\mc{P}_j$ be the induced analytic parametrisations and let $g$ be an admissible metric for $(J,\boldsymbol{\zeta})$.  Assume \eqref{ass:star} holds. 
The normalised conformal block amplitude of $(\mc{P}_j,J,{\bf x}_j,\boldsymbol{\alpha}_j,\boldsymbol{\zeta}_j)$, defined as the 
$L^2(\mc{T}^{b_j},\mu_{\mc{T}}^{b_j})\simeq \mc{H}_{\mc{T}}^{\otimes b_j}$ valued measurable function of ${\bf p}_j=(p_{j1},\dots,p_{jb_j})\in \R_+^{b_j}$ by
\[ \tilde{\mc{B}}_{\mc{P}_j,J,{\bf x}_j,\boldsymbol{\alpha}_j,\boldsymbol{\zeta}_j}({\bf p}_j,\boldsymbol{\nu}_j):=
\frac{\big\cjg \mc{A}_{\mc{P}_j,g,{\bf x}_j,\boldsymbol{\alpha}_j,\boldsymbol{\zeta}_j}, \bigotimes_{k=1}^{b_j}{\bf C}^{(\sigma_{jk}+1)/2}H_{Q+ip_{jk},\nu_{jk},0}\big\cjd}{C(\mc{P}_j,g, \boldsymbol{\alpha}_j, {\bf p}_j){\bf C}^{\rm DOZZ}_{\gamma,\mu}({\boldsymbol{\alpha}_j},{\bf p}_j,\boldsymbol{\sigma}_j)} \]
is such that 
\[ \int_{\R^{b_j}}\int_{\mc{T}}  C_{\mc{P}_j,g}(\boldsymbol{\alpha}_j, {\bf p}_j)^2|{\bf C}^{\rm DOZZ}_{\gamma,\mu}({\boldsymbol{\alpha}_j},{\bf p}_j,\boldsymbol{\sigma}_j)|^2 |\tilde{\mc{B}}_{\mc{P}_j,J,{\bf x}_j,\boldsymbol{\alpha}_j,\boldsymbol{\zeta}_j}({\bf p}_j,\boldsymbol{\nu}_j)|^2 d\mu_{\mc{T}}(\boldsymbol{\nu}_j)\dd {\bf p_j}<\infty\] 
where ${\bf C}^{\rm DOZZ}_{\gamma,\mu}({\boldsymbol{\alpha}_j},{\bf p}_j,\boldsymbol{\sigma}_j)$ is the DOZZ structure constant given by
\begin{equation}\label{dozzblock}
{\bf C}^{\rm DOZZ}_{\gamma,\mu}({\boldsymbol{\alpha}_j},{\bf p}_j,\boldsymbol{\sigma}_j):=\left\{
\begin{array}{ll}
C^{\rm DOZZ}_{\gamma,\mu}(Q+i\sigma_{j1}p_{j1},\alpha_{j1},\alpha_{j2}), & b_j=1\\[0.1cm]
C^{\rm DOZZ}_{\gamma,\mu}(Q+i\sigma_{j1}p_{j1},Q+i\sigma_{j2}p_{j2},\alpha_{j1}), & b_j=2\\[0.1cm]
C^{\rm DOZZ}_{\gamma,\mu}(Q+i\sigma_{j1}p_{j1},Q+i\sigma_{j2}p_{j2},Q+i\sigma_{j3}p_{j3}), & b_j=3
\end{array}\right.
\end{equation}
with
$\boldsymbol{\sigma}_j:=(\sigma_{j1},\dots,\sigma_{jb_j})$ are the orientation signs of each boundary of $\pl\mc{P}_j$
 \begin{equation}\label{metricconstantanomaly}
C_{\mc{P}_j,g}(\boldsymbol{\alpha}_j, {\bf p}_j):=\left\{
\begin{array}{ll}
C(\mc{P}_j,g, \Delta_{Q+ip_{j1}},\Delta_{\alpha_{j2}},\Delta_{\alpha_{j3}}), & b_j=1\\[0.1cm]
C(\mc{P}_j,g,\Delta_{Q+ip_{j1}},\Delta_{Q+ip_{j2}},\Delta_{\alpha_{j1}}), & b_j=2\\[0.1cm]
C(\mc{P}_j,g,\Delta_{Q+ip_{j1}},\Delta_{Q+ip_{j2}},\Delta_{Q+ip_{j3}}), & b_j=3
\end{array}\right.
\end{equation}
where  $C(\mc{P}_j,g, \cdot)$ is the constant \eqref{metricconstant}.
\end{lemma}
\begin{proof}
For $\alpha_{jk}$ satisfying \eqref{ass:star}  and $p_{jk}\in \R_+$, the DOZZ constant ${\bf C}^{\rm DOZZ}_{\gamma,\mu}({\boldsymbol{\alpha}_j},{\bf p}_j,{\boldsymbol{\sigma}_j})$ is never vanishing. By Corollary 
\ref{boundednessofPi_V}, we then obtain the announced property of $\tilde{\mc{B}}_{\mc{P}_j,J,{\bf x}_j,\boldsymbol{\alpha}_j,\boldsymbol{\zeta}_j}$  since $\mc{A}_{\mc{P}_j,g,{\bf x}_j,\boldsymbol{\alpha}_j,\boldsymbol{\zeta}_j}\in e^{\eps \bar{c}_-}L^2(M^{b_j})$ for some small $\eps>0$ by Theorem \ref{integrcf}. 
\end{proof}
 
The normalised block amplitude depends on the complex structure $J_j=J|_{\mc{P}_j}$ but does not depend on $g$ since the Weyl anomaly of $\mc{A}_{\mc{P}_j,g,{\bf x}_j,\boldsymbol{\alpha}_j,\boldsymbol{\zeta}_j}$ cancels out with that of $C(\mc{P}_j,g, \boldsymbol{\Delta}_{\boldsymbol{\alpha}_j, {\bf p}_j})$. This amplitude is \emph{normalised} in the sense that (as a consequence of Theorem \ref{pantDOZZ})
\[ \tilde{\mc{B}}_{\mc{P}_j,J,{\bf x}_j,\boldsymbol{\alpha}_j,\boldsymbol{\zeta}_j}({\bf p}_j,0,\dots,0)=1.\]

\subsection{Step 3: Holomorphic factorisation of Segal amplitudes for complex building blocks}
Next, we describe the holomorphic dependence of the conformal block amplitudes  with respect to the moduli space parameters, namely the plumbing parameters, and at the same time the holomorphic factorisation of the Segal amplitude of $\mc{P}_j$ proved in Section \ref{sec:computingamplitudes} (Theorem \ref{pantDOZZ}).
Indeed, we shall apply later Proposition \ref{prop:decompositionAmplitude} and Lemma \ref{lem:blockamplitudes} in the case of a holomorphic families of Riemann surfaces $(\Sigma,J({\bf q}))$  depending on a complex parameter ${\bf q}\in \D^{L}$, and the core of the conformal bootstrap formula is to extract the ``holomorphic part" of the correlation functions with respect to ${\bf q}$ from Theorem \ref{pantDOZZ}. 

We use the plumbing construction on the  moduli space described in Sections \ref{plumbings} and \ref{subsubsec:plumbingparameters} to define the family $(\Sigma,J({\bf q}))$, as we now explain. 
By Section \ref{subsubsec:plumbingparameters}, there are disks $\mc{D}_{jk}\subset \hat{\mc{P}}_j$ containing $\mc{D}_{jk}(\delta)$ and some 
biholomorphisms $\omega_{jk}:\mc{D}_{jk}\subset \hat{\mc{P}}_j\to \D$ induced by analytic extension of the inverse of the maps $z\in\{|z|=\delta\}\mapsto \zeta_{jk}(z/\delta)$ . We then obtain biholomorphisms
\[\psi_{jk}:= \psi_j\circ \omega_{jk}^{-1}: \D\to \psi_{jk}(\D)\subset \hat{\C},\quad  \textrm{ with }\Sigma_j=\hat{\C}\setminus \bigcup_{k=1}^{b_j}\psi_{jk}(\D_\delta)\]
and denote 
\begin{equation}\label{bfz_0} 
{\bf z}_0=(\psi_{j1}(0),\psi_{j2}(0),\psi_{j3}(0))=(-1/2,1/2,i\sqrt{3}/2).
\end{equation}
Now choose a parameter $q_i\in \bbar{\D}$ for each $i=1,\dots,L$, and 
for $i$ so that $\pl_{k}\mc{P}_j=\mc{C}_i$ denote
\[q_{jk}=
\left\{\begin{array}{ll}
q_i, & \zeta_{jk} \textrm{ incoming }\\
1, & \zeta_{jk} \textrm{ outgoing }
\end{array}\right.,\] 
and finally we set ${\bf q}_j:=(q_{j1},\dots,q_{jb_j})$ and ${\bf q}:=(q_{1},\dots,q_{L})$.

We define $\mc{P}_j({\bf q}_j)\subset \hat{\mc{P}}_j$ and  $\Sigma_j({\bf q}_j)\subset \hat{\C}$ 
the Riemann surfaces 
\[\Sigma_j({\bf q}_j):=\hat{\C}\setminus \bigcup_{k=1}^{b_j}\psi_{jk}\Big(\D_{|q_{jk}|}\Big), \quad 
\mc{P}_j({\bf q}_j):=\psi_j^{-1}(\Sigma_j({\bf q}_j))\]
where $\Sigma_j({\bf q}_j)$ is equipped with the canonical complex structure of $\hat{\C}$ and 
$\mc{P}_j({\bf q}_j)$ is equipped with the induced complex structure $\hat{J}_j$ of $\hat{\mc{P}}_j$.
The plumbing construction of Section \ref{subsubsec:plumbingparameters} produces a family of Riemann surfaces $(\Sigma({\bf q}),J({\bf q}))$
obtained by gluing the complex blocks $\mc{P}_j({\bf q}_j)$ along their boundary according to the rules given by the graph $\mc{G}$, using the following analytic parametrisations of the boundary circles of $\mc{P}_j({\bf q}_j)$  
\[ \zeta_{jk}^{{\bf q}}(z):=\left\{\begin{array}{ll}
\omega_{jk}^{-1}\Big(\frac{q_{jk}}{z}\Big), & \pl_k\mc{P}_j \textrm{ is incoming}\\
\omega_{jk}^{-1}(z), & \pl_k\mc{P}_j \textrm{ is outgoing}.
\end{array}\right.\]
This corresponds exactly to performing the plumbing construction with the equation
$\omega_{jk}(x)\omega_{j'k'}(x')=q_i$ for $x\in \omega_{jk}^{-1}(\mathbb{A}_{|q_{i}|})$, $x'\in \omega_{j'k'}^{-1}(\mathbb{A}_{|q_{i}|})$  if $\pl_k\mc{P}_j=\pl_{k'}\mc{P}_{j'}=\mc{C}_i$ in $\Sigma$: if we use the generic notation ${\bf 1}$
 for the vectors $(1,\dots,1)$, the surface $(\Sigma({\bf q}),J({\bf q}))$ corresponds to cutting $(\Sigma({\bf 1}),J({\bf 1}))$ into the pairs of pants $\mc{P}_j({\bf 1})$ and gluing a flat annulus of modulus $|q_i|$ with a twist angle $q_i/|q_i|$ between $\pl_{k}\mc{P}_j({\bf 1})$ and $\pl_{k'}\mc{P}_{j'}({\bf 1})$ if 
 $\mc{C}_i=\pl_{k}\mc{P}_j=\pl_{k'}\mc{P}_{j'}$.
In particular for ${\bf q}_{\delta}:=(\delta^2,\dots,\delta^2)$ (i.e. $q_i=\delta^2$ for each $i$), 
the obtained surface $(\Sigma({\bf q}_{\delta}),J({\bf q}_{\delta}))$ is just our original surface $(\Sigma,J)$, and in terms of $(\Sigma,J)$, $(\Sigma({\bf q}),J({\bf q}))$ can be viewed as gluing an annulus of modulus $|q_i|/\delta^2$ between $\pl_k\mc{P}_j$ and $\pl_{k'}\mc{P}_{j'}$ with twist angle $q_i/|q_i|$, if $|q_i|\leq \delta^2$, while if $|q_i|>\delta^2$ it amounts to remove an annulus of modulus $\delta^2/q_i$ around $\mc{C}_i$, again with a twist angle $q_i/|q_i|$, and glue back.
 To describe the family of Riemann surfaces $(\Sigma({\bf q}),J({\bf q}))$ as elements in moduli space $\mc{M}_{{\bf g},m}$, it is more convenient to pull $(\Sigma({\bf q}),J({\bf q}))$ back on the fixed surface $\Sigma$ by a family of diffeomorphisms $\psi_{\bf q}:\Sigma\to \Sigma({\bf q})$. Since all $\mc{P}_j({\bf q}_j)$ are subsets of $\hat{\mc{P}}_j$, the marked points ${\bf x}_j$ can all be seen as points on $\Sigma({\bf q})$. 
By the plumbing construction, since for each ${\bf q},{\bf q}'$, the blocks $\mc{P}_j({\bf q}_j)$ and $\mc{P}_j({\bf q}'_j)$ are diffeomorphic with diffeomorphisms isotopic to the identity 
and equal to the identity near ${\bf x}_j$, $\psi_{\bf q}:\Sigma\to \Sigma({\bf q})$ can be constructed to be fixing the marked points and so that $\psi_{\bf q}$ are all isotopic to identity. 
The Teichm\"uller class of $(\Sigma({\bf q}),J({\bf q}))$ can then be represented by a holomorphic family $\Phi_{\mc{G},J,\boldsymbol{\zeta}}({\bf q})=(\Sigma,J'({\bf q}))$ of complex structures $J'({\bf q}):=\psi_{{\bf q}}^*J({\bf q})\in \mc{T}_{{\bf g},m}$ on the fixed surface $\Sigma$.  
Denote by $(\mc{C}_i({\bf q}))_i\subset \Sigma({\bf q})$ the family of simple curves bounding the complex blocks $\mc{P}_j({\bf q}_j)$, with $\mc{C}_i({\bf q})$ in the free homotopy class of $\mc{C}_i$, by $\boldsymbol{\zeta}_j^{\bf q}=(\zeta^{{\bf q}}_{j1},\dots,\zeta^{\bf q}_{jb_j})$ the associated parametrisation of $\pl\mc{P}_j({\bf q}_j)$ and by $\boldsymbol{\zeta}^{\bf q}$ the induced parametrisation of the curves $\mc{C}_i({\bf q})$.
By using the metric $(\omega_{jk}^{-1})^*(|dz|^2/|z|^2)$ on $\mc{P}_j({\bf q}_j)\cap \mc{D}_{jk}$
we obtain a well-defined family of smooth metrics $g_{\bf q}$ on $\Sigma({\bf q})$ as the metric $|dz|^2/|z|^2$ is invariant by the transformation $z\mapsto q/z$ for each $q\in \C^*$, and that it is compatible with the complex structure $J({\bf q})$ of $\Sigma({\bf q})$ by the definition of the plumbing construction. Moreover it is an admissible 
metric for $(J({\bf q}),\boldsymbol{\zeta}^{\bf q})$. We shall write $g'_{\bf q}:=\psi_{{\bf q}}^*g_{\bf q}.$\\

We next write the holomorphic factorisation that will be proved in Theorem \ref{pantDOZZ} in Section \ref{sec:computingamplitudes} (and which follows from Ward identity in Section \ref{sec:ward}), 
and the dependence of the normalised conformal block amplitudes with respect to the plumbing parameters, contained in
Corollaries \ref{plumbedpant} and \ref{annulusDOZZ}. Since $\mc{P}_j({\bf q}_j)$ is obtained from $\mc{P}_j$ by gluing (resp. removing) an annulus $\mathbb{A}_{|q_{jk}|/\delta^2}$ (resp. $\mathbb{A}_{\delta^2/|q_{jk}|}$) to $\pl_{k}\mc{P}_j$ if $|q_{jk}|\leq \delta^2$ (resp. $|q_{jk}|\geq \delta^2$) with a parametrisation 
 of $\pl_k\mc{P}_j({\bf q}_j)$ given by $\zeta_{jk}^{{\bf q}}(z)=\zeta_{jk}(q_{jk}z/|q_{jk}|)$, we obtain that  
(using notation \eqref{notationboldsymbol}) there are complex valued functions 
\begin{equation}\label{defomeganu} 
(\boldsymbol{\nu}_j,{\bf p}_j,\boldsymbol{\alpha}_j)\mapsto w_{\boldsymbol{\nu}_j}^{\mc{P}_j}(\boldsymbol{\alpha}_j,{\bf p}_j)\in \C 
\end{equation} 
given by the expressions (see Theorem \ref{pantDOZZ} for the definition of $\omega_{\mc{P}}(\boldsymbol{\Delta_{\alpha}},\boldsymbol{\nu},{\bf z})$) 
\[ w_{\boldsymbol{\nu}_j}^{{\mc{P}_j}}(\boldsymbol{\alpha}_j,{\bf p}_j):=\left\{\begin{array}{ll}
w_{\mc{P}_j}(\Delta_{Q+ip_{j1}},\Delta_{\alpha_{j1}},\Delta_{\alpha_{j2}},\boldsymbol{\nu}_j,{\bf z}_0) & \textrm{ if }b_j=1\\
w_{\mc{P}_j}(\Delta_{Q+ip_{j1}},\Delta_{Q+ip_{j2}},\Delta_{\alpha_{j1}},\boldsymbol{\nu}_j,{\bf z}_0) & \textrm{ if }b_j=2\\
w_{\mc{P}_j}(\Delta_{Q+ip_{j1}},\Delta_{Q+ip_{j2}},\Delta_{Q+ip_{j3}},\boldsymbol{\nu}_j,{\bf z}_0)) & \textrm{ if }b_j=3
\end{array}\right.
\]
such that 
\begin{equation} \begin{split}
\label{wardresume}
  & \big\cjg \mc{A}_{\mc{P}_j({\bf q}_j),g_{\bf q},{\bf x}_j,\boldsymbol{\alpha}_j,\boldsymbol{\zeta}_j},\bigotimes_{k=1}^{b_j}{\bf C}^{(\sigma_{jk}+1)/2}\Psi_{Q+ip_{jk},\nu_{jk},\tilde{\nu}_{jk}}\big\cjd_{\mc{H}} \\
 & = C_{\mc{P}_j,g}(\boldsymbol{\alpha}_j,{\bf p}_j){\bf C}^{\rm DOZZ}_{\gamma,\mu}(\boldsymbol{\alpha}_j,{\bf p}_j,\boldsymbol{\sigma}_j)  w_{\boldsymbol{\nu}_j}^{\mc{P}_j}( \boldsymbol{\alpha}_j,{\bf p}_j)\,\bbar{w_{\tilde{\boldsymbol{\nu}}_j}^{\mc{P}_j}(\boldsymbol{\alpha}_j,{\bf p}_j)}  \prod_{k=1}^{b_j}
\big(\frac{|q_{jk}|}{\delta^2}\big)^{-\frac{c_{\rm L}}{12}+2\Delta_{Q+ip_{jk}}}\big(\frac{q_{jk}}{\delta^2}\big)^{|\nu_{jk}|}\big(\frac{\bar{q}_{jk}}{\delta^2}\big)^{|\tilde{\nu}_{jk}|}.
\end{split}\end{equation}
Corollaries \ref{plumbedpant} and \ref{annulusDOZZ} applied with no Young diagrams also yield that the metric anomaly constant  \eqref{metricconstantanomaly} satisfies 
\begin{equation}\label{metricanomalyconstant}
C_{\mc{P}_j({\bf q}_j),g_{\bf q}}(\boldsymbol{\alpha}_j,{\bf p}_j)=C_{\mc{P}_j,g}(\boldsymbol{\alpha}_j,{\bf p}_j) \prod_{k=1}^{b_j}
\big(\frac{|q_{jk}|}{\delta^2}\big)^{-\frac{c_{\rm L}}{12}+2\Delta_{Q+ip_{jk}}}.\end{equation}
For ${\bf n}_j=(n_{j1},\dots,n_{jb_j})\in\N^{b_j}$, we define ${\bf w}_{{\bf n}_j}^{\mc{P}_j}(\boldsymbol{\alpha}_j,{\bf p}_j)\in \bigotimes_{k=1}^{b_j}\C^{d_{n_{jk}}}$ 
to be the vector with components
\begin{equation}\label{defwnj} 
{\bf w}_{{\bf n}_j}^{\mc{P}_j}(\boldsymbol{\alpha}_j,{\bf p}_j):=
(w_{\boldsymbol{\nu}_j}^{\mc{P}_j}(\boldsymbol{\alpha}_j,{\bf p}_j))_{\boldsymbol{\nu}_j\in \mc{T}_{{\bf n}_j}}
\end{equation} 
where $\mc{T}_{{\bf n}_j}:=\mc{T}_{n_{j1}}\times\dots\times \mc{T}_{n_{jb_j}}$ and recall (below \eqref{decofT}) that $d_{n}=|\mc{T}_n|$.
Next, recall from Section \ref{subsubsec:spectral} that for $n\in\N$ and $p\in\R_+$,  
the Schapovalov form ${\bf F}_{Q+ip,n}$ is a positive definite quadratic form on $\C^{d_n}$, with matrix coefficients denoted by ${\bf F}_{Q+ip,n}:=(F_{Q+ip}(\nu,\nu'))_{\nu,\nu'\in\mc{T}_n}$. 
As a consequence of \eqref{wardresume}, \eqref{metricanomalyconstant}, Lemma \ref{lem:blockamplitudes} and the definition of $\Pi_{\mc{V}_{p_{jk}}}$, we obtain
\[ \tilde{\mc{B}}_{\mc{P}_j({\bf q}_j),\hat{J}_j,{\bf x}_j,\boldsymbol{\alpha}_j,\boldsymbol{\zeta}_j^{\bf q}}({\bf p}_j)=\sum_{{\bf n}_j\in\N^{b_j}}\prod_{k=1}^{b_j}\big(\frac{q_{jk}}{\delta}\big)^{n_{jk}}\Big(\bigotimes_{k=1}^{b_j}{\bf C}^{(\sigma_{jk}+1)/2}\bbar{{\bf F}^{-1/2}_{Q+ip_{jk},n_{jk}}}\Big){\bf w}_{{\bf n}_j}^{\mc{P}_j}(\boldsymbol{\alpha}_j,{\bf p}_j).
\]
It is a holomorphic function of the variable ${\bf q}_j$ for  ${\bf q}_j\in \D^{b_j}$ which satisfies 
\[ \lim_{{\bf q}_j\to 0} \tilde{\mc{B}}_{\mc{P}_j({\bf q}_j),J,{\bf x}_j,\boldsymbol{\alpha}_j,\boldsymbol{\zeta}_j^{\bf q}}({\bf p}_j)(\nu_1,\dots,\nu_{b_j})=\prod_{k=1}^{b_j}\delta_0(\nu_k)\]
if $\delta$ is the dirac mass at $0$.
\subsection{Step 4: Definition of the normalised conformal blocks}
We are now in position to define the \emph{normalised conformal block} of the family $(\Sigma({\bf q}),g_{\bf q},{\bf x},\boldsymbol{\alpha},\boldsymbol{\zeta}^{\bf q})$.
\begin{definition}[\textbf{Normalised conformal block}]\label{def:conformalblock}
The normalised conformal block of $(\Sigma({\bf q}),g_{\bf q},{\bf x},\boldsymbol{\alpha},\boldsymbol{\zeta}^{\bf q})$
is defined for almost all ${\bf p}\in \R_+^{L}$ as the function of ${\bf q}\in \D^{3{\bf g}-3+m}$ 
\[\tilde{\mc{F}}_{\bf p}(\boldsymbol{\alpha},{\bf q}):=
\int_{\mc{T}^{L}}\big\cjg \delta_{\mc{G}}(\boldsymbol{\nu}_1,\dots,\boldsymbol{\nu}_L)\delta_{\mc{G}}({\bf p}_1,\dots,{\bf p}_L),\prod_{j=1}^N
\tilde{\mc{B}}_{\mc{P}_j({\bf q}_j),g_{\bf q},\boldsymbol{z}_j,\boldsymbol{\alpha}_j,\boldsymbol{\zeta}_j^{\bf q}}({\bf p}_j,\boldsymbol{\nu}_j)\big\cjd \dd\mu^L_{\mc{T}}(\boldsymbol{\nu})\]
where $\delta_{\mc{G}}$ is defined by \eqref{deltamcG}, and we use the notations \eqref{notationboldsymbol}. 
\end{definition}
This holomorphic function is called ``normalised'' since $\tilde{\mc{F}}_{\bf p}(\boldsymbol{\alpha},0)=1$. It is sometimes called conformal block in the physics literature, and when thought of as locally defined on the moduli space, it contains the whole algebraic structure arising from the Virasoro algebra via the Ward coefficients \eqref{defwnj}. On the other hand, there is a notion of conformal blocks that are viewed as holomorphic sections of a holomorphic line bundle over $\mc{T}_{{\bf g},m}$ satisfying 
another type of Ward identity. In the physics treatment of Liouville CFT, they appear in the work of Teschner \cite[Section 12]{Teschner_Vartanov}. We do not discuss further these blocks here, and refer rather to the forthcoming work \cite{BGKR2} (see also the review \cite{guillarmou2024}) where this global holomorphic structure on $\mc{T}_{{\bf g},m}$ as sections of a holomorphic line bundle is treated in details. They are obtained by multiplying the normalised block amplitudes by some appropriate holomorphic functions of the moduli parameters on Teichm\"uller space (rather than moduli space), 
namely the ``holomorphic part of the conformal radii" appearing in the constant \eqref{metricconstant}.\\ 

\noindent\textbf{Example in genus 2:} in our example of Figure \ref{surfacegenre2}, the normalised block amplitudes are (recall $({\bf x},\boldsymbol{\alpha})=\emptyset$ since there is no marked point)
\[
\tilde{\mc{B}}_{\mc{P}_1({\bf q}_1),J({\bf q}),\boldsymbol{\zeta}_1^{\bf q}}(p_1,p_2,p_3)=\sum_{n_1,n_2,n_3\in\N}\delta^{-\sum_{k=1}^3n_k}q_3^{n_3}\bbar{{\bf F}^{-1/2}_{Q+ip_1,n_1}}\bbar{{\bf F}^{-1/2}_{Q+ip_2,n_2}}{\bf F}^{-1/2}_{Q+ip_3,n_3}{\bf w}_{n_1,n_2,n_3}^{\mc{P}_1}(p_1,p_2,p_3)
\]
\[
 \tilde{\mc{B}}_{\mc{P}_2({\bf q}_2),J({\bf q}),\boldsymbol{\zeta}_2^{\bf q}}(p_1,p_2,p_3)=   \sum_{n_1,n_2,n_3\in\N}q_1^{n_1}q_2^{n_2}\delta^{-\sum_{k=1}^3n_k}{\bf F}^{-1/2}_{Q+ip_1,n_1}{\bf F}^{-1/2}_{Q+ip_2,n_2}\bbar{{\bf F}^{-1/2}_{Q+ip_3,n_3}}
{\bf w}_{n_1,n_2,n_3}^{\mc{P}_2}(p_1,p_2,p_3)
\]
and the modified conformal block is given by  
\[\tilde{\mc{F}}_{(p_1,p_2,p_3)}(q_1,q_2,q_3)= \sum_{\substack{\nu_1,\nu_2,\nu_3\in\mc{T}^3\\ \nu'_1,\nu'_2,\nu'_3\in\mc{T}^3}}\Big(\prod_{k=1}^3
\big(\frac{q_k}{\delta^2}\big)^{|\nu_k|}F^{-1}_{Q+ip_k}(\nu_k,\nu_k')\Big) w_{\nu_1',\nu_3',\nu_3}^{\mc{P}_2}(p_1,p_3,p_3)w_{\nu_1,\nu_2',\nu_2}^{\mc{P}_1}(p_1,p_2,p_2).\]
\subsection{Step 5: Conformal bootstrap formula}
We can end this section with our main theorem expressing the correlation functions in terms of normalised conformal blocks and DOZZ constants, in the plumbing coordinates associated to a family of $L=3{\bf g}-3+m$ parametrised cutting curves $\zeta_j: \T\to \mc{C}_j$.  Below, we write ${\bf q}_j={\bf 1}$ when $q_{jk}=1$ for all $k=1,\dots,b_j$.
\begin{theorem}[\textbf{Conformal bootstrap in plumbing coordinates}]\label{th:fullexpressioncorrel} 
Assume \eqref{ass:star} and let $L:=3{\bf g}-3+m$ and $N=2{\bf g}-2+m$ where ${\bf g}$ is the genus of $\Sigma$ and $m$ the number of marked points. The correlation function of the family of surfaces $(\Sigma({\bf q}),g_{\bf q})$ with $m$ 
insertions $({\bf x},\boldsymbol{\alpha})$ is given by the expression
\begin{equation}\label{theoremmodifiedconfblock}
\mc{A}_{\Sigma({\bf q}),g_{\bf q},{\bf x},\boldsymbol{\alpha}} = \frac{2^{\frac{L}{2}}\prod_{j=1}^{L} |q_j|^{\frac{-c_{\rm L}}{12}} }{(2\pi)^{2L-1}}\int_{\R_+^L}C_{\Sigma,g}(\boldsymbol{\alpha},{\bf p})\boldsymbol{\rho}(\boldsymbol{\alpha},{\bf p})\prod_{j=1}^{L} |q_j|^{2\Delta_{Q+ip_j}} 
\Big|\tilde{\mc{F}}_{\bf p}(\boldsymbol{\alpha},{\bf q})\Big|^2\dd{\bf p},
\end{equation}
\[\textrm{with }\,\,  \boldsymbol{\rho}(\boldsymbol{\alpha},{\bf p}) :=  \prod_{j=1}^N{\bf C}^{\rm DOZZ}_{\gamma,\mu}(\boldsymbol{\alpha}_j,{\bf p}_j,\boldsymbol{\sigma}_j)|_{p_{{\rm v}_1(e_i)}=p_{{\rm v}_2(e_i)}=p_i}, \quad C_{\Sigma,g}(\boldsymbol{\alpha},{\bf p}):=\prod_{j=1}^NC_{\mc{P}_j({\bf 1}),g}(\boldsymbol{\alpha}_j, {\bf p}_j)|_{p_{{\rm v}_1(e_i)}=p_{{\rm v}_2(e_i)}=p_i}\]
where $\tilde{\mc{F}}_{\bf p}(\boldsymbol{\alpha},{\bf q})$ is the normalised conformal block associated to 
$(\Sigma({\bf q}),J({\bf q}),{\bf z},\boldsymbol{\alpha},\boldsymbol{\zeta}^{\bf q})$
 and  $C_{\mc{P}_j({\bf 1}),g}(\boldsymbol{\alpha}_j,{\bf p}_j)$ are explicit constants independent of ${\bf q}$ and given by \eqref{metricconstantanomaly} and \eqref{metricconstant}.
\end{theorem}
\begin{proof}
From the definition \eqref{defhatA} (applied to $\mc{P}_j({\bf q}_j)$ rather than $\mc{P}_j$) and viewing $\widehat{\mc{A}}_{\mc{P}_j({\bf q}_j),g_j,{\bf x}_j,\boldsymbol{\alpha}_j,\boldsymbol{\zeta}_j^{\bf q}}$ as an element in $L^2(\R_+^{b_j},d{\bf p}_j;\mc{H}_{\mc{T}}^{\otimes b_j}\otimes \mc{H}_{\mc{T}}^{\otimes b_j})$, we obtain by \eqref{wardresume}  (as an element in $\mc{H}_{\mc{T}}^{\otimes b_j}\otimes \mc{H}_{\mc{T}}^{\otimes b_j}$)
\[\begin{split}
\widehat{\mc{A}}_{\mc{P}_j({\bf q}_j),g_{\bf q},{\bf x}_j,\boldsymbol{\alpha}_j,\boldsymbol{\zeta}_j^{\bf q}}({\bf p}_j)=&
{\bf C}^{\rm DOZZ}_{\gamma,\mu}(\boldsymbol{\alpha}_j,{\bf p}_j,\boldsymbol{\sigma}_j)  \mc{B}_{\mc{P}_j({\bf q}_j),J_j,{\bf x}_j,\boldsymbol{\alpha}_j,\boldsymbol{\zeta}_j^{\bf q}}({\bf p}_j)\otimes \bbar{\mc{B}_{\mc{P}_j({\bf q}_j),J_j,{\bf x}_j,\boldsymbol{\alpha}_j,\boldsymbol{\zeta}_j^{\bf q}}({\bf p}_j)}\\
& =\prod_{k=1}^{b_j}
\big(\frac{|q_{jk}|}{\delta^2}\big)^{-\frac{c_{\rm L}}{12}+2\Delta_{Q+ip_{jk}}}  C_{\mc{P}_j,g}(\boldsymbol{\alpha}_j, {\bf p}_j){\bf C}^{\rm DOZZ}_{\gamma,\mu}(\boldsymbol{\alpha}_j,{\bf p}_j,\boldsymbol{\sigma}_j)\\
&\quad \times \tilde{\mc{B}}_{\mc{P}_j({\bf q}_j),J_j,{\bf x}_j,\boldsymbol{\alpha}_j,\boldsymbol{\zeta}_j^{\bf q}}({\bf p}_j)\otimes \bbar{\tilde{\mc{B}}_{\mc{P}_j({\bf q}_j),J_j,{\bf x}_j,\boldsymbol{\alpha}_j,\boldsymbol{\zeta}_j^{\bf q}}({\bf p}_j)}\\
\end{split}\]  
Combining this with Proposition \ref{prop:decompositionAmplitude} and \eqref{metricanomalyconstant}, we see that
  $\mc{A}_{\Sigma({\bf q}),g_{\bf q},{\bf x},\boldsymbol{\alpha}}$ can be written as \eqref{theoremmodifiedconfblock}.
\end{proof}
In particular, using the Hinich-Vaintrob plumbing coordinates explained in Section \ref{subsubsec:plumbingparameters}, which cover the whole moduli space, 
one obtains an expression of the correlation functions in terms of noramlised conformal blocks and DOZZ constants for all Riemann surfaces.\\

\noindent \textbf{Example in genus 2:} in our example in genus $2$ with no marked point, the partition function is 
\[\begin{split}
\mc{A}_{\Sigma({\bf q}),g_{{\bf q}}} = \frac{2^{\frac{3}{2}}|q_1q_2q_3|^{-\frac{c_{\rm L}}{12}}}{(2\pi)^5}  \int_{\R_+^3} & C^{\rm DOZZ}_{\gamma,\mu}(Q-ip_1,Q-ip_2,Q+ip_2)
C^{\rm DOZZ}_{\gamma,\mu}(Q+ip_1,Q+ip_3,Q-ip_3)\\
&  \times C_{\mc{P}_1({\bf 1}),g}(p_1,p_2,p_2)C_{\mc{P}_2({\bf 1}),g}(p_2,p_3,p_3) \times \prod_{j=1}^3 |q_j|^{2\Delta_{Q+ip_j}} |\tilde{\mc{F}}_p({\bf q})|^2\dd p.
\end{split}\]

    %%%%%%%%%%%%%%%%%%%%%%%%%%%%%%%%%%%%%%%%%%%%%%%%%%%%%%%%%%%%%%%%%%%%%%%%%%%%%%%%%%%%%%%%%%%%%%%%%%%%%%%%%%%%%%%%%%%%%%%%%%%%%%%%%%%%%%%%%%%%
 \section{Special cases}\label{sec:special}
   %%%%%%%%%%%%%%%%%%%%%%%%%%%%%%%%%%%%%%%%%%%%%%%%%%%%%%%%%%%%%%%%%%%%%%%%%%%%%%%%%%%%%%%%%%%%%%%%%%%%%%%%%%%%%%%%%%%%%%%%%%%%%%%%%%%%%%%%%%%%%%%%%%%%%%%%%%%%%%%%%%%%%%%%%%%%%%%%%%%%%%%%%%%%%%%%%%%%%%%%%%%%%%%%%%%%%%%%%%%%%%%%%%%%%%%%%%%%%%%%%%%%%%%%%%%%%%%%%%%%%%%%%%%%%%%%%%%%%%%
In this section, we implement concretely Theorem \ref{th:fullexpressioncorrel}  in the context of the torus (Section \ref{sub:torus}) or the Riemann sphere (Section \ref{sub:sphere})  to illustrate how to implement the conformal bootstrap and also because these special cases have many applications in physics.
\subsection{The torus}

On the torus, we can use more explicit complex parameters for the moduli than those described above using plumbings, as we now explain. 

\subsubsection{1-point correlation function on $\T^2$}\label{sub:torus}
Consider the torus $\T^2_\tau:=\C/(2\pi\Z+2\pi\tau\Z)$ equipped with its canonical complex structure $J_{\C}$ induced by $\C$, where $\tau=\tau_1+i\tau_2$ satisfies $\tau_2>0$, and let us define 
$q=e^{2\pi i\tau}$ which satisfies $|q|<1$. We take the marked point $0$ and recall that the moduli space $\mc{M}_{1,1}=\mathbb{H}^2/{\rm PSL}_2(\Z)$ where $\mathbb{H}^2\subset \C$ is the upper half plane, with parameter $\tau$. 
The parameter $q$ can also be viewed as a parameter for the chosen torus with one marked point.  The torus $\T_\tau^2$ is complex equivalent to the quotient $\C\setminus \{0\}/\cjg z\mapsto qz\cjd$ by the dilation, as can be seen by using the map $z\mapsto e^{iz}$, and the marked point is now $z=1$.
We can then realise $\T^2_\tau$ with $0$ as marked point 
as the gluing of the exterior circle $\mc{C}_1=\{|z|=1\}$, equipped with the canonical parametrisation $\zeta_{1}(e^{i\theta})=e^{i\theta}$, of the annulus $\mathbb{A}_{q}$, with the interior circle $\{|z|=|q|\}$ equipped with the parametrisation  $\zeta_2(e^{i \theta})=qe^{i \theta}$. We use the flat metric $|dz|^2/|z|^2$ on $\C\setminus\{0\}$ and denote by $\boldsymbol{\zeta}:=(\zeta_1,\zeta_2)$. There is only one complex building block $\mc{P}_1(q)=\mathbb{A}_q$ with the marked point $z=1$. To avoid having the marked point on the glued circle, and since the $1$-point correlation function of the torus is independent of the point, we can freely assume the insertion is at $z=|q|^{1/2}$ in the torus $\mathbb{A}_q/\cjg z\mapsto qz\cjd\simeq \mathbb{T}^2_\tau$.

According to Definition \ref{def:conformalblock} (we can use Corollary  \ref{annulusDOZZ} in that case rather than \eqref{wardresume}), the modified conformal block an conformal block for the torus 
$(\T^2_\tau,J_{\C},z,\alpha_1,\boldsymbol{\zeta})$ are given by the series
\begin{equation}\label{conformalblockT2-1point}
\tilde{\mc{F}}_{p}(\alpha_1,q)=\sum_{n\in \N} q^{n}
{\rm Tr}({\bf F}_{Q+ip,n}^{-1}{\bf w}_{n}^{\mathbb{A}}(\alpha_1,p,p))
\end{equation}
converging for almost all $p\in \R_+$ in the disk $\{|q|<1\}$, where ${\bf w}_{n}^{\mathbb{A}}(\alpha_1,p,p):=(w_{\mathbb{A}}(\Delta_{\alpha_1},\Delta_{Q+ip},\Delta_{Q+ip},\nu))_{\nu\in \mc{T}_n}$ using the $w_{\mathbb{A}}$ coefficients of Corollary  \ref{annulusDOZZ}.

Applying Theorem \ref{th:fullexpressioncorrel} (or rather its proof) to this case, we obtain:
\begin{theorem}\label{Torus1pt}
For $\alpha_1\in(0,Q)$, the Liouville $1$-point correlation function on $\T^2_\tau$ is given 
by the $L^1$-integral 
\[\cjg V_{\alpha_1}(0)\cjd_{\T^2_\tau}=\frac{|q|^{-\frac{{\rm c}_L}{12}}}{2e}
\int_0^\infty  |q|^{2\Delta_{Q+ip}} C^{\rm DOZZ}_{\gamma,\mu}(Q+ip,\alpha_1,Q-ip)|\tilde{\mc{F}}_{p}(\alpha_1,q)|^2  \dd p.\]

\end{theorem}
\subsubsection{k-point correlation function on $\T^2$}

We use the same notations as for the $1$-point function. Take $x_1=0$ and $x_2,\dots,x_k\in \C$ 
with $0<{\rm Im}(x_2)<\dots<{\rm Im}(x_k)<2\pi\tau_2$, which we view as marked points on $\T_\tau^2$. Let us define the marked points $z_j:=e^{ix_j}\in \mathbb{A}_q$. In order to avoid  having $z_1$ on a glued circle, we can scale the points by $\la\in (0,1)$ for $\la<1$ arbitrarily close to $1$, without changing the $k$-point correlation function, the new points become $z'_j=\la z_j$. 
If $\la<1$ is arbitrarily close to $1$, the torus $\mathbb{A}_q/\cjg z\mapsto qz\cjd$ can be cut along the curves $\mc{C}_j:=\{|z|=|z_{j}|\}$ for $j=1,\dots,k$ and $\mc{C}_{k+1}=\{|z|=|q|\}$, where $\mc{C}_1=\{|z|=1\}$ is glued with $\mc{C}_{k+1}$ (the parametrisations of these circles are the canonical ones,  except for $\mc{C}_{k+1}$ which is parametrised by $\zeta_k(e^{i \theta})=qe^{i \theta}$). 
This decomposes $\mathbb{T}^2_\tau$ as $k$ complex building blocks, denoted $\mc{P}_j$ for $j=1,\dots, k$, which are annuli bounded by $\mc{C}_{j}\cup \mc{C}_{j+1}$ each with $1$ marked point given by $z'_j$ (see Figure \ref{figtorus}). By Proposition \ref{Weyl}, the amplitude of $\mc{P}_j$ can also be written 
as $\mc{A}_{\mathbb{A}_{|q_j|},g_{\mathbb{A}},z'_j/|z_j|,\alpha_j,\boldsymbol{\zeta}}$ with $q_j:=z_{j+1}/z_j\in \D$ for $j<k$, and 
$\mc{A}_{\mathbb{A}_{q/|z_k|},g_{\mathbb{A}},z'_k/|z_k|,\alpha_j,\boldsymbol{\zeta}}$ for the last one, with $\boldsymbol{\zeta}$ being the canonical parametrisation except fr the interior circle for $j=k$ having a shift by ${\rm arg}(q)$.  We will also set $q_k:=q/z_k$.

The normalised conformal block in that case can be written, using Definition \ref{def:conformalblock} and Corollary \ref{annulusDOZZ}, as the series 
\begin{equation}\label{torus:conformalblockkpt}
\tilde{\mc{F}}_{\bf p}(\boldsymbol{\alpha},{\bf q}):=\sum_{{\bf n}\in \N^k}(\prod_{j=1}^k q_j^{n_j})\tilde{\mc{F}}_{{\bf p},{\bf n}}(\boldsymbol{\alpha}) , \quad \textrm{ with }
\end{equation}
\[\tilde{\mc{F}}_{{\bf p},{\bf n}}(\boldsymbol{\alpha}):={\rm Tr}({\bf F}_{p_k,n_k}^{-1}{\bf w}_{n_k,n_{k-1}}^{\mathbb{A}}(p_k,\alpha_{k},p_{k-1}){\bf F}_{p_{k-1},n_{k-1}}^{-1}\dots {\bf w}_{n_2,n_1}^{\mathbb{A}}(p_2,\alpha_2,p_1){\bf F}_{p_1,n_1}^{-1}{\bf w}_{n_1,n_{k}}^{\mathbb{A}}(p_1,\alpha_1,p_{k}))\]
where ${\bf p}=(p_1,\dots,p_k)$, ${\bf q}=(q_1,\dots,q_k)$, ${\bf n}=(n_1,\dots,n_k)$ (using the $w_{\mathbb{A}}(\boldsymbol{\Delta}_{\boldsymbol{\alpha}},\boldsymbol{\nu})$ of  Corollary \ref{annulusDOZZ})
\[{\bf w}^{\mathbb{A}}_{n,n'}(p,\alpha,p'):=(w_{\mathbb{A}}(\Delta_{Q+ip},\Delta_{\alpha},\Delta_{Q+ip'},\nu,\nu'))_{(\nu,\nu')\in \mc{T}_n\times \mc{T}_{n'}}.\] 
Here we have used Proposition \ref{Weyl} and the scaling $z\mapsto z/|z_j|$ to reduce the $\mc{P}_j$ amplitude to the amplitude of $\mc{A}_{\mathbb{A}_{|q_j|},g_{\mathbb{A}},z'_j/|z_j|,\alpha_j,\zeta_j}$  to be able to apply Corollary \ref{annulusDOZZ}. 
 The series $\tilde{\mc{F}}_{{\bf p},{\bf n}}(\boldsymbol{\alpha})$ is absolutely converging for almost all ${\bf p}$, and  it does not depend on $\la$.  

The proof of Theorem \ref{th:fullexpressioncorrel} also gives:
\begin{theorem}\label{th:kpointtorus}
Let $\boldsymbol{\alpha}:=(\alpha_1,\dots,\alpha_k)\in (0,Q)^k$.
The $k$-point correlation function on $\T^2_\tau$ equipped with the Euclidean metric $|dz|^2$ is given for $x_1=0$ and ${\rm Im}(x_j)<{\rm Im}(x_{j+1})<2\pi \tau_2$ by the $L^1$-integral
\[\begin{split}
\cjg \prod_{j=1}^kV_{\alpha_j}(x_j)\cjd_{\T_\tau}=\frac{\prod_{j=1}^k|q_j|^{-\frac{c_{\rm L}}{12}}}{2^{2k-1}\pi^{k-1}e^k} 
\int_{\R_+^k} \prod_{j=1}^k|q_j|^{2\Delta_{Q+ip_j}}\Big(\prod_{j=1}^k C^{\rm DOZZ}_{\gamma,\mu}(Q+ip_j,\alpha_j,Q-ip_{j-1})\Big)\Big|_{p_0=p_k}|\tilde{\mc{F}}_{\bf p}(\boldsymbol{\alpha},{\bf q})|^2\dd{\bf p}.
\end{split}\]
with $q_j=z_{j+1}/z_j$ for $j=1,\dots,k-1$ and  $q_k=q/z_k$.
\end{theorem}

\begin{figure}[h] 
\centering
{\begin{tikzpicture}[xscale=1,yscale=1] 
\tikzstyle{sommet}=[circle,draw,fill=yellow,scale=0.3] 
\draw[color=red,fill=blue!30] (0,0) circle (3);
\draw[color=red,fill=white] (0,0) circle (0.7);
\node[color=red] (C) at (0.6,0.5){ $\mathbf{+}$  $q$};
\node[color=red,rotate=60] (D) at (3,0){ $\mathbf{+}$};
\node[color=red,rotate=60] (D) at (2.7,0){ $\mathbf{+}$};
\node[color=red,rotate=60] (D) at (1.9,0){ $\mathbf{+}$};
\node[color=red,rotate=60] (D) at (1.3,0){ $\mathbf{+}$};
\draw[color=red,dashed] (0,0) circle (1.3);
\node[rotate=30]  (E) at (1.1,-0.5){ $\mathbf{+}$};
\node[]  (E) at (1,-0.2){ $\lambda z_1$};
\draw[color=red,dashed] (0,0) circle (1.9);
\node[rotate=30]  (F) at (0,-1.8){ $\mathbf{+}$};
\node[]  (G) at (0.3,-1.6){ $\lambda z_2$};
\node[]  (H) at (2.2,0){ $\dots$};
\draw[color=red,dashed] (0,0) circle (2.7);
\node[rotate=30]  (F) at (2,2.1){ $\mathbf{+}$};
\node[]  (G) at (2.5,2.1){ $\lambda z_k$};
\end{tikzpicture}
}
\caption{Cutting for $k$ point correlation functions on the torus. The marked points are black, the dashed circles are the cutting circles $\mc{C}_j$ with radii $r_j=|z_j|$ and red points stand for the origins of the parametrised boundaries.}
\label{figtorus}
\end{figure}

\subsection{The Riemann sphere  in the linear channel}\label{sub:sphere}

  We consider the Riemann sphere $\mathbb{S}^2=\hat{\C}$    with the DOZZ metric $g_{\rm dozz}$ and let $(z_j)_{j=1,\dots,k}$  be $k\geq 4$ points with $z_1=0$, $z_k=\infty$ and $|z_j|<|z_{j+1}|$. Without loss of generality we assume $|z_2|<1$ and $|z_{k-1}|>1$. 
 We also choose $(\alpha_j)_{j=1,\dots,k}\in (0,Q)^k$ some weights satisfying $\sum_j \alpha_j>2Q$ and set $\boldsymbol{\alpha}=(\alpha_1,\dots,\alpha_k)$.
 We can cut $\hat \C$ into $k-2$ complex building blocks $\mc{P}_j$ separated by 
$k-3$ circles $(\mc{C}_j)_{j=2,\dots,k-2}$ chosen as follows. Let $\la>1$ chosen arbitrarily close to 
$1$ and set $r_j=\la |z_j|$, then define $\mc{C}_j:=\{|z|=r_j\}$ and let 
$\mc{P}_1=\bbar{\D}_{r_2}$, $\mc{P}_j=\{|z|\in [r_j,r_{j+1}]\}$ for $j=2,\dots,k-3$ and 
$\mc{P}_{k-2}:=\hat{\C}\setminus \D_{r_{k-2}}$. Each circle is parametrised canonically by $\zeta_j(e^{i\theta})=r_j e^{i\theta}$. We choose a conformal metric $g=e^{h}g_{\rm dozz}$   which is equal to $|dz|^2/|z|^2$ on $|z|\in [\eps,\eps^{-1}]$ for $\min(|z_2|,|z_{k-1}|^{-1})>\eps>0$ small enough and equal to $g_{\rm dozz}$  near $0$, invariant by $z\mapsto 1/z$. Recall, by Weyl anomaly, that the choice of the metric is irrelevant up to an overall multiplicative factor. One can then view the $k$ point correlation functions with points $z_j$ and weights $\alpha_j$ as a gluing of amplitudes.
By Proposition \ref{Weyl} (using both 1) and 2)), the amplitudes of $\mc{P}_1$ and $\mc{P}_{k-2}$ are respectively given by  
\[\begin{gathered}
\mc{A}_{\mathbb{D}_{r_2},g,(0,z_2),(\alpha_1,\alpha_2),r_2\zeta_0}=\Big|\frac{r_2}{z_2}\Big|^{-2\Delta_{\alpha_2}}r_2^{-2\Delta_{\alpha_1}}e^{(1+6Q^2)S_{\rm L}^0(\D_{r_2},g_{\D},h)}\mc{A}_{\mathbb{D},g_{\D},(0,\frac{z_2}{r_2}),(\alpha_1,\alpha_2),\zeta_0}, \\ 
\mc{A}_{\mathbb{D}_{r_{k-2}^{-1}},g,(0,\frac{1}{z_{k-1}}),(\alpha_k,\alpha_{k-1}),r_{k-2}^{-1}\bbar{\zeta}_0}=\Big|\frac{z_{k-1}}{r_{k-2}}\Big|^{-2\Delta_{\alpha_{k-1}}}r_{k-2}^{2\Delta_{\alpha_k}}e^{(1+6Q^2)S_{\rm L}^0(\D_{1/r_{k-2}},g_{\D},h)}\mc{A}_{\mathbb{D},g_{\D},(0,\frac{r_{k-2}}{z_{k-1}}),(\alpha_k,\alpha_{k-1}),\bbar{\zeta}_0}
\end{gathered}\]
where $\zeta_0(e^{i\theta})=e^{i\theta}$ (we applied an inversion $z\mapsto 1/z$ in the second case, and some dilations $z\mapsto rz$ in both cases for $r=r_2$ and $r=r_{k-2}^{-1}$). Again using Proposition \ref{Weyl} we also have that the amplitude of $\mc{P}_j$ for $j=2,\dots,k-3$ is 
\[\mc{A}_{\mathbb{A}_{r_j/r_{j+1}},g_{\mathbb{A}},\frac{z_{j+1}}{r_{j+1}},\alpha_{j+1},\boldsymbol{\zeta}}\]
with $\boldsymbol{\zeta}(e^{i\theta})=(e^{i\theta}r_j/r_{j+1},e^{i\theta})$. By \cite[Proposition 7.12]{GKRV20_bootstrap}\footnote{We use that $\mc{A}^0_{\D,g_\D}=1$ and therefore $\mc{A}_{\mathbb{D},g_{\D},(0,z),(\alpha_1,\alpha_2),\zeta_0}=Z_{\D,g_\D}U_{\alpha_1,\alpha_2}(0,z)$ in the terminology of \cite{GKRV20_bootstrap})}
we can write for $z\in \D$
\begin{equation}\label{warddisk}
\begin{split}
\cjg \mc{A}_{\mathbb{D},g_{\D},(0,z),(\alpha_1,\alpha_2),\zeta_0},\bbar{\Psi_{Q+ip,\nu,\tilde{\nu}}} \cjd_{\mc{H}}=& \frac{Z_{\mathbb{D},g_{\D}}}{2}C^{\rm DOZZ}_{\gamma,\mu}(\alpha_1,\alpha_2,Q+ip)z^{|\nu|}\bar{z}^{\tilde{\nu}}|z|^{2(\Delta_{Q+ip}-\Delta_{\alpha_1}-\Delta_{\alpha_2})}\\
& \times w_{\D}(\Delta_{\alpha_1},\Delta_{\alpha_2},\Delta_{Q+ip},\nu)\bbar{w_{\D}(\Delta_{\alpha_1},\Delta_{\alpha_2},\Delta_{Q+ip},\tilde{\nu})} 
\end{split}\end{equation}
where $w_{\D}$ are explicit coefficients similar to those in Corollary \ref{annulusDOZZ} (see eq (7.36) of \cite{GKRV20_bootstrap} for the formula). Define 
${\bf w}^{\D}_n(\alpha_1,\alpha_2,p)=(w_{\D}(\Delta_{\alpha_1},\Delta_{\alpha_2},\Delta_{Q+ip},\nu))_{\nu \in \mc{T}_n}$.
Let us define the series 
\begin{equation}\label{sphere:conformalblockkpt}
\hat{\mc{F}}_{\bf p}(\boldsymbol{\alpha},{\bf q}):= \Big(\prod_{j,0<|z_j|<1} |z_j|^{-\Delta_{\alpha_j}}\Big)\Big(\prod_{j, 1\leq |z_j|<\infty} |z_j|^{\Delta_{\alpha_j}}\Big)
 |z_2|^{-\Delta_{\alpha_1}}|z_{k-1}|^{\Delta_{\alpha_{k}}}(\prod_{j=2}^{k-2} |q_j|^{\Delta_{Q+ip_j}})\sum_{{\bf n}\in \N^{k-3}}(\prod_{j=2}^{k-2} q_j^{n_j})\tilde{\mc{F}}_{{\bf p},{\bf n}}(\boldsymbol{\alpha}) , 
\end{equation}
\[   \textrm{ with }\quad \tilde{\mc{F}}_{{\bf p},{\bf n}}(\boldsymbol{\alpha}):=\cjg {\bf F}_{p_2,n_2}^{-1}{\bf w}^{\mathbb{A}}_{n_2,n_3}(p_2,\alpha_3,p_3){\bf F}_{p_3,n_3}^{-1}\dots {\bf F}_{p_{k-2},n_{k-2}}^{-1}{\bf w}^{\D}_{n_{k-2}}(p_{k-2},\alpha_{k-1},\alpha_k),{\bf w}^{\D}_{n_2}(p_2,\alpha_2,\alpha_1)\cjd_{\mc{H}_{\mc{T}}}\]
where we have set  $q_j:=z_{j}/z_{j+1}\in \D$ for $2\leq j\leq k-2$, while ${\bf p}=(p_2,\dots,p_{k-2})$, ${\bf q}=(q_2,\dots,q_{k-2})$, ${\bf n}=(n_2,\dots,n_{k-2})$. We observe that 
\[ e^{(1+6Q^2)S_{\rm L}^0(\D_{r_2},g_{\D},h)}e^{(1+6Q^2)S_{\rm L}^0(\D_{1/r_{k-2}},g_{\D},h)}\prod_{j=2}^{k-2}|q_j|^{-\frac{c_{\rm L}}{12}}=
e^{(1+6Q^2)S_{\rm L}^0(\hat{\C},g_{\rm dozz},g)}\]
is the global Weyl anomaly relating the metrics $g=e^hg_{\rm dozz}$ to the DOZZ metric $g_{\rm dozz}$ on $\hat\C$. 
Using Proposition \ref{prop:decompositionAmplitude}, Corollary \ref{annulusDOZZ}, \eqref{warddisk},
we obtain 
\[\cjg \prod_{j=1}^kV_{\alpha_j}(z_j)\cjd_{\hat\C,g}=\frac{2^{-\frac{3}{2}} Z_{\D,g_\D}^2}{(2\pi)^{k-3} (2e)^{k-4}}e^{(1+6Q^2)S_{\rm L}^0(\hat{\C},g_{\rm dozz},g)-\sum_{j=2}^{k-1}\Delta_{\alpha_j}h(z_j)}\int_{\R_+^{k-3}}\boldsymbol{\rho}(\boldsymbol{\alpha},{\bf p})|\hat{\mc{F}}_{\bf p}(\boldsymbol{\alpha},{\bf q})|^2d{\bf p}\]
with 
\begin{equation}\label{rhosphere}
\boldsymbol{\rho}(\boldsymbol{\alpha},{\bf p}):=C^{\rm DOZZ}_{\gamma,\mu}(\alpha_1,\alpha_2,Q-ip_2)C^{\rm DOZZ}_{\gamma,\mu}(\alpha_k,\alpha_{k-1},Q+ip_{k-2})\prod_{j=2}^{k-3} C^{\rm DOZZ}_{\gamma,\mu}(Q+ip_j,\alpha_{j+1},Q-ip_{j+1}).
\end{equation}
By the Weyl anomaly \eqref{ax1},  we obtain
 \begin{theorem}\label{th:kpointsphere}
Let $\boldsymbol{\alpha}:=(\alpha_1,\dots,\alpha_k)\in (0,Q)^k$ and $\sum_{j}\alpha_j>2Q$.
The $k$-point correlation function on $\hat\C$ equipped with the DOZZ metric $g_{\rm dozz}$ 
is given for $z_1=0$, $z_k=\infty$ and $|z_j|<|z_{j+1}|$ by the $L^1$-integral
\[\begin{split}
\cjg \prod_{j=1}^kV_{\alpha_j}(z_j)\cjd_{\hat\C,g_{\rm dozz}}=\frac{2^{-\frac{3}{2}}Z_{\D,g_\D}^2}{(2\pi)^{k-3} (2e)^{k-4}}\int_{\R_+^{k-3}}\boldsymbol{\rho}(\boldsymbol{\alpha},{\bf p})|\hat{\mc{F}}_{\bf p}(\boldsymbol{\alpha},{\bf q})|^2d{\bf p}
\end{split}\]
with $q_j=z_{j}/z_{j+1}$ for $j=2,\dots,k-2$, ${\bf p}=(p_2,\dots,p_{k-2})$, ${\bf q}=(q_2,\dots,q_{k-2})$, $\boldsymbol{\alpha}=(\alpha_1,\dots,\alpha_k)$ and $\boldsymbol{\rho}(\boldsymbol{\alpha},{\bf p})$ given by \eqref{rhosphere}.
\end{theorem}
We note here that the constant $ Z_{\D,g_\D}$ is given by $Z_{\D,g_\D}= e^{\frac{1}{4}}\times 2^{\frac{1}{12}}\pi^{\frac{1}{4}}e^{\frac{5}{24}+\zeta_R'(-1)}$. The proof is given at the end of the proof of Corollary \ref{annulusDOZZ}.

\section{Ward identities  }\label{sec:ward}
%%%%%%%%%%%%%%%%%%%%%%%%%%%%%%%%%%%%%%%%%%%%%%%%%%%%%%%%%%%%%%%%%%%%%%%%%%%%%%%%%%%%%%%%%%%%%%%%%%%%%%%%%%%%%%%%%%%%%%%%%%%%%%%%%%%%%%%%%%%%%%%%%%%%%%%%%%%%%%%%%%%%%%%%%%%%%%%%%%%%%%%%%%%%%%%%%%%%%%%%%%%%%%%%%%%%%%%%%%%%%%%%%%%%%%%%%%%%%%%%%%%%%%%%%%%%%%%%%%%%%%%%%%%%%%%%%%%%%%%%%%%%%%%%%%%%%%%%%%%%%%%% 
%%%%%%%%%%%%%%%%%%%%%%%%%%%%%%%%%%%%%%%%%%%%%%%%%%%%%%%%%%%%%%%%%%%%%%%%%%%%%%%%%%%%%%%%%%%%%%%%%%%%%%%%%%%%%%%%%%%%%%%%%%%%%%%%%%%%%%%%%%%%%%%%%%%%%%%%%%%%%%%%%%%%%% 

We have seen that the LCFT correlation functions on a Riemann surface can be expressed as compositions of amplitudes of spheres with  $b$ disks removed and $3-b$ punctures, $b=1,2,3$. Furthermore the amplitudes are Hilbert-Schmidt operators and their compositions can be expressed in terms of the spectral resolution of the LCFT Hamiltonian. This leads to the problem of evaluating the amplitudes on the LCFT eigenstates $\Psi_{\alpha,\nu,\tilde\nu}$ discussed in Proposition \ref{defprop:desc}. In \cite{GKRV20_bootstrap} the eigenstates $\Psi_{\alpha,\nu,\tilde\nu}$ were given a probabilistic expression in the case of sufficiently negative $\alpha\in\R$. In the language of this paper, this expression involves a  disk amplitude where, in addition to the vertex operator $V_\alpha(0)$, there are a number of insertions of the {\it Stress Energy Tensor} (SET). These SET insertions encode the   action of the Virasoro generators  on the eigenstates\footnote{ In \cite{BGKR1}, a more geometric description of this action is explored.}
As will be explained in section \ref{sec:computingamplitudes}, the gluing of the building block amplitude with these disk amplitudes leads to an LCFT correlation function of three vertex operators together with a number of SET insertions on  the Riemann sphere   $\hat \C$. In the following, we introduce the SET insertions in subsection \ref{SET}, then we define the  correlation functions with SET insertions in subsection \ref{sub:ampSEThole} and finally  we   show how the dependence on SET insertions can be  evaluated explicitly using the conformal Ward identities  in Proposition \ref{propward}, which is the main result of this section.

\subsection{Stress Energy Field}\label{SET}

%%%%%%%%%%%%%%%%%%%%%%

We work on a surface $\Sigma$ where either   $\Sigma=\hat\C$ or  $\Sigma\subset\C$  a bounded region (and thus $\partial\Sigma\neq\emptyset$). We equip $\Sigma$  with an admissible metric  $g=e^{\omega(z)}|\dd z|^2$.  The stress energy tensor (SET) does not make sense as a random field but can be given sense at the level of correlation functions as the limit of a regularised SET. Let us introduce the field (recall \eqref{defliouvillefield} for the definition of $\phi_g$) 
\begin{align*}
\Phi_g:=\phi_g+\tfrac{Q}{2}\omega
\end{align*}
and its $|dz|^2$-regularisation $\Phi_{g,\epsilon}$ with the following caveat: we  make a special choice of regularisation by choosing a  nonnegative radial smooth compactly supported function $\rho:\C\to \R$ such that $\int \rho(z)\,\dd z=1$ where $\dd z$ denotes the Lebesque measure and thus $\int z\partial_z \rho(z)\,\dd z=-1$. Then, as usual, we set $\rho_\epsilon(z)=\epsilon^{-2}\rho(z/\epsilon)$ and denote by $h_\epsilon=h\star \rho_\epsilon$ the regularisation of a distribution $h$.  In particular $\Phi_{g,\epsilon}(z)$ is a.s. smooth.
Then the regularised SET is defined by
\begin{equation}\label{defSET}
\quad T_{g,\epsilon}(z):=Q\partial^2_{z}\Phi_{g,\epsilon}(z)-(\partial_z \Phi_{g,\epsilon}(z))^2 +a_{\Sigma,g,\epsilon}(z)
\end{equation}
with the renormalisation constant  given by $a_{\Sigma,g,\epsilon}:=\E[(\partial_z X_{g,D,\epsilon}(z))^2]$ if $\partial\Sigma\not= \emptyset$ and  $a_{\Sigma,g,\epsilon}:=\E[(\partial_z X_{g,\epsilon}(z))^2]$ if $\partial\Sigma=\emptyset$.
 We denote also by $\bbar T_{g,\epsilon}(z)$  the complex conjugate of $T_{g,\epsilon}(z)$. The regularised SET  field $T_{g,\epsilon}(z)$ is a proper random field and, in the case when $\Sigma$ has a boundary, the SET   thus depends on the boundary fields $\boldsymbol{\tilde{\varphi}} \in(H^{s}(\T))^{b}$.  

In this section, we will also sometimes consider regularised vertex operators $V_{\alpha,g,\epsilon}(z)= \epsilon^{\alpha^2/2}e^{\alpha \phi_{g,\epsilon}(z)}$ where $\phi_{g,\epsilon}$ denotes regularisation in the flat background metric; this simplifies computations and this is why we use this regularisation. We will explicitely say when we use the flat metric regularisation for the vertex operator and therefore by default (unless explicitely stated) $V_{\alpha,g,\epsilon}(z)$ will denote regularisation in the $g$ metric; also the formal notation $V_{\alpha,g}(z)$ denotes the limit of the $g$-regularised vertex operator. Both regularisations yield the same limiting quantity up to multiplication by the smooth factor $g(z)^{\alpha^2/2}$ if $\Sigma= \hat\C$.

The renormalising constant $a_{\Sigma,g,\epsilon}$ will play an important role in the following. As a preliminary result, we study how this constant reacts to geometric changes:
 
\begin{lemma}\label{constanteSET}
1) in case $\partial\Sigma=\emptyset$ then $a_{\Sigma,g,\epsilon}(z)=o(1)$ as $\epsilon\to 0$.\\
2)  in case $\partial\Sigma\not=\emptyset$ then $a_{\Sigma,g,\epsilon}(z) $ converges as $\epsilon\to 0$ towards some limit denoted by $a_{\Sigma,g}(z)$ satisfying:
\begin{align*}%\label{}
a_{\Sigma,e^{\sigma}g}(z)=a_{\Sigma,g}(z)
\end{align*}
 for any bounded smooth function $\sigma:\bar\Sigma\to\R$ and 
\begin{align}\label{atransf}
a_{\Sigma,\psi^\ast g}(z)=\psi'(z)^2a_{\Sigma', g}(\psi(z))-\frac{1}{12}  {\rm S}_\psi(z)
\end{align}
where $\psi:\Sigma\to\Sigma'$ is a biholomorphism and
\begin{align}\label{Schwarzian}
{\rm S}_\psi:=\psi'''/\psi'-\frac{3}{2}(\psi''/\psi')^2
\end{align}
 is  the Schwarzian derivative of  $\psi$.
\end{lemma} 

\begin{proof}
The constant is given by $a_{\Sigma,g,\epsilon}(z)=\epsilon^{-2}\int\int \partial_z\rho(u) \partial_z\rho(v)G(z+\epsilon u,z+\epsilon v)\dd u\dd v$ with $G$ the Green function in the metric $g$ with boundary condition given by Dirichlet or vanishing mean depending on $\partial\Sigma\not= \emptyset$ or not. We will use the two following computations to prove all our claims: let $h(z,\bar z,z',\bar z')$ be a smooth function and $\psi$ a holomorphic map, both of them  defined in a neighborhood of $z$ for each of their variables and $\psi'(z)\not =0$
\begin{align}
&\epsilon^{-2}\int\int\partial_z\rho(u)\partial_z\rho(v)h(z+u\epsilon,\bar z+\epsilon\bar u,z+v\epsilon,\bar z+\epsilon\bar v)\dd u\dd v=\partial^2_{zz'} h(z,\bar z,z,\bar z)+o(1)\label{cst1}\\
&\epsilon^{-2}\int\int\partial_z\rho(u)\partial_z\rho(v)\log\frac{1}{|\psi(z+u\epsilon)-\psi( z+v\epsilon)|}\dd u\dd v=- \frac{1}{12}{\rm S}_\psi+o(1).\label{cst2}
\end{align}
Indeed, the first claim follows by Taylor expanding $h$ in \eqref{cst1} ($h$ and its derivatives below are evaluated at $z,\bar z,z,\bar z$)
\[\begin{split}
h(z+u\epsilon, \bar z+\epsilon\bar u,z+v\epsilon,\bar z+\epsilon\bar v)=& h +\epsilon u\partial_zh+\epsilon\bar u\partial_{\bar z}h+\epsilon v\partial_{z'}h+\epsilon\bar v\partial_{\bar z'}h+\frac{\epsilon^2}{2}\big(u^2\partial^2_{zz}h+\bar u^2\partial^2_{\bar z\bar z}h+v^2\partial^2_{z'z'}h+\bar v^2\partial^2_{\bar z'\bar z'}h)\\
&+\frac{\epsilon^2}{2}\big(2u\bar u \partial^2_{z\bar z}h+ 2uv \partial^2_{zz'}h+2u\bar v \partial^2_{z\bar z'}h+2\bar uv \partial^2_{\bar z  z'}h+2\bar u\bar v \partial^2_{\bar z  \bar z'}h+2v\bar v \partial^2_{z'  \bar z'}h\big)+o(\epsilon^2)
\end{split}\] 
and then using that 
\begin{equation}\label{various}
\int\partial_z\rho(u)\dd u=0,\quad\int\partial_z\rho(u)u\dd u=-1,\quad\int\partial_z\rho(u)\bar u\dd u=0.
\end{equation}

The second claim follows by Taylor expanding $\psi$ at order $3$ to get 
\begin{align*}
\epsilon^{-2}\int\int&\partial_z\rho(u)\partial_z\rho(v)\log\frac{1}{|\psi(z+u\epsilon)-\psi( z+v\epsilon)|}\dd u\dd v\\
&=\epsilon^{-2}\int\int\partial_z\rho(u)\partial_z\rho(v)\log\frac{1}{|\epsilon (u-v)\psi'(z)+\frac{\epsilon^2}{2}(u^2-v^2)\psi''(z)+\frac{\epsilon^3}{6}(u^3-v^3)\psi'''(z)+o(\epsilon^3) |}\dd u\dd v\\
&=\epsilon^{-2}\int\int\partial_z\rho(u)\partial_z\rho(v)\log\frac{1}{|  1+\frac{\epsilon}{2}(u+v)\frac{\psi''(z)}{\psi'(z)}+\frac{\epsilon^2}{6}(u^2+uv+v^2)\frac{\psi'''(z)}{\psi'(z)}+o(\epsilon^2) |}\dd u\dd v
\end{align*} 
where we have used $\int\int\partial_z\rho(u)\partial_z\rho(v)\log\frac{1}{|  u-v|}\dd u\dd v=0$. We can then Taylor expand the $\log$ and use \eqref{various} to conclude as previously.

Now in case $\partial\Sigma=\emptyset$ and $g$ is conformal to $|dz|^2$,  then the Green function is of the form $G_g(z,z')=\log\frac{1}{|z-z'|}+h(z,\bar z)+h(z',\bar z')$. Combining our two claims \eqref{cst1}+\eqref{cst2} with $\psi(z)=z$ proves 1). In the  case $\partial\Sigma\not =\emptyset$ then $G_{g,D}$ is of the form $G_{g,D}(z,z')=\log\frac{1}{|z-z'|}+h(z,\bar z,z',\bar z')$ with $h$ smooth and our two claims give that the limit of $a_{\Sigma,g,\epsilon}(z)$ is given by $\partial^2_{zz'} h(z,\bar z,z,\bar z)$. Furthermore $h$ does not depend on the conformal factor. Finally if $\psi:\Sigma\to\Sigma'$ be a biholomorphism then  $G_{\psi^*g,D}(z,z')=\log\frac{1}{|\psi(z)-\psi(z')|}+h(\psi(z),\psi(\bar z),\psi(z'),\psi(\bar z'))$ with $G_g(z,z')=\log\frac{1}{|z-z'|}+h(z,\bar z,z',\bar z')$ the Green function of $(\Sigma',g)$. Our claims then give that $a_{\Sigma,\psi^* g}(z)=\partial_{zz'}(h(\psi(z),\psi(\bar z),\psi(z'),\psi(\bar z')))_{z'=z} -\frac{1}{12}  {\rm S}_\psi(z)$ from which the result follows.
 \end{proof}

\subsection{Correlation functions with SET insertions}\label{sub:ampSEThole}
%%%%%%%%%%%%%%%%%%%%%%%%%

Now we turn to studying correlation functions with $T_{g,\epsilon}$ insertions and their limits as $\epsilon\to 0$. As will be explained in Section \ref{sec:computingamplitudes}, the analytic continuation to real values of the complex weight $\alpha$ of the eigenstates leads to disk amplitudes  functions with further SET insertions.  When building block amplitudes are then evaluated at the eigenstates, this  yield (by the gluing lemma) correlation functions with SET insertions.   We will construct the SET insertions with the caveat that the Liouville potential is removed from a neighborhood, called hole, of the insertions.  These holes emerge from the intertwining property \eqref{intert:desc}, which is the bridge between the analytic definition of the amplitude and their probabilistic representation (disk amplitude) for real values of the complex weight $\alpha$. Technically speaking, they will be convenient to deal with because they have a regularising effect: indeed, singularities in the potential coming from vertex operators are located inside these holes, hence taming the effects of these singularities.

Locations of the SET insertions will be labelled with a couple $({\bf u},{\bf v})$ (${\bf u}$ collecting the $T$-insertions, ${\bf v}$ for the $\bar T$-insertions) and vertex operator locations  will still be labelled with ${\bf z}$. We will need to keep all these  parameters distinct. Formalizing, given  a bounded open set $U \subset \C$ and ${\bf z}\in \C^m$ for some $m\geq 3$, we introduce the sets
\begin{align*}
 \caO^{\bf z}_{\C,U}=\{({\bf u},{\bf v})\in U^{k+\tilde k}\,|\, u_j,v_{j'} \text{ all distinct and } \forall j ,\forall i, \:  u_j\neq z_i\; ,v_j\neq z_i  \},\\
  \caO^{\rm ext}_{\C,U} =\{({\bf z},{\bf u},{\bf v})\in \C^m\times U^{k+\tilde k}\,|\, z_i,u_j,v_{j'} \text{ all distinct} \}.
\end{align*}

For $\mathbf{u}= (u_1, \dots, u_k)$ we set (recall \eqref{defSET})
\begin{equation*}
  T_{g,\epsilon}( \mathbf{u}):= \prod_{i=1}^k  T_{g,\epsilon}(u_i)
\end{equation*}
and similarly for $ T_{g,\epsilon}( \mathbf{v})$. We then define
\begin{align}\label{amplitudeSET1}%\label{}
  \langle T_g( \mathbf{u})\bbar T_g( \mathbf{v})\prod_{i=1}^mV_{\alpha_i,g}(z_i)\rangle_{\hat\C,U,g}
   & 
:=  \big(\frac{{\det}'(\Delta_{g})}{{\rm v}_{g}(\Sigma)}\big)^{-1/2}  \lim_{\epsilon\to 0} \lim_{\epsilon'\to 0} \,\int_\R  \E\Big[ \prod_{i=1}^m V_{\alpha_i,g,\epsilon'}( z_i)  T_{g,\epsilon}({\bf u})\bbar T_{g,\epsilon}({\bf v}) 
 \\&
 \times  \exp\Big( -\tfrac{Q}{4\pi}\int_{\C}K_{g}(c+ X_{g} )\,{\rm dv}_{g} - \mu  e^{\gamma c}M^g_\gamma( \phi_g, \C\setminus U)  \Big) \Big]\,\dd c . \nonumber
\end{align}
The set $U$ stands for the holes in the potential.  Notice that the   SET insertions are located inside the holes.
 The following statement is a  straightforward adaptation from \cite[Prop 9.1]{GKRV20_bootstrap}. We stress again that the role of the hole $U$ is crucial as it removes any problematic singularity in the treatment of correlation functions. So the statement below can be seen as soft
  and it does not imply any property for the correlation functions of LCFT   with SET insertions (meaning with empty hole $U$, 
  which is more difficult).

\begin{proposition}\label{ampSETholomorphic}
Let $\alpha_i$ satisfy the Seiberg bounds \eqref{seiberg1} and \eqref{seiberg2}. Then
 the limit 
  in \eqref{amplitudeSET1}  exists  and  defines  a function which is smooth in  $\bf u$ and $\bf v$ and continuous in $\bf z$  for  $({\bf z},{\bf u},{\bf v}) \in \caO^{\rm ext}_{\C,U}$. It is also smooth in the variables $z_i$'s that belong to  $  U$.
\end{proposition}

\proof For later reference let us summarise the argument in \cite[Prop 9.1]{GKRV20_bootstrap}. We set
\begin{equation*}
\tilde{G}_{g,\epsilon,\epsilon'}(z,z')= \E[  \tilde{X}_{g,\epsilon}(z) \tilde{X}_{g,\epsilon'}(z')  ]
\end{equation*}
the regularisation of $\tilde{G}_{g}$ the Green function of the GFF $\tilde{X}_g$ which has a zero mean with 
respect to the curvature of the metric $g=e^\omega|dz|^2$, i.e. $\tilde{X}_g= X_g- \frac{1}{4 \pi} \int_{\C} X_g K_g \dd v_g$ and therefore
\begin{equation}\label{tildeG}
\tilde{G}_{g}(z,z')= \log \frac{1}{|z-z'|}- \frac{1}{4} \omega(z)-\frac{1}{4}  \omega(z')+C 
\end{equation}
for some constant $C$. Let us stress here that $\tilde{X}_{g,\epsilon}$ denotes regularisation in the flat background metric with respect to $\rho_\epsilon$. In this context, $V_{\alpha_i ,g,\epsilon}(z_i)$ will be $\epsilon^{\alpha_i^2/2} e^{\alpha_i X_{g,\epsilon}}$ where $X_{g,\epsilon}$ denotes regularisation in the flat background metric and this only affects the limit \eqref{amplitudeSET1} up to a smooth factor depending on the $z_i$.
Using  Gaussian integration by parts, the expectation on the RHS of \eqref{amplitudeSET1}
can be written as a linear combination of terms of the form  
\begin{align}\label{basicterm0}
\int_{(\C\setminus U)^l}{\boldsymbol \partial}^{2k} \tilde{G}_{g,\epsilon,\epsilon'}({\bf u},{\bf u}')\bar{\boldsymbol \partial}^{2\tilde k} \tilde{G}_{g,\epsilon,\epsilon'}({\bf v},{\bf v}')
 \langle \prod_{j=1}^l V_{\gamma,g }(w_j)\prod_{i=1}^m V_{\alpha_i ,g,\epsilon'}(z_i) \rangle_{\hat\C,U,g}
  \prod_{j=1}^l \dd{\rm v}_g(w_j) 
\end{align}
  where ${\bf u}'=({\bf u},{\bf w},{\bf z})$ (with ${\bf w}= (w_1, \cdots, w_l)$) and  ${\bf v}'=({\bf v},{\bf w},{\bf z})$ and ${\boldsymbol \partial}^{2k}  \tilde{G}_{g,\epsilon,\epsilon'}({\bf u},{\bf u}')$ is a shorthand for a product of holomorphic derivatives of  regularised Green functions 
  $${\boldsymbol\partial}^{2 k}  \tilde{G}_{g,\epsilon,\epsilon'}({\bf u},{\bf u}')=\prod_{\alpha,\beta}\partial_{u_\alpha}^{a_{\alpha\beta}}\partial_{u_\beta'}^{b_{\alpha\beta}}\tilde{G}_{g,\epsilon,\epsilon'}(u_\alpha,u_\beta')
  $$ 
  where $u_\alpha\neq u_\beta'$ and  $\bar\partial^{2\tilde k}  G_{\epsilon,\epsilon'}({\bf v},{\bf v}')$ similarly for  anti-holomorphic derivatives. Here $a_{\alpha\beta}, b_{\alpha\beta}\in \{0,1,2\}$ and  $b_{\alpha\beta}\neq 0$ only if $u_\beta'=u_r$ for some $r$; in other words, all the derivatives act on the $\bf u$ variables. Here $2k$ and ${2\tilde k}$ are the total degrees 
      \begin{align*}
  \sum_{\alpha\beta}(a_{\alpha\beta}+ b_{\alpha\beta})=2k,\ \ \sum_{\alpha\beta}(\tilde a_{\alpha\beta}+ \tilde b_{\alpha\beta})=2\tilde k.
  \end{align*}
  Second, a non-trivial fact is the following bound (see \cite[Lemma 3.3]{KRV_DOZZ} or \cite[Lemma 9.2]{GKRV20_bootstrap})
  \begin{align}\label{nontrivial}
  \int_{(\C\setminus U)^m}
   \sup_{\epsilon'} \: \langle \prod_{j=1}^l V_{\gamma,g }(w_j)\prod_{i=1}^m V_{\alpha_i ,g,\epsilon'}(z_i) \rangle_{\hat\C,U,g}   \prod_{j=1}^l \dd{\rm v}_g(w_j) \leq C\langle \prod_{i=1}^m V_{\alpha_i,g }(z_i) \rangle_{\hat\C,U,g}.
\end{align}
 Since the Green function is   smooth at non coinciding points one then concludes the limit \eqref{basicterm0} exists, is continuous in  $({\bf z},{\bf u},{\bf v}) \in \caO^{\rm ext}_{\C,U}$, smooth in  $\bf u$ and $\bf v$ and the limit doesn't depend on the regularisation.  
 Finally, to prove smoothness in the $z_i$'s that belong to $U$   we note that, by Gaussian integration by parts, 
 \begin{align}\label{basicterm1}
\partial_{z_i}  \langle \prod_{j=1}^l V_{\gamma }(w_j)\prod_{j=1}^m V_{\alpha_j}(z_j) \rangle_{\hat\C,U,g}&=\alpha_i(\sum_{j\neq i}\alpha_j\partial_{z_i} \tilde{G}_g(z_i,z_j)+\gamma\sum_{u=1}^l \partial_{z_i} \tilde{G}_g(z_i,w_u)) \langle \prod_{j=1}^l V_{\gamma,g }(w_j)\prod_{i=1}^m V_{\alpha_i,g }(z_i) \rangle_{\hat\C,U,g}\\&-\mu\alpha_i\gamma\int_{\C\setminus U}\partial_{z_i} \tilde{G}_g(z_i,w_{m+1}) \langle \prod_{j=1}^{m+1}V_{\gamma,g }(w_j)\prod_{i=1}^mV_{\alpha_i,g }(z_i) \rangle_{\hat\C,U,g} \dd{\rm v}_g(w_{m+1})\nonumber.
\end{align}
Let us stress here that in the above formula the vertex operators $V_{\alpha_j}(z_j)$ are defined as the limit of regularised vertex operators $V_{\alpha_j, \epsilon}(z_j)$ in the background metric $g$ (and not the flat metric): with this regularisation, one can check that the diagonal term $\E[  \partial_{z_i} \tilde{X}_{g,\epsilon}(z_i) \tilde{X}_{g,\epsilon}(z_i)    ]$ goes to $0$ as $\epsilon$ goes to $0$ hence the sum $\sum_j$ in \eqref{basicterm1} is indeed on $j\neq i$. Iterating this formula we deduce smoothness  again from smoothness of the Green function at non-coinciding points together with the a priori bound \eqref{nontrivial}. \qed

 %%%%%%%%%%%%%%%%%%%%%%%%%%%%%%%%%%%%%%%%%%%%%%%%%%%%%%%%%%%%%%%%%%%%%%%%%%%%%%%%%%%%%%%%%%%%%%%%%%%%%%%%%%%%%%%%%%%%%%%%%%%%%%%%%%%%%%
\subsection{Ward identities}\label{setupward}
%%%%%%%%%%%%%%%%%%%%%%%%%%%%%%%%%%%%%%%%%%%%%%%%%%%%%%%%%%%%%%%%%%%%%%%%%%%%%%%%%%%%%%%%%%%%%%%%%%%%%%%%%%%%%%%%%%%%%%%%%%%%%%%%%%%%%% 
 
 The Ward identities allow to express the correlation functions \eqref{amplitudeSET1} with SET insertions in terms of partial differential operators of correlation functions without SET insertions. They can be formulated  quite generally but since we will use them for a very specific purpose, namely in the computation of the building block amplitudes, we will specialize immediately to that setup (recall subsection \ref{geometricblocks}).
 
We first introduce the notations related to the holes. We will consider 3 holes of two types: $b$ of them will contain  SET insertions and $3-b$ will not  contain any SET insertions. Then we will place the first three components of ${\bf z}\in\C^m$ in each of these holes, which will play a special role (other marked points $z_i$ for i$\geq 4$ will be called later artificial as they will just be added to satisfy the Seiberg bound but their respective weights will be sent to $0$ in the end).  We start with the material needed for the holes that will contain SET. For $b \leq 3$, let  $\psi_i:\D\to \C$ for $i=1,\dots,b$ be conformal maps. We suppose $\caD_i=\psi_i(\D)$ are disjoint. For each $i$, let $A^i_j$ (with $j=1,\dots, k_i$) and $ \tilde A^i_j $ (for $j=1,\dots, \tilde k_i$) be disjoint annuli in $\D$ surrounding a disjoint ball $B^i$ centered at origin. For $t\geq 0$ we set  $\caD_{i,t}=\psi_i(e^{-t}\D)$ and $U^{\rm set}_t=\cup_{i=1}^b \caD_{i,t}$, which will stand for holes containing SET insertions. Let $\caA^i_{j,t}=\psi_i(e^{-t}A^i_j)$ and similarly for $\tilde\caA^i_{j,t}$ and $\caB^i_{t}$. Thus we have a family of conformal  annuli  surrounding a conformal ball in each $\caD_{t,i}$ and all these sets are separated from each other and from the boundary $\partial \caD_{i,t}$ by a distance $\caO(e^{-t})$  (see Figure \ref{figsetup}). We also introduce $ {U}^{\rm noset}_t= \cup_{i=b+1}^3 \caD_{t,i} $ where $ \caD_{t,i}$ is the ball of center $z_i$ and radius $e^{-t}$ for $i=b+1,\dots,3$. The set ${U}^{\rm noset}_t$ is made up of holes containing the $z_i$'s (for $i\leq 3$) that will be surrounded by no SET insertions.

\begin{figure}[h] 
\centering
{\begin{tikzpicture}[xscale=0.7,yscale=0.7] 
\draw[color=red,fill=white] (0,0) circle (3.5);
\draw[color=red,fill=blue!30] (0,0) circle (3.3);
\draw[color=red,fill=white] (0,0) circle (2.9);
\draw[color=red,fill=green!30] (0,0) circle (2.6);
\draw[color=red,fill=white] (0,0) circle (2.1);
\draw[color=red,fill=green!30] (0,0) circle (1.9);
\draw[color=red,fill=white] (0,0) circle (0.7);
\draw[color=red,fill=red!30] (0,0) circle (0.5);
\node[] (C) at (0,0){    $B_i$};
\node[]  (G) at (3,2.2){ $\D$};
\node[]  (H) at (2.2,2.2){ $A_1^1$};
\node[]  (E) at (1.1,2.1){ $\tilde A_1^1$};
\node[]  (F) at (0.6,1.1){ $\tilde A_2^1$};
\draw[domain=0:6.28,samples=100,color=red,fill=white] plot ({10+2.7*cos(\x r)*(1+0.2*cos(10*\x r))},{2.7*sin(\x r)});
\draw[domain=0:6.28,samples=100,color=red,fill=blue!30] plot ({10+2.4*cos(\x r)*(1+0.2*cos(10*\x r))},{2.5*sin(\x r)});
\draw[domain=0:6.28,samples=100,color=red,fill=white] plot ({10+2*cos(\x r)*(1+0.2*cos(10*\x r))},{2.1*sin(\x r)});
\draw[domain=0:6.28,samples=100,color=red,fill=green!30] plot ({10+1.5*cos(\x r)*(1+0.2*cos(10*\x r))},{1.9*sin(\x r)});
\draw[domain=0:6.28,samples=100,color=red,fill=white] plot ({10+1.3*cos(\x r)*(1+0.2*cos(10*\x r))},{1.7*sin(\x r)});
\draw[domain=0:6.28,samples=100,color=red,fill=green!30] plot ({10+1*cos(\x r)*(1+0.2*cos(10*\x r))},{1.5*sin(\x r)});
\draw[domain=0:6.28,samples=100,color=red,fill=white] plot ({10+0.5*cos(\x r)*(1+0.2*cos(10*\x r))},{1.2*sin(\x r)});
\draw[domain=0:6.28,samples=100,color=red,fill=red!30] plot ({10+0.3*cos(\x r)*(1+0.2*cos(10*\x r))},{0.6*sin(\x r)});
\node[]  (A) at (13, 2){ $\caD_{1,t}$};
\node[]  (A) at (11, 2){ $\caA^1_{1,t}$};
\node[]  (A) at (10, 0){ $\caB^1_{t}$};
\node[]  (A) at (10.5, 1){ $\tilde \caA^1_{2,t}$};
\node[]  (A) at (9.3, 1.5){ $\tilde \caA^1_{1,t}$};
\draw[->] (4,0) -- (6,0) node[above,midway]{$\psi_1(e^{-t}\cdot)$};
\end{tikzpicture}
}
\caption{Configuration of disjoint annuli in $\D$ for $k_1=1$ and $\tilde k_1=2$ and their images under $\psi_1(e^{-t}\cdot)$.}
\label{figsetup}
\end{figure}

We then consider the correlation function \eqref{amplitudeSET1} with $U=U_t=U^{\rm set}_t \cup U^{\rm noset}_t $, ${\bf u}=({\bf u}^1,\dots,{\bf u}^b)$  and  ${\bf v}=({\bf v}^1,\dots,{\bf v}^b)$ where $\mathbf{u}^i=(u^i_{1},\dots,u^i_{k_i})$ with $u^i_{j}\in \caA^i_{j,t}$, and similarly for   $\mathbf{v}^i$. Hence ${\bf u}\in\C^k$ with $k=\sum_{i=1}^kk_i$ and  ${\bf v}\in\C^{\tilde k}$ with $\tilde k=\sum_{i=1}^k\tilde k_i$. 
By Proposition \ref{ampSETholomorphic}  it is smooth in $(\bf u,\bf v)\in  \caO^{\bf z}_{\C,U}$  and smooth in the first three variables on   $(z_1,z_2,z_3)\in\prod_{i=1}^b \caB^i_{t} \times \prod_{i=b+1}^3 \caD_{t,i} $ and continuous in $({\bf z},{\bf u},{\bf v}) \in \caO^{\rm ext}_{\C,U}$. The Ward identity will involve derivatives in all the variables $\bf u,v,z$ and we will consider these in the distributional sense in the variables $z_i$ for $i=4,\dots,m$ \footnote{Using the method of \cite{Joona2019} one can in fact show smoothness in all the variables $  z_i$ at noncoinciding points for all $t$ but to keep the discussion simple we will not embark on this.}.  For this purpose, we pick disjoint balls $B_i \subset U_{t_0}^c$ for some $t_0 \geq 0$ and centered at $z_i$, $i=4,\dots, m$. Take  $f_i\in C_0^\infty( B_i)$.  When we  will smear  \eqref{amplitudeSET1}  with test functions in these variables, we denote the result  (by slight abuse of notation) by  $\langle T_g( \mathbf{u})\bar T_g( \mathbf{v}) \prod_{i=1}^3V_{\alpha_i,g}(z_i)\prod_{i=4}^mV_{\alpha,g}(f_i)\rangle_{\hat\C,U_t,g}$.
Finally, indexing the ${\bf u}$ as ${\bf u}=(u_1,\dots,u_k)$ we will write ${\bf u}^{(l)}$ for the vector  in $\C^{k-1}$ with the same entries as  ${\bf u}$ with $l$-th entry removed and we will sometimes iterate this procedure to remove further entries and write   ${\bf u}^{(l,l')}$ and so on. The same notations are used for the vector ${\bf v}$.  
 
 With these definitions we have:
 \begin{proposition}[{\bf Ward Identity}]\label{propward}
 Recall that $c_{\rm L}=1+6Q^2$ and the metric we consider is $g=e^{\omega}|dz|^2$. Then 
 \begin{align}\nonumber
 \langle T_g( \mathbf{u})\bbar T_g( \mathbf{v})\prod_{i=1}^mV_{\alpha_i,g}(z_i)\rangle_{\hat\C,U_t,g} &= \frac{1}{2}\sum_{l=1}^{k-1}\frac{c_{\rm L}}{(u_k-u_l)^4} \langle  T_g(\mathbf{u}^{(k,l)})\bbar T_g({ \mathbf{v}})\prod_{i=1}^mV_{\alpha_i,g}(z_i)\rangle_{\hat\C,U_t,g}
 \\
 &+ \sum_{l=1}^{k-1}\Big(\frac{2}{(u_k-u_l)^2} +\frac{1}{(u_k-u_l) }\partial_{u_l} \Big)
  \langle  T_g(\mathbf{u}^{(k)})\bbar T_g({ \mathbf{v}})\prod_{i=1}^mV_{\alpha_i,g}(z_i)\rangle_{\hat\C,U_t,g}\nonumber
   \\
 &+\sum_{i=1}^{m}\Big(\frac{\Delta_{\alpha_i}}{(u_k-z_i)^2} +\frac{\partial_{z_i}+\Delta_{\alpha_i} \partial_{z_i}\omega }{(u_k-z_i) }\Big)
  \langle  T_g(\mathbf{u}^{(k)})\bbar T_g({ \mathbf{v}})\prod_{i=1}^mV_{\alpha_i,g}(z_i)\rangle_{\hat\C,U_t,g}\nonumber
  \\
  &+  \frac{i\mu}{2}\oint_{\partial{U}_t}\frac{1}{  u_k-  w}\langle   T_g(\mathbf{u}^{(k)})  \bbar T_g({ \mathbf{v}})  V_{\gamma,g }(w)\prod_{i=1}^mV_{\alpha_i,g}(z_i) \rangle_{\hat\C,U_t,g} g(w) \dd\bar w\nonumber
 \end{align}
 and  
 \begin{align}\nonumber
 \langle T_g( \mathbf{u})\bbar T_g( \mathbf{v})\prod_{i=1}^mV_{\alpha_i,g}(z_i)\rangle_{\hat\C,U_t,g} &= \frac{1}{2}\sum_{l=1}^{\tilde k-1}\frac{c_{\rm L}}{(\bar v_{\tilde k}-\bar v_l)^4} \langle  T_g(\mathbf{u})\bbar T_g({ \mathbf{v}^{(\tilde k,l)} })\prod_{i=1}^mV_{\alpha_i,g}(z_i)\rangle_{\hat\C,U_t,g}
 \\
 &+ \sum_{l=1}^{\tilde k-1}\Big(\frac{2}{(\bar v_{\tilde k}-\bar v_l)^2} +\frac{1}{(\bar v_{\tilde k}-\bar v_l) }\partial_{\bar v_l} \Big)
  \langle  T_g(\mathbf{u})\bbar T_g({ \mathbf{v}^{(\tilde k)}})\prod_{i=1}^mV_{\alpha_i,g}(z_i)\rangle_{\hat\C,U_t,g}\nonumber
   \\
 &+\sum_{i=1}^{m}\Big(\frac{\Delta_{\alpha_i}}{(\bar v_{\tilde k}-\bar z_i)^2} +\frac{\partial_{\bar z_i}+\Delta_{\alpha_i}\partial_{\bar z_i}\omega }{(\bar v_{\tilde k}-\bar z_i) } \Big)
  \langle  T_g(\mathbf{u})\bar T_g({ \mathbf{v}}^{(\tilde k)})\prod_{i=1}^mV_{\alpha_i,g}(z_i)\rangle_{\hat\C,U_t,g}\nonumber
  \\
  &+  \frac{i\mu}{2}\oint_{\partial{U}_t}\frac{1}{ \bar  v_{\tilde k}-  \bar w}\langle   T_g(\mathbf{x})  \bbar T_g({ \mathbf{v}}^{(\tilde k)})  V_{\gamma ,g}(w)\prod_{i=1}^mV_{\alpha_i,g}(z_i) \rangle_{\hat\C,U_t,g} g(w) \dd w\nonumber
 \end{align}
The derivatives in the variables $z_{4}, \dots, z_m$ are in the sense of distributions and the derivatives in $z_{1}, \dots, z_3$ are in the classical sense.

  \end{proposition}   

The proof is a computationally  tedious extension of the one in \cite{GKRV20_bootstrap} and we defer it to Appendix \ref{wardproof}.

\vskip 2mm

Iterating the  Ward identity allows us to express the correlation function with SET insertions in terms of derivatives of correlation function without SET insertions.  To state the result we need some notation.  
 Given $\tilde {\bf u}\in\C^p, \tilde{\bf v}\in \C^{p'}$ with all $\tilde u_i, \tilde v_j$ distinct and  ${\bf q}=
(q_{ij})_{1\leq i<j\leq p},{\bf n}=
(n_{ij})_{1\leq i\leq p,1\leq j\leq \tilde p}$ where $q_{ij}\geq 0$  and $n_{ij}\geq 0$ are integers, let
\begin{align}\label{defpn}
p_{\bf q}(\tilde {\bf u}):=\prod_{i<j} (\tilde u_i-\tilde u_j)^{-q_{ij}},\ \ p_{\bf n}(\tilde {\bf u},\tilde {\bf v}):=\prod_{i, j} (\tilde u_i-\tilde v_j)^{-n_{ij}}.
\end{align}
We denote the homogeneity degrees by $|{\bf q}|=\sum _{i<j}q_{ij}$ and $|{\bf n}|=\sum _{i, j}n_{ij}$.
Also, given   ${\bf r}=
(r_{i})_{1\leq i\leq m}$ with $r_i\geq 0$ integers, we define   the differential operator $\partial^{\bf r}_{\bf z}:=\prod_{i=1}^m\partial_{z_i}^{r_i}$ and $|{\bf r}|=\sum_i r_i$. Then we have

\begin{proposition}\label{wardite}
The following identity holds
 \begin{align}
 \prod_{i=1}^m e^{\Delta_{\alpha_i} \omega (z_i)} \langle T_g( \mathbf{u})\bbar T_g( \mathbf{v})\prod_{i=1}^mV_{\alpha_i,g}(z_i)\rangle_{\hat \C,g,U_t}
  &= D({\bf u},{\bf z}) D({\bf \bar v},{\bf \bar z}) \left (    \prod_{i=1}^m e^{\Delta_{\alpha_i} \omega (z_i)} \langle  \prod_{i=1}^mV_{\alpha_i,g }(z_i) \rangle_{\hat \C,g,U_t}\right )+
 \epsilon_t{(\bf u,v,z})
  \label{wrdd0}
 \end{align}
 where  
 \begin{align*}
  D({\bf u},{\bf z})
  =\sum_{{\bf q},{\bf n}, {\bf r}}  a_{\bf q,n, r} p_{\bf q}({\bf u}) p_{\bf n}({\bf u},{\bf z})\partial^{\bf r}_{\bf z} ,
\end{align*}
with $ |{\bf q}|+ |{\bf n}|+|{\bf r}|=2k$ (and similarly for the antiholomorphic part $D({\bf \bar v},{\bf \bar z})$) and the coefficients $a_{\bf q,n, r}\in\R$ are polynomials in the conformal weights $\Delta_{\alpha_i}$. The right hand side of the above relation is to be considered in the following sense: the derivatives in $z_i$ for $i \leq 3$ are in the classical sense and the derivatives in $z_i$ for $i \geq 4$ are in the sense of distributions. The remainder satisfies uniformly in ${\bf u,v}, z_1, \dots,z_3 $ (satisfying the previous conditions) 
\begin{align}\label{reminder}
|  \int_{\C^{m-3}}    \epsilon_t({\bf u,v,z})  f_{4}(z_{4})  \dots f_m(z_m)  \dd z_{4} \dots \dd z_m  |\leq Ce^{(2( k+\tilde k-1)+\gamma\alpha)t}  \prod_{i=4}^m\|f_i\|_{k,\tilde k}
\end{align}
where  $\alpha=\max_{i\leq 3}\alpha_i$ and $\|f\|_{k,\tilde k}=\sum_{a=0}^k\sum_{c=0}^{\tilde k}\|\partial^a_z \partial^c_{\bar{z}} f(z) \|_\infty$.

\end{proposition}

\proof
We iterate the Ward identity $k+\tilde k$ times at each step applying  it only to the terms that do not have the contour integral. This way we end up with terms having no 
 SET insertions and contour integral terms which form the remainder $\epsilon_t$.  The former terms take the form of the first term in the right hand side of identity \eqref{wrdd0}. In this iteration procedure where we iterate the identity of Proposition \ref{propward} to the terms which are not under the form of a contour integral, the remainder terms appear as integrals with respect to one contour integral, i.e.
 \begin{equation*}
  \epsilon_t({\bf u,v,z})=  \frac{i}{2}\oint_{\partial{U}_t}\frac{1}{  u_k-  w}\langle   T_g(\mathbf{u}^{(k)})  \bbar T_g({ \mathbf{v}})  V_{\gamma,g }(w)\prod_{i=1}^mV_{\alpha_i,g}(z_i) \rangle_{\hat\C,U_t,g} g(w) \dd\bar w\nonumber+ \cdots
 \end{equation*}
In the sequel of the proof, for simplicity we suppose that $\omega(z_i)=0$ for all $i$.
To write the remainder smeared with $f_i$ let us introduce the operators
\begin{align*}%\label{}
\caD_{u_l,i}=\frac{\Delta_{\alpha_i}}{(u_l-z_i)^2} -\partial_{z_i}\frac{1}{(u_l-z_i) }
\end{align*}
i.e. the operator such that for $F,G$ real valued functions $\int_\C( \frac{\Delta_{\alpha_i}}{(u_l-z_i)^2} +\frac{1}{(u_l-z_i) }\partial_{z_i})F G \dd z_i=\int_\C F \caD_{u_l,i} G \dd z_i $ where $\dd z_i$ denotes the standard Lebesgue measure (recall that $\frac{\Delta_{\alpha_i}}{(u_l-z_i)^2} +\frac{1}{(u_l-z_i) }\partial_{z_i}$ is the operator appearing in the Ward identity). In fact, the remainder terms will involve functions $g_i$ of the form
\begin{equation}\label{gijoe}
g_i=(\prod_{j\in J_i} \caD_{u_j,i})(\prod_{j\in\tilde J_i} \bbar\caD_{v_j,i})f_i.
\end{equation}  
where $J_i$ is a subset of $\lbrace 1, \dots, k \rbrace$ and $\tilde J_i$ is a subset of $\lbrace 1, \dots, \tilde k \rbrace$.
In this context, the remainder   $\epsilon_t({\bf u,v,z})$ integrated against $ f_{4}(z_{4})  \dots f_m(z_m)  \dd z_{4} \dots \dd z_m$  is a sum of terms of the form 
\begin{align}\label{expansion}
\oint_{\partial{\caD_{t,j}}}\frac{1}{(\bar u_l-  \bar w)^{1+\tau}} p({\bf u},{\bf v},{\bf z'})
\partial^{\bf m}_{\bf u'} \partial^{\bf n}_{\bar{\bf v}'}\partial^{\bf p}_{\bf z'}\partial^{\bf q}_{\bf \bar z'}\langle  T_g(\mathbf{u}')  \bbar T_g({ \mathbf{v}'}) V_{\gamma,g}(w) \prod_{i=1}^3 V_{\alpha_i,g}(z_i)\prod_{i=4}^mV_{\alpha_i}(g_i)\rangle_{\hat\C,U_t,g}   g(w) \dd w
\end{align}
or integration with respect to $\frac{1}{  u_l-  w}\dd\bar w$ instead where $\tau$ is an integer and $g_i$ is of the form \eqref{gijoe} with non intersecting $J_i$ and non intersecting $\tilde J_i$ (recall that the notation $V_{\alpha_i}(g_i)$ stands for integration against $g_i(z_i) \dd z_i$). Here $j=1,\dots,m$, ${\bf z'}=(z_1,\dots, z_b)$ and  ${\bf u'}\in\C^a,{\bf v'}\in\C^{\tilde a}$ are subsets of the coordinates of $\bf u$  and ${\bf v}$. The term $p({\bf u},{\bf v},{\bf z'})$ is a product 
\[
p({\bf u},{\bf v},{\bf z'})=p_{\bf k}({\bf u})p_{\bf l}(\bar{\bf v})p_{\bf r}({\bf u},{\bf z'})p_{\bf s}(\bar{\bf v},\bar{\bf z}')
\] 
The total homogeneity degree satisfies
\begin{equation}\label{homo}
\tau+ |{\bf k}|+ |{\bf l}|+|{\bf r}|+|{\bf s}|+|{\bf m}|+|{\bf n}|+|{\bf p}|+|{\bf q}|+2a +2 \tilde a+2\sum_{i=b+1}^m(|J_i|+|\tilde J_i|)= 2(k+\tilde k-1).
\end{equation}
 For the estimates  note that  the differences $|u_i-u_j|$, $|u_i-z_l|$, $|v_i-v_j|$, $|v_i-z_l|$  are comparable to $ e^{-t}$. Hence 
\[
|p({\bf u},{\bf v},{\bf z'})|\leq Ce^{(|{\bf k}|+ |{\bf l}|+|{\bf r}|+|{\bf s}|)t}.
\]
Estimation of the derivative of $\langle  T_g(\mathbf{u}')  \bar T_g({ \mathbf{v}'}) V_{\gamma,g}(w) \prod_{i=1}^3 V_{\alpha_i,g}(z_i)\prod_{i=4}^mV_{\alpha_i}(g_i)\rangle_{\hat\C,U_t,g} $ which appears in \eqref{expansion} is completely straightforward using  expansion into terms of the form \eqref{basicterm0} and then formula \eqref{basicterm1} ($ \tilde{G}_g$ was introduced in the proof of proposition \ref{ampSETholomorphic})
once we realise that all the Green functions that will enter in these expressions involve points whose distance is bounded from below by $ce^{-t}$. First, by using integration by parts on the SET insertions we expand in terms of the form \eqref{basicterm0} and then use formula \eqref{basicterm1} to get rid of the $z$ derivatives. As a result the term 
\[\partial^{\bf m}_{\bf u'} \partial^{\bf n}_{\bar{\bf v}'}\partial^{\bf p}_{\bf z'}\partial^{\bf q}_{\bar{\bf z}'}\langle  T_g(\mathbf{u}')  \bbar T_g({ \mathbf{v}'}) V_{\gamma,g}(w) \prod_{i=1}^3 V_{\alpha_i,g}(z_i)\prod_{i=4}^mV_{\alpha_i}(g_i)\rangle_{\hat\C,U_t,g}\] 
is a linear combination of terms 

\begin{align}\label{basicterm11}
\int_{(\C\setminus U_t)^p}
 \prod_{\alpha,\beta}\partial^{a_{\alpha\beta}} \tilde{G}_g(\tilde{u}_\alpha,\tilde{u}_\beta)
 \prod_{\alpha,\beta}\bar\partial^{b_{\alpha\beta} } \tilde{G}_g (\tilde{v}_\alpha,\tilde{v}_\beta)
 \langle  V_{\gamma,g}(w)\prod_{i=1}^pV_{\gamma,g}(w_i) \prod_{i=1}^3V_{\alpha_i,g}(z_i)\prod_{i=4}^mV_{\alpha_i,g}(g'_i)\rangle_{\hat\C,U_t,g} \prod_{i=1}^p \dd{\rm v}_g(w_{i})  
\end{align}
where $\tilde{u}_\alpha,\tilde{u}_\beta, \tilde{v}_\alpha,\tilde{v}_\beta$ belong to $({\bf u',v',z',w},w)$ (with ${\bf w}= (w_1, \dots , w_p)$)  so that $|\tilde{u}_\alpha-\tilde{u}_\beta|, |\tilde{v}_\alpha-\tilde{v}_\beta|\geq ce^{-t}$. Furthermore
$g_i'=\partial_{\bf u'}^{\bf e}\bar\partial_{\bf v'}^{\bf d}g_i$ for some ${\bf e}$,  ${\bf d}$ and it is readily verified that
\begin{equation}  \label{boundgf}%\label{}
\|g_i'\|_\infty\leq C\|f_i\|_{k,\tilde k}.
\end{equation}
Indeed there are no more than $k$ $z_i$-derivatives and no more than $\tilde k$ $\bar z_i$-derivatives acting on $f_i$ and all the rational functions occurring are bounded by $\caO(1)$ since they involve differences of coordinates inside $U_t$ and at $\mc{O}(1)$ distance from $U_t$. Using the a priori bound \eqref{nontrivial} to control the $(w_1, \dots, w_p)$ integral we end up with the bound (recall that ${\bf u'}\in\C^a,{\bf v'}\in\C^{\tilde a}$)
\begin{align*}
|\eqref{basicterm11}|\leq Ce^{(|{\bf m}|+|{\bf n}| +|{\bf p}|+|{\bf q}|+2a+2\tilde a)t}\langle  V_{\gamma,g}(w) \prod_{i=1}^3V_{\alpha_i}(z_i)\prod_{i=4}^mV_{\alpha_i}(|g'_i|)\rangle_{\hat\C,U_t,g}.
\end{align*}

 Gathering the previous considerations we get that the absolute value of \eqref{expansion} is bounded by 
 \begin{equation}\label{BoundBound}
 C e^{(|{\bf k}|+ |{\bf l}|+|{\bf r}|+|{\bf s}|+|{\bf m}|+|{\bf n}| +|{\bf p}|+|{\bf q}|+2a+2\tilde a)t}  \oint_{\partial{\caD_{t,j}}}   \Big|\frac{1}{\bar u_l-  \bar w}\Big|^{1+\tau}   \langle  V_{\gamma,g}(w) \prod_{i=1}^3 V_{\alpha_i}(z_i)\prod_{i=4}^mV_{\alpha_i}(|g'_i|)\rangle_{\hat\C,U_t,g}  |dw|
 \end{equation}
 We now bound the contour integral in \eqref{BoundBound}. We use the following upper bound for the correlation function (as can be seen using the explicit expression for the correlations: see \cite[Theorem 3.4]{GKRV20_bootstrap}) 
\begin{align}\label{upperapri}
\langle V_{\gamma,g}(w) \prod_{i=1}^mV_{\alpha_i,g}(z_i)\rangle_{\hat\C,U_t,g}\leq C\prod_{i=1}^m |w-z_i|^{-\gamma\alpha_j}\prod_{1 \leq i<j \leq m}|z_i-z_j|^{-\alpha_i\alpha_j}.
\end{align}
When $w \in \partial{\caD_{t,j}}$ we have $|w-z_j|\sim e^{-t}$ whereas the other distances in the bound \eqref{upperapri} are $\caO(1)$ hence 
\begin{equation*}
 \oint_{\partial{\caD_{t,j}}}   \Big|\frac{1}{\bar u_l-  \bar w}\Big|^{1+\tau}   \langle  V_{\gamma,g}(w) \prod_{i=1}^3 V_{\alpha_i}(z_i)\prod_{i=4}^mV_{\alpha_i}(|g'_i|)\rangle_{\hat\C,U_t,g}  |dw|   \leq C e^{(u+\gamma \alpha_j) t}  \prod_{i=4}^m \|g_i'\|_\infty \leq C  e^{(\tau+\gamma \alpha_i )t}   \prod_{i=4}^m \|f_i\|_{k,\tilde k}
\end{equation*}
where we have used \eqref{boundgf} for the last inequality. This yields the result.\qed

\begin{remark}
We are going to apply this result in the next section to the case where $\alpha$ is very negative so that   the contribution of $\epsilon_t$ tends to 0 as $t\to\infty$. 
\end{remark}

So far we have proven differentiability of the correlation functions $\langle \prod_{i=1}^m V_{\alpha_i,g}(z_i)\rangle_{\hat\C,U_t,g}
$ in the variables $z_1,\dots,z_3$ in a region that shrinks to points as $t\to\infty$. However, if the corresponding weights $\alpha_i$ are negative enough   differentiability  can be established in the limit $t$ going to infinity. 
 \begin{lemma}\label{difflemma}
Let  $N<2-\gamma\alpha$ where $\alpha=\max_{i=1,\dots,3}\alpha_i$ and fix $t_0 \geq 0$. For all $t \geq 0$, the correlation functions $\langle \prod_{i=1}^mV_{\alpha_i,g}(z_i)\rangle_{\hat\C,U_t,g}
$  are  $C^N$ in $z_i \in \caD_{i,t_0}$ for $i=1 \dots 3$ and $C^0$  in $(z_{4},\dots, z_n)$ in the region $z_i\in \C\setminus \bar U_{t_0}$ with $z_i\neq z_j$ and they converge uniformly on compact subsets of the aforementionned region together with all these derivatives as $t\to \infty$ to $\langle \prod_{i=1}^m V_{\alpha_i,g}(z_i)\rangle_{\hat\C,g}
$.
\end{lemma}
 \proof
 We will show that $\langle \prod_{i=1}^m V_{\alpha_i,g}(z_i)\rangle_{\hat\C,g}$ is $C^N$ in $z_i \in \caD_{i,t_0}$  for $i=1 \dots 3$ for $N<2-\gamma\alpha$.  As an example let us show the existence of 
$\partial_{z_1}^{N}\langle \prod_{i=1}^mV_{\alpha_i,g,\epsilon}(z_i)\rangle_{\hat\C,g}$ for $N<2-\gamma\alpha_1$. Let us consider the regularised correlation function $\langle \prod_{i=1}^mV_{\alpha_i,g,\epsilon}(z_i)\rangle_{\hat\C,U,g}$ where we have replaced $V_{\alpha_i,g}$ with the regularised $V_{\alpha_i,g,\epsilon}$ (here we smooth the GFF with $\rho_\epsilon$ rather than with the background metric in the definition of $V_{\alpha_i,g,\epsilon}$; this only affects the correlations up to a smooth prefactor). These are smooth in $z_i$ and we can use the Gaussian integration by parts formula \eqref{basicterm1} (with regularised Green functions) as in deriving  \eqref{basicterm11}. We get that 
$\partial_{z_1}^{N}\langle \prod_{i=1}^mV_{\alpha_i,g,\epsilon}(z_i)\rangle_{\hat\C,U_t,g}$ is a linear combination of terms 
\begin{equation}\label{partialterms} 
 \prod_{j=1}^m\partial^{a_{j}} \tilde{G}_{g,\epsilon,\epsilon}(z_1,z_j)
\int_{\C^p}
 \prod_{i=1}^p\partial^{b_{i}} \tilde{G}_{g,\epsilon}(z_{1},w_{i})
 \langle \prod_{i=1}^pV_{\gamma,g}(w_i) \prod_{i=1}^mV_{\alpha_i,g,\epsilon}(z_i)\rangle_{\hat\C,g}   \prod_{i=1}^p \dd{\rm v}_g(w_i)
\end{equation} 
  where $\tilde{G}_{g,\epsilon,\epsilon}=\rho_\epsilon\ast \tilde{G}_g\ast \rho_\epsilon$ and $\tilde{G}_{g,\epsilon}=\rho_\epsilon\ast \tilde{G}_g$ are smooth functions ($\tilde{G}_g$ was introduced in  \eqref{tildeG}) and $\sum_{j=1}^m a_j+\sum_{i=1}^p b_i=N$. The prefactor $ \prod_{j=1}^m\partial^{a_{j}} \tilde{G}_{g,\epsilon,\epsilon}(z_1,z_j)$ converges to a smooth limit (including the $j=1$ term). Outside of a ball $B$ around $z_1$ the $\partial^{b_{i}} \tilde{G}_{g,\epsilon}(z_{1},w_{i})$ terms are bounded and one has the following bound (see \cite[Lemma 3.3]{KRV_DOZZ} or \cite[Lemma 9.2]{GKRV20_bootstrap})
  \begin{align}\label{nontrivialagain}
  \int_{\C^m}
   \sup_{\epsilon'} \: \langle \prod_{j=1}^l V_{\gamma,g }(w_j)\prod_{i=1}^m V_{\alpha_i ,g,\epsilon'}(z_i) \rangle_{\hat\C,g}   \prod_{i=1}^p \dd{\rm v}_g(w_i) \leq C\langle \prod_{i=1}^m V_{\alpha_i,g }(z_i) \rangle_{\hat\C, g}.
\end{align}
Hence we can concentrate on the most singular part of the integral where all the $w_i$ are in the ball $B$. 
  
 For this consider, for $\pi\in S_p$ a permutation of $\lbrace1, \dots, p \rbrace$, the indicator function $\chi_\pi({\bf w})$ of the set 
 $\{|z_1-w_{\pi(1)}|\leq  |z_1-w_{\pi(2)}|\leq\dots\leq |z_1-x_{\pi(p)}|\}$.
  \begin{align*}
\Big|\int_{B^p}\chi_\pi({\bf w})& \prod^p_{l=1}\partial^{b_{l}} \tilde{G}_{g,\epsilon}(z_{1},w_{l})
 \langle \prod_{i=1}^pV_{\gamma}(w_i) \prod_{i=1}^mV_{\alpha_i,g,\epsilon}(z_i)\rangle_{\hat\C,g}  \prod_{i=1}^p \dd{\rm v}_g(w_i)\Big|
 \\
 &\leq C\int_{B^p}(1+(|z_1-w_{\pi(1)}|+\epsilon)^{-N})   \langle \prod_{i=1}^pV_{\gamma,g}(w_i) \prod_{i=1}^mV_{\alpha_i,g,\epsilon}(z_i)\rangle_{\hat\C,U_t,g}   \prod_{i=1}^p \dd{\rm v}_g(w_i) \\
 &\leq C\int_{B}(1+(|z_1-w_{\pi(1)}|+\epsilon)^{-N}) \Big ( \int_{B^{p-1}} \langle \prod_{i=1}^pV_{\gamma,g}(w_i) \prod_{i=1}^mV_{\alpha_i,g,\epsilon}(z_i)\rangle_{\hat\C, g}   \prod_{i \not = \pi(1)}   \dd{\rm v}_g(w_i)  \Big )  \dd{\rm v}_g(w_{\pi(1)})    \\
 &\leq C\int_{B}(1+(|z_1-w_{\pi(1)}|+\epsilon)^{-N}\langle V_{\gamma,g}(w_{\pi(1)}) \prod_{l=1}^mV_{\alpha_l,g,\epsilon}(z_l)\rangle_{\hat\C, g} {\rm dv}_g(w_{\pi(1)})  
\end{align*}
where we used 
$$
|\partial^{m}_{z_i}\tilde{G}_{g,\epsilon}(z_1,w)|\leq C(1+(|z_1-w|+\epsilon)^{-m})
$$
and the following bound which is a regularised version of \eqref{nontrivialagain} 
 \begin{equation*}
  \int_{\C^m}
   \langle \prod_{j=1}^l V_{\gamma,g }(w_j)\prod_{i=1}^m V_{\alpha_i ,g,\epsilon}(z_i) \rangle_{\hat\C,g}   \prod_{i=1}^p \dd{\rm v}_g(w_i) \leq C\langle \prod_{i=1}^m V_{\alpha_i,g, \epsilon }(z_i) \rangle_{\hat\C, g}.
\end{equation*}
to do the $w_i$ integrals with $i\neq\pi(1)$.
Finally the following regularised version of the bound \eqref{upperapri} is true for $w_{\pi(1)}$ around $z_1$
\[
\langle V_\gamma(w_{\pi(1)}) \prod_{l=1}^mV_{\alpha_l,g,\epsilon}(z_l)\rangle_{\hat\C,g}\leq C (|w_{\pi(1)}-z_1|+\epsilon)^{-\gamma \alpha_1}.
\]
Hence one can take the limit as $\epsilon$ goes to $0$  of $\partial_{z_1}^{N}\langle \prod_{i=1}^mV_{\alpha_i,g,\epsilon}(z_i)\rangle_{\hat\C,g}$. The other terms in the expansion go obviously the same way.
 \qed

 %%%%%%%%%%%%%%%%%%%%%%%%%%%%%%%%%%%%%%%%%%%%%%%%%%%%%%%%%%%%%%%%%%%%%%%%%%%%%%%%%%%%%%%%%%%%%%%%%%%%%%%%%%%%%%%%%%%%%%%%%%%%%%%%%%%%%%%%%%%%%%%%%%%%%%%%%%%%%%%%%%%%%%%%%%%%%%%%%%%%%%%%%%%%%%%%%%%%%%%%%%%%%%%%%%%%%%%%%%%%%%%%%%%%%%%%%%%%%%%%%%%%%%%%%%%%%%%%%%%%%%%%%%

%

 %%%%%%%%%%%%%%%%%%%%%%%%%%%%%%%%%%%%%%%%%%%%%%%%%%%%%%%%%%%%%%%%%%%%%%%%%%%%%%%%%%%%%%%%%%%%%%%%%%%%%%%%%%%%%%%%%%%%%%%%%%%%%%%%%%%%%%%%%%%%%%%%%%%%%%%%%%%%%%%%%%%%%%%%%%%%%%%%%%%%%%%%%%%%%%%%%%%%%%%%%%%%%%%%%%%%%%%%%
\section{Amplitudes of building blocks}\label{sec:computingamplitudes}
%%%%%%%%%%%%%%%%%%%%%%%%%%%%%%%%%%%%%%%%%%%%%%%%%%%%%%%%%%%%%%%%%%%%%%%%%%%%%%%%%%%%%%%%%%%%%%%%%%%%%%%%%%%%%%%%%%%%%%%%%%%%%%%%%%%%%%%%%%%%%%%%%%%%%%%%%%%%%%%%%%%%%%%%%%%%%%%%%%%%%%%%%%%%%%%%%%%%%%%%%%%%%%%%%%%%%%%%%%%%%%%%%%%%%%%%%%%%%%%%%%%%%%%%%%%%%%%%%%%%%%%%%%
In this section we will use the Ward identities established in Proposition \ref{propward} to compute the  amplitudes of the building blocks that are involved in the construction of conformal blocks. In particular we need  to evaluate these amplitudes at the eigenfunctions $\Psi_{Q+ip,\nu,\tilde\nu}$ of the LCFT Hamiltonian and establish the factorisation of the result in its dependence on the two Virasoro algebra representation labels $\nu,\tilde\nu$ as stated in Theorem \ref{pantDOZZ} below, which is the main result of this section.

\subsection{Holomorphic factorisation}
%%%%%%%%%%%%%%%%%%%%%%%%%%%%
Let $\{\caP,{\bf x}^{\rm mp},\boldsymbol{\zeta}\}$ be a complex building block in the sense of Section \ref{geometricblocks}. 
Here $\partial\caP=\cup_{i=1}^b\caC_i$, $b\in \{1,2,3\}$ and $x_i$, $i\in \{b+1,\dots, 3\}$, are  distinct points in the interior of $\caP$ collected in the vector ${\bf x}^{\rm mp}$ (the superscript ${\rm mp}$ refers to ``marked point''). The surface $\caP$  is equipped with a complex structure $J$ and $\boldsymbol{\zeta}=(\zeta_1,\dots,\zeta_b)$ are analytic parametrisations of   $\partial\caP$.
 We denote often for brevity  $\{\caP,{\bf x}^{\rm mp},\boldsymbol{\zeta}\}$ by $\caP$. %and labels $\alpha_j<Q$ attached to the marked points. 
 We fix also an admissible metric $g_\caP$ on $\caP$. Recall also that we set $\sigma_i=-1$ if the boundary component $\caC_i$ has outgoing orientation,  i.e. if its orientation induced by the parametrisation agrees with that inherited from the orientation of $\caP$, and  we set $\sigma_i=1$ if $\caC_i$  has incoming orientation. 
 We collect in the vector $\boldsymbol{\alpha}^{\rm mp}$ the weights $\alpha_i<Q$, $i=b+1,\dots, 3$, attached to the marked points in ${\bf x}^{\rm mp}$. 
  The corresponding  amplitude    is a function $\caA_{\caP, g_\caP,{\bf x}^{\rm mp},\boldsymbol{\alpha}^{\rm mp},\boldsymbol{\zeta}}:H^s(\T)^b\to \C$. Let us use the notation
 \begin{align}\label{ampf}
 \big\cjg \caA_{\caP, g_\caP,{\bf x}^{\rm mp},\boldsymbol{\alpha}^{\rm mp},\boldsymbol{\zeta}},\otimes_{j=1}^{b}f_j\big\cjd_{\mc{H}}:=\int \caA_{\caP, g_\caP,{\bf x}^{\rm mp},\boldsymbol{\alpha}^{\rm mp},\boldsymbol{\zeta}}(\tilde{\boldsymbol{\varphi}})\big(\prod_{j=1}^b \bbar{f_j(\tilde\varphi_j)}\big) \dd \mu_0^{\otimes_b}(  \tilde{\boldsymbol{\varphi}}) .
\end{align}

As was explained in Section \ref{Section:blocks}, we need to
 evaluate \eqref{ampf} with  $f_j={\bf C}^{(1+\sigma_j)/2}\Psi_{Q+ip_j,\nu_j,\tilde\nu_j}$  where ${\bf C}$ is the complex conjugation.   
To state our result, we need some more material related to the uniformization  of the surface $\caP$. As explained in Section \ref{subsubsec:plumbingparameters} and Section \ref{sub:cfb}, we can glue a disc $\mc{D}_j\simeq \D$  to each boundary curve $\caC_j$ using the parametrisation $\zeta_j:\T\to \mc{C}_j$
 so that the resulting surface $\hat\caP$ is conformal  to $\hat\C$ with its standard complex structure. 
 Denote by $x_j\in \hat\caP$ the center of $\mc{D}_j$, $j=1,\dots,b$ and collect these entries in ${\bf x}^{\rm di}=(x_j)_{j=1,\dots,b}$ (the superscript ${\rm di}$ refers to ``disk insertion'', i.e. to emphasize that it collects the added $x_j$'s lying in the disks glued to the pant to obtain $\hat{\C}$). 
 We will also write $\hat{{\bf x}}:=(x_1,x_2,x_3)=({\bf x}^{\rm di}, {\bf x}^{\rm mp})$ for the set of $3$ punctures of the punctured sphere $\hat{\mc{P}}$
 (the hat refers to $\hat{\mc{P}}$).

Therefore  $\hat\caP$ has an atlas $\{(U_i,\rho_i)\}_{i=1}^N$, $N>b$ with $\D\subset \rho_j(U_j)$ for $j=1,\dots,b$  and there is a biholomorphic map $\psi:\hat\caP\to\hat\C$ so that $\psi(x_j)=z_j$ for $j=1,\dots,3$ and $z_1,z_2,z_3\in\hat\C$  are  three prescribed points which we may choose at will. For later technical reasons, we will take all  $z_1,z_2,z_3$ in  the unit disk $\D$.  Similarly to the notations used for the vectors associated to the $x_j$'s, we will use the notations ${\bf z}^{\rm di}=(z_{1},\dots,z_b)$, ${\bf z}^{\rm mp}=(z_{b+1},\dots,z_3)$ and  $\hat{{\bf z}}=(z_{1},z_2,z_3)$. We also define $\psi_j:=\psi\circ \rho_j^{-1}:\D\to\psi(\mc{D}_j)$ for $j=1,\dots,b$, hence $\psi_j(0)=z_j$. Furthermore $\hat\C\setminus\psi(\caP)=\cup_{j=1}^b\psi(\caD_j)$ and $\psi(\caD_j)$ are simply connected regions in ${\C}$ with analytic boundaries. Define the conformal radius of $\psi(\mc{D}_j)$ by
\begin{equation}\label{defrad} 
{\rm Rad}_{z_j}(\psi(\mc{D}_j)):=|\psi_j'(0)|.
\end{equation}
Finally we consider the function $\hat\omega:\hat\C\to\R$ defined by $\psi_* g_{\mc{P}}=e^{\hat\omega}g_{\rm dozz}$ where $g_{\rm dozz}=(|z|\vee 1)^{-4}|dz|^2$  is the DOZZ metric.  Our purpose is to prove the following statement, which will be sometimes referred to as the  {\it holomorphic factorisation}:

\begin{theorem}[\textbf{Holomorphic factorisation}]\label{pantDOZZ}
Let $(\caP,{\bf x}^{\rm mp},\boldsymbol{\zeta})$ be a building block. Let $\alpha_j= Q+ip_j$, $p_j\in\R$, $j=1,\dots, b$, let $\alpha_j<Q$ for $j=b+1,..,3$ and denote $\hat{\boldsymbol{\alpha}}=(\alpha_1,\alpha_2,\alpha_3)$.  Assume (recall that $\chi(\mc{P})$ is the Euler characteristic of $\mc{P}$)
\begin{equation}\label{ass:starward}
   \sum_{j=b+1}^{3}\alpha_{j }-\chi(\mc{P})Q>0 .
\end{equation}
Then, if $\boldsymbol{\nu}=(\nu_1,\dots,\nu_{b})\in \mc{T}^b$ and $\boldsymbol{\tilde{\nu}}=(\tilde{\nu}_1,\dots,\tilde{\nu}_{b})\in \mc{T}^b$, 
the following formula holds

 \begin{align}\label{pantclaim}
\big\cjg  \caA_{\caP, g_\caP,{\bf x}^{\rm mp},\boldsymbol{\alpha}^{\rm mp},\boldsymbol{\zeta}},  \otimes _{j=1}^b{\bf C}^{(1+\sigma_j)/2}\Psi_{\alpha_j,\nu_j,\tilde\nu_j}  \big\cjd_{\mc{H}}
=&  C(\mc{P},g_{\mc{P}},\boldsymbol{\Delta}_{\hat{\boldsymbol{\alpha}}})
C_{\gamma,\mu}^{{\rm DOZZ}} (\boldsymbol{\tilde\alpha}) \nonumber \\&\times 
  w_{\caP}(\boldsymbol{\Delta}_{\hat{\boldsymbol{\alpha}}},
 \boldsymbol{\nu},\hat{{\bf z}})
 \overline{ w_{\caP}(\boldsymbol{\Delta}_{\hat{\boldsymbol{\alpha}}},
 \boldsymbol{\tilde\nu},\hat{{\bf z}})}P(\hat{{\bf z}}) 
 \end{align}
 
where   $\boldsymbol{\tilde \alpha}=(\tilde\alpha_1,\tilde\alpha_2,\tilde\alpha_3)$ with  $\tilde\alpha_j=Q+i \sigma_j p_j$ for $j=1,..,b$ and $\tilde\alpha_j=\alpha_j$  for  $j=b+1,..,3$. Furthermore:
\begin{itemize}
\item the function $w_{\caP}$ is a polynomial in the conformal weights $\boldsymbol{\Delta}_{\hat{\boldsymbol{\alpha}}}=(\Delta_{\alpha_1},\Delta_{\alpha_2},\Delta_{\alpha_3})$ with coefficients depending only on the complex structure of $\caP$ and on the Young diagram  
\begin{align*}%\label{}
 w_{\caP}(\boldsymbol{\Delta}_{\hat{\boldsymbol{\alpha}}},
 \boldsymbol{\nu},\hat{{\bf z}})=\sum_{{\bf n}=(n_1,n_2,n_3)} a_{\bf n}(\caP, \boldsymbol{\nu},\hat{{\bf z}})\boldsymbol{\Delta}_{\hat{\boldsymbol{\alpha}}}^{\bf n},
 \end{align*}
where   $w_{\caP}(\Delta_{\alpha_1} ,\Delta_{\alpha_2},\Delta_{\alpha_3},\emptyset,\hat{{\bf z}})=1$ and the sum runs over ${\bf n}\in\N^3$ with finitely many coefficients $a_{\bf n}(\caP, \boldsymbol{\nu},\hat{{\bf z}})$ that are non zero.\\
\item $P$ is given by
\begin{equation}\label{defp0}
 P(\hat{{\bf z}}):= |z_1-z_2|^{2(\Delta_{\alpha_3}-\Delta_{\alpha_2} -\Delta_{\alpha_1})}|z_3-z_2|^{2(\Delta_{\alpha_1}-\Delta_{\alpha_2} -\Delta_{\alpha_3})}|z_1-z_3|^{2(\Delta_{\alpha_2}-\Delta_{\alpha_1} -\Delta_{\alpha_3})}
\end{equation}
\item the metric dependent constant $C(\mc{P},g_{\mc{P}},\boldsymbol{\Delta}_{\hat{\boldsymbol{\alpha}}})$ is given by
\begin{align}\label{metricconstant}
C(\mc{P},g_{\mc{P}},\boldsymbol{\Delta}_{\hat{\boldsymbol{\alpha}}})= &\frac{(\sqrt{2}\pi)^{b-1}}{2\sqrt{2}Z_{\D,g_\D}^{b}}  \big({\frac{{\rm v}_{g_{\rm dozz}}(\hat\C)}{{\det}'(\Delta_{g_{\rm dozz}})}} \big)^\hf \prod_{j=1}^b {\rm Rad}_{z_j}(\psi(\mc{D}_j))^{2\Delta_{\alpha_j}}\prod_{j=b+1}^3e^{-\Delta_{\alpha_j}\hat\omega(z_j)}  \\ \times&e^{c_{\rm L}(S^0_{\rm L}(\psi(\caP),g_{\rm dozz},\psi_*g_{\mc{P}})+\sum_{j=1}^bS^0_{\rm L}(\psi_j(\caD_j),g_{\rm dozz},(\psi_j)_*|dz|^2))}\nonumber
 \end{align}
where %the metric $g_{\rm dozz}$ on $\hat\C$ is defined by  $g_{\rm dozz}= |z|_+^{-4}|dz|^2$ (and 
we made a slight abuse of notations by identifying $S_{\rm L}^0(\hat\C,g,g')$ with $S_{\rm L}^0(\hat\C,g,\omega)$ in the case when the metric $g'=e^\omega g$ is conformal to $g$.
\end{itemize}
\end{theorem}
\begin{remark} For the benefit of the readers familiar with the physics literature we explain briefly how
the coefficients $w_{\mc{P}}$ can actually be computed recursively. We discuss the case when $\mc{P}$ is the complex plane with $3$ disjoint Euclidean disks removed with incoming boundary components (with canonical parametrisation) and respective radii $r_1,r_2,r_3$ and centers $z_1,z_2,z_3$. The 
general case of  "in/out" boundary components follows as explained in the very beginning of the proof of Theorem \ref{pantDOZZ}) and the case with $b$ disks removed and $3-b$ marked points is a consequence when observing that dealing with a  marked point amounts to removing a disk and then gluing a disk amplitude with marked point $\Psi_{\alpha_j,\emptyset,\emptyset} $ and empty Young diagrams. To express the recursion rules, it will be convenient to consider the quantity 
\[ (\mathbf{L}_{-\nu_1}\widetilde{\mathbf{L}}_{- \tilde \nu_1}V_{\alpha_1}|\mathbf{L}_{-\nu_2}\widetilde{\mathbf{L}}_{-\tilde \nu_2}V_{\alpha_2}|\mathbf{L}_{-\nu_3}\widetilde{\mathbf{L}}_{-\tilde \nu_3}V_{\alpha_3} ):=\frac{\big\cjg   \caA_{\caP, g_\caP, \boldsymbol{\zeta}},  \otimes _{j=1}^3{\bf C} \Psi_{\alpha_j,\nu_j,\tilde\nu_j}  \big\cjd_{\mc{H}}}{C(\mc{P},g_{\mc{P}},\boldsymbol{\Delta}_{\hat{\boldsymbol{\alpha}}})\prod_{j=1}^3r_j^{2\Delta_{\alpha_j+| \nu_j|+|\tilde\nu_j|}}}\] 
for $\nu_i,\tilde{\nu}_i\in \mc{T}$ (i=1,2,3), where $\mathbf{L}_{- \nu }:=  \mathbf{L}_{-\nu(1)} \dots \mathbf{L}_{-\nu(k)}$ (where $k=s( \nu)$ is the size of $ \nu$) stands for a formal composition of operators $ (\mathbf{L}_{n})_{n\in \Z}$ and similarly for $\widetilde{\mathbf{L}}_{- \nu}$ in term of a family $ (\widetilde{\mathbf{L}}_{n})_{n\in \Z}$. This quantity is a function of $z_1,z_2,z_3$ and of the Young diagrams $\nu_j,\tilde\nu_j$ for $j=1,2,3$. For $u\in\C$, $\nu,\tilde{\nu}\in \mc{T}$ and $I\subset \{1,2,3\}$, consider the differential operators $\mathbf{D}^{(u,I)}_{\nu}$, $\tilde{\mathbf{D}}^{(u,I)}_{\tilde{\nu}}$   defined by 
\begin{equation}
\mathbf{D}^{(u,I)}_{\nu}= \mathbf{D}_{\nu(k)}^{(u,I)}\dots \mathbf{D}^{(u,I)}_{\nu(1)}\quad \text{ and } 
\quad \tilde{\mathbf{D}}_{\tilde{\nu}}^{(u,I)}= \tilde{\mathbf{D}}_{\tilde{\nu}(j)}^{(u,I)}\dots \tilde{\mathbf{D}}_{\tilde{\nu}(1)}^{(u,I)}
\end{equation}
where for $n\in\Z$
\begin{align}
\mathbf{D}_{n}^{(u,I)}=&\sum_{i\in I}\Big(-\frac{1}{(z_i-u)^{n-1}}\partial_{z_i}+\frac{(n-1) }{(z_i-u)^n}\Delta_{\alpha_i}\Big)\\
\tilde{\mathbf{D}}_{n}^{(u,I)}=&\sum_{i\in I}\Big(-\frac{1}{(\bar z_i-\bar u)^{n-1}}\partial_{\bar z_i}+\frac{(n-1) }{(\bar z_i-\bar u)^n}\Delta_{\alpha_i}\Big).
\end{align}
 The following rules are then in force:
\begin{enumerate}
\item the $(\mathbf{L}_{n} )_n$ and $ (\widetilde{\mathbf{L}}_{m})_m$ commute, i.e. in each entry we can replace a symbol   $\widetilde{\mathbf{L}}_{m}\mathbf{L}_{n}$ by $\mathbf{L}_{n}\widetilde{\mathbf{L}}_{m}$.
\item (Commutation relations) the commutator of the $(\mathbf{L}_{n} )_n$ (resp. $(\widetilde{\mathbf{L}}_{n} )_n$) satisfies
\begin{align}
[\mathbf{L}_n,\mathbf{L}_m]=&(n-m)\mathbf{L}_{n+m}+\frac{c_{\mathbf{L}}}{12}(n^3-n)\delta_{n=-m}{\rm I}\label{commut1}\\
[\widetilde{\mathbf{L}}_n,\widetilde{\mathbf{L}}_m]=&(n-m)\widetilde{\mathbf{L}}_{n+m}+\frac{c_{\mathbf{L}}}{12}(n^3-n)\delta_{n=-m}{\rm I}\label{commut2}.
\end{align}
\item  (one descendant case) if $( \nu_1,\tilde \nu_1)$ are Young diagrams then, for $I=\{2,3\}$
$$ (\mathbf{L}_{- \nu_1}\widetilde{\mathbf{L}}_{-\tilde \nu_1}V_{\alpha_1}| V_{\alpha_2}| V_{\alpha_3} )  
= \mathbf{D}^{(z_1,I)}_{  \nu_1}\tilde{\mathbf{D}}^{(z_1,I)}_{{ \tilde \nu_1}}\langle V_{\alpha_1}(z_1) V_{\alpha_2}(z_2)V_{\alpha_3}(z_3)\rangle_{\hat\C,g_0} .$$
\item (switching  the 3rd entry) for $m\in \N$  
\begin{align*}
 (\mathbf{L}_{- \nu_1}\widetilde{\mathbf{L}}_{- \tilde \nu_1}&V_{\alpha_1}|\mathbf{L}_{- \nu_2}\widetilde{\mathbf{L}}_{- \tilde \nu_2}V_{\alpha_2}| \mathbf{L}_{-m}\mathbf{L}_{- \nu_3}\widetilde{\mathbf{L}}_{-\tilde \nu_3}V_{\alpha_3} )\\
=&
\sum_{0\leq k\leq |\nu_2|+1}C^{m-2}_{k+m-2}(-1)^{k-1}(z_2-z_3)^{1-m-k}(\mathbf{L}_{- \nu_1}\widetilde{\mathbf{L}}_{-\tilde \nu_1}V_{\alpha_1}|\mathbf{L}_{k-1}\mathbf{L}_{-\nu_2}\widetilde{\mathbf{L}}_{-\tilde \nu_2}V_{\alpha_2}|  \mathbf{L}_{-\nu_3}\widetilde{\mathbf{L}}_{-\tilde \nu_3}V_{\alpha_3} )
\\& 
+ \sum_{0\leq k\leq |\nu_1|+1}C^{m-2}_{k+m-2}(-1)^{k-1}(z_1-z_3)^{1-m-k}  (\mathbf{L}_{k-1}\mathbf{L}_{-\nu_1}\widetilde{\mathbf{L}}_{-\tilde \nu_1}V_{\alpha_1}|\mathbf{L}_{-\nu_2}\widetilde{\mathbf{L}}_{-\tilde \nu_2}V_{\alpha_2}|  \mathbf{L}_{-\nu_3}\widetilde{\mathbf{L}}_{-\tilde \nu_3}V_{\alpha_3} )
\end{align*}
and a similar relation for $\widetilde{\mathbf{L}}_{-m}$.
\item (switching  the 2nd entry) for $m\in \N$  and $I=\{z_3\}$
\begin{align*}
 (\mathbf{L}_{- \nu_1}\widetilde{\mathbf{L}}_{- \tilde \nu_1}&V_{\alpha_1}| \mathbf{L}_{-m}\mathbf{L}_{- \nu_2}\widetilde{\mathbf{L}}_{- \tilde \nu_2}V_{\alpha_2}|  V_{\alpha_3} )\\
=&
 - \mathbf{D}^{(z_2,I)}_{ m}(\mathbf{L}_{- \nu_1}\widetilde{\mathbf{L}}_{-\tilde \nu_1}V_{\alpha_1}|\mathbf{L}_{-\nu_2}\widetilde{\mathbf{L}}_{-\tilde \nu_2}V_{\alpha_2}|   V_{\alpha_3} )
\\& 
+ \sum_{0\leq k\leq |\nu_1|+1}C^{m-2}_{k+m-2}(-1)^{k-1}(z_1-z_2)^{1-m-k}  (\mathbf{L}_{k-1}\mathbf{L}_{-\nu_1}\widetilde{\mathbf{L}}_{-\tilde \nu_1}V_{\alpha_1}|\mathbf{L}_{-\nu_2}\widetilde{\mathbf{L}}_{-\tilde \nu_2}V_{\alpha_2}|  V_{\alpha_3} )
\end{align*}
and a similar relation for $\widetilde{\mathbf{L}}_{-m}$.
 \item (Annihilation) if $m\in\N$ satisfies $m>|\nu|$   then $ \mathbf{L}_{m}\mathbf{L}_{- \nu}\widetilde{\mathbf{L}}_{-\tilde \nu}V_{\alpha_j}=0$,
 and a similar relation for the $\widetilde{\mathbf{L}}_{m}$.
  \item (Eigenstates)  $\mathbf{L}_{0}V_{\alpha_j}=  \Delta_{\alpha_j}V_{\alpha_j}$   and a similar relation for $\widetilde{\mathbf{L}}_{0}$.
  
 These relations can be established by computing the residues in the Ward identities as explained in \cite{Gawedzki96_CFT}. In the case when the complex plane has analytic disks removed, then the coefficients can be expressed in terms of the coefficients computed above and the coefficients of the Taylor expansion of the uniformizing maps of the analytic disks. We won't give any further details in this latter situation.
\end{enumerate}
\end{remark}

 In a forthcoming work, we will address the question of the global definition of the conformal blocks on the   Teichm\"uller space. For this, an important input will be the following formula. 
 \begin{proposition}\label{remarkblocks}
 Let $(\caP,{\bf x}^{\rm mp},\boldsymbol{\zeta})$ be a building block. Let $\alpha_j= Q+ip_j$, $p_j\in\R$, $j=1,\dots, b$, let $\alpha_j<Q$ for $j=b+1,..,3$ and denote $\hat{\boldsymbol{\alpha}}=(\alpha_1,\alpha_2,\alpha_3)$.  Assume  
\begin{equation}\label{ass:starwardrem}
   \sum_{j=b+1}^{3}\alpha_{j }-\chi(\mc{P})Q>0 .
\end{equation}
For each $n\geq 0$, denote by $ w^n_{\caP}(\boldsymbol{\Delta}_{\hat{\boldsymbol{\alpha}}},
  \hat{{\bf z}})$ the coefficient  $w_{\caP}(\boldsymbol{\Delta}_{\hat{\boldsymbol{\alpha}}},
 \boldsymbol{\nu},\hat{{\bf z}})$ where    $\boldsymbol{\nu}=(\nu_1,\dots,\nu_{b})\in \mc{T}^b$ is chosen in such a way that   $\nu_{j}=\emptyset$ for $j>1$ and  $\nu_1$ is the Young diagram $\nu_1=(n,0,\dots,0,\dots)$ (with the convention that $ w^n_{\caP}(\boldsymbol{\Delta}_{\hat{\boldsymbol{\alpha}}},
  \hat{{\bf z}})=0$ for $n<0$). Consider a meromorphic vector field $v=v(z)\partial_z$ on $\hat\C$ with a unique pole at $z_1$, and vanishing at $z_2$ and $z_3$, which we can expand in coordinates as $v\circ \psi_1(z)/\psi_1'(z)=\sum_{n\in\Z}v_nz^{1-n}$. Then 
the following formula holds
\begin{equation}
 \sum_{n\geq 1} w^n_{\caP}(\boldsymbol{\Delta}_{\hat{\boldsymbol{\alpha}}},
  \hat{{\bf z}})\frac{1}{2i\pi}\oint_{\sigma} \frac{v\circ \psi_1(z)}{\psi_1'(z)}z^{n-2}\,dz=-\frac{{ c}_{\rm L}}{12}{\rm Res}_{z_1} (v   S_{\psi_1^{-1}}) -\Delta_{\alpha_1}v_0 -\Delta_{\alpha_2}v'(z_2)-\Delta_{\alpha_3}v'(z_3)
\end{equation}
where $\sigma$ is a small contour around $0$.
\end{proposition}
This formula is a consequence of  the proof of Theorem \ref{pantDOZZ}, as will be explained later.
\medskip

Now we state a corollary of Theorem \ref{pantDOZZ} used in   subsection \ref{sub:cfb}. Let $\caP=(\caP,{\bf x}^{\rm mp},\boldsymbol{\zeta})$ be a building block equipped with an admissible metric $g_{\mc{P}}$ as previously. Let ${\bf q}=(q_1,\dots,q_b)\in \D^b$ be complex parameters. For each $i=1,\dots,b$ we can glue an annulus  $\mathbb{A}_{q_i}$  defined in \eqref{defaq} (with  metric and parametrisation  defined there) to each boundary curve $\mc{C}_i$ to obtain a Riemann surface denoted  $\caP({\bf q})$, with metric $g_{\mc{P}}^{\bf q}$ and parametrisation $\boldsymbol{\zeta}^{\bf q}$ obtained by gluing. We claim

\begin{corollary}{\bf (Plumbed pants)}\label{plumbedpant}
Assume \eqref{ass:starward}. For all ${\bf q}\in \D^{b}$, ${\bf p}=(p_1,\dots,p_b)\in \R^b$ and 
$\boldsymbol{\nu}=(\nu_1,\dots,\nu_b)\in \mc{T}^b$, $\boldsymbol{\tilde{\nu}}=(\tilde{\nu}_1,\dots,\tilde{\nu}_b)\in \mc{T}^b$, we have
\begin{align*}
 \big\cjg &\mc{A}_{\mc{P}({\bf q}), g_{\mc{P}}^{\bf q},{\bf x}^{\rm mp},\boldsymbol{\alpha}^{\rm mp},\boldsymbol{\zeta}^{\bf q}},\otimes_{j=1}^{b}{\bf C}^{(\sigma_{j}+1)/2}\Psi_{Q+ip_{j},\nu_{j},\tilde{\nu}_{j}}\big\cjd_{\mc{H}} \\
& =  \big\cjg \mc{A}_{\mc{P}, g_{\mc{P}},{\bf x}^{\rm mp},\boldsymbol{\alpha}^{\rm mp},\boldsymbol{\zeta}},\otimes_{j=1}^{b}{\bf C}^{(\sigma_{j}+1)/2}\Psi_{Q+ip_{j},\nu_{j},\tilde{\nu}_{j}}\big\cjd_{\mc{H}}
  \prod_{j=1}^{b}|q_{j}|^{-c_{\rm L}/12}q_{j}^{\Delta_{Q+ip_{j}}+|\nu_{j}|}\bar{q}_{j}^{\Delta_{Q+ip_{j}}+|\tilde{\nu}_{j}|}\nonumber 
 \end{align*}
with  $c_{\rm L}=1+6Q^2$ the central charge. 
\end{corollary}

\begin{proof} By  the gluing lemma (Proposition \ref{glueampli}), the $\mc{P}({\bf q})$-amplitude is the composition of the $\mc{P}$-amplitude with the annuli $\mathbb{A}_{q_i}$-amplitudes, namely
\begin{align*}
\big\cjg  \mc{A}_{\mc{P}({\bf q}), g_{\mc{P}}^{\bf q},{\bf x}^{\rm mp},\boldsymbol{\alpha}^{\rm mp},\boldsymbol{\zeta}^{\bf q}},\otimes_{j=1}^{b}f_j\big\cjd_{\mc{H}} =& \frac{1}{(\sqrt{2}\pi)^b } 
\big\cjg  \mc{A}_{\mc{P}, g_{\mc{P}},{\bf x}^{\rm mp},\boldsymbol{\alpha}^{\rm mp},\boldsymbol{\zeta} },\otimes_{j=1}^{b}   \mc{A}_{\mathbb{A}_{q_j}, g_{\mathbb{A}}, \boldsymbol{\zeta} }  f_j\big\cjd_{\mc{H}} .
 \end{align*}
The annulus amplitude is then computed with Proposition \ref{prop:annulussimple} and finally we use Propositions \ref{defprop:desc} and \ref{propPi} to get that 
\[ e^{i \arg(q_j)\boldsymbol{\Pi}} e^{\log |q_j|\mathbf{H}}\Psi_{Q+ip_{j},\nu_{j},\tilde{\nu}_j}=q_{j}^{\Delta_{Q+ip_j}+|\nu_j|}\bar{q}_j^{\Delta_{Q+ip_j}+|\tilde{\nu}_{j}|}\Psi_{Q+ip_{j},\nu_{j},\tilde{\nu}_{j}}.\]
This implies the result.
\end{proof}

Finally, we state a last corollary, used in Section \ref{sec:special}, whose proof is postponed to the very end of this section:

\begin{corollary}{\bf (Plumbed annulus)}\label{annulusDOZZ}
Assume $\alpha_3>0$.  For all $|z|<1$ and $|q/z|<1$, $\boldsymbol{\nu}=(\nu_1,\nu_2)\in \mc{T}^2$, 
$\tilde{\boldsymbol{\nu}}=(\tilde \nu_1,\tilde\nu_2)\in \mc{T}^2$. Then using the convention \eqref{defaq} for the parametrised annulus $\mathbb{A}_q$, we have 
\begin{align}\label{claimannulus}
\big\cjg &\mc{A}_{\mathbb{A}_q, g_{\mathbb{A}},z, \alpha_1,\boldsymbol{\zeta} },\otimes_{j=1}^{2}{\bf C}^{(\sigma_{j}+1)/2}\Psi_{Q+ip_j,\nu_j,\tilde{\nu}_j}\big\cjd_{\mc{H}} \\
=&   C 
C_{\gamma,\mu}^{{\rm DOZZ}} (\boldsymbol{\tilde\alpha}) w_{\mathbb{A}}(\boldsymbol{\Delta}_{\hat{\boldsymbol{\alpha}}},
 \boldsymbol{\nu} )
 \overline{ w_{\mathbb{A}}(\boldsymbol{\Delta}_{\hat{\boldsymbol{\alpha}}},
 \tilde{\boldsymbol{\nu}} )}  |q |^{-\frac{{c}_{\rm L}}{12}+2\Delta_{Q+ip_{1}}} |z|^{2\Delta_{Q+ip_{2}}-   2\Delta_{Q+ip_{1}}   }(q/z)^{|\nu_{1}|}(\bar{q}/\bar{z})^{ |\tilde{\nu}_{1}|} z^{|\nu_{2}|}\bar{z}^{|\tilde\nu_{2}|}\nonumber 
 \end{align}
with the constant $C_{\gamma,\mu}^{{\rm DOZZ}} (\boldsymbol{\tilde\alpha})$ defined as in Theorem   \ref{pantDOZZ}, ${c}_{\rm L}=1+6Q^2$,
\begin{equation}\label{constanttorus}
C= \frac{\pi}{2^{1/2}e},%\frac{\pi}{2}  Z_{\D,g_\D}^{-2}\sqrt{\frac{{\rm v}_{g_0}(\hat\C)}{{\det}'(\Delta_{\hat{\C},g_0})}} e^{S_{\rm L}^0(\hat\C,g_0,g_{\rm dozz})} ,
\end{equation}
the coefficients $w_{\mathbb{A}}(\boldsymbol{\Delta}_{\hat{\boldsymbol{\alpha}}}, \boldsymbol{\nu} )$ are polynomials in the conformal weights $\boldsymbol{\Delta}_{\hat{\boldsymbol{\alpha}}}=(\Delta_{\alpha_1},\Delta_{\alpha_2},\Delta_{\alpha_3})$  and  $w_{\mathbb{A}}(\boldsymbol{\Delta}_{\hat{\boldsymbol{\alpha}}},\emptyset )=1$.
\end{corollary}

\subsection{Proof of Theorem \ref{pantDOZZ}}\label{debutproof}
%%%%%%%%%%%%%%%%%%%%%%%%%%%%%%%%%

The rest of this section is devoted to the proof of Theorem \ref{pantDOZZ}. We need some preliminaries.  
First, recall  that flipping the orientation of a boundary circle is implemented by the map ${\bf O}$ on $e^{-\beta c_-}L^2(\R\times \Omega_\T)$, see subsection \ref{sub:revert}. Hence if we denote by $\tilde{\boldsymbol{\zeta}}$ the parametrisation where all the boundary components are "in" then 
\begin{align*}%\label{}
 \big\cjg \caA_{\caP, g_\caP,{\bf x}^{\rm mp},\boldsymbol{\alpha}^{\rm mp},\boldsymbol{\zeta}},\otimes_{j=1}^{b}f_j\big\cjd_{\mc{H}}= \big\cjg \caA_{\caP, g_\caP,{\bf x}^{\rm mp},\boldsymbol{\alpha}^{\rm mp},{\boldsymbol{\tilde\zeta}}},\otimes_{j=1}^{b}{\bf O}^{(1-\sigma_j)/2}f_j\big\cjd_{\mc{H}}.
\end{align*}
Recalling that ${\bf CO}\Psi_{Q+ip,\nu,\tilde\nu}=\Psi_{Q-ip,\nu,\tilde\nu}$ (Prop. \eqref{revert}) we obtain
\begin{align*}%\label{}
 \big\cjg \caA_{\caP, g_\caP,{\bf x}^{\rm mp},\boldsymbol{\alpha}^{\rm mp},\boldsymbol{\zeta}}, \otimes_{j=1}^b{\bf C}^{(1+\sigma_j)/2}\Psi_{Q+ip_j,\nu_j,\tilde\nu_j}\big\cjd_{\mc{H}}=\caA_{\caP, g_\caP,{\bf x}^{\rm mp},\boldsymbol{\alpha}^{\rm mp},\tilde{\boldsymbol{\zeta}}}(\otimes_{j=1}^b{\bf C}\Psi_{Q+i\sigma_jp_j,\nu_j,\tilde\nu_j})
\end{align*}
Hence it suffices to prove the theorem for the case $\sigma_j=1$ for all $j$.   

Let  $g$  be an admissible metric on $\D$. Beside being admissible, we are free to choose this metric as we please and  we take   $g$ to be Euclidean in a neighborhood of the origin for later computational convenience.   Then $g_\caP$ and $\rho_j^\ast g$ (recall   $\{(U_i,\rho_i)\}_{i=1}^N$ is the atlas on $\hat\caP$) glue to a smooth metric  $g_{\hat\caP}$ on $\hat\caP$ and $\hat g:=\psi_\ast g_{\hat\caP}$ is a smooth metric on $\hat\C$. Obviously $\hat g_{\caP}:=\hat g|_{\psi(\caP)}$ and $\hat g_{j}:=\hat g|_{\psi(\caD_j)}$ are admissible on $\psi(\caP)$ and $\psi(\caD_j)$ and glue to $\hat g$.  

{\bf Briefly, our strategy   is the following:}

\vskip 1mm

\noindent (a) First we analytically continue the states  $\Psi_{\alpha,\nu,\tilde\nu}$   in the parameter $\alpha$ to the real line. For $\alpha$ real, and  small enough (depending on $\nu,\tilde\nu$), we use the intertwining property \eqref{limtpsi} to get a probabilistic expression for the descendant states $\Psi_{\alpha,\nu,\tilde\nu}$  in Lemma \ref{TTLemma}, which can be expressed  in terms of generalised amplitudes on the disk, namely amplitudes with further SET insertions in addition to the vertex insertions (see Definition \ref{def:ampholes}), in Lemma \ref{ampdiskSET}. This generalised disk amplitude can be expressed   as a generalised amplitude on $\psi(\caD_j)$ using the conformal map $\psi$ (section \ref{Conformal map}). Therefore, evaluating a building block amplitudes at the eigenstates $\Psi_{\alpha,\nu,\tilde\nu}$, namely \eqref{pantclaim},  can be interpreted, using Segal's gluing axioms, as correlation function on the Riemann sphere with further SET insertions.
 \vskip 1mm

\noindent (b) There is a caveat regarding the gluing of the aforementioned amplitudes: the states $\Psi_{\alpha,\nu,\tilde\nu}$ for $\alpha\in\R$ small  are in general not in the domain of the amplitude $\caA_{\psi(\caP), \hat g_\caP,{\bf z}^{\rm mp},\boldsymbol{\alpha}^{\rm mp},\boldsymbol{\zeta}}$ due to violation of the Seiberg bounds. To remedy this we regularise this  amplitude by inserting several additional vertex operators  $V_{\alpha_j,\hat g_\caP}(z_j)$, $z_j\in \psi(\caP)$, $j=4,\dots,m$ with $\alpha_j<Q$ and $\sum_j\alpha_j$ large enough, both conditions being dictated by the Seiberg bounds (section \ref{Pant amplitude}).  The resulting amplitude  
can then be evaluated at the states $\Psi_{\alpha_j,\nu_j,\tilde\nu_j}$  and the result will be given as a contour integral of  LCFT correlation function on $\hat\C$ with SET insertions (section \ref{pantglue}).

\vskip 1mm

\noindent (c) The correlation function with SET insertions is then evaluated using the Ward identities after which the contour integrals can be computed using the residue theorem. The result is an expression in terms of $z_j$ derivatives of the correlation function $\langle \prod_{j=1}^mV_{\alpha_j,\hat g}(z_j)\rangle_{\hat\C,\hat g}$ (section \ref{residue}). 
In the latter one we will use analyticity to take the $\alpha_j\to 0$ for $j>3$ resulting to the LCFT three-point function and eventually to the DOZZ structure constant. Then analytic continuation back to the desired values of $\alpha_j$ yields the claim (section \ref{finalst}).

\medskip
Let us now proceed.

\subsubsection{Eigenfunctions and generalised amplitudes}\label{Eigenfunctions and generalised amplitudes}
%%%%%%%%%%%%%%%%%%%%%%%%%%%%%%%%%%%%%%%%%%%%%%%%%%%%%%%%%%%%%%%%%%%%%%%%%%%%%%%%%%%%%%%%%%%%

We start by recalling some results from \cite{GKRV20_bootstrap}. Basically, the content of this subsection is to recall how the generalised eigenfunctions $\Psi_{\alpha,\nu,\tilde\nu}$ can be analytically continued in $\alpha$ from the spectrum line $Q+i\R$ to real values of $\alpha$, region over which they admit a probabilistic representation in terms of LCFT correlation functions with SET insertions. To make this claim precise, let us reformulate a direct consequence of  Proposition \ref{defprop:desc}:

\begin{proposition}\label{eigenana}
Let  $\nu,\tilde\nu$  be Young diagrams and $\ell=|\nu|+|\tilde\nu|$.  There exists a real number $a\in (-\infty,Q-\gamma) $ such that $Q+i\R$ and  $(-\infty ,a]$ belong to the same connected component of $W_\ell$  and $ [-\infty ,a]\subset \mathcal{I}_{\nu,\tilde{\nu}}$. 
\end{proposition}

Recall that the mapping $\alpha \in W_\ell \mapsto \Psi_{\alpha,\nu,\tilde\nu}\in e^{-\beta c_-}L^2(\R\times \Omega_\T)$ (with $\beta>Q-{\rm Re}(\alpha)$)
is analytic. Via the above proposition and by analyticity,  we will deduce the factorisation property on the spectrum line $\alpha\in Q+i\R$ from relations obtained over the region $\alpha\in (-\infty ,a]$. The  crucial fact  related to this latter region is the intertwining property: for $\alpha\in\R$ with $\alpha<a\wedge(Q-\gamma)$, Proposition \ref{defprop:desc} (taking $\chi=1$) yields
\begin{equation}\label{limtpsi}
 \lim_{t\to +\infty}e^{t (2\Delta_\alpha+|\nu|+|{\tilde{\nu}}|)}e^{-t\mathbf{H}}\Psi^0_{\alpha,\nu,{\tilde{\nu}}}=\Psi_{\alpha,\nu,{\tilde{\nu}}}
 \end{equation}
  in $e^{-\beta c_- }L^2(\R\times \Omega_\T)$ for any $\beta>(Q-\alpha)$.  This relation is the starting point of our probabilistic representation of $\Psi_{\alpha,\nu,\tilde\nu}$. Indeed, recall next that the free eigenfunctions $\Psi^0_{\alpha,\nu,{\tilde{\nu}}}$ have a probabilistic representation in terms of contour integrals of GFF expectations with SET insertions, which will be preserved, up to adding the Liouville potential, when applying the propagator $e^{-t\mathbf{H}}$ using the Feynman-Kac formula \ref{FKgeneral}. To formulate this more precisely, we introduce some notation from \cite{GKRV20_bootstrap}. Let  $f:\C^k\times\C^{\tilde k}\to\C$ and  ${\bf a} \in \C^k$,  ${\bf b} \in \C^{\tilde k}$.       We will denote multiple   nested contour integrals of $f$ as follows:
\begin{align*}
\oint_{|\mathbf{u}-\boldsymbol{a}|=\boldsymbol{\delta} }\oint_{|\mathbf{v} -\boldsymbol{\tilde a}|=\tilde{\boldsymbol{\delta}}}f ({\bf u},{\bf v})\dd {\bf \bar v}\dd {\bf u}:=
\oint_{|u_k-a_k|=\delta_k}\dots \oint_{|u_1-a_1|=\delta_1} \oint_{|v_{\tilde k}-\tilde a_{\tilde k}|=\tilde{\delta}_{\tilde k}}\dots  \oint_{|v_1-\tilde a_1|=\tilde{\delta}_1}  f({\bf u},{\bf v}) \dd \bar v_1\dots \dd \bar v_{j}  \dd u_1\dots \dd u_k\nonumber
\end{align*}
where $\boldsymbol{\delta} :=(\delta_1,\dots, \delta_k)$ with $0<\delta_1<\dots<\delta_k<1$ and similarly for $\tilde{\boldsymbol{\delta}}$.  We always suppose $\delta_i\neq \tilde\delta_j$ for all $i,j$. Given Young diagrams $\nu$, $\tilde\nu$  we denote
\begin{equation}\label{notation}
\mathbf{u}^{1-\nu}:=\prod u_i^{1-\nu(i)},\ \ \ \bar{\mathbf{v}}^{1-\tilde{\nu}}:=\prod \bar {v}_i^{1-\tilde{\nu}(i)}.
\end{equation}
More generally, we will often make use of the shorthand $f({\bf u}):=\prod_i f(u_i)$ if $f:\C\to \C$ is a function and ${\bf u}\in \C^k$.
Recall also   the definitions for Young diagrams in subsection \ref{sub:virasoro}, in particular $s(\nu)$ is the size of $\nu$.

Now we introduce the notations for the SET. We write $g_\D=|dz|^2$ for the flat metric on $\D$. Let $g=e^\omega g_\D$ where $\omega$ is smooth on $\D$. Recall that, on $\D$, $\Phi_g=\phi_g+ \tfrac{Q}{2}\omega$ with $\phi_{g}:=X_{g,D}+P\tilde\varphi$ and we denote $\Phi_{g,\epsilon}$ its $|dz|^2$-regularisation (as in subsection \ref{SET}). The SET is then given by
$$
T_{g,\epsilon}=Q\partial^2_{z}\Phi_{g,\epsilon}-\big(\partial_z(\Phi_{g,\epsilon})\big)^2 +\E[(\partial_z X_{g,D,\epsilon}(z))^2]
$$
and $\bbar T_{g,\epsilon}$ is its complex conjugate.  Then from \cite{GKRV20_bootstrap}   Lemma 7.7  we have:

\begin{lemma}\label{TTLemma} Let $\alpha \in\R$, $\boldsymbol{ \delta},\boldsymbol{\tilde\delta}\in \R^k$. Then
\begin{align}\label{tlim}
 e^{t (2\Delta_\alpha+|\nu|+|{\tilde{\nu}}|)}e^{-t\mathbf{H}}\Psi^0_{\alpha,\nu,{\tilde{\nu}}} 
 =&
  \frac{1}{(2\pi i)^{s(\nu)+s(\tilde{\nu})}}  
 \oint_{|\mathbf{u}|=\boldsymbol{\delta}_t}   \oint_{|\mathbf{v}|=\boldsymbol{\tilde\delta}_t}  
\mathbf{u}^{1-\nu}\bar{\mathbf{v}}^{1-\tilde\nu}  
\Psi_{t, \alpha}(\mathbf{u},\mathbf{v})\dd   \bar{\mathbf{v}}\dd   \mathbf{u}
 \end{align}
where $\boldsymbol{\delta}_t:=e^{-t}\boldsymbol{\delta},\boldsymbol{\tilde \delta}_t:=e^{-t}\boldsymbol{\tilde\delta}$, and 
\begin{equation}\label{limphieps}
 \Psi_{t, \alpha}(\mathbf{u},\mathbf{v}) :=\lim_{{\boldsymbol{\epsilon}\to 0}}\Psi_{t,\epsilon,\alpha}(\mathbf{u},\mathbf{v})  
\end{equation}    with
\begin{align}\label{psiepsi}
\Psi_{t,\epsilon,\alpha}(\mathbf{u},\mathbf{v}) =e^{-Q c}
\E\Big( T_{g_\D,\epsilon}(\mathbf{u})\bbar T_{g_\D,\epsilon}(\mathbf{v}) V_{\alpha,g_\D}(0) e^{-\mu e^{\gamma c}M_\gamma(\phi_{g_\D},\D\setminus\D_{e^{-t}})}\Big) \quad \in e^{-\beta c_- }L^2(\R\times \Omega_\T),
 \end{align}
the expectation is over the Dirichlet field $X_{g_\D,D}$  and we denoted $\D_{e^{-t}}=e^{-t}\D$. The limit \eqref{limphieps} holds in $ e^{-\beta c_- }L^2(\R\times \Omega_\T)$, for $\beta>(Q-\alpha)$, uniformly over the compact subsets of $\{(\mathbf{u},\mathbf{v})|u_i,v_j\in \D_{e^{-t}}\setminus\{0\}, \text{all distinct}\}$.
  \end{lemma}
  
Note that the expression  \eqref{psiepsi} is the expectation involved in the definition of amplitudes \eqref{amplitude} with the function $F$ given by  
$$F=T_{g_\D,\epsilon}(\mathbf{u})\bbar T_{g_\D,\epsilon}(\mathbf{v}) e^{\mu e^{\gamma c}M_\gamma(\phi_{g_\D},\D_{e^{-t}})}.$$
It has  a hole in the potential, i.e. $\D_{e^{-t}}$, and further SET insertions.
Hence we need to extend the definition of amplitudes for the limiting object as $\epsilon\to 0$, having in mind to  extend  later the gluing rule of amplitudes to this notion of generalised amplitudes.

 The framework is the following. Consider an admissible surface $(\Sigma,g, {\bf z},\boldsymbol{\zeta})$ embedded in the Riemann sphere $\hat{\C}$ (viewed as the complex plane), with marked points ${\bf z}=(z_1,\dots,z_m)$ and associated weights  $\boldsymbol{\alpha}=(\alpha_1,\dots,\alpha_m)\in\R^m$ satisfying $\alpha_i<Q$ for all $i$. Furthermore we consider an open set $U\subset\Sigma$, which will stand for   holes in $\Sigma$ and which is not necessarily connected: this will typically be the case as we will put holes around many different vertex insertions. We also consider two vectors $\mathbf{u}=(u_1,\dots,u_k)\in U^k ,\mathbf{v}=(v_1,\dots,v_{\tilde k})\in U^{\tilde k}$. 
Recall now the definition of the sets
\begin{align}\label{caOt}
 \caO^{\bf z}_{\Sigma,U}=\{({\bf u},{\bf v})\in U^{k+\tilde k}\,|\, u_j,v_{j'} \text{ all distinct and } \forall j ,\forall i, \:  u_j\neq z_i\; ,v_j\neq z_i  \},\\
  \caO^{\rm ext}_{\Sigma,U} =\{({\bf z},{\bf u},{\bf v})\in \Sigma^m\times U^{k+\tilde k}\,|\, z_i,u_j,v_{j'} \text{ all distinct} \}.
\end{align}

\begin{definition}[\textbf{Generalised amplitudes}]\label{def:ampholes}
Let $({\bf u},{\bf v})\in  \caO^{\bf z}_{\Sigma,U}$.

\noindent {\bf (A) }
If $\partial\Sigma=\emptyset$, i.e. $\Sigma=\hat\C$, then we assume that the Seiberg bound \eqref{seiberg1} (thus with ${\bf g}=0$)
is satisfied and   the generalised amplitude with holes and SET insertions  is   defined by (recall \eqref{amplitudeSET1})
\[
 \caA_{\Sigma,U, g, {\bf z},\boldsymbol{\alpha}, \boldsymbol{\zeta}} (T_g( \mathbf{u})\bbar T_g( \mathbf{v}))  :=   \langle T_g( \mathbf{u})\bbar T_g( \mathbf{v})\prod_{i=1}^mV_{\alpha,g}(z_i)\rangle_{\hat\C,U,g}.
\]

 \vskip 3mm
 
\noindent {\bf (B)}  If  $\partial\Sigma$ has $b>0$ boundary connected components then for boundary fields $\tilde{\boldsymbol{\varphi}}:=(\tilde\varphi_1,\dots,\tilde\varphi_b)\in (H^{s}(\T))^b$ with $s<0$ we set  
\begin{multline}
 \label{amplitudeSET}
 \caA_{\Sigma,U, g, {\bf z},\boldsymbol{\alpha}, \boldsymbol{\zeta}}( T_g( \mathbf{u})\bbar T_g( \mathbf{v}),\tilde{\boldsymbol{\varphi}}) \\
 :=\lim_{\epsilon\to 0}  Z_{\Sigma,g}\caA^0_{\Sigma,g}(\tilde{\boldsymbol{\varphi}})
 \E \big[T_{g,\epsilon}({\bf u})\bbar T_{g,\epsilon}({\bf v})\prod_{i=1}^m V_{\alpha_i,g }(z_i)e^{-\frac{Q}{4\pi}\int_\Sigma K_g\phi_g\dd {\rm v}_g-\frac{Q}{2\pi}\int_{\partial\Sigma}k_g\phi_g\dd \ell_g -\mu M_\gamma^g (\phi_g,\Sigma\setminus U)}\big].
 \end{multline}
where  the  expectation $\E$ is over the Dirichlet GFF $X_{g,D}$ and $Z_{\Sigma,g},\caA^0_{\Sigma,g}(\tilde{\boldsymbol{\varphi}})$ are defined in \eqref{znormal}, \eqref{amplifree}.
\end{definition}

The above definition (existence of the limit) is well grounded. Indeed, the following statement is a straightforward adaptation from \cite[Prop 9.1]{GKRV20_bootstrap}

\begin{proposition}\label{ampSETholo}
The limit $\epsilon\to 0$ in Definition \ref{def:ampholes} is well defined and  defines a  continuous function of the variables $({\bf z},{\bf u},{\bf v})\in   \caO^{\rm ext}_{\Sigma,U} $, $\mu_0^{\otimes b}$ almost surely in $\tilde{\boldsymbol{\varphi}}\in (H^{s}(\T))^b$, with $s<0$.  
\end{proposition} 

Now we have formulated the definition of generalised amplitudes, we summarise below two important results. First, in view of manipulating later the geometry of  these amplitudes,   we state the Weyl covariance  (Prop. \ref{Weyl}) for generalised amplitudes:
\begin{proposition}\label{WeylSET}{\bf  (Weyl covariance)}  Let   $\omega\in H^1_0(\Sigma)$ if $\partial\Sigma\not=\emptyset$ or $\omega\in H^1(\Sigma)$ if $\partial\Sigma=\emptyset$. Then % for $F:H^{-s}(\Sigma,g)\to \R^+$  measurable  
\begin{align*}
\caA_{\Sigma,U,e^\omega g,{\bf z},\boldsymbol{\alpha}, \boldsymbol{\zeta}}(T_{e^\omega g}( \mathbf{u})\bbar T_{e^\omega g}( \mathbf{v}),\tilde{\boldsymbol{\varphi}})= e^{c_LS_{\rm L}^0(\Sigma,g,\omega)-\sum_i\Delta_{\alpha_i}\omega(z_i) }\caA_{\Sigma, U,g,{\bf z},\boldsymbol{\alpha},\boldsymbol{\zeta} }\big( T_g( \mathbf{u})\bbar T_g( \mathbf{v}), \tilde{\boldsymbol{\varphi}}\big).
\end{align*}
\end{proposition}

\begin{proof}
This results from Prop \ref{Weyl} applied to $F=T_{  e^{\omega}g,\epsilon}(\mathbf{u})\bbar T_{e^\omega g,\epsilon}(\mathbf{v}) e^{\mu e^{\gamma c}M_\gamma(\phi_g+\frac{Q}{2}\omega,U)}$ and then taking the limit $\epsilon\to 0$ using Prop. \ref{ampSETholo}.
\end{proof}

Next, and as a consequence of the previous discussion,  we  write the analytically continued eigenstates (more precisely \eqref{tlim}) as   generalised amplitudes on $\D$ with boundary  $\partial\D$ equipped with the  parametrisation $\zeta_\D(e^{i\theta})=e^{i\theta}$: 
\begin{lemma} \label{ampdiskSET}
 Let $\alpha\in\R$ with $\alpha<Q$ and $\nu,\tilde{\nu}\in \mc{T}$. Then 
\begin{align*}
  e^{(2\Delta_{\alpha}+|\nu|+|\tilde\nu|)t}   e^{-t\mathbf{H}}\Psi^0_{\alpha,\nu,\tilde\nu}(\tilde\varphi)= &\frac{Z_{\D,g_\D} ^{-1}}{(2\pi i)^{s(\nu)+s(\tilde \nu)}} 
 \oint_{|\mathbf{u}|=\boldsymbol{\delta}_t}   \oint_{|\mathbf{v}|=\boldsymbol{\tilde\delta}_t}  
\mathbf{u}^{1-\nu}\bbar{\mathbf{v}}^{1-\tilde\nu} \mathcal{A}_{\D,\D_{e^{-t}},g_\D,0,\alpha ,\zeta_\D}(T_g({\bf u})   \bbar T_g({\bf v}),\tilde\varphi)    \dd   \bbar{\mathbf{v}}\dd   \mathbf{u}.
 \end{align*}
\end{lemma}

\begin{proof}
 To get this expression, we combine Lemma \ref{TTLemma} with the obervations that  the flat metric $g_\D$  has $0$ curvature and geodesic curvature $1$ (so that the geodesic curvature term contributes the $e^{-Qc}$ in \eqref{psiepsi}) and $\caA^0_{\D,g_\D}(\tilde{\boldsymbol{\varphi}})=1$.
\end{proof}

\subsubsection{Conformal map} \label{Conformal map} 

 We want to write the disk amplitudes in Lemma \ref{ampdiskSET} and the pant amplitude as   amplitudes on $\caD_j$ and $\psi(\caP)$. Let us start with the former. We fix once and for all an admissible metric  $g=e^\omega g_\D$ on $\D$ where for later convenience we take $\omega=0$ in a neighborhood of the origin. % We have %$\caP_{g_\D}\varphi=\caP_{g}\varphi$ and  $X_{{g_\D},D}\stackrel{d}=X_{g,D}$. 
By the Weyl covariance in Proposition \ref{WeylSET}  (using   $\omega(0)=0$) we obtain
\begin{align} \label{defzg}
\Psi_{t,\alpha_j}(\mathbf{u},\mathbf{v}) =
 Z(g)Z_{\D,g_\D}^{-1}\caA_{\D,e^{-t}\D,g,0,\alpha_j,\zeta_\D} (T_g(\mathbf{u})\bar T_g(\mathbf{v}))
\end{align}
where $$Z(g)=e^{-c_{\rm L}S_{\rm L}^0(\D,g_\D,\omega)}.
$$
 Next, to transport this expression to the surface $\psi(\caD_i)$ we need the transformation property of the SET under conformal maps. Recall the definition of  the Schwarzian derivative 
\begin{align}\label{Schwarzian2}
 {\rm S}_\psi:=\psi'''/\psi'-\frac{3}{2}(\psi''/\psi')^2.
\end{align}
We have then

 \begin{proposition}\label{diffeoSET} 
 {\bf (Conformal changes of coordinates)}  Let $(\Sigma,g)$ and $(\Sigma',g')$ be two admissible surfaces embedded in the complex plane with $\partial \Sigma\not=\emptyset$ and such that there exists a biholomorphism  $\psi :\Sigma\to\Sigma'$ with $g'=\psi_\ast g$.  We suppose $g$ is of  the form  $g=e^{\omega(z)}|dz|^2$ with $\omega\in C^\infty(\Sigma )$.  
 Define
 \begin{align}\label{Schwarzian1}
T^\psi_{g'}:=\psi'^2T_{g'}\circ\psi+\frac{Q^2}{2}
 {\rm S}_\psi.
\end{align} 
Then
\begin{align*}
\caA_{\Sigma,U, g,{\bf z},\boldsymbol{\alpha},\boldsymbol{\zeta}}(T_{g}({\bf u})\bar T_{g} ({\bf v}))=  \caA_{\Sigma',U', g',{\bf z}',\boldsymbol{\alpha},\boldsymbol{\zeta}'}( %T_{g'}^\psi({\bf u}) \bar T_{g'}^\psi({\bf v})
 T^\psi_{g'}({\bf u})  \bar T^\psi_{g'}({\bf v})).
\end{align*}
where $U$ is an open set containing $u_i,v_j$, $U'=\psi(U)$,  ${\bf z}'=\psi({\bf z})$ and $\boldsymbol{\zeta'}=\psi\circ\boldsymbol{\zeta}$.
\end{proposition}

\begin{proof} From Proposition \ref{WeylSET}, it suffices to treat the case $\omega=0$ in $U$. Also, for notational readability, we only treat the case of $T$ insertions but the argument is the same if including $\bar T$ insertions. %Let $k$ be the number of $T$ insertions. 
Let $T_{g,\epsilon}$ be the regularised SET  \eqref{defSET}.  By diffeomorphism covariance (i.e. Prop. \ref{Weyl}  item 2)
\begin{align}\label{regset}
\caA_{\Sigma,\mc{D}, g,{\bf z},\boldsymbol{\alpha},\boldsymbol{\zeta}}(T_{g,\epsilon}({\bf u})%\bar T_{g,\epsilon} ({\bf v})
)=  \caA_{\Sigma',\mc{D}', g',{\bf z}',\boldsymbol{\alpha},\boldsymbol{\zeta}'}( T'_{g',\epsilon}({\bf u})%\bar T'_{g,\epsilon} ({\bf v})
)
\end{align}
where  on the r.h.s. 
$$ T'_{g',\epsilon}(z):=Q\partial^2_{z}\Phi'_{\epsilon}(z)-(\partial_z \Phi'_{\epsilon}(z))^2 +a_{\Sigma,g,\epsilon}(z),\quad \text{ with } \quad \Phi'_{\epsilon}:=\rho_\epsilon\ast(\Phi_{g'}\circ\psi+Q\log|\psi'|)$$
and we still use the notation $ T'_{g',\epsilon}({\bf u})$ for the product of such quantities evaluated at the entries of the vector ${\bf u}$.  We stress that the constant $a_{\Sigma,g,\epsilon}$ has remained unchanged. The rhs admits a limit as $\epsilon\to 0$ for the same reason as regularised SET insertions do (i.e. \cite[Prop 9.1]{GKRV20_bootstrap}): let us denote by $T'_{g'}$ the limit of $T'_{g',\epsilon}$ as $\epsilon\to 0$ when inserted into the amplitude $ \caA_{\Sigma',\mc{D}', g',{\bf z}',\boldsymbol{\alpha},\boldsymbol{\zeta}'}$. Now we analyze this limit. Denote by $ P_{\Sigma'}\boldsymbol{\tilde \varphi}$ the harmonic extension   on $ \Sigma'$ with parametrisation $\boldsymbol{\zeta'}$ and by $G_{g',D}(z,z')$  the Dirichlet Green function on $\Sigma'$. Next, a simple computation   shows  that $T'_{g}(u)$ can be decomposed as the sum
$$T'_{g'}(u)=T^{\psi,1}(u) + T^{\psi,2}(u)+ \frac{Q^2}{2}S_\psi(u) $$
with
\begin{align*}
T^{\psi,1} (u)=\lim_{\epsilon\to 0}T^{\psi,1}_\eps(u), & &T^{\psi,1}_\eps(u):=&-\Big(\partial_z(\rho_\epsilon\ast(X_{g',D}\circ \psi))(u)\Big)^2+a_{\Sigma,g,\eps}(u)
\\
T^{\psi,2}(u)=\lim_{\epsilon\to 0}T^{\psi,2}_\eps(u),& &T^{\psi,2}_\eps(u):=&(\psi'(u))^2\Big(Q\partial_{zz}^2( \rho_\epsilon\ast  X_{g',D})(\psi(u))+Q\partial_{zz}^2P_{\Sigma'}\boldsymbol{\tilde \varphi}(\psi(u)) -(\partial_zP_{\Sigma'}\boldsymbol{\tilde \varphi}(\psi(u)))^2\\
& & &-2 \partial_z( \rho_\epsilon\ast  X_{g',D})(\psi(u))\partial_zP_{\Sigma'}\boldsymbol{\tilde \varphi}(\psi(u))\Big)  
\end{align*}
and the above limits are understood in the sense of insertions inside amplitudes. The term $T^{\psi,2}(u)$ gathers all terms that are linear in the $X_{g',D}$-derivatives: these terms are easy to handle because they obey standard composition rules of differential calculus. In particular, to get the above decomposition, we have  used the facts that 
\begin{align*}
\lim_{\epsilon\to 0}\partial_{z}( \rho_\epsilon\ast  (X_{g',D}\circ\psi ))(u)=& \lim_{\epsilon\to 0}\,\,\psi'(u)\partial_{z}( \rho_\epsilon\ast   X_{g',D} )(\psi(u))\\
\lim_{\epsilon\to 0}\partial^2_{zz}( \rho_\epsilon\ast  (X_{g',D}\circ\psi ))(u)=&\lim_{\epsilon\to 0}\,\,(\psi'(u))^2\partial^2_{zz}( \rho_\epsilon\ast   X_{g',D} )(\psi(u))+\psi''(u)  \partial_{z}( \rho_\epsilon\ast   X_{g',D} )(\psi(u))
\end{align*}
 when inserted inside amplitudes. The first term $T^{\psi,1} (u)$ is more subtle as it involves a Wick normalisation of the squared derivative. We recall the following elementary Gaussian IBP. Assume $Y\in\C^m$ is a Gaussian vectors with entries of the type  $Y_j\in \{X_{g',D}(x_j),\partial_zX_{g',D}(x_j),\partial^2_{zz}X_{g',D}(x_j)\}$ for some points $x_j\in\Sigma$ all distinct and distinct of $u$, and    $F$ is a smooth function on $  \C^{m}$. Then
\begin{align*}
-\E[T^{\psi,1} (u)F(Y)]=& -\lim_{\eps\to 0}\E[T^{\psi,1}_\eps (u)F(Y)]
\\
=&\lim_{\eps\to 0}\sum_{j,j'}\E[\partial_z(\rho_\epsilon\ast(X_{g',D}\circ \psi))(u)Y_j]\E[\partial_z(\rho_\epsilon\ast(X_{g',D}\circ \psi))(u)Y_{j'}]\E[\partial^2_{Y_jY_{j'}}F(Y)]
\end{align*}
 Now, if $Y_j=\partial^a_zX_{g',D}(x_j)$  for some $a\in \{0,1,2\}$ then
\begin{align*}
\lim_{\eps\to 0}\E[\partial_z(\rho_\epsilon\ast(X_{g',D}\circ \psi))(u)Y_j]=&\psi'(u)\partial_z\partial_{z'}^aG_{g',D}(\psi(u),x_j)\\
=&\lim_{\eps\to 0}\psi'(u)\E[\partial_z(\rho_\epsilon\ast X_{g',D})(\psi(u))Y_j].
\end{align*}
 and this shows that 
 $$\E[T^{\psi,1} (u)F(Y)]=(\psi'(u))^2\lim_{\eps\to 0}\E[T^{1}_\eps (\psi(u))F(Y)]$$
 with $T^{1}_\eps(u):=-\Big(\partial_z( \rho_\epsilon\ast X_{g',D})(u)\Big)^2+a_{\Sigma',g',\eps}(u)$. We deduce that $ T'_{g',\epsilon}=T^\psi_{g'}$ when inserted inside amplitudes.
  \end{proof}

Combining this proposition with \eqref{defzg}, we deduce that
\begin{align} \label{psitpsiP}
\Psi_{t,\alpha_j}(\mathbf{u},\mathbf{v}) =
 Z(g)Z_{\D,g_\D}^{-1}\caA_{\psi_j(\D),\psi_j(e^{-t}\D),\hat g_j,z_j,\alpha_j,\psi_j\circ \zeta_\D} (  T^{\psi_j}_{\hat g_j}(\mathbf{u})\bar T^{\psi_j}_{\hat g_j}(\mathbf{v})).
\end{align}
This is the disk amplitude representation of the analytically continued eigenfunctions we were aiming to.

\subsubsection{Pant amplitude} \label{Pant amplitude} 

Now we focus on the pant amplitude and its regularity with respect to its parameters. First, from Proposition \ref{Weyl}, we can embed the pant amplitude in the complex plane
\begin{equation}\label{pantP}
\caA_{\caP, g_\caP,{\bf x}^{\rm mp},\boldsymbol{\alpha}^{\rm mp},\boldsymbol{\zeta}}=\caA_{\psi(\caP), \hat g_\caP,{\bf z}^{\rm mp},\boldsymbol{\alpha}^{\rm mp},\psi\circ \boldsymbol{\zeta}}.
\end{equation} 
Next, we want to regularise this amplitude by inserting additional vertex operators; these extra vertex operators will be dubbed artificial   as they are just used for regularisation and they will be removed later. So, for $m>3$ fixed (whose precise value will be chosen later), let us consider ${\bf  z}^{\rm ar}:=( z_4,\dots,z_m)\in (\psi(\caP)\cap\D)^{m-3} $  such that $z_i\not =z_j$ for $i\not=j$ and $i,j=1,\dots,m$  and call $\boldsymbol{\alpha}^{\rm ar}:=( \alpha_4,\dots,\alpha_m ) \in \R^{m-3}$ their associated weights; here the superscript ${\rm ar}$ stands for ``artificial''.  We claim the following:

\begin{proposition}
\label{anala} There exists a complex neighborhood $ \mc{U}$ of the set $\{  (\boldsymbol{\alpha}^{\rm mp},\boldsymbol{\alpha}^{\rm ar})  \in \R^{3-b}\times\R^{m-3} \, |\,\forall j=b+1,\dots, m ,\, \alpha_j<Q\}$  s.t. the function   
\begin{align*}%\label{}
(\boldsymbol{\alpha}^{\rm mp},\boldsymbol{\alpha}^{\rm ar}, {\bf   z}^{\rm mp},{\bf   z}^{\rm ar})\to\caA_{\psi(\caP), \hat g_\caP,({\bf   z}^{\rm mp},{\bf   z}^{\rm ar}),(\boldsymbol{\alpha}^{\rm mp},\boldsymbol{\alpha}^{\rm ar}),\psi\circ \boldsymbol{\zeta}}
\end{align*}
extends holomorphically in the variables $(\boldsymbol{\alpha}^{\rm mp},\boldsymbol{\alpha}^{\rm ar})\in \mc{U}$ (and ${\bf   z}^{\rm mp}, {\bf   z}^{\rm ar}$ fixed).
 This extension is continuous in 
  $(\boldsymbol{\alpha}^{\rm mp},\boldsymbol{\alpha}^{\rm ar}, {\bf   z}^{\rm mp}, {\bf  z}^{\rm ar})\in  \mc{U}\times \{({\bf   z}^{\rm mp},{\bf  z}^{\rm ar})\in \psi(\mc{P})^{m-b}| z_i\not =z_j\}$   with values in 
$$e^{\beta(\bar c\wedge 0)-R(\bar c\vee 0 )-a\sum_{j=1}^b(c_j-\bar c)^2} L^2(H^{-s}(\T)^{b},\mu_0^{\otimes_b})$$ for some $a>0$, arbitrary $\beta<s$ with $s:=  \sum_{j=b+1}^m{\rm Re}(\alpha_j)-Q\chi(\mc{P})$ and arbitrary $R>0$. 
\end{proposition}

\begin{proof}
The holomorphic extension/continuity is similar  to \cite[appendix C.1]{GKRV20_bootstrap} and details are thus left to the reader. The bound (for possible complex values of the weights) is then proved as in  Theorem \ref{integrcf}.  
 \end{proof}
 
 Finally,  we will also  consider the generalised amplitude $\caA_{\psi(\caP), U^{\rm noset}_t, \hat g_\caP,{\bf z}^{\rm mp},\boldsymbol{\alpha}^{\rm mp},\psi\circ \boldsymbol{\zeta}}$ where $U^{\rm noset}_t= \cup_{i=b+1}^3 \caD_{t,i}$ where $\caD_{t,i}$ is the ball of center $z_i$ and radius $e^{-t}$ (recall the setup detailed in Subsection \ref{setupward}). This amplitude has a hole, namely $U^{\rm noset}_t$, but no SET insertions.  The reason why we introduce it is to gain regularity of correlation functions in the variables $z_{b+1},\cdots,z_3$ since these variables are in the holes of the amplitude (recall Prop. \ref{ampSETholomorphic}). The following lemma will be used to control  this generalised amplitude when we  take $t\to \infty$:
 
 \begin{lemma}\label{opeconv}
 The following limit exists
 \begin{align}\label{opecon}
 \lim_{t\to\infty}\caA_{\psi(\caP),U^{\rm noset}_t,\hat g_\caP,({\bf   z}^{\rm mp}, {\bf   z}^{\rm ar}),(\boldsymbol{\alpha}^{\rm mp},\boldsymbol{\alpha}^{\rm ar}),\psi\circ \boldsymbol{\zeta}}=\caA_{\psi(\caP),\hat g_\caP,({\bf   z}^{\rm mp}, {\bf   z}^{\rm ar}),(\boldsymbol{\alpha}^{\rm mp},\boldsymbol{\alpha}^{\rm ar}),\psi\circ \boldsymbol{\zeta}}
\end{align}
where the convergence is as continuous linear functionals on $\bigotimes_{i=1}^be^{-\beta_i c_-}L^2(\R\times \Omega_\T)$    for $\beta_i>Q-\alpha_i$ and  $i=1,\dots,b$. 
\end{lemma}
 \begin{proof}
 We proceed as in the proof of Theorem \ref{integrcf}. 
Denote the amplitudes in \eqref{opecon} by $\caA_t$ and $\caA$ respectively. We have
\begin{align*}
\caA_t(\boldsymbol{\tilde\varphi})-\caA(\boldsymbol{\tilde\varphi})=
\caA^0_{\psi(\caP),\hat g_\caP}(\tilde{\boldsymbol{\varphi}})e^{-\frac{Q}{4\pi}\int_{\psi(\caP)} K_{\hat g_\caP}  P\tilde{\boldsymbol{\varphi}} dv_{\hat g_\caP}} \caE_t(\boldsymbol{\tilde\varphi})\end{align*}
where $\caE_t$ is the expectation
\begin{align*}%\label{}
\caE_t:= \E \big[ \prod_{i=b+1}^m V_{\alpha_i,\hat g_\caP}(z_i)e^{-\frac{Q}{4\pi}\int_{\psi(\caP)} K_{\hat g_\caP} X_{\hat g_\caP,D}\dd {\rm v}_{\hat g_\caP}}\big(e^{-\mu M_\gamma^{\hat g_\caP}(\phi_{\hat g_\caP},\psi(\caP))}-e^{-\mu M_\gamma^{\hat g_\caP}(\phi_{\hat g_\caP},\psi(\caP)\setminus U^{\rm noset}_t)}\big)\big]
\end{align*}
Set $S_t:=\big(e^{-\mu M_\gamma^{\hat g_\caP}(\phi_{\hat g_\caP},U^{\rm noset}_{t_0})}-e^{-\mu M_\gamma^{\hat g_\caP}(\phi_{\hat g_\caP},U^{\rm noset}_{t_0}\setminus    U^{\rm noset}_t )}\big)$. Then
$$ \caE_t=\E \big[ \prod_{i=b+1}^m V_{\alpha_i,\hat g_\caP}(z_i)e^{-\frac{Q}{4\pi}\int_{\psi(\caP)} K_{\hat g_\caP} X_{\hat g_\caP,D}\dd {\rm v}_{\hat g_\caP}}  e^{-\mu M_\gamma^{\hat g_\caP}(\phi_{\hat g_\caP},\psi(\caP)\setminus  U^{\rm noset}_{t_0} )}S_t \big].$$
Note that $|S_t|\leq 2$ and $\mu_0$-almost surely in $\boldsymbol{\tilde\varphi}$, $S_t$ converges almost surely (w.r.t. Dirichlet GFF measure) to $0$. Hence $\caE_t$ converges to $0$ $\mu_0$-almost surely in $\boldsymbol{\tilde\varphi}$, and so  does $\caA_t(\boldsymbol{\tilde\varphi})-\caA(\boldsymbol{\tilde\varphi})$.
 Also, using $|S_t|\leq 2$, and Theorem \ref{integrcf}, we deduce that
 \begin{align*}%\label{}
 |\caA_t(\boldsymbol{\tilde\varphi})-\caA(\boldsymbol{\tilde\varphi})|\leq e^{\beta(\bar c\wedge 0)-R(\bar c\vee 0 )-d\sum_{j=1}^b(c_j-\bar c)^2} A(\boldsymbol{\tilde\varphi})
\end{align*}
  for some $A\in L^2(H^{-s}(\T)^{b},\mu_0^{\otimes_b})$, some $d>0$, arbitrary $\beta<s$ with $s:=  \sum_{j=b+1}^m \alpha_j -Q\chi(\mc{P})$ and arbitrary $R>0$.  Hence by the dominated convergence theorem, and for the norm on $\bigotimes_{i=1}^be^{-\beta_i c_-}L^2(\R\times \Omega_\T)$,
  \begin{align*}%\label{}
\lim_{t\to\infty}
\|\caA_t-\caA\|^2&=\lim_{t\to\infty}\int |\caA_t(\boldsymbol{\tilde\varphi})-\caA(\boldsymbol{\tilde\varphi})|^2\prod_{i=1}^be^{-2\beta_i(c_i)_{-}}\,\dd \mu_0^{\otimes_b}(\boldsymbol{\tilde\varphi})=0
\end{align*}
yielding the claim.\end{proof}

\subsubsection{Gluing disks to pants}\label{pantglue} 

Now we want to glue the disk amplitudes with SET insertions to the pant amplitude. Recall the metric $\hat g$ defined in the beginning of Subsection \ref{debutproof}. Recall also, from Proposition \ref{ampSETholo}, that the correlation functions of $T_{\hat g}({\bf u})$ and $\bbar T_{\hat g}({\bf v})$ in the measure $\langle\cdot\rangle_{\hat\C,U,\hat g  }$ are well defined 
at non coinciding points $u_i,v_j\in U$.  We define 
 \begin{align}\label{TpsiSET}
 \psi_j^\ast T_{\hat g}:={\psi'_j}^2  T_{\hat g}\circ\psi_j+\tfrac{c_{\rm L}}{12}S_{\psi_j} .
\end{align}
Note that the difference with the definition of $T^{\psi_j}_{\hat g}$ in Proposition \ref{diffeoSET}  is the constant in front of the Schwarzian derivative: $c_{\rm L}=1+6Q^2$. 
 We claim

\begin{proposition}\label{glluu}
Let $\alpha_j\in\R $, for $j=1,\dots,m$ satisfy the Seiberg bounds \eqref{seiberg1} and \eqref{seiberg2} with $\mathbf{g}=0$. For $j=1,\dots,b$, pick $k_j,\tilde{k}_j\in\N$ and 
$\mathbf{u}_j\in (\D_{e^{-t}})^{k_j}$, $\mathbf{v}_j\in (\D_{e^{-t}})^{\tilde{k}_j}$. Then
\begin{align}\label{pantclaim1}
 \frac{\sqrt{2}}{(\sqrt{2}\pi)^{b-1}} & \caA_{\psi(\caP),  U^{\rm noset}_t ,\hat g_\caP,({\bf   z}^{\rm mp}, {\bf   z}^{\rm ar}),(\boldsymbol{\alpha}^{\rm mp},\boldsymbol{\alpha}^{\rm ar}),\psi\circ \boldsymbol{\zeta}}(\otimes_{j=1}^b\Psi_{t, \alpha_j}(\mathbf{u}_j,\mathbf{v}_j))\\
 = & ( Z(g)Z_{\D,g_\D}^{-1})^b\langle \prod_{j=1}^b \psi_j^\ast T_{\hat g}(\mathbf{u}_j){\psi_j^\ast \bar T_{\hat g}}(\mathbf{v}_j)\prod_{j=1}^mV_{\alpha_j,\hat g}(z_j)  
\rangle_{\hat\C,\hat g,U_t}\nonumber
 \end{align}
where $U_t=\cup_{j=1}^b\psi_j(\D_{e^{-t}}) \cup U^{\rm noset}_t$.
\end{proposition}

\begin{proof} First we observe that the LHS is well defined. Indeed, the  generalised amplitudes   $\Psi_{t, \alpha_j}(\mathbf{u}_j,\mathbf{v}_j)$ belong respectively to  $e^{-\beta_j c_- }L^2(\R\times \Omega_\T)$ with $\beta_j>Q-\alpha_j$ for $j=1,\dots,b$ and, from the bound in Proposition \ref{anala}, the LHS makes sense provided that  $\sum_{j=1}^b(\alpha_j-Q)+s>0$ (with $\chi(\mc{P})=2-b$), which is satisfied due to the Seiberg bound.  

Now we want to see the LHS as a gluing of generalised amplitudes. 
For this, we first observe that the generalised amplitude $ \Psi_{t, \alpha_j}(\mathbf{u}_j,\mathbf{v}_j)$ can be written as in \eqref{psitpsiP} and the pant amplitude as in \eqref{pantP}. Next,  consider  \eqref{psitpsiP} with regularised SET insertions, i.e. with $T_{\hat g_j,\epsilon}$ regularised as in \eqref{defSET}. Likewise let  $T_{\hat g,\epsilon}$ denote the regularised whole plane SET. 
Recalling that $\hat g_{j}=\hat g|_{\psi_j(\D)}$ and denoting ${\bf u}=({\bf u}_1,\dots,{\bf u}_b)$ and ${\bf v}=({\bf v}_1,\dots,{\bf v}_b)$, we obtain 
by Proposition \ref{glueampli}  
\begin{align}\label{gglluu}
 \caA_{\hat \C,U_t, \hat g,{\bf   z} ,\boldsymbol{\alpha}}(\prod_{j=1}^b \psi_j^\ast T_{\hat g,\epsilon}(\mathbf{u}) \psi_j^\ast \bar T_{\hat g,\epsilon}(\mathbf{v}))
  &= \frac{\sqrt{2}}{(\sqrt{2}\pi)^{b-1}} \int  \caA_{\psi(\caP),  \tilde{U}_t, \hat g_\caP,({\bf   z}^{\rm mp},{\bf   z}^{\rm ar}),(\boldsymbol{\alpha}^{\rm mp},\boldsymbol{\alpha}^{\rm ar}),\psi\circ \boldsymbol{\zeta}}(\boldsymbol{\tilde \varphi})\\
  &\times\prod_{j=1}^b\caA_{\psi(\mc{D}_j),\psi(\D_{e^{-t}}),\hat g_j,z_j,\alpha_j,\psi\circ \zeta_\D}(  T^{(j)}_{\epsilon}(\mathbf{u})\bar T^{(j)}_{\epsilon}(\mathbf{v}),
  \boldsymbol{\tilde \varphi}))\,\dd\mu_0^{\otimes_b}( \boldsymbol{\tilde \varphi})
\end{align}
where 
\begin{align}\label{apu}
T^{(j)}_{\epsilon}(u):=T^{\psi_j}_{\hat g_j,\epsilon}(u)+ (\psi'(u))^2a_{\hat\C,\hat g,\epsilon}(\psi(u))-\psi'_j(u)^2a_{\psi(\mc{D}_j),\hat g_j}(\psi_j(u))+\frac{1}{12}S_{\psi_j}(u).
\end{align}
From Lemma \ref{constanteSET} we have $\lim_{\epsilon\to 0}a_{\hat\C,\hat g,\epsilon}=0$ and 
\begin{align}\label{apu1}
\lim_{\epsilon\to 0}\psi'_j(u)^2a_{\psi(\mc{D}_j),\hat g_j}(\psi_j(u))=\frac{1}{12}S_{\psi_j}(u)+a_{\D,g}(u)=\frac{1}{12}S_{\psi_j}(u)
\end{align}
since $a_{\D,g}(u)=0$.
The equality \eqref{gglluu} holds  provided that   the integrand in the RHS is   integrable but this is again a consequence of the Seiberg bound and the bound for the $\psi(\mc{P})$-amplitude in Proposition \ref{anala} (recall that the limit \eqref{limphieps} holds in $ e^{-\beta c_- }L^2(\R\times \Omega_\T)$, for $\beta>(Q-\alpha)$). Also, these estimates ensure that we can pass to the limit as $\epsilon\to 0$ in      the RHS to obtain the LHS of our claim.  Using \eqref{apu} and \eqref{apu1} we then conclude that the LHS converges to the RHS of the claim.
\end{proof} 
 
For later technical reasons  we will need to work with another metric than $\hat g$: recall that this metric depends implicitly on the arbitrary but fixed points $z_1,z_2,z_3$ and we want to remove this dependence. For this we could choose  any metric that is  Euclidean in the unit disk and our choice will be   the DOZZ metric $g_{\rm dozz}=|z|_+^{-4} |dz|^2$, especially because it will be easier later to connect with the DOZZ formula.  Proposition \ref{WeylSET} gives
\begin{align}\label{Zzhat1}
\langle \prod_{j=1}^b \psi_j^\ast T_{\hat g}(\mathbf{u}_j){\psi_j^\ast \bar T_{\hat g}}(\mathbf{v}_j)\prod_{j=1}^mV_{\alpha_j,\hat g}(z_j)  
\rangle_{\hat\C,\hat g,U_t}=Z({\bf z},\boldsymbol{\alpha},\hat g)\langle \prod_{j=1}^b \psi_j^\ast T_{g_{\rm dozz}}(\mathbf{u}_j){\psi_j^\ast \bar T_{g_{\rm dozz}}}(\mathbf{v}_j)\prod_{j=1}^mV_{\alpha_j,g_{\rm dozz}}(z_j)  
\rangle_{\hat\C,g_{\rm dozz},U_t}
\end{align}
with 
\begin{equation}\label{Zzhat}
Z({\bf z},\boldsymbol{\alpha},\hat g):=e^{c_{\rm L} S^0_{\rm L}(\hat\C,g_{\rm dozz},\hat \omega)-\sum_{j=1}^m\Delta_{\alpha_j}\hat\omega(z_j)}
\end{equation}
where we wrote  $\hat g =e^{ \hat\omega} g_{\rm dozz}$ for some   function $ \hat\omega\in H^1(\hat\C)$.
We then have our final result about gluing
(recall the definition of $a$ in Proposition \ref{eigenana}):

\begin{proposition}\label{witht}
Let $\alpha_j\in\R $, for $j=1,\dots,m$, satisfy the Seiberg bounds \eqref{seiberg1} and \eqref{seiberg2} with $\mathbf{g}=0$,   with the condition $\alpha_j<a\wedge(Q-\gamma)$ for  $j=1,..,b$. Consider Young diagrams $\nu_j,\tilde\nu_j$ for $j=1,..,b$. Then 
\begin{align}\label{pantclaimmm}
  \frac{\sqrt{2}}{(\sqrt{2}\pi)^{b-1}}&  \caA_{\psi(\caP), \hat g_\caP,({\bf   z}^{\rm mp},{\bf z}^{\rm ar}),(\boldsymbol{\alpha}^{\rm mp},\boldsymbol{\alpha}^{\rm ar}),\psi\circ \boldsymbol{\zeta}}\big(\otimes_{j=1}^b\Psi_{\alpha_j,\nu_j,\tilde\nu_j}\big)
\\
= & ( Z(g)Z_{\D,g_\D}^{-1})^bZ({\bf z},\boldsymbol{\alpha},\hat g) \frac{1}{(2\pi i)^{\sum_{j=1}^b(s(\nu_j)+s(\tilde{\nu}_j))}} \lim_{t\to\infty} \nonumber
\\
  &    \oint_{|\mathbf{u}|=\boldsymbol{\delta}_t}   \oint_{|\mathbf{v}|=\tilde{\boldsymbol{\delta}}_t}  
\mathbf{u}^{1-\boldsymbol{\nu}}\bar{\mathbf{v}}^{1-\tilde{\boldsymbol{\nu}}} \langle \prod_{j=1}^b \psi_j^\ast T_{g_{\rm dozz}}(\mathbf{u}_j){\psi_j^\ast\bar  T_{g_{\rm dozz}}}(\mathbf{v}_j)\prod_{j=1}^mV_{\alpha_j,g_{\rm dozz}}(z_j)  
\rangle_{\hat\C,g_{\rm dozz},U_t}\dd  \mathbf{u}\dd  \mathbf{v}   \nonumber
\end{align}
where $Z(g)$ was defined just after \eqref{defzg} and we used the notations 
${\bf u}=({\bf u}_1,\dots,{\bf u}_b)$ and ${\bf v}=({\bf v}_1,\dots,{\bf v}_b)$, $\boldsymbol{ \nu}=( \nu_1,\dots,\nu_b)$ and $\tilde{\boldsymbol{\nu}}=( \tilde\nu_1,\dots, \tilde\nu_b)$, the powers $\mathbf{u}^{1-\boldsymbol{\nu}}$ and $\bar{\mathbf{v}}^{1-\tilde{\boldsymbol{\nu}}} $ are shorthands  respectively for  $\prod_{j=1}^b\mathbf{u}_j^{1-\nu_j}$ and $\prod_{j=1}^b\bar{\mathbf{v}}_j^{1-\tilde \nu_j}$, and $\boldsymbol{\delta}_t=e^{-t}\boldsymbol{\delta}$ (similarly for $\tilde{\boldsymbol{\delta}}_t$, with $\boldsymbol{\delta},\tilde{\boldsymbol{\delta}}$ as described in section \ref{Eigenfunctions and generalised amplitudes}).
\end{proposition} 

\begin{proof}
We combine \eqref{pantclaim1} with the Weyl formula \eqref{Zzhat1} and take the contour  integrals of the resulting identity.
Using Lemma  \ref{TTLemma} we may  invert expectation and contour integrals.  The convergence   \eqref{limtpsi}, Lemma  \ref{TTLemma}  and the convergence of the amplitudes Lemma \ref{opeconv}  then show that the contour integrals of the  LHS of  \eqref{pantclaim1} converge to the LHS of \eqref{pantclaimmm}. 
\end{proof}

  We will now turn to computing the limit of the RHS of \eqref{pantclaimmm} using the Ward identities.   
  
\subsubsection{Residue calculation}\label{residue}

Denote ${\bf u}=({\bf u}_1,..,{\bf u}_b)$ with $k_i$ the size of ${\bf u}_i$ (and similarly for ${\bf v}$ with $\tilde k_i$ the size of ${\bf v}_i$) with
${\bf u}_i=(u_{i1},\dots,u_{ik_i})$ and similarly for ${\bf v}_i$. Define the vectors $\boldsymbol{\xi}_i,\boldsymbol{\eta}_i$ for $i=1,\dots, b$ by $\xi_{ij}=\psi_i(u_{ij})$ ($j=1,\dots k_i$) and $\eta_{ij}=\psi_i(v_{ij})$ ($j=1,\dots \tilde k_i$). Also, if $I_i\subset \{1,\dots,k_i\}$ (resp. $J_i\subset \{1,\dots,\tilde k_i\}$) we denote by $\boldsymbol{\xi}_{I_i}$ (resp. $\boldsymbol{\eta}_{J_i}$) the vector $(\xi_{ij})_{j\in I_i}$ (resp. $(\eta_{ij})_{j\in J_i}$) and finally we set $\boldsymbol{\xi}=(\boldsymbol{\xi}_1,\dots,\boldsymbol{\xi}_b)$ and $\boldsymbol{\eta}=(\boldsymbol{\eta}_1,\dots,\boldsymbol{\eta}_b)$.

With these notations and recalling that  $\psi_i^\ast T_{g_{\rm dozz}}(\mathbf{u}_i)$ is a product of sums made up of two terms given  by \eqref{TpsiSET}, we can expand this product (as well as the product $ \psi_i^\ast T_{g_{\rm dozz}}(\mathbf{v}_i)$) in the expectation in the r.h.s. of Proposition \ref{witht}  to obtain
\begin{align}\label{conformal11}
 \langle \prod_{i=1}^b( \psi_i^\ast &T_{g_{\rm dozz}}(\mathbf{u}_i) )(\overline{ \psi_i^\ast T_{g_{\rm dozz}}}(\mathbf{v}_i)) \prod_{j=1}^mV_{\alpha_j,g_{\rm dozz}}(z_j)  
\rangle_{\hat \C,g_{\rm dozz},U_t}
\\
=&\sum_{\bf I,J}
F_{\bf I}({\bf u})\overline{F_{\bf J}({\bf v})} \langle \prod_{i=1}^b( T_{g_{\rm dozz}}(\boldsymbol{\xi}_{I_i})\bbar T_{g_{\rm dozz}}(\boldsymbol{\eta}_{J_i}) \prod_{j=1}^mV_{\alpha_j,g_{\rm dozz}}(z_j) 
\rangle_{\hat \C,g_{\rm dozz},U_t}\nonumber
\end{align}
where the sum runs over ${\bf I}=(I_1,..,I_b)$, ${\bf J}=(J_1,..,J_b)$ with $I_i\subset\{1,\dots, k_i\}$ and $J_i\subset\{1,\dots, \tilde k_i\}$. 
The function  $F_{\bf I}$   is a  product of terms ${\psi'_i}(u_{ij})^2$ and $S_{\psi_i}(u_{ij})$ 
  and therefore analytic and bounded in $u_{ij}\in \D_{e^{-t}}$ .

We can now apply the Ward identities (Proposition \ref{wardite}) to the correlation function in \eqref{conformal11}: if we set ${\bf z}=(z_1,\dots,z_m)$, then 
 \begin{align}
 \langle \prod_{i=1}^b&( T_{{g_{\rm dozz}}}(\boldsymbol{\xi}_{I_i})\bar T_{{g_{\rm dozz}}}(\boldsymbol{\eta}_{J_i}) \prod_{j=1}^mV_{\alpha_j,g_{\rm dozz}}(z_j) 
\rangle_{\hat\C,{g_{\rm dozz}},U_t}\nonumber\\
  &=\sum_{ {\bf r},  \tilde{ \bf r}} 
  \caP_{\bf r}^{\bf I}(\boldsymbol{\xi},{\bf z}) {\caP_{\tilde{\bf r}}^{\bf J}(\bbar{\boldsymbol{\eta}},\bbar{\bf z})}\partial^{\bf r}_{\bf z} \partial^{\tilde{\bf r}}_{\bar{\bf z}}
  \langle  \prod_{j=1}^mV_{\alpha_j,g_{\rm dozz} }(z_j) \rangle_{\hat\C,{g_{\rm dozz}},U_t}+\epsilon_t{(\bf u,v,z})
  \label{wrdd}
 \end{align}
 where   (recall the definition of  $p_{\bf s}(\boldsymbol{\xi})$ and  $p_{\bf n}(\boldsymbol{\xi},{\bf z})$ in \eqref{defpn})  
 \begin{align*}
 \caP_{\bf r}^{\bf I}(\boldsymbol{\xi},{\bf z})=\sum_{{\bf s},{\bf n}}  a_{\bf s,n, r}^{\bf I} p_{\bf s}(\boldsymbol{\xi}) p_{\bf n}(\boldsymbol{\xi},{\bf z}),%
\end{align*}
the sum runs over the set of indices ${\bf r} \in \N^m$ with $ |{\bf n}|+ |{\bf s}|+|{\bf r}|=2\sum_{i=1}^b|I_i|$ (and similarly for ${\tilde {\bf r}}$) and the coefficients $ a^{\bf I}_{\bf s,n, r} \in\R$ are polynomials in the conformal weights $\Delta_{\alpha_i}$. Furthermore $ a^{\bf I}_{\bf s,n, r} =0$ if there is $n_{ij}\neq 0$ and $r_j=0$. Here we have used the fact that $\hat\omega(z_j)=1$ for all $j$ for the DOZZ metric in order to remove both products in front of the correlation functions in the statement of Proposition \ref{wardite}:  this is  one reason why we chose a metric Euclidean over the unit disk.
The derivatives in \eqref{wrdd} are in the sense of distributions in the variables $z_4, \dots, z_m$. 
Note that the set $U_t$ depends on $\hat{{\bf z}}=(z_1,z_2,z_3)=({\bf z}^{\rm di},{\bf z}^{\rm mp})$ since the conformal map $\psi:\hat\caP\to\hat\C$ does. However, the derivatives $\partial^{\bf r}_{\bf z} \partial^{\tilde{\bf r}}_{\bar{\bf z}}
$ do not act on this dependence: the conclusions of Proposition \ref{wardite} are that $\partial^{\bf r}_{\bf z} \partial^{\tilde{\bf r}}_{\bar{\bf z}}
 \langle  \prod_{j=1}^mV_{\alpha_j,g_{\rm dozz} }(z_j) \rangle_{\C,g_{{\rm dozz}},U_t}$ is to be understood as $\partial^{\bf r}_{\bf w} \partial^{\tilde{\bf r}}_{\bar{\bf w}}
  \langle  \prod_{j=1}^mV_{\alpha_j,{\rm dozz} }(w_j) \rangle_{\C,g_{{\rm dozz}},U_t}$ evaluated at ${{\bf w}={\bf z}}$. So, here and in the calculations that follow, we treat $U_t$ as fixed.

The remainder (given by Proposition \ref{wardite}) satisfies for smooth $(f_i)_{i=4, \dots, m}$   with  disjoint compact support (as in Proposition \ref{wardite})
 \begin{align*}
\int_{\C^{m-3}}  \oint_{|\mathbf{u}|=\boldsymbol{\delta}_t}   \oint_{|\mathbf{v}|=\boldsymbol{\tilde\delta}_t}  
|\mathbf{u}^{1-\boldsymbol{\nu}}\bar{\mathbf{v}}^{1-\tilde{\boldsymbol{\nu}}} \epsilon_t{(\bf u,v},{\bf z}) \prod_{i=4}^m f_i(z_i) |\dd{\bf u}\dd{\bf v} \dd {\bf z}^{\rm ar}\leq C e^{(|\boldsymbol{\nu}|+|\tilde{\boldsymbol{\nu}}|-2k-2\tilde k)t}e^{(2(k+\tilde k-1)+\gamma\alpha)t} 
\end{align*}
where $k=\sum_ik_i$ and  $\tilde k=\sum_i\tilde k_i$, $\alpha=\max_{i\leq 3}\alpha_i$ and $C$ is a constant depending on the uniform norm of the derivatives of the $(f_i)_{i=4, \dots, m}$.
Hence the remainder vanishes as $t\to\infty$ provided $\alpha$ is small enough, which we may assume by adding the condition $\alpha_j< a\wedge(Q-\gamma)$ for $j=b+1, \dots, 3$ and changing the value of $a$. %in Prop. \ref{eigenana}.

We are thus left with studying the contour integral
\begin{align*}%\label{}
\caI_{t,\bf I,s,n}({\bf z})= \oint_{|\mathbf{u}|=\boldsymbol{\delta}_t}  \mathbf{u}^{1-\boldsymbol{\nu}}F_{\bf I}({\bf u})p_{\bf s}(\boldsymbol{\xi})p_{\bf n}(\boldsymbol{\xi},{\bf z})\dd{\bf u}
\end{align*}
(the ${\bf v}$-contour integrals are of the same form). 

Now we rewrite both $p_{\bf s}(\boldsymbol{\xi})$ and $p_{\bf n}(\boldsymbol{\xi},{\bf z})$
in a form that will be convenient to compute residues.
For $i=1,\dots,b$, since   $\psi_i(0)=z_i$, we have $\psi_i(u)=z_i+u\zeta_i(u)$ with $\zeta_i$ holomorphic and non-vanishing  on $\D$. 
Setting $\lambda_i=\psi'_i(0)$ we   get then for $q\in\N$
\begin{align}\label{psirel1}
%(x^i_j-x^i_k)^{-1}=(\psi_i(u^i_j)-\psi_i(v^i_k))^{-1}=\xi(u^i_j,v^i_k) ((u^i_j-v^i_k)^{-1}
(\psi_i(u)-\psi_i(u'))^{-q}=\lambda_i^{-q}(u-u')^{-q}(1+\xi_{i,q}(u,u') )
\end{align}
 where  $\xi_i$ is holomorphic in $\D\times\D$ with $\xi_i(0,0)=0$. Similarly for $u\in\D,u'\in\D$, $i\not=j$ and $q\in\N$
 \begin{align}\label{psirel2}
(\psi_i(u)-\psi_j(u'))^{-q}=&(z_i-z_j+u\zeta_i(u)-u'\zeta_j(u'))^{-q}=(z_i-z_j)^{-q}(1+\frac{u\zeta_i(u)-u'\zeta_j(u')}{z_i-z_j})^{-q}\nonumber\\
=&(z_i-z_j)^{-q}\sum_{n\geq 0}c_{q,n}(z_i-z_j)^{-n}( u\zeta_i(u)-u'\zeta_j(u'))^n\nonumber\\
=&(z_i-z_j)^{-q}\sum_{n\geq 0}c_{q,n}(z_i-z_j)^{-n}\sum_{\ell=0}^n   u^{\ell}(u')^{n-\ell} h_{ijn\ell}(u,u')
\end{align}
where $h_{ijn\ell}(u,u')$ is is holomorphic around $(0,0)$. The second line was obtained by expanding the function $(1+x)^{-q}$ as a series and the third line by expanding the  term $( u\zeta_i(u)-u'\zeta_j(u'))^n$ with the Newton binomial formula. Recalling that $\xi_{ij}=\psi(u_{ij})$ we obtain
 \begin{align}\label{pmx}
 p_{\bf s}(\boldsymbol{\xi})=p_{{\bf s}_0}( {\bf z}^{\rm di})\Big(\prod_{i=1}^b p_{{\bf s}_i}(\lambda_i{\bf u}_i)\Big)
 \sum_{\bf s'}\theta_{ {\bf s},{\bf s}'}({\bf u})p_{{\bf s}'}( {\bf z}^{\rm di})  
\end{align}
for some ${\bf s}_i$, $i=1,..,b$ with $|{\bf s}_0|+\sum_{i=1}^b|{\bf s}_i|=|{\bf s}|$, and  where $\theta_{ {\bf s}, \bf s'}$ is  of the form 
$$\theta_{  {\bf s},\bf s'}({\bf u})=\sum_{|\boldsymbol{\ell}|=|{ \bf s'}|}{ \bf u}^{\boldsymbol{\ell}}\theta_{\boldsymbol{\ell},  {\bf s},\bf s'}({\bf u})$$
with $\theta_{\boldsymbol{\ell},  {\bf s},\bf s'}$ analytic in ${\bf u}$ over a neighborhood of $0$ and $\theta_{  {\bf s},\bf 0}({\bf u})=1$. To get this expression, starting from the definition of  $p_{\bf s}(\boldsymbol{\xi})$ by a product \eqref{defpn}, we have just gathered   the terms producing $\lambda_i(u_{ij}-u_{ij'})$ with $j,j'\in I_i$, i.e. from \eqref{psirel1}, to form $p_{{\bf s}_i}(\lambda_i{\bf u}_i)$  and gathered the terms producing the terms $(z_i-z_j)$ with $i\not =j$, i.e. from \eqref{psirel2}, to form $p_{{\bf s}_0}( {\bf z}^{\rm di})$    up to holomorphic error terms given by the series. The series $ \sum_{\bf s'}\theta_{  {\bf s},\bf s'}({\bf u})p_{{\bf s}'}( {\bf z}^{\rm di})  $ converges uniformly in a ($t$-independent) neighborhood of $0$.

Now we proceed similarly to analyze $p_{\bf n}(\boldsymbol{\xi},{\bf z})$.
Taking $u'=0$, we also have for $u\in\D$ and for $q\in\N$
$$
(\psi_i(u)-z_i)^{-q}=\lambda_i^{-q} u^{-q}(1+\eta_{i,q}(u))
$$
with analytic $\eta_{i,q}$ vanishing at $0$. Also, for $z_i\not=z_j$,
\[
 (\psi_i(u)-z_j)^{-q}=(z_i-z_j)^{-q}(1+\frac{u\zeta_{i}(u)}{z_i-z_j})^{-1}.
 \]
Arguing as we did in \eqref{psirel2} , we deduce
  \begin{align}\label{pmxiz}
 p_{\bf n}(\boldsymbol{\xi},{\bf z})=\prod_{i=1}^b(\lambda_i {\bf u}_i)^{-\boldsymbol{\rho}_i}p_{{\bf n}_0}({\bf z}^{\rm di}) 
 \sum_{\bf n'}\vartheta_{ {\bf n},\bf  n'}({\bf u})p_{{\bf n'}}({\bf z}) 
\end{align}
for some (vectorial) exponents $\boldsymbol{\rho}_i$ and ${\bf n}_0$ s.t. $\sum_{i=1}^b|\boldsymbol{\rho}_i|+|{\bf n}_0|=|{\bf n}|$ and 
\[
\vartheta_{{\bf n}, \bf n'}({\bf u})=\sum_{|\boldsymbol{\ell}|=|{ \bf n'}|}{ \bf u}^{\boldsymbol{\ell}}\vartheta_{\boldsymbol{\ell},{\bf n}, \bf n'}({\bf u})
\]
with $\vartheta_{\boldsymbol{\ell},{\bf n}, \bf n'}$ analytic in ${\bf u}$ over a neighborhood of $0$ and $\vartheta_{ {\bf n},\bf 0}({\bf u})=1$.  Again, the series $ \sum_{\bf n'}\vartheta_{{\bf n},\bf  n'}({\bf u})p_{{\bf n'}}({\bf z})$ converges uniformly in a neighborhood of $0$.

Define now the series
$$\sum_{\bf m} t_{\bf s,n,m}({\bf u})p_{{\bf m}}({\bf z}):=\sum_{\bf s',\bf n'}\vartheta_{ {\bf n},\bf  n'}({\bf u}) \theta_{ {\bf s},{\bf s}'}({\bf u})p_{{\bf n'}}({\bf z}) p_{{\bf s}'}( {\bf z}^{\rm di})  $$
with the coefficients $t_{\bf s,n,m}$ of the form $t_{\bf s,n,m}({\bf u})=\sum_{|\boldsymbol{\ell}|=|{ \bf m}|}{ \bf u}^{\boldsymbol{\ell}}t_{\bf \boldsymbol{\ell},s,n,m}({\bf u})$ and $t_{\bf \boldsymbol{\ell},s,n,m}$ analytic in ${\bf u}$ over a neighborhood of $0$ and $t_{\bf s,n,0}({\bf u})=1$. Furthermore, the series $\sum_{\bf m}p_{{\bf m}}({\bf z})t_{\bf s,n,m}({\bf u})$ converges uniformly in a neighborhood of $0$. 
We deduce
\begin{align}  
 \caI_{t,\bf I,s,n}({\bf z})
 =C
  p_{\bf s_0}({\bf z}^{\rm di})p_{\bf n_0}({\bf z}^{\rm di})\sum_{\bf m}p_{{\bf m}}({\bf z}) \oint_{|\mathbf{u}|=\boldsymbol{\delta}_t}  t_{\bf s,n,m}({\bf u}) F_{\bf I}({\bf u})\prod_{i=1}^b\mathbf{u}_i^{1-\nu_i-\boldsymbol{\rho}_i}p_{{\bf s}_i}({\bf u}_i)\dd{\bf u}\label{relforI}
 \end{align}
where the constant $C$ gathers the contribution from the various powers of the $\lambda_i$'s.

To compute the contour integral, we make use of the following elementary result
\begin{lemma}\label{contourp}
Let ${\bf p}=(p_1,\dots,p_k)\in\Z^k$ and ${\bf e}=(e_{ij})_{1\leq i<j\leq k}\in \N^{k(k-1)/2}$. Let $\boldsymbol{\delta}=(\delta_1,\dots,\delta_k)\in \R_+^k$ with $\delta_1<\dots<\delta_k$. Then the contour integral $\oint_{|{\bf u}|=\boldsymbol{\delta}}p_{{\bf e}}({\bf u}){\bf u}^{-{\bf p}}\dd {\bf u}$ vanishes for $|{\bf p}|+|{\bf e}|\not=k$. 
\end{lemma}

\begin{proof} By holomorphicity, the mapping $\lambda\in\R_+ \mapsto \oint_{|{\bf u}|=\lambda\boldsymbol{\delta}}p_{{\bf e}}({\bf u}){\bf u}^{-{\bf p}}\dd {\bf u}$ is actually constant. Furthermore a simple scaling argument shows that $$ \oint_{|{\bf u}|=\lambda\boldsymbol{\delta}}p_{{\bf e}}({\bf u}){\bf u}^{-{\bf p}}\dd {\bf u} = \lambda^{-|{\bf e}|-|{\bf p}|+k}\oint_{|{\bf u}|= \boldsymbol{\delta}}p_{{\bf e}}({\bf u}){\bf u}^{-{\bf p}}\dd {\bf u}$$
from which our claim follows.
\end{proof}

To use this lemma, we expand
\begin{align*}
 \oint_{|\mathbf{u}|=\boldsymbol{\delta}_t}  t_{\bf s,n,m}({\bf u})F_{\bf I}({\bf u}) \prod_{i=1}^b\mathbf{u}_i^{1-\nu_i-\boldsymbol{\rho}_i}p_{{\bf s}_i}({\bf u}_i)\dd{\bf u} 
 &= 
  \sum_{|\boldsymbol{\ell}|=|{ \bf m}|} \oint_{|\mathbf{u}|=\boldsymbol{\delta}_t}  { \bf u}^{\boldsymbol{\ell}} \prod_{i=1}^b\mathbf{u}_i^{1-\nu_i-\boldsymbol{\rho}_i}p_{{\bf s}_i}({\bf u}_i)t_{\bf \boldsymbol{\ell},s,n,m}({\bf u})F_{\bf I}({\bf u})\dd{\bf u}.
\end{align*}
Then we expand again $t_{\bf \boldsymbol{\ell},s,n,m}({\bf u})F_{\bf I}({\bf u})= \sum_{{\bf p} } b_{\bf p}{\bf u}^{\bf p}$ and apply the previous elementary lemma to see that the contour integral is $0$ if $\sum_{i=1}^b |{\bf s}_i|+\sum_{i=1}^b|\boldsymbol{\rho}_i| + |{\boldsymbol{\nu}}|-k-|{\bf p}|-|{\bf m}|\not=k$. In particular, since $\sum_{i=1}^b |I_i| <k$, this contour integral is $0$ if 
\begin{equation}\label{conditionintnonzero}
|{\bf m}|>\sum_{i=1}^b |{\bf s}_i|+\sum_{i=1}^b|\boldsymbol{\rho}_i| + |{\boldsymbol{\nu}}|-2\sum_{i=1}^b |I_i|-|{\bf p}|.
\end{equation}

Considering then the product of the contour integral with $  p_{\bf s_0}({\bf z}^{\rm di})p_{\bf n_0}({\bf z}^{\rm di})$, this gives that 
\[ \caI_{t,\bf I,s,n}({\bf z})=\sum_{\bf m}b_{\bf s,n,m}^{\bf I}p_{\bf m}(\bf z)\]
 where the coefficients $b^{\bf I}_{\bf s,n,m}$ do not depend on $t$ and  $b^{\bf I}_{\bf s,n,m}=0$ if   $|{\bf m}|> |{\bf s}|+ |{\bf n}| + |{\boldsymbol{\nu}}|-2\sum_{i=1}^b |I_i| =|{\boldsymbol{\nu}}| -|{\bf r}|$ (this can be seen by using \eqref{conditionintnonzero} along with $\sum_{i=1}^b|\boldsymbol{\rho}_i|= |{\bf n}|- |{\bf n}_0|$ with $|{\bf n}_0| \geq 0$ and also the relation $ |{\bf n}|+ |{\bf s}|+|{\bf r}|=2\sum_{i=1}^b|I_i|$).

Now we consider the differential operator
 \begin{equation}\label{diffop}
 \caD_{\boldsymbol{\alpha},\boldsymbol{\nu},{\bf z}}=\sum_{\bf r,m}a_{\bf r,m,\boldsymbol{\alpha},\boldsymbol{\nu}}p_{\bf m}({\bf z})\partial^{\bf r}_{\bf z}  \quad \text{with}\quad a_{\bf r,m,\boldsymbol{\alpha},\boldsymbol{\nu}}:=\sum_{\bf I}\sum_{{\bf s},{\bf n},\bf m}b^{\bf I}_{\bf s,n,m} a^{\bf I}_{\bf s,n,r}.
 \end{equation}
 The sum is finite since it runs over $ |{\bf n}|+ |{\bf s}|+|{\bf r}|=2\sum_{i=1}^b|I_i|$  and $|{\bf m}|\leq|\boldsymbol{\nu}|-|{\bf r}|$. Also, the coefficients  are polynomials in $\Delta_{\alpha_j}$.

To summarise, we have obtained: 
  \begin{proposition}\label{withtbis}
 Given $\boldsymbol{\nu}=(\nu_1,\dots,\nu_b)\in \mc{T}^b$, $\tilde{\boldsymbol{\nu}}=(\tilde{\nu}_1,\dots,\tilde{\nu}_b)$, $\alpha_j<a\wedge(Q-\gamma)$ for  $j=1,..,3$ and  $\alpha_j<Q$ for $j>3$ satisfying $\sum_j\alpha_j>2Q$, we have  (in the strong sense in $\hat {\bf z}$ and in the distributional sense in ${\bf z}^{\rm ar}$)
  \begin{align*}
 \frac{1}{(2\pi i)^{\sum_{j=1}^b(s(\nu_j)+s(\tilde{\nu}_j))}} & \oint_{|\mathbf{u}|=\boldsymbol{\delta}_t}   \oint_{|\mathbf{v}|=\boldsymbol{\tilde\delta}_t}  
\mathbf{u}^{1-\boldsymbol{\nu}}\bbar{\mathbf{v}}^{1-\tilde{\boldsymbol{\nu}}}
 \langle \prod_{i=1}^b( \psi_i^\ast T_{g_{\rm dozz}}(\mathbf{u}_i) )( \psi_i^\ast \bar T_{g_{\rm dozz}}(\mathbf{v}_i)) \prod_{j=1}^mV_{\alpha_j,g_{\rm dozz}}(z_j)  
\rangle_{\C,g_{\rm dozz},U_t}
\dd{\bf u}\dd{\bf v}\\
  &=\caD_{\boldsymbol{\alpha},\boldsymbol{\nu},{\bf z}} \caD_{\boldsymbol{\alpha},\tilde{\boldsymbol{\nu}},\bar{\bf z}} 
  \langle  \prod_{j=1}^mV_{\alpha_j,g_{\rm dozz}}(z_j) \rangle_{\hat \C,{g_{\rm dozz}},U_t}+\epsilon_t({\bf z} )
%\label{wrdd}
 \end{align*}
where $\int_{\C^{m-3}} |\epsilon_t({\bf z}) \prod_{i=4}^m f_i(z_i)| \dd {\bf z}^{\rm ar} \to 0$ as $t\to\infty$  for smooth $(f_i)_{i=4, \dots, m}$ satisfying the conditions of Proposition \ref{wardite}.  
\end{proposition}

By choosing $a$ sufficiently small, one can apply Lemma \ref{difflemma} to get the following convergence
$$
\caD_{\boldsymbol{\alpha},\boldsymbol{\nu},{\bf z}} \caD_{\boldsymbol{\alpha},\boldsymbol{\tilde\nu},\bar{\bf z}} 
  \langle  \prod_{j=1}^mV_{\alpha_j,g_{\rm dozz} }(z_j) \rangle_{\C,{g_{\rm dozz}},U_t}\to \caD_{\boldsymbol{\alpha},\boldsymbol{\nu},{\bf z}} \caD_{\boldsymbol{\alpha},\tilde{\boldsymbol{\nu}},\bar{\bf z}} 
  \langle  \prod_{j=1}^mV_{\alpha_j,g_{\rm dozz} }(z_j) \rangle_{\C,{g_{\rm dozz}}}
$$
as $t\to \infty$ in the strong sense in the variables $z_1,z_2,z_3$ and in the distributional sense for the variables ${\bf z}^{\rm ar}$. We may now combine this result with Proposition \ref{witht}, Proposition \ref{anala} and 
Lemma  \ref{TTLemma} to get:
\begin{proposition}\label{pantalphas}
Given $\boldsymbol{\nu},\tilde{\boldsymbol{\nu}}$ Young diagrams, $\alpha_j<a\wedge(Q-\gamma)$ for  $j=1,..,3$ and  $\alpha_j<Q$ for $j>3$ satisfying $\sum_j\alpha_j>2Q$ (recall the definition \eqref{Zzhat})
 \begin{align}\label{final1}
 & \frac{\sqrt{2}}{(\sqrt{2}\pi)^{b-1}}   \caA_{\psi(\caP),\hat g_\caP,({\bf   z}^{\rm mp}, {\bf   z}^{\rm ar}),(\boldsymbol{\alpha}^{\rm mp},\boldsymbol{\alpha}^{\rm ar}),\psi\circ \boldsymbol{\zeta}}\big(\otimes_{j=1}^b\Psi_{\alpha_j,\nu_j,\tilde{\nu}_j}\big)\\
 & = ( Z(g)Z_{\D,g_\D}^{-1})^bZ({\bf z},\boldsymbol{\alpha},\hat g)
\caD_{\boldsymbol{\alpha},\boldsymbol{\nu},{\bf z}} \caD_{\boldsymbol{\alpha},\tilde{\boldsymbol{\nu}},\bar{\bf z}} 
\langle \prod_{j=1}^mV_{\alpha_j,g_{\rm dozz}}(z_j)
\rangle_{\hat\C,g_{\rm dozz}}.\nonumber
 \end{align}
 The LHS in  \eqref{final1} is continuous in $\bf z$ and holomorphic in $\boldsymbol{\alpha}$ in a complex neighborhood of the above region and the 
 equality \eqref{final1} holds in the distributional sense. 
 \end{proposition}

We note that the above statement does not make use anymore of strong derivatives. Indeed, strong derivatives were used to keep $z_1,z_2,z_3$ fixed, and therefore the domain $U_t$ fixed, so that IBP with respect to the operator $ \caD_{\boldsymbol{\alpha},\boldsymbol{\nu},{\bf z}}$ are valid in the previous statements: if distributional derivatives in $z_1,z_2,z_3$ were considered instead, then  $U_t$ would have an annoying dependence in these variables and IBP would be problematic. But now that the $U_t$ dependence is removed there is no need anymore to use strong derivatives.

\subsubsection{Proof of Theorem \ref{pantDOZZ}}\label{finalst}
%%%%%%%%%%%%%%%%%%%%%%%%%%%%%%%

We want to integrate \eqref{final1} against test functions to perform an integration by parts in the RHS.  Let $\caO=\{({\bf z} \in\C^{m}  \,|\,\forall i\not= j,\, z_i\neq z_j\}$ and let $f\in C_0^\infty(\caO)$. We test \eqref{final1} against $f$ 
(with $\caD^\dagger_{\boldsymbol{\alpha},\boldsymbol{\nu},{\bf z}}  $ the formal adjoint of $\caD_{\boldsymbol{\alpha},\boldsymbol{\nu},{\bf z}}  $)
\begin{align}\label{test1}
\int
 \frac{\sqrt{2}}{(\sqrt{2}\pi)^{b-1}} &  \caA_{\psi(\caP), \hat g_\caP,({\bf   z}^{\rm mp},{\bf z}^{\rm ar}),(\boldsymbol{\alpha}^{\rm mp},
 \boldsymbol{\alpha}^{\rm ar}),\psi\circ \boldsymbol{\zeta}}\big(\otimes_{j=1}^b\Psi_{\alpha_j,\nu_j, \tilde{\nu}_j}\big)f({\bf z})\dd {\bf z}\nonumber\\
 =&
( Z(g)Z_{\D,g_\D}^{-1})^b
\int
\langle \prod_{j=1}^mV_{\alpha_j,g_{\rm dozz}}(z_j)
\rangle_{\hat\C,g_{\rm dozz}}\caD^\dagger_{\boldsymbol{\alpha},\boldsymbol{\nu},{\bf z}} \caD^\dagger_{\boldsymbol{\alpha},\tilde{\boldsymbol{\nu}},\bar{\bf z}} \big(Z({\bf z},\boldsymbol{\alpha},\hat g)f({\bf z})\big)
\dd\bf z.
\end{align}
Since $a_{\bf r,\bf m,\boldsymbol{\alpha},\boldsymbol{\nu}}$ are polynomials in $\Delta_{\alpha_j}$ (recall \eqref{diffop}) and since the LCFT correlation function in \eqref{test1} is holomorphic in a neighborhood of the region $\alpha_j<Q$, $\sum\alpha_j>2Q$ \cite[Th. 6.1]{KRV_DOZZ} the same holds for the RHS of  \eqref{test1}. By Proposition \ref{pantalphas}  the LHS  is holomorphic in a neighborhood of the region $\alpha_j<a\wedge (Q-\gamma)$, $j\leq 3$, $\alpha_j<Q$ for $j>3$ and $\sum_j\alpha_j>2Q$.

On the other hand the eigenvectors $\Psi_{\alpha_j,\nu_j,\tilde\nu_j}$ are holomorphic in $\alpha_j$ in a connected region  containing $\alpha_j<a\wedge (Q-\gamma)$ and the spectrum line $Q+i\R$. Also, from Proposition \ref{defprop:desc}
and Proposition \ref{anala}, the mapping $\boldsymbol{\alpha}\mapsto \text{RHS of }\eqref{test1}$  is holomorphic in the region 
\[ \Omega:=\{\boldsymbol{\alpha}\in \C^{m} \, | \, \forall j=1,\dots,b, \alpha_j\in W_\ell, \,\,  (\alpha_j)_{j \geq b+1}\in\mc{U},\,  \sum_j{\rm Re}(\alpha_j)>2Q\}.\] 
We can then choose $m$ large enough so that this set contains a non empty connected component which   has a non empty intersection with $(-\infty,a)^3\times    \R^{m-3}$ and contains a neighborhood (in the topology of $\{{\rm Re}(\alpha)\leq Q\}^3\times \C^{m-3}$) 
of  
\[ \mc{S}_b:=\Big\{(\alpha_1,\dots,\alpha_3,0,\dots,0)\in \C^m \, |\, \forall j=1,\dots,b, \alpha_j\in Q+i\R, \, \forall j=b+1,\dots,3,\,  \alpha_j<Q, \, \sum_{j=b+1}^3\alpha_j>\chi(\mc{P})Q\Big\}.\]
($\mc{S}_b$ stands  for ``Seiberg bounds'' since Seiberg bounds are satisfied in this set) .

For $\boldsymbol{\alpha}\in\R^m $ (satisfying the Seiberg bounds), we have the relation (Prop. \ref{pantalphas} with empty Young diagrams)
\begin{align}\label{test2}
 \frac{\sqrt{2}}{(\sqrt{2}\pi)^{b-1}} &  \caA_{\psi(\caP), \hat g_\caP,({\bf z}^{\rm mp}, {\bf z}^{\rm ar}),(\boldsymbol{\alpha}^{\rm mp},\boldsymbol{\alpha}^{\rm ar}),\psi\circ \boldsymbol{\zeta}}\big(\otimes_{j=1}^b\Psi_{\alpha_j,\emptyset,\emptyset}\big) \nonumber\\
 =&
( Z(g)Z_{\D,g_\D}^{-1})^b
\langle \prod_{j=1}^mV_{\alpha_j,g_{\rm dozz}}(z_j)
\rangle_{\hat\C, g_{\rm dozz}} Z({\bf z},\boldsymbol{\alpha},\hat g)  
.
\end{align}
Therefore we can repeat the same argument for the analyticity for the RHS of  \eqref{test2}  with $W_\ell$ replaced by $W_0$. Since $W_\ell\subset W_0$, we conclude that the RHS of \eqref{test2}  can also be continued to a neighborhood of $\mc{S}_b$ on the same connected component as the RHS. Therefore we can take $\boldsymbol{\alpha} \in\mc{S}_b $ in \eqref{test2} (in particular  $\alpha_j=0$ for $j\geq 4$). From \eqref{3pointDOZZ}
\begin{align}\label{3pointDOZZagain}
\langle \prod_{j=1}^3 V_{\alpha_j,  g_{\rm dozz}}(z_j)
\rangle_{\hat\C, g_{\rm dozz}}= &\frac{1}{2}P(\hat{{\bf z}}) C_{\gamma,\mu}^{{\rm DOZZ}} (\alpha_1,\alpha_2,\alpha_3 ) \big({\frac{{\rm v}_{ g_{\rm dozz}}(\Sigma)}{{\det}'(\Delta_{ g_{\rm dozz}})}}\big)^\hf  
\end{align}
{ with  }
\begin{align*}
P(\hat{{\bf z}}):=&|z_1-z_2|^{2(\Delta_{\alpha_3}-\Delta_{\alpha_2} -\Delta_{\alpha_1})}|z_3-z_2|^{2(\Delta_{\alpha_1}-\Delta_{\alpha_2} -\Delta_{\alpha_3})}|z_1-z_3|^{2(\Delta_{\alpha_2}-\Delta_{\alpha_1} -\Delta_{\alpha_3})}.
\end{align*}
We stress here that we have used the fact that $g_{\rm dozz}(z)=1$ for $z\in\D$ to get rid of the product  $\prod_{i=1}^3g_{\rm dozz}(z_i)^{-\Delta_{\alpha_i}}$ in \eqref{3pointDOZZ} and this is the other reason why we choose to work with the metric $g_{\rm dozz}$. 
Therefore, the analytic continuation of the rhs of \eqref{test1} is given by the same expression with $\langle \prod_{j=1}^mV_{\alpha_j,g_{\rm dozz}}(z_j)
\rangle_{\hat\C, g_{\rm dozz}}$  replaced by \eqref{3pointDOZZagain}    and $Z({\bf z},\boldsymbol{\alpha},\hat g) $ replaced by $Z(\hat{{\bf z}},\hat{\boldsymbol{\alpha}},\hat g) $. Since this expression is smooth in $\hat{{\bf z}}$ we conclude that
 \begin{align*}
  \frac{\sqrt{2}}{(\sqrt{2}\pi)^{b-1}}& \mathcal{A}_{\psi(\caP),\hat g_\caP,{\bf z}^{\rm mp},\boldsymbol{\alpha}^{\rm mp},\psi\circ \boldsymbol{\zeta}}
\big(\otimes_{i=1}^b\Psi_{\alpha_i,\nu_i,\tilde\nu_i}\big)=\frac{1}{2} C_{\gamma,\mu}^{{\rm DOZZ}} (\alpha_1,\alpha_2,\alpha_3 ) ( Z(g)Z_{\D,g_\D}^{-1})^b  \big({\frac{{\rm v}_{ g_{\rm dozz}}(\Sigma)}{{\det}'(\Delta_{ g_{\rm dozz}})}}   \big)^\hf \\
 &\times
 Z(\hat{{\bf z}},\hat{\boldsymbol{\alpha}},\hat g) \big(\caD_{ \boldsymbol{\alpha},\boldsymbol{\nu}, {\bf z}} \caD_{ \boldsymbol{\alpha},\tilde{\boldsymbol{\nu}},\bar{ \bf z}}
P(\hat{{\bf z}}) \big) .
 \end{align*}
Next we define $w_{\caP}(\boldsymbol{\Delta}_{\hat{\boldsymbol{\alpha}}},
 \boldsymbol{\nu},\hat{{\bf z}})$, $\tilde w_{\caP}(\boldsymbol{\Delta}_{\hat{\boldsymbol{\alpha}}},
 \boldsymbol{\tilde\nu}, \hat{{\bf z}})$ through
\begin{align*}%\label{}
\caD_{ \boldsymbol{\alpha},\boldsymbol{\nu}, {\bf z}} \caD_{ \boldsymbol{\alpha},\tilde{\boldsymbol{\nu}},\bar{ \bf z}}
\big( 
P(\hat{\bf z})\big)  =w_{\caP}(\boldsymbol{\Delta}_{ \boldsymbol{\alpha}},
 \boldsymbol{\nu}, {\bf z})\tilde w_{\caP}(\boldsymbol{\Delta}_{ \boldsymbol{\alpha}},
 \boldsymbol{\tilde\nu}, {\bf z})P(\hat{\bf z}).
\end{align*}
Here we have just used that $\partial^{\bf r}_{\bf z}\partial^{\bf r'}_{\bar {\bf z}}(f({\bf z})f(\bar {\bf z}))=\partial^{\bf r}_{\bf z}  f({\bf z})\partial^{\bf r'}_{\bar {\bf z}}f(\bar {\bf z})$.  Obviously they are polynomials in the conformal weights $\boldsymbol{\Delta}_{\hat{\boldsymbol{\alpha}}}=(\Delta_{\alpha_1},\Delta_{\alpha_2},\Delta_{\alpha_3})$ (since $\Delta_{\alpha_j}=0$ for $j\geq 4$):
\begin{align*}%\label{}
 w_{\caP}(\boldsymbol{\Delta}_{\hat{\boldsymbol{\alpha}}},
 \boldsymbol{\nu}, {\bf z})=\sum_{{\bf n}=(n_1,..,n_3)} a_{\bf n}(\caP, {\bf z})
 \boldsymbol{\Delta}_{\hat{\boldsymbol{\alpha}}}^{\bf n}, \ \ \ \ \tilde w_{\caP}(\boldsymbol{\Delta}_{\hat{\boldsymbol{\alpha}}},
 \boldsymbol{\nu},{\bf z})=\sum_{{\bf n}=(n_1,..,n_3)} \overline{a_{\bf n}(\caP,{\bf z})}\boldsymbol{\Delta}_{\hat{\boldsymbol{\alpha}}}^{\bf n}.
\end{align*}
Also, if each weight  $\alpha_j$ belongs either to the spectrum line $Q+i\R$ or to the real line $\R$ then the conformal weight $\Delta_{\alpha_j}$ is real so that  $\tilde w_{\caP}(\boldsymbol{\Delta}_{\hat{\boldsymbol{\alpha}}},
 \boldsymbol{\nu}, {\bf z})=\overline{w_{\caP}(\boldsymbol{\Delta}_{\hat{\boldsymbol{\alpha}}},
 \boldsymbol{\nu}, {\bf z})}$. 

Finally, since $\hat g=e^{\hat \omega}g_{\rm dozz}$, and recalling $\hat g=(\psi_i)_\ast g$ in $\psi(\caD_i)$ with $g=g_\D$ near origin,  we have $\hat \omega(z_j)=-2\log |\psi'_j(0)|$  for $j=1,\dots,b$ so that
$$Z(\hat{\bf z},\hat{\boldsymbol{\alpha}},\hat g) =e^{c_{\rm L}S^0_{\rm L}(\hat\C,g_{\rm dozz},\hat \omega) }\prod_{j=1}^b  |\psi'_j(0)|^{2\Delta_{\alpha_j}}\prod_{j=b+1}^3e^{-\Delta_{\alpha_j}\hat\omega(z_j)}.$$
The final constant \eqref{metricconstant} is then obtained by combining the metric dependent constants together and playing with the fact that the Liouville functional $S^0_{\rm L}$ is invariant under biholomorphism, i.e. $S^0_{\rm L}(\Sigma,g,g')=S^0_{\rm L}(\psi(\Sigma),\psi_*g,\psi_*g')$, and is a cocycle 
$S^0_{\rm L}(\Sigma,g,g'')=S^0_{\rm L}(\Sigma,g,g')+S^0_{\rm L}(\Sigma,g',g'')$.

 This finishes the proof of Theorem \ref{pantDOZZ} for the case of incoming boundaries. \qed

   \medskip

   \noindent {\it Proof of Proposition \ref{remarkblocks}.} The proof follows essentially from   computations we have already made as we explain now. Recall that $v\circ \psi_1(z)/\psi_1'(z)=\sum_{n\in\Z}v_nz^{1-n}$ where the sum is for $  n\leq N$ (recall that $v$ is meromorphic with a unique pole at $0$). Also note that $ w^n_{\caP}(\boldsymbol{\Delta}_{\hat{\boldsymbol{\alpha}}},
  \hat{{\bf z}})=0$ for $n<0$ and $ w^n_{\caP}(\boldsymbol{\Delta}_{\hat{\boldsymbol{\alpha}}},
  \hat{{\bf z}})=1$ for $n=0$. Then
   $$ \sum_{n\geq 0} w^n_{\caP}(\boldsymbol{\Delta}_{\hat{\boldsymbol{\alpha}}},
  \hat{{\bf z}})\frac{1}{2i\pi}\oint_{\sigma} \frac{v\circ \psi_1(z)}{\psi_1'(z)}z^{n-2}\,dz=\sum_{n\geq 0} w^n_{\caP}(\boldsymbol{\Delta}_{\hat{\boldsymbol{\alpha}}},
  \hat{{\bf z}})v_n.$$
 Next we sum over $n$ the content of Proposition \ref{witht} to get (recall \eqref{TpsiSET})
 \begin{align}
 \frac{\sqrt{2}}{(\sqrt{2}\pi)^{b-1}} & \sum_{n\geq 1}  v_n \caA_{\psi(\caP), \hat g_\caP,({\bf   z}^{\rm mp},{\bf z}^{\rm ar}),(\boldsymbol{\alpha}^{\rm mp},\boldsymbol{\alpha}^{\rm ar}),\psi\circ \boldsymbol{\zeta}}\big(\Psi_{\alpha_1,n,\emptyset}\otimes_{j=2}^b\Psi_{\alpha_j,\emptyset,\emptyset}\big)\nonumber\\
 =&
  ( Z(g)Z_{\D,g_\D}^{-1})^bZ({\bf z},\boldsymbol{\alpha},\hat g)\lim_{t\to\infty}
   \frac{1}{(2\pi i)}  \oint_{|u|= {\delta}_t}    \sum_{n\in \Z}  v_n
u^{1-n}  \langle  \psi_1^\ast T_{g_{\rm dozz}}(u)  \prod_{j=1}^mV_{\alpha_j,g_{\rm dozz}}(z_j)  
\rangle_{\hat\C,g_{\rm dozz},U_t}\dd u\nonumber\\
 &-
  ( Z(g)Z_{\D,g_\D}^{-1})^bZ({\bf z},\boldsymbol{\alpha},\hat g)\lim_{t\to\infty}
   \frac{v_0}{(2\pi i)}  \oint_{|u|= {\delta}_t}    
u  \langle  \psi_1^\ast T_{g_{\rm dozz}}(u)  \prod_{j=1}^mV_{\alpha_j,g_{\rm dozz}}(z_j)  
\rangle_{\hat\C,g_{\rm dozz},U_t}\dd u.\label{expansionrem}
 \end{align}
 Indeed, in the sum of the rhs above, the $n<0$ don't contribute cause  $\Psi_{\alpha_1,n,\emptyset}=0$ for $n<0$, and we have subtracted the $n=0$ term.
 The first contour integral in the rhs equals
 $$  \frac{1}{(2\pi i)}  \oint_{|u|= {\delta}_t}    \frac{v\circ \psi_1(u)}{\psi_1'(u)}  \langle  \psi_1^\ast T_{g_{\rm dozz}}(u)  \prod_{j=1}^mV_{\alpha_j,g_{\rm dozz}}(z_j)  
\rangle_{\hat\C,g_{\rm dozz},U_t}\dd u$$
 and we   make a change of variables so that the integration contour becomes $ \psi_1(\{u\in\C \,|\, |u|=\delta_t\})$, producing
 \begin{align*}
   \frac{1}{2\pi i}&  \oint_{\psi_1(|u|= {\delta}_t)}    v(u)  \langle   T_{g_{\rm dozz}}(u)  \prod_{j=1}^mV_{\alpha_j,g_{\rm dozz}}(z_j)  
\rangle_{\hat\C,g_{\rm dozz},U_t}\dd u \\
&-\frac{c_{\rm L}}{12} \frac{1}{(2\pi i)}  \oint_{\psi_1(|u|= {\delta}_t)}   v(u)    S_{\psi_1^{-1}}(u)\,du\langle  \prod_{j=1}^mV_{\alpha_j,g_{\rm dozz}}(z_j)  
\rangle_{\hat\C,g_{\rm dozz},U_t} .
 \end{align*}
We use the Ward identities (Prop. \ref{propward}) in the first term    to see that it is equal to (writing $g_{\rm dozz}=e^{\omega_{\rm dozz}}|dz|^2$)
\begin{equation}\label{residueeasy}
 \frac{1}{(2\pi i)}   \oint_{\psi_1(|u|= {\delta}_t)}    v(u) \Big(\sum_{i=1}^{m}\Big(\frac{\Delta_{\alpha_i}}{(u-z_i)^2} +\frac{\partial_{z_i}+\Delta_{\alpha_i} \partial_{z_i}\omega_{\rm dozz} }{(u-z_i) }\Big)\Big)\langle      \prod_{j=1}^mV_{\alpha_j,g_{\rm dozz}}(z_j)  
\rangle_{\hat\C,g_{\rm dozz},U_t}\dd u +B_t
\end{equation}
where $B_t$ is the contour integral of the boundary term appearing in the Ward identity, 
$$B_t:=  \frac{i\mu}{2}   \frac{1}{2\pi i}    \oint_{\psi_1(|u|= {\delta}_t)}        \oint_{\partial{U}_t}\frac{v(u) }{  u-  w}\langle    V_{\gamma,g_{\rm dozz} }(w)\prod_{i=1}^mV_{\alpha_i,g_{\rm dozz}}(z_i) \rangle_{\hat\C,U_t,g_{\rm dozz}} g_{\rm dozz}(w)\dd u \dd\bar w$$
which tends to $0$ as $t\to\infty$, for fixed $z_1,\dots, z_m$,  for $\alpha_1,\alpha_2,\alpha_3\in \R$ and negative enough. Indeed,   a direct use of the Girsanov transform to the term $V_{\gamma,g_{\rm dozz} }(w)$ shows that $$\langle    V_{\gamma,g_{\rm dozz} }(w)\prod_{i=1}^mV_{\alpha_i,g_{\rm dozz}}(z_i) \rangle_{\hat\C,U_t,g_{\rm dozz}}\leq C|w-z_i|^{-\alpha_i\gamma}$$ when $w\to z_i$, uniformly in $t$. The reader can compare with the situation in Prop. \ref{wardite}, which was significantly harder because the boundary term still contained  SET insertions. It remains to compute the contour integral in \eqref{residueeasy}.  In the following we denote by $\sigma(z_i)$ a small contour around $z_j$, which surrounds no $z_{j'}$ for $j\not=j'$. By the residue theorem, we have the relation
$$  \oint_{\psi_1(|u|= {\delta}_t)}  \dots +\sum_{i=2}^m\oint_{\sigma(z_i)}\dots=0$$
where $\dots$ stands for the integrand in \eqref{residueeasy}. It is then an easy task to compute the residues using our assumptions on $v$ (also, recall that $\omega_{\rm dozz}=0$ in $\D$) and get, using that $v(z_2)=v(z_3)=0$,
$$\frac{1}{(2\pi i)}   \oint_{\sigma(z_i)}  \dots= \Big(\Delta_{\alpha_i}v'(z_i) +\delta_{i\geq 4}v(z_i)\partial_{z_i}\Big)  \langle      \prod_{j=1}^mV_{\alpha_j,g_{\rm dozz}}(z_j)  
\rangle_{\hat\C,g_{\rm dozz},U_t}.$$
For the same reasons, the last term in \eqref{expansionrem} can be computed to be equal to (up to the multiplicative factor   $( Z(g)Z_{\D,g_\D}^{-1})^bZ({\bf z},\boldsymbol{\alpha},\hat g)$)
 $$\lim_{t\to\infty} \frac{v_0}{(2\pi i)}   \oint_{\psi_1(|u|= {\delta}_t)}    \frac{\psi_1^{-1}(u)}{(\psi_1^{-1})'(u)} \Big(\sum_{i=1}^{m}\Big(\frac{\Delta_{\alpha_i}}{(u-z_i)^2} +\frac{\partial_{z_i}+\Delta_{\alpha_i} \partial_{z_i}\omega_{\rm dozz} }{(u-z_i) }\Big)\Big)\langle      \prod_{j=1}^mV_{\alpha_j,g_{\rm dozz}}(z_j)  
\rangle_{\hat\C,g_{\rm dozz},U_t}\dd u +B_t,$$
with $B$ the boundary term coming from the Ward identities as before. The contour involved here is a small contour around $z_1$ and, given the fact that   $ \frac{\psi_1^{-1}}{(\psi_1^{-1})'}$ is holomorphic near $z_1$ with an expansion of the form 
$$ \frac{\psi_1^{-1}(u)}{(\psi_1^{-1})'(u)}=(u-z_1)+(u-z_1)^2h(u)$$
with $h$ holomorphic near $z_1$, this contour integral can be easily evaluated to produce in the limit $t\to\infty$
$$v_0\Delta_{\alpha_1} \langle      \prod_{j=1}^mV_{\alpha_j,g_{\rm dozz}}(z_j)  
\rangle_{\hat\C,g_{\rm dozz}}.$$
Gathering the previous considerations  and taking $t\to\infty$, we have proved
\begin{align*}
 \frac{\sqrt{2}}{(\sqrt{2}\pi)^{b-1}} & \sum_{n\geq 1}  v_n \caA_{\psi(\caP), \hat g_\caP,({\bf   z}^{\rm mp},{\bf z}^{\rm ar}),(\boldsymbol{\alpha}^{\rm mp},\boldsymbol{\alpha}^{\rm ar}),\psi\circ \boldsymbol{\zeta}}\big(\Psi_{\alpha_1,n,\emptyset}\otimes_{j=2}^b\Psi_{\alpha_j,\emptyset,\emptyset}\big)\\
 =&
  ( Z(g)Z_{\D,g_\D}^{-1})^bZ({\bf z},\boldsymbol{\alpha},\hat g)\Big(-\Delta_{\alpha_1}v_0 -\sum_{i=2}^m \Big(\Delta_{\alpha_i}v'(z_i) +\delta_{i\geq 4}v(z_i)\partial_{z_i}\Big) \\
  & -\frac{c_{\rm L}}{12} {\rm Res}_{z_1} (v   S_{\psi_1^{-1}}) \Big)
\langle  \prod_{j=1}^mV_{\alpha_j,g_{\rm dozz}}(z_j)  
\rangle_{\hat\C,g_{\rm dozz}}  .
   \end{align*}
Using a similar argument as in the end of the proof of Theorem \ref{pantDOZZ}, we can then analytically continue this relation to $\hat{\boldsymbol{\alpha}}=(\alpha_1,\alpha_2,\alpha_3)$ and $\alpha_i=0$ for $i\geq 4$. Therefore, and using \eqref{test2}, we deduce  
\begin{align*}
  & \sum_{n\geq 1}  v_n \caA_{\psi(\caP), \hat g_\caP,{\bf   z}^{\rm mp},\boldsymbol{\alpha}^{\rm mp},\psi\circ \boldsymbol{\zeta}}\big(\Psi_{\alpha_1,n,\emptyset}\otimes_{j=2}^b\Psi_{\alpha_j,\emptyset,\emptyset}\big)\\
 =&
\Big(-\Delta_{\alpha_1}v_0  -\sum_{i=2}^3 \Delta_{\alpha_i}v'(z_i)   -\frac{c_{\rm L}}{12} {\rm Res}_{z_1} (v   S_{\psi_1^{-1}})   \Big)
 \caA_{\psi(\caP), \hat g_\caP,{\bf   z}^{\rm mp},\boldsymbol{\alpha}^{\rm mp},\psi\circ \boldsymbol{\zeta}}\big( \otimes_{j=1}^b\Psi_{\alpha_j,\emptyset,\emptyset}\big) ,
   \end{align*}
 from which our claim follows using Theorem \ref{pantDOZZ} in both sides of this relation (and using $w^0_{\caP}(\boldsymbol{\Delta}_{\hat{\boldsymbol{\alpha}}},
  \hat{{\bf z}})=1$).
   \qed

\medskip

\noindent {\it Proof of Corollary \ref{annulusDOZZ}.} Let us consider the conformal map $\psi(u)=u/z$, which maps the annulus $\mathbb{A}_q$ to the annulus $(1/z)\mathbb{A}_q=\{u\in\C\,|\,|q/z|\leq u\leq |1/z|\}$, with parametrisation of boundary given by $(z^{-1}\zeta_1,z^{-1}\zeta_2)$. 
By Proposition \ref{Weyl}, we have
$\mc{A}_{\mathbb{A}_q, g_{\mathbb{A}},z, \alpha_1,\boldsymbol{\zeta} }=\mc{A}_{(1/z)\mathbb{A}_q, g_{\mathbb{A}},1, \alpha_1,\psi^{-1}\circ \boldsymbol{\zeta} }$. We fix now arbitrarily $\delta \in (0,1)$ (close to $1$). Then, for $|z|<\delta$ and $|q/z|<\delta$ the annulus $(1/z)\mathbb{A}_q$ can be seen as the gluing of $\mc{P}_\delta:=\delta^{-1}\mathbb{A}_{\delta^2}$, equipped with the metric $g_{\mathbb{A}}=|dz|^2/|z|^2$, with the annuli $\delta \mathbb{A}_{\delta^{-1}q/z}$ and $(1/z) \mathbb{A}_{z\delta^{-1} }$. Corollary \ref{plumbedpant} produces   for all $|z|<\delta$ and $|q/z|<\delta$ (choosing the points to that $P(\hat{\bf z})=1$ as above) 
\begin{align*}
& \langle \mc{A}_{\mathbb{A}_q, g_{\mathbb{A}},z,  \alpha_1,\boldsymbol{\zeta} }, \Psi_{Q+ip_1,\nu_1,\tilde{\nu}_1}\otimes \bbar{ \Psi_{Q+ip_2,\nu_2,\tilde{\nu}_2}}\rangle \\
& =  \cjg\mc{A}_{\delta^{-1}\mathbb{A}_{\delta^2},g_{\mathbb{A}},1,\alpha_1} , \Psi_{Q+ip_1,\nu_1,\tilde{\nu}_1}\otimes \bbar{ \Psi_{Q+ip_2,\nu_2,\tilde{\nu}_2}}\cjd\times  \Big|\frac{z}{\delta}\Big|^{-\frac{c_{\rm L}}{12}+2\Delta_{Q+ip_2}}(\frac{z}{\delta})^{|\nu_2|}(\frac{\bar{z}}{\delta})^{|\tilde{\nu}_2|}
\Big|\frac{q}{\delta z}\Big|^{-\frac{c_{\rm L}}{12}+2\Delta_{Q+ip_1}}(\frac{q}{\delta z})^{|\nu_1|}(\frac{\bar{q}}{\delta\bar{z}})^{|\tilde{\nu}_1|}
\\
 & =   C^{{\rm DOZZ}}_{\gamma,\mu}(\alpha_1,Q-ip_1,Q+ip_2) C(\delta^{-1}\mathbb{A}_{\delta^2},g_{\mathbb{A}},\boldsymbol{\Delta}_{\boldsymbol{\hat\alpha}})  \delta^{\frac{c_{\rm L}}{6}-2\Delta_{Q+ip_1}-2\Delta_{Q+ip_2}-|\nu_1|-|\tilde \nu_1|-|\nu_2|-|\tilde\nu_2|}    \\
  &\qquad \times  |z|^{-\frac{c_{\rm L}}{12}+2\Delta_{Q+ip_2}}z^{|\nu_2|}(\bar{z})^{|\tilde{\nu}_2|}
\Big|\frac{q}{z}\Big|^{-\frac{c_{\rm L}}{12}+2\Delta_{Q+ip_1}}(\frac{q}{z})^{|\nu_1|}(\frac{\bar{q}}{\bar{z}})^{|\tilde{\nu}_1|}w_{\delta^{-1}\mathbb{A}_{\delta^2}}(\boldsymbol{\Delta}_{\hat{\boldsymbol{\alpha}}},
 \boldsymbol{\nu},\hat{{\bf z}})
 \overline{ w_{\delta^{-1}\mathbb{A}_{\delta^2}}(\boldsymbol{\Delta}_{\hat{\boldsymbol{\alpha}}},
 \tilde{\boldsymbol{\nu}},\hat{{\bf z}})} .
\end{align*} 
 Let us set 
 \[
 C_\delta:=C(\delta^{-1}\mathbb{A}_{\delta^2},g_{\mathbb{A}},\boldsymbol{\Delta}_{\boldsymbol{\hat\alpha}})  \delta^{\frac{c_{\rm L}}{6}-2\Delta_{Q+ip_1}-2\Delta_{Q+ip_2}} ,\qquad 
w_{\delta}(\boldsymbol{\Delta}_{\hat{\boldsymbol{\alpha}}}, \boldsymbol{\nu} ) :=
 w_{\delta^{-1}\mathbb{A}_{\delta^2}}(\boldsymbol{\Delta}_{\hat{\boldsymbol{\alpha}}},
 \boldsymbol{\nu},\hat{{\bf z}}) \delta^{ -|\nu_1|-|\nu_2| } 
\]
in such a way that
 \begin{align}
  & \langle \mc{A}_{\mathbb{A}_q, g_{\mathbb{A}},z,  \alpha_1,\boldsymbol{\zeta} }, \Psi_{Q+ip_1,\nu_1,\tilde{\nu}_1}\otimes  \bbar{\Psi_{Q+ip_2,\nu_2,\tilde{\nu}_2}}\rangle \label{intermediate}\\
  &= C_\delta C^{{\rm DOZZ}}_{\gamma,\mu}(\alpha_1,p_1,p_2)w_{\delta}(\boldsymbol{\Delta}_{\hat{\boldsymbol{\alpha}}},
 \boldsymbol{\nu} )\overline{w_{\delta}(\boldsymbol{\Delta}_{\hat{\boldsymbol{\alpha}}},
 \boldsymbol{\nu} )} |z|^{-\frac{c_{\rm L}}{12}+2\Delta_{Q+ip_2}}\Big|\frac{q}{z}\Big|^{-\frac{c_{\rm L}}{12}+2\Delta_{Q+ip_1}}z^{|\nu_2|}(\bar{z})^{|\tilde{\nu}_2|}
 (\frac{q}{z})^{|\nu_1|}(\frac{\bar{q}}{\bar{z}})^{|\tilde{\nu}_1|}. \nonumber
\end{align}
 We are going to compute $C_\delta$ by taking $\boldsymbol{\nu},\tilde{\boldsymbol{\nu}}=(\emptyset,\emptyset)$, in which case $w_{\delta}(\boldsymbol{\Delta}_{\hat{\boldsymbol{\alpha}}},
 \boldsymbol{\nu} )=1$. For this we will proceed similarly as before with a twist that will allow us to keep track of the uniformizing map. 
 Let $g$ still be an admissible metric on $\D$ of the form $g=e^\omega g_\D$, Euclidean over a neighborhood of the origin. Let $\mc{D}_1=q\D$ and $\mc{D}_2=\D^c$. We can glue these two disks to $\mathbb{A}_q$ in such a way that the resulting surface is now the complex plane equipped with a metric $\hat g$ obtained as the gluing of the metric $g_{\mathbb{A}}$ on $\mathbb{A}_q$ with the metrics $(\psi_j)_*g$ on  $\mc{D}_j$ for $j=1,2$ where $\psi_1:\D\to\mc{D}_1$ and $\psi_2:\D\to\mc{D}_2$ are respectively given by $\psi_1(z)=qz$ and $\psi_2(z)=1/z$. As before, the whole plane correlation functions can be obtained as the gluing of the annulus with the two disks: for $\alpha_1,\alpha_2$ real satisfying the Seiberg bound $\sum_{i=1}^3\alpha_i>2Q$ and $\alpha_j<Q$ for $j=1,2$ then (with $\tilde{\boldsymbol{\varphi}}=(\tilde{\varphi}_1,\tilde{\varphi}_2)$)
\[\langle V_{\alpha_1}(0)V_{\alpha_2}(\infty)V_{\alpha_3}(z)\rangle_{\hat \C,\hat g} =\frac{1}{\pi}\int \mc{A}_{\mathbb{A}_q, g_{\mathbb{A}},z, \alpha_3,\boldsymbol{\zeta} }(\tilde{\boldsymbol{\varphi}})\mc{A}_{\mc{D}_1, (\psi_1)_*g,0, \alpha_1,\zeta_1 }(\tilde{\varphi}_1)\mc{A}_{\mc{D}_2, (\psi_2)_*g,\infty, \alpha_2, \zeta_2 }(\tilde \varphi_2)\dd\mu_0^{\otimes 2} (\tilde{\boldsymbol{\varphi}}).\]
Observe now that taking the limit $t\to\infty$ in Lemma \ref{ampdiskSET} produces  (with $\zeta(e^{i\theta})=e^{i\theta}$),
$$\mc{A}_{\D,g_\D,0, \alpha,\zeta } =Z_{\D,g_\D}  \Psi_{\alpha,\emptyset,\emptyset}.$$
Using Proposition \ref{Weyl}, we deduce
\[\mc{A}_{\mc{D}_1, (\psi_1)_*g,0, \alpha_1,\zeta_1 } =e^{c_{\rm L}S_{\rm L}^0(\D,g_\D,\omega)} Z_{\D,g_\D}  \Psi_{\alpha_1,\emptyset,\emptyset},\quad \mc{A}_{\mc{D}_2, (\psi_2)_*g,\infty, \alpha_2,\zeta_2 } =e^{c_{\rm L}S_{\rm L}^0(\D,g_\D,\omega)} Z_{\D,g_\D}  \Psi_{\alpha_2,\emptyset,\emptyset}.
\]
Therefore
$$\big\cjg  \mc{A}_{\mathbb{A}_q, g_{\mathbb{A}},z, \alpha_1,\boldsymbol{\zeta} },\otimes_{j=1}^{2} \Psi_{\alpha_j,\emptyset,\emptyset}\big\cjd_{\mc{H}}=\pi\langle V_{\alpha_1}(0)V_{\alpha_2}(\infty)V_{\alpha_3}(z)\rangle_{\hat \C,\hat g} Z_{\D,g_\D}^{-2}e^{-2c_{\rm L}S_{\rm L}^0(\D,g_\D,\omega)}.$$
On the other hand, from \eqref{3pointDOZZ}, we have
$$\langle V_{\alpha_1}(0)V_{\alpha_2}(\infty)V_{\alpha_3}(z)\rangle_{\hat \C,\hat g} =|q|^{2\Delta_{\alpha_1}}|z|^{2\Delta_{\alpha_2}-2\Delta_{\alpha_1}}\frac{1}{2}C_{\gamma,\mu}^{{\rm DOZZ}} (\alpha_1,\alpha_2,\alpha_3 ) \sqrt{\frac{{\rm v}_{\hat g}(\hat\C)}{{\det}'(\Delta_{\hat g})}} e^{-6Q^2 S_{\rm L}^0(\hat\C,\hat g,g_{\rm dozz})}.$$
Cancelling out the $S_{\rm L}^0$ terms, we are left with 
\begin{align*}
\big\cjg & \mc{A}_{\mathbb{A}_q, g_{\mathbb{A}},z, \alpha_1,\boldsymbol{\zeta} },\otimes_{j=1}^{2} \Psi_{\alpha_j,\emptyset,\emptyset}\big\cjd_{\mc{H}}
\\
=&\pi|q|^{2\Delta_{\alpha_1}}|z|^{2\Delta_{\alpha_2}-2\Delta_{\alpha_1}}\frac{1}{2}C_{\gamma,\mu}^{{\rm DOZZ}} (\alpha_1,\alpha_2,\alpha_3 )  Z_{\D,g_\D}^{-2}\sqrt{\frac{{\rm v}_{g_0}(\hat\C)}{{\det}'(\Delta_{g_0})}} e^{S_{\rm L}^0(\hat\C,g_0,g_{\rm dozz})} e^{(1+6Q^2)S_{\rm L}^0(\mathbb{A}_q,|dz|^2,\hat g)}
\end{align*}
where we have used the conformal anomaly for the regularised determinant $ \sqrt{\frac{{\rm v}_{\hat g}(\hat\C)}{{\det}'(\Delta_{\hat g})}}= \sqrt{\frac{{\rm v}_{g_0}(\hat\C)}{{\det}'(\Delta_{g_0})}} e^{ S_{\rm L}^0(\hat\C,g_0,\hat g)}$ and the fact that $S_{\rm L}^0$ is a cocycle, in particular $S_{\rm L}^0(\hat\C,g_0,\hat g)=S_{\rm L}^0(\hat\C,g_0,g_{\rm dozz})+S_{\rm L}^0(\hat\C,g_{\rm dozz},\hat g)$. The last anomaly term on the annulus can be computed  easily
$$  S_{\rm L}^0(\mathbb{A}_q,|dz|^2,\hat g)=(-\log|q|)/12. $$
Combining, we get the value of the constant $C_\delta=\frac{\pi}{2}  Z_{\D,g_\D}^{-2}\sqrt{\frac{{\rm v}_{g_0}(\hat\C)}{{\det}'(\Delta_{\hat{\C},g_0})}} e^{S_{\rm L}^0(\hat\C,g_0,g_{\rm dozz})} $, which turns out to be independent of $\delta$. This implies that the coefficients $w_{\delta}(\boldsymbol{\Delta}_{\hat{\boldsymbol{\alpha}}}, \boldsymbol{\nu} )$ do not depend on $\delta$ either (as the l.h.s. does not).  Therefore, \eqref{intermediate} holds for $|z|<\delta$ and $|q/z|<\delta$ and for some constant $C:=C_\delta$ and coefficients $w (\boldsymbol{\Delta}_{\hat{\boldsymbol{\alpha}}}, \boldsymbol{\nu} ):=w_{\delta}(\boldsymbol{\Delta}_{\hat{\boldsymbol{\alpha}}}, \boldsymbol{\nu} )$, which do not depend on $\delta$ (and will thus be denoted $w_{\mathbb{A}}(\boldsymbol{\Delta}_{\hat{\boldsymbol{\alpha}}}, \boldsymbol{\nu} )$).
 In conclusion,  \eqref{claimannulus}  holds for some constant $C$ for all $|z|<1$ and $|q/z|<1$. Finally, we notice that the constant \eqref{constanttorus} is given by \cite{QuineChoi,Wei}:
\[ {\det}'(\Delta_{\hat{\C},g_0})=e^{\frac{1}{2}-4\zeta_{R}'(-1)}, \quad Z_{\D,g_\D}= e^{\frac{1}{4}}\times 2^{\frac{1}{12}}\pi^{\frac{1}{4}}e^{\frac{5}{24}+\zeta_R'(-1)}, \quad {\rm v}_{g_0}(\hat\C)=4\pi,\quad S_{\rm L}^0(\hat\C,g_0,g_{\rm dozz})=\frac{1-2\log 2}{6}\]
and thus the constant $C= \frac{\pi}{2^{1/2}e}$.\qed

  %%%%%%%%%%%%%%%%%%%%%%%%%%%%%%%%%%%%%%%%%%%%%%%%%%%%%%%%%%%%%%%%%%%%%%%%%%%%%%%%%%%%%%%%%%%%%%%%%%%%%%%%%%%%%%%%%%%%%%%%%%%%%%%%%%%%%%%%%%%%

%%%%%%%%%%%%%%%%%%%%%%%%%%%%%%%%%%%%%%%%%%%%%%%%%%%%%%%%%%%%%%%%%%%%%%%%%%%%%%%%%%%%%%%%%%%%%%%%%%%%%%%%%%%%%%%%%%%%%%%%%%%%%%%%%%%%%%%%%%%%%
%%%%%%%%%%%%%%%%%%%%%%%%%%%%%%%%%%%%%%%%%%%%%%%%%%%%%%%%%%%%%%%%%%%%%%%%%%%%%%%%%%%%%%%%%%%%%%%%%%%%%%%%%%%%%%%%%%%%%%%%%%%%%%%%%%%%%%%%%%%%
\appendix
\section*{Appendix}
 %%%%%%%%%%%%%%%%%%%%%%%%%%%%%%%%%%%%%%%%%%%%%%%%%%%%%%%%%%%%%%%%%%%%%%%%%%%%%%%%%%%%%%%%%%%%%%%%%%%%%%%%%%%%%%%%%%%%%%%%%%%%%%%%%%%%%%%%%%%%
%%%%%%%%%%%%%%%%%%%%%%%%%%%%%%%%%%%%%%%%%%%%%%%%%%%%%%%%%%%%%%%%%%%%%%%%%%%%%%%%%%%%%%%%%%%%%%%%%%%%%%%%%%%%%%%%%%%%%%%%%%%%%%%%%%%%%%%%%%%%
%%
\section{Adjoint Poisson kernel}\label{appendix:P*}

\begin{lemma}\label{lem:adjointP}
Let $(\Sigma,g)$ be a smooth Riemannian surface with smooth boundary and let $P:C^\infty(\pl \Sigma)\to C^\infty(\Sigma)$ be the Poisson operator defined by $\Delta_gP\varphi=0$ and $(P\varphi)|_{\pl \Sigma}=\varphi$, and denote by $P(x,y)$ its integral kernel.  Let $P^*:C_c^\infty(\Sigma^\circ)\to C^\infty(\pl \Sigma)$ 
be the adjoint operator defined by $P^*f(y):=\int_{\Sigma}P(x,y)f(x)d{\rm v}_g(x)$. Then $P^*$ extends as a continuous map $P^*:C^\infty(\Sigma)\to C^\infty(\pl \Sigma)$.
\end{lemma}
\begin{proof} By  \eqref{parametrixoperatorP}, it suffices to prove the result for the case of the disk $\Sigma=\bbar{\D}$, which in turn is equivalent to showing that for $f\in C^\infty(\bbar{\D})$
\begin{equation}\label{utheta}
u(\theta):=  \frac{1}{2i}\int_{\D} ze^{-i\theta}(1-ze^{-i\theta})^{-1}f(z)dz\wedge d\bar{z}   
\end{equation} 
is smooth in $\theta$. In polar coordinate $z=re^{i\theta'}$, write $f(r,\theta')=\sum_{k=0}^{K}(1-r)^kf_k(\theta')+s_K(r,\theta')$ with $s_K=\mc{O}((1-r)^{K+1})$ at $r=1$ and $f_k\in C^\infty(\T)$. Define $u_k$ similarly as $u$ (in  \eqref{utheta})  with $(1-r)^kf_k(\theta')$ replacing $f$, and define $S_k$ similarly by replacing $f$ by $s_k$. It is easy to check that $S_K\in C^{K}(\theta)$ and 
\[u_k(\theta)= \int_{0}^1\int_{\T}re^{i(\theta-\theta')}\frac{(1-r)^k}{1-re^{i(\theta-\theta')}}f_k(\theta')rdrd\theta'=
\sum_{n=1}^\infty \hat{f}_k(n)e^{in\theta}\int_{0}^1r^{n+1}
(1-r)^kdr.\]
The last identity follows from the fact that $f_k$ are smooth so that $\hat{f}_k(n)=\mc{O}((1+|n|)^{-\infty})$ so that the series and integral converge. Since  $|\int_{0}^1r^{n+1}
(1-r)^kdr|<1$, we deduce that $u_k\in C^\infty(\T)$. Since $u=\sum_{k=0}^K u_k+S_K$, we obtain that $u\in C^K(\T)$, and since $K$ is arbitrary, $u\in C^\infty(\T)$.
\end{proof}

\section{The DOZZ formula}\label{app:dozz}
%%%%%%%%%%%%%%%%%%%%%%%%%%%%%%%%%%%%%%%%%%%
We set $l(z)=\frac{\Gamma (z)}{\Gamma (1-z)}$ where $\Gamma$ denotes the standard Gamma function. We introduce Zamolodchikov's special holomorphic function $\Upsilon_{\frac{\gamma}{2}}(z)$ by the following expression for $0<\Re (z)< Q$
\begin{equation}\label{def:upsilon}
\log \Upsilon_{\frac{\gamma}{2}} (z)  = \int_{0}^\infty  \left ( \Big (\frac{Q}{2}-z \Big )^2  e^{-t}-  \frac{( \sinh( (\frac{Q}{2}-z )\frac{t}{2}  )   )^2}{\sinh (\frac{t \gamma}{4}) \sinh( \frac{t}{\gamma} )}    \right ) \frac{dt}{t}.
\end{equation}
The function $\Upsilon_{\frac{\gamma}{2}}$  is then defined on all $\C$ by analytic continuation of the expression \eqref{def:upsilon} as expression \eqref{def:upsilon}  satisfies the following remarkable functional relations: 
\begin{equation}\label{shiftUpsilon}
\Upsilon_{\frac{\gamma}{2}} (z+\frac{\gamma}{2}) = l( \frac{\gamma}{2}z) (\frac{\gamma}{2})^{1-\gamma z}\Upsilon_{\frac{\gamma}{2}} (z), \quad
\Upsilon_{\frac{\gamma}{2}} (z+\frac{2}{\gamma}) = l(\frac{2}{\gamma}z) (\frac{\gamma}{2})^{\frac{4}{\gamma} z-1} \Upsilon_{\frac{\gamma}{2}} (z).
\end{equation}
The function $\Upsilon_{\frac{\gamma}{2}}$ has no poles in $\C$ and the zeros of $\Upsilon_{\frac{\gamma}{2}}$ are simple (if $\gamma^2 \not \in \Q$) and given by the discrete set $(-\frac{\gamma}{2} \N-\frac{2}{\gamma} \N) \cup (Q+\frac{\gamma}{2} \N+\frac{2}{\gamma} \N )$. 
With these notations, the DOZZ formula is defined for $\alpha_1,\alpha_2,\alpha_3 \in \C$ by the following formula where we set $\bar{\alpha}=\alpha_1+\alpha_2+\alpha_3$ 
\begin{equation}\label{theDOZZformula}
C_{\gamma,\mu}^{{\rm DOZZ}} (\alpha_1,\alpha_2,\alpha_3 )  = (\pi \: \mu \:  l(\frac{\gamma^2}{4})  \: (\frac{\gamma}{2})^{2 -\gamma^2/2} )^{\frac{2 Q -\bar{\alpha}}{\gamma}}   \frac{\Upsilon_{\frac{\gamma}{2}}'(0)\Upsilon_{\frac{\gamma}{2}}(\alpha_1) \Upsilon_{\frac{\gamma}{2}}(\alpha_2) \Upsilon_{\frac{\gamma}{2}}(\alpha_3)}{\Upsilon_{\frac{\gamma}{2}}(\frac{\bar{\alpha}}{2}-Q) 
\Upsilon_{\frac{\gamma}{2}}(\frac{\bar{\alpha}}{2}-\alpha_1) \Upsilon_{\frac{\gamma}{2}}(\frac{\bar{\alpha}}{2}-\alpha_2) \Upsilon_{\frac{\gamma}{2}}(\frac{\bar{\alpha}}{2} -\alpha_3)   }
\end{equation}      
 The DOZZ formula is meromorphic with poles corresponding to the zeroes of the denominator of expression \eqref{theDOZZformula}.

\section{Markov property of the GFF}
 %%%%%%%%%%%%%%%%%%%%%%%%%%%
 
 \begin{proposition}\label{decompGFF}
Let $(\Sigma,g)$ be a Riemannian manifold with smooth boundary $\partial \Sigma$. Let $\mathcal{C}$ be a  union of  smooth non overlapping closed simple curves separating  $\Sigma$ into two connected components $\Sigma_1$ and $\Sigma_2$.\\
1) if $\partial \Sigma\not=\emptyset$ then the Dirichlet GFF $Y_g$ on $\Sigma$ admits the following decomposition in law as a sum of independent processes
$$Y_g\stackrel{law}{=}Y_1+Y_2+P$$
with $Y_q$ a Dirichlet GFF on $\Sigma_q$ for $q=1,2$ and $P$ the harmonic extension on $\Sigma\setminus\mathcal{C}$ of the restriction of $Y_g$ to $\mathcal{C}$ with boundary value $0$ on $\partial \Sigma$.\\
2) if $\partial \Sigma=\emptyset$ then the GFF $X_g$ on $\Sigma$ admits the following decomposition in law  
$$X_g\stackrel{law}{=}Y_1+Y_2+P-c_g$$
where $Y_1,Y_2,P$ are independent, $Y_q$ is a Dirichlet GFF on $\Sigma_q$ for $q=1,2$, $P$ is the harmonic extension on $\Sigma\setminus\mathcal{C}$ of the restriction of $X_g$ to $\mathcal{C}$  and $c_g:=\frac{1}{{\rm v}_g(\Sigma)}\int_\Sigma(Y_1+Y_2+P)\,\dd {\rm v}_g $.
 \end{proposition}     
 
  \begin{proof} We first prove 1). Let  $G_D:=G_{\Sigma_1}\mathbf{1}_{\Sigma_1\times \Sigma_1}+G_{\Sigma_2}\mathbf{1}_{\Sigma_2\times \Sigma_2}$ with $G_{\Sigma_1},G_{\Sigma_2}$ the Dirichlet Green functions on $\Sigma_1,\Sigma_2$. Set
  \[ u(x,y):=G_{g,D}(x,y)-G_D(x,y)\]
  with $G_{g,D}$ the Dirichlet Green function on $\Sigma$. Then  $u\in C^0(\Sigma\times \Sigma\setminus {\rm diag}(\Sigma))$. For $y\in \Sigma\setminus  \mathcal{C}$ fixed, we have for $x\in \Sigma\setminus \mathcal{C}$ 
\[\Delta_gu(\cdot,y)=\delta(y)-\delta(y)=0 , \quad u(\cdot,y)|_{\mathcal{C}}=G_{g,D}(\cdot,y)|_{\mathcal{C}}\]
 and,  if $\mc{P}_g$ is the Poisson kernel for the pair $(\Sigma,\mathcal{C})$, we deduce
\[  u(x,y)=\int_{\mc{C}} \mc{P}_g(x,x')G_{g,D}(x',y){\rm v}_\mathcal{C}(\dd x') \]
for all $x\in \Sigma,y\in \Sigma\setminus \mathcal{C}$, with ${\rm v}_\mathcal{C}$ the length measure on $\mc{C}$. But this extends by continuity to $(\Sigma\times \Sigma)\setminus \textrm{diag}(\Sigma)$. Similarly, for $x\in \Sigma\setminus \mathcal{C}$, we have for $y\in \Sigma\setminus \mathcal{C}$
\[\Delta_yu(x,y)=0,\quad \textrm{ with }  u(x,\cdot)|_{\mathcal{C}}=\int_{\mc{C}}\mc{P}_g(x,x')G_g(x',\cdot)|_{\mathcal{C}}{\rm v}_\mathcal{C}(\dd x').\] 
This implies that for $y\in \Sigma$, $x\in \Sigma\setminus \mathcal{C}$
\[  u(x,y)= \int_{\mc{C}}\int_{\mc{C}} \mc{P}_g(x,x')G_g(x',y')\mc{P}_g(y,y'){\rm v}_\mathcal{C}(\dd x'){\rm v}_\mathcal{C}(\dd y').\]
This extends by continuity to $(x,y)\in (\Sigma\times \Sigma)\setminus \textrm{diag}(\Sigma)$. Hence our claim.

  Now we prove 2). Let  $G_D:=G_{\Sigma_1}\mathbf{1}_{\Sigma_1\times \Sigma_1}+G_{\Sigma_2}\mathbf{1}_{\Sigma_2\times \Sigma_2}$ with $G_{\Sigma_1},G_{\Sigma_2}$ the Dirichlet Green functions on $\Sigma_1,\Sigma_2$.  Let us consider
  \[ u(x,y):=G_g(x,y)-G_D(x,y)+\frac{1}{{\rm v}_g(\Sigma)}\int_\Sigma G_D(x,y)\,  {\rm v}_g(\dd x) +\frac{1}{{\rm v}_g(\Sigma)}\int_\Sigma G_D(x,y)\, {\rm v}_g(\dd y).\]
Note that $u\in C^0(\Sigma\times \Sigma\setminus {\rm diag}(\Sigma))$. 
For $y\in \Sigma\setminus  \mathcal{C}$ fixed, we have for $x\in \Sigma\setminus \mathcal{C}$ 
\[\Delta_gu(\cdot,y)=\delta(y)-\frac{1}{{\rm v}_g(\Sigma)}-\delta(y)+\frac{1}{{\rm v}_g(\Sigma)}=0 , \quad u(\cdot,y)|_{\mathcal{C}}=G_g(\cdot,y)|_{\mathcal{C}}+c(y)\]
where $c(y):=\frac{1}{{\rm v}_g(\Sigma)}\int_\Sigma G_D(x,y)\,{\rm v}_g(\dd x)$.
Thus, if $\mc{P}_g$ is the Poisson kernel for the pair $(\Sigma,\mathcal{C})$ 
\[ u(x,y)=\int_\Sigma \mc{P}_g(x,x')G_g(x',y){\rm v}_\mathcal{C}(\dd x')+c(y)\]
for all $x\in \Sigma,y\in \Sigma\setminus \mathcal{C}$, with ${\rm v}_\mathcal{C}$ the length measure on $\mc{C}$. But this extends by continuity to $(\Sigma\times \Sigma)\setminus \textrm{diag}(\Sigma)$.
Note that $c(y)=0$ if $y\in \mathcal{C}$. 

Now similarly, for $x\in \Sigma\setminus \mathcal{C}$, we have for $y\in \Sigma\setminus \mathcal{C}$
\[\Delta_yu(x,y)=0,\quad \textrm{ with } u(x,\cdot)|_{\mathcal{C}}=\int_\Sigma \mc{P}_g(x,x')G_g(x',\cdot)|_{\mathcal{C}}{\rm v}_\mathcal{C}(\dd x').\] 
This implies that for $y\in \Sigma$, $x\in \Sigma\setminus \mathcal{C}$
\[u(x,y)= \int_\Sigma \mc{P}_g(x,x')G_g(x',y')\mc{P}_g(y',y){\rm v}_\mathcal{C}(\dd x'){\rm v}_\mathcal{C}(\dd y').\]
This extends by continuity to $(x,y)\in (\Sigma\times \Sigma)\setminus \textrm{diag}(\Sigma)$.
Thus,  denoting  by $\mc{S}_g$ the restriction of $G_g$  to $\mathcal{C}\times \mathcal{C}$
\[ G_g(x,y)=G_D(x,y)-\frac{1}{{\rm v}_g(\Sigma)}\int_\Sigma G_D(x,y) {\rm v}_g(\dd x) -\frac{1}{{\rm v}_g(\Sigma)}\int_\Sigma G_D(x,y)\,{\rm v}_g(\dd y)+\mc{P}_g\mc{S}_g\mc{P}_g^*(x,y).\]
Observe that, if $C:=\int G_D(x',y'){\rm v}_g(\dd x'){\rm v}_g(\dd y')$,  
\[\int_\Sigma u(x,y){\rm v}_g(\dd x)=\int_\Sigma u(x,y){\rm v}_g(\dd y)=C\frac{1}{{\rm v}_g(\Sigma)}, \quad \int_{\Sigma\times \Sigma}u(x,y){\rm v}_g(\dd x){\rm v}_g(\dd y)=C.\]
This implies that 
\[\begin{split}
\mc{P}_g\mc{S}_g\mc{P}_g^*(x,y)=& \mc{P}_g\mc{S}_g\mc{P}_g^*(x,y)-\frac{1}{{\rm v}_g(\Sigma)}\int_\Sigma\mc{P}_g\mc{S}_g\mc{P}_g^*(x,y'){\rm v}_g(\dd y')-\frac{1}{{\rm v}_g(\Sigma)}\int_\Sigma\mc{P}_g\mc{S}_g\mc{P}_g^*(x',y){\rm v}_g(\dd x')\\
& +\frac{1}{{\rm v}_g(\Sigma)^2}\int_{\Sigma\times \Sigma}\mc{P}_g\mc{S}_g\mc{P}_g^*(x',y'){\rm v}_g(\dd y'){\rm v}_g(\dd x')+\frac{1}{{\rm v}_g(\Sigma)^2}\int G_D(x',y'){\rm v}_g(\dd x'){\rm v}_g(\dd y').
\end{split}\]
In particular we have proved that 
\[\begin{split}
G_g(x,y)=& G_D(x,y)+\mc{P}_g\mc{S}_g\mc{P}_g^*(x,y) -\frac{1}{{\rm v}_g(\Sigma)}\int_\Sigma G_D(x',y){\rm v}_g(\dd x') -\frac{1}{{\rm v}_g(\Sigma)}\int_\Sigma G_D(x,y'){\rm v}_g(\dd y')\\ 
 & +\frac{1}{{\rm v}_g(\Sigma)^2} \int_\Sigma G_D(x',y'){\rm v}_g(\dd x'){\rm v}_g(\dd y')-\frac{1}{{\rm v}_g(\Sigma)}\int_\Sigma\mc{P}_g\mc{S}_g\mc{P}_g^*(x,y'){\rm v}_g(\dd y')\\\
 & -\frac{1}{{\rm v}_g(\Sigma)}\int_\Sigma\mc{P}_g\mc{S}_g\mc{P}_g^*(x',y){\rm v}_g(\dd x')+\frac{1}{{\rm v}_g(\Sigma)^2}\int_{\Sigma\times \Sigma}\mc{P}_g\mc{S}_g\mc{P}_g^*(x',y'){\rm v}_g(\dd x'){\rm v}_g(\dd y').
\end{split}\]
This proves our second claim.
   \end{proof}
 %%%%%%%%%%%%%%%%%%%%%%%%%%%%%%%%%%%%%%%%%%%%%%%%%%%%%%%%%%%%%%%%%%%%%%%%%%%%%%%%%%%%%%%%%%%%%%%%%%%%%%%%%%%%%%%%%%%%%%%%%%%%%%%%%%%%%%%%%%%%%%%%%%%%%%%%%%%%%%%%%%%%%%%%%%%%%%%%%%%%%%

\section{Proof of the Ward identity}\label{wardproof}
Here we prove Proposition \ref{propward}. We will only check the first identity, the second (antiholomorphic) one can be proved in a similar way. Furthermore, for simplicity, we will only treat the case when there are only $T$-insertions. Also, we will suppose that the metric $g$ is the round metric on the sphere; the result for other metrics can be deduced via the Weyl anomaly formula. Recall that in the round metric the conformal factor $\omega$ satisfies
\begin{equation}\label{eqconffactor}
\partial_z^2 \omega= \frac{1}{2} (\partial_z \omega)^2
\end{equation}
and this relation will be used very often in the sequel.
We will also consider slightly different correlation functions, i.e. we will work with the vertex $\tilde{V}_{\alpha,g}$ defined as the limit $\epsilon\to 0$ of
\begin{equation*}
\tilde{V}_{\alpha,g,\epsilon}(z)= \epsilon^{\alpha^2/2}  e^{\alpha  \Phi_{g,\epsilon}(z)}
\end{equation*}
(where we recall $\Phi_g=c+X_g+\frac{Q}{2}\omega$) and prove the Ward identity with this vertex definition. In this situation, the Ward identity is slightly different, i.e. without the $\partial_{\bar z_i}\omega$ terms. The Ward identity with $V_{\alpha}$ can then be deduced from the fact that  
\begin{equation*}
 \langle T_g( \mathbf{u})\bar T_g( \mathbf{v})\prod_{i=1}^m\tilde{V}_{\alpha_i,g}(z_i)\rangle_{\hat\C,U_t,g}= \left (  \prod_{i=1}^m e^{\Delta_{\alpha_i}\omega(z_i) }   \right )  \langle T_g( \mathbf{u})\bar T_g( \mathbf{v})\prod_{i=1}^mV_{\alpha_i,g}(z_i)\rangle_{\hat\C,U_t,g}. 
\end{equation*}

For $F$ some functional, we recall the following crucial identity (see \cite[Lemma 3.3]{KRV_DOZZ} or \cite[Lemma 9.2]{GKRV20_bootstrap})
 which will be used throughout the proof
\begin{equation}\label{KPZtypeidentity}
\mu \gamma \int_{\C \setminus U_t}   \langle  F(X_g) \tilde{V}_{\gamma,g} (x)  \prod_{i=1}^m\tilde{V}_{\alpha_i,g }(z_i) \rangle_{\hat\C,U_t,g}  \dd x= (\sum_{k}\alpha_k  -2Q)     \langle F(X_g)  \prod_{i=1}^m \tilde{V}_{\alpha_i,g }(z_i) \rangle_{\hat\C,U_t,g}.
\end{equation}

\noindent{\bf Gaussian integration by parts.} 
We apply the Gaussian integration by parts to the SET-insertion $T(u_k)$ (a reminder about Gaussian IBP, in the context used here, is presented in \cite[Section 9.2 and proof of Prop 9.1]{GKRV20_bootstrap}) as in section \ref{sub:ampSEThole} and we work directly in the $\epsilon\to 0$ limit. In the case of the round sphere, using identity \eqref{eqconffactor}, we get the following expression for the SET-insertion \eqref {defSET} $T_g(u)$ for $u \in \C$
\begin{equation}\label{defSETsimple}
T_g(u)= Q \partial_u^2 X_g(u) - (\partial_u X_g(u))^2- Q\partial_u  \omega(u) \partial_u X_g(u)
\end{equation}
where we also used  that by 1) of Lemma \ref{constanteSET} $a_{\hat\C, g}=0$. The contractions $\E [ \partial_u^a X_g(u) \partial_u^b X_g(u')]$ with $a>0$ entering the integration by parts formula are obtained from 
\begin{align*}%\label{}
\E  [\partial_u X_g(u)  X_g(v)]=-\frac{1}{2}(\frac{1}{u-v}+\hf \partial_u\omega(u))
\end{align*}
by differentiation. Also we define for $u,v\in\C$
    \begin{align}\label{greenderi}
    C(u,v)=-\frac{1}{2}\frac{1}{u-v},\ \quad C_{\epsilon,\epsilon'}(u,v)=\int \rho_{\epsilon}(u-u') \rho_{\epsilon'}(v-v')C(u',v')\dd u\dd v
 \end{align}
 with $( \rho_{\epsilon})_\epsilon$ a mollifying family of the type $\rho_\epsilon(\cdot)=\epsilon^{-2}\rho(|\cdot|/\epsilon)$ for $\rho\in C_c^\infty$.

Applying the Gaussian IBP formula to the two first terms in \eqref{defSETsimple} of the SET-insertion $T(u_k)$  produces plenty of terms which we group in five contributions:
\begin{equation}\label{IPPstress}
\Big\langle  T_g(\mathbf{u})   \prod_{i=1}^m\tilde{V}_{\alpha_i,g }(z_i) \Big\rangle_{\hat\C,U_t,g} =M(\mathbf{u},\mathbf{z})+T(\mathbf{u},\mathbf{z}) + N(\mathbf{u},\mathbf{z})+P(\mathbf{u},\mathbf{z}) + C(\mathbf{u},\mathbf{z})  .
\end{equation}
The first contribution in \eqref{IPPstress} collects the contractions hitting only one $\tilde{V}_{\alpha_p,g}$:
\begin{align}\label{IPPstressM}
 M(\mathbf{u},\mathbf{z})=  \sum_{p=1}^{n} \Big(\frac{Q\alpha_p}{2}- \frac{\alpha_p^2}{4}\Big) \frac{1}{(u_k-z_p)^2}  \langle    T(\mathbf{u}^{(k)})  \prod_{i=1}^m \tilde{V}_{\alpha_i,g }(z_i) \rangle_{\hat\C,U_t,g} +M'(\mathbf{u},\mathbf{z}) 
\end{align}
where $M'(\mathbf{u},\mathbf{z})$ gathers the terms which depend on the metric $g$ and is given by
\begin{align*}
M'(\mathbf{u},\mathbf{z}) & =   - Q \sum_{p=1}^n \frac{\alpha_p}{4} \partial_{u_k}^2 \omega (u_k) \langle    T_g(\mathbf{u}^{(k)})  \prod_{i=1}^n\tilde{V}_{\alpha_i,g }(z_i) \rangle_{\hat\C,U_t,g} \; \; \; (1) \\
& \quad -\sum_{p=1}^n \frac{\alpha_p^2}{4(u_k-z_p)} \partial_{u_k}  \omega (u_k) \langle    T_g(\mathbf{u}^{(k)})  \prod_{i=1}^m \tilde{V}_{\alpha_i,g }(z_i) \rangle_{\hat\C,U_t,g} \; \; \; (2) \\
& \quad - \sum_{p=1}^n \frac{\alpha_p^2}{16}  (\partial_{u_k}  \omega (u_k))^2 \langle    T_g(\mathbf{u}^{(k)})  \prod_{i=1}^m \tilde{V}_{\alpha_i,g }(z_i) \rangle_{\hat\C,U_t,g}. \; \; \; (3)
\end{align*}

The second contribution collects the terms coming from contractions of SET insertions and producing lower order SET insertions:
 \begin{align}\label{OtherstressT}
 T(\mathbf{u}, \mathbf{z})= \frac{1}{2}   \sum_{\ell=1}^{k-1}\frac{1+6Q^2}{(u_k-u_\ell)^4}  \Big\langle   T_g(\mathbf{u}^{(\ell,k)})   \prod_{i=1}^m \tilde{V}_{\alpha_i,g }(z_i) \Big\rangle_{\hat\C,U_t,g}
 \end{align}
 The third contribution is given by terms where all contractions hit $V_{\gamma,g}$:
\begin{align}\nonumber
N(\mathbf{u},\mathbf{z})=&\Big(\frac{\mu\gamma^2}{4}-\frac{\mu\gamma Q}{2}\Big) \int_{\C \setminus U_t} \frac{1}{(u_k-x)^2}\langle   T_g(\mathbf{u}^{(k)})   \tilde{V}_{\gamma,g}(x)\prod_{i=1}^m\tilde{V}_{\alpha_i,g }(z_i) \rangle_{\hat\C,U_t,g}\,\dd x+ N'(\mathbf{u},\mathbf{z})
\\=
  & - \mu   \int_{\C \setminus U_t} \frac{1}{(u_k-x)^2}\langle   T_g(\mathbf{u}^{(k)}) \tilde{V}_{\gamma,g}(x)\prod_{i=1}^m\tilde{V}_{\alpha_i ,g}(z_i) \rangle_{\hat\C,U_t,g} \,\dd x+N'(\mathbf{u},\mathbf{z}) \label{Nterm} %this term is $0$ if $\nu_k=1$
\end{align}
where 
\begin{align*}
N'(\mathbf{u},\mathbf{z}) & =  \frac{\mu \gamma Q}{4} \partial^2_{u_k}  \omega (u_k)  \int_{\C \setminus U_t} \langle   T_g(\mathbf{u}^{(k)})   \tilde{V}_{\gamma,g}(x)\prod_{i=1}^m\tilde{V}_{\alpha_i,g }(z_i) \rangle_{\hat\C,U_t,g}\,\dd x \; \; \; (4) \\
& \quad + \frac{\mu \gamma^2}{16} (\partial_{x_k}  \omega (x_k))^2  \int_{\C \setminus U_t} \langle   T(\mathbf{u}^{(k)})   \tilde{V}_\gamma(x)\prod_{i=1}^m\tilde{V}_{\alpha_i,g }(z_i) \rangle_{\hat\C,U_t,g}\,\dd x  \; \; \; (5) \\
& \quad + \frac{\mu \gamma^2}{4} \partial_{u_k}  \omega (x_k)  \int_{\C \setminus U_t} \frac{1}{u_k-x}\langle   T(\mathbf{u}^{(k)})   \tilde{V}_{\gamma,g}(x)\prod_{i=1}^m\tilde{V}_{\alpha_i,g }(z_i) \rangle_{\hat\C,U_t,g}\,\dd x\,. \; \; \; (6) \\
\end{align*}
%Let us also register here the following
The fourth term  corresponds to the third term in the decomposition \eqref{defSETsimple}
\begin{align*}
&P(\mathbf{u},\mathbf{z}) :=
-Q\partial_{u_k}  \omega (u_k)     \langle    \partial_{u_k} X_g(u_k)   T_g(\mathbf{u}^{(k)})  \prod_{i=1}^m\tilde{V}_{\alpha_i ,g}(z_i)  \rangle_{\hat\C,U_t,g}\\
&= \frac{Q}{2} \partial_{u_k} \omega (u_k)   \sum_{p=1}^n \alpha_p (\frac{1}{u_k-z_p}+\hf\partial_{u_k} \omega(u_k)) \langle   T_g(\mathbf{u}^{(k)})   \prod_{i=1}^m \tilde{V}_{\alpha_i,g }(z_i) \rangle_{\hat\C,U_t,g} \\
&  \quad  -\frac{Q\mu \gamma}{2} \partial_{u_k}  \omega (u_k)  \int_{\C \setminus U_t} (\frac{1}{u_k-x}+\hf\partial_{u_k}\omega(u_k)) \langle   T(\mathbf{u}^{(k)})   \tilde{V}_\gamma(x)\prod_{i=1}^m \tilde{V}_{\alpha_i,g }(z_i) \rangle_{\hat\C,U_t,g}\,\dd x  \\
&\quad + Q^2  \partial_{u_k} \omega (u_k) \sum_{\ell=1}^{k-1}  \frac{1}{(u_k-u_\ell)^3} \langle  T_g(\mathbf{u}^{(k,l)})  \prod_{i=1}^m \tilde{V}_{\alpha_i,g }(z_i) \rangle_{\hat\C,U_t,g}   \\
& \quad -   Q \partial_{u_k} \omega (u_k) \sum_{\ell=1}^{n-1} \frac{1}{(u_k-u_\ell)^2}   \langle \partial_{u_\ell} \Phi_g(u_\ell) T_g(\mathbf{u}^{(k,l)})  \prod_{i=1}^m \tilde{V}_{\alpha_i,g }(z_i) \rangle_{\hat\C,U_t,g}  .
\end{align*}
Using the identity  \eqref{KPZtypeidentity} this becomes
\begin{align*}
&P(\mathbf{u},\mathbf{z})=
 \frac{Q^2}{2} (\partial_{u_k} \omega (u_k))^2 \langle   T(\mathbf{u}^{(k)})   \prod_{i=1}^m \tilde{V}_{\alpha_i,g }(z_i) \rangle_{\hat\C,U_t,g}  \; \; \; (7a)\\
& \quad + Q^2  \partial_{u_k} \omega (u_k) \sum_{\ell=1}^{k-1}  \frac{1}{(u_k-u_\ell)^3} \langle  T_g(\mathbf{u}^{(k,l)})  \prod_{i=1}^m \tilde{V}_{\alpha_i,g }(z_i) \rangle_{\hat\C,U_t,g}  \; \; \; (7b)  \\
& \quad -   Q \partial_{u_k} \omega (u_k) \sum_{\ell=1}^{n-1} \frac{1}{(u_k-u_\ell)^2}   \langle \partial_{u_\ell} \Phi_g(u_\ell) T_g(\mathbf{u}^{(k,l)})  \prod_{i=1}^m \tilde{V}_{\alpha_i,g }(z_i) \rangle_{\hat\C,U_t,g}  (7c) \\
&  \quad  -\frac{Q\mu \gamma}{2} \partial_{u_k}  \omega (u_k)  \int_{\C \setminus U_t} \frac{1}{u_k-x} \langle   T(\mathbf{u}^{(k)})   \tilde{V}_\gamma(x)\prod_{i=1}^m \tilde{V}_{\alpha_i,g }(z_i) \rangle_{\hat\C,U_t,g}\,\dd x   \; \; \; (7d)\\
&\quad + \frac{Q}{2} \partial_{u_k} \omega (u_k)   \sum_{p=1}^n  \frac{\alpha_p}{u_k-z_p} \langle   T_g(\mathbf{u}^{(k)})   \prod_{i=1}^m \tilde{V}_{\alpha_i,g }(z_i) \rangle_{\hat\C,U_t,g}  \; \; \; (7e)
\end{align*}
Finally $C(\mathbf{u},\mathbf{z}) $ gathers all the other terms and is given by
\begin{align*}%\label{}
C(\mathbf{u},\mathbf{z}) =\sum_{i=1}^{15}C_i(\mathbf{u},\mathbf{z}) 
\end{align*}
with the following definitions:
%Contraction 1 is
\begin{align*}
 %&-Q^2  \E[    \partial_{u_k} X_g(u_k)  \partial^2_{u_{\ell}} X_g(u_{\ell})    ] \E[  \partial_{u_k} X_g(u_k)   \partial_{u_{\ell'}}^2 X_g(u_{\ell'})    ] = -\frac{ Q^2}{(u_k-u_\ell)^3(u_k-u_{\ell'})^3} \\
& C_1(\mathbf{u} ,{\bf z}):=  - \sum_{\ell\not =\ell'=1}^{k-1}\frac{ Q^2}{(u_k-u_\ell)^3(u_k-u_{\ell'})^3}  \langle      T_g(\mathbf{u}^{(k,\ell,\ell')})   \prod_{i=1}^m \tilde{V}_{\alpha_i,g }(z_i) \rangle_{\hat\C,U_t,g}\\
%\end{align*}
%Contraction 2 is
%\begin{align*}
%&  -2 Q^2  \E[    \partial_{u_k}^2 X_g(u_k)  \partial_{u_{\ell}} X_g(u_{\ell})    ] \E[  \partial_{u_{\ell}} X_g(u_{\ell})   \partial_{u_{\ell'}}^2 X_g(u_{\ell'})    ] = \frac{ 2 Q^2}{(u_k-u_\ell)^3(u_{\ell}-u_{\ell'})^3},   \\
& C_2(\mathbf{u} ,{\bf z}):=  \sum_{\ell\not =\ell'=1}^{k-1}\frac{2Q^2}{(u_k-u_\ell)^3(u_\ell-u_{\ell'})^3}  \langle     T_g(\mathbf{u}^{(\ell,k,\ell')})   \prod_{i=1}^m\tilde{V}_{\alpha_i,g }(z_i) \rangle_{\hat\C,U_t,g}\\
 &  \qquad\quad\quad\quad  - \sum_{\ell=1}^{k-1}\frac{Q^2}{(u_k-u_\ell)^3}  \partial_{u_{\ell}} \omega (u_{\ell}) \langle     T_g(\mathbf{u}^{(\ell,k})   \prod_{i=1}^m\tilde{V}_{\alpha_i,g }(z_i) \rangle_{\hat\C,U_t,g}\\
%\end{align*}
%Contraction 3 is
%\begin{align*}
%&  2 Q  \E[    \partial_{u_k} X_g(u_k)  \partial_{u_{\ell}}^2 X_g(u_{\ell})    ] \E[  \partial_{u_{k}} X_g(u_{k})   \partial_{u_{\ell'}} X_g(u_{\ell'})    ] = \frac{   Q}{(u_k-u_\ell)^3(u_{k}-u_{\ell'})^2}, \\
& C_3(\mathbf{u} ,{\bf z}):=    \sum_{\ell\not =\ell'=1}^{k-1}\frac{  Q}{(u_k-u_\ell)^3(u_k-u_{\ell'})^2}  \langle    \partial_{u_{\ell'}}\Phi_g(u_{\ell'})  T(\mathbf{u}^{(k,\ell,\ell')})   \prod_{i=1}^m\tilde{V}_{\alpha_i ,g}(z_i) \rangle_{\hat\C,U_t,g}\\
%\end{align*}
%Contraction 4 is
%\begin{align*}
 %&  4 Q  \E[    \partial_{u_k}^2 X_g(u_k)  \partial_{u_{\ell}} X_g(u_{\ell})    ] \E[  \partial_{u_{\ell}} X_g(u_{\ell})   \partial_{u_{\ell'}} X_g(u_{\ell'})    ] = -\frac{  2 Q}{(u_k-u_\ell)^3(u_{\ell}-u_{\ell'})^2} ,\\
& C_4(\mathbf{u} ,{\bf z}):= - \sum_{\ell\not =\ell'=1}^{k-1}\frac{2Q}{(u_k-u_\ell)^3(u_\ell-u_{\ell'})^2}  \langle    \partial_{u_{\ell'}} \Phi_g (u_{\ell'})   T(\mathbf{u}^{(\ell,k,\ell')})   \prod_{i=1}^m\tilde{V}_{\alpha_i,g }(z_i) \rangle_{\hat\C,U_t,g}\\
%\end{align*}
%Contraction 5 is 
%\begin{align*}
%&  -2Q \alpha_p \E[    \partial_{u_k}^2 X_g(u_k)  \partial_{u_\ell} X_g(u_\ell)    ] \E[   \partial_{u_\ell} X_g(u_\ell)  X_g(z_p)   ] = \frac{ Q\alpha_p}{(x_\ell-z_p) (u_k-u_{\ell})^3}+  \frac{Q \alpha_p}{2 (u_k-u_{\ell})^3} \partial_{x_\ell} \omega(x_\ell),\\
& C_{5}(\mathbf{u} ,{\bf z}):=  \sum_{\ell=1}^{k-1} \sum_{p=1}^{m}\frac{ Q\alpha_p}{(u_\ell-z_p) (u_k-u_{\ell})^3}  \langle     T_g(\mathbf{u}^{(k,\ell )})   \prod_{i=1}^m\tilde{V}_{\alpha_i,g }(z_i) \rangle_{\hat\C,U_t,g} 
\\ 
& \qquad \qquad\quad +  \sum_{\ell=1}^{k-1} \sum_{p=1}^{m}  \frac{Q \alpha_p}{2 (u_k-u_{\ell})^3} \partial_{u_\ell} \omega (u_\ell)     \langle     T_g(\mathbf{u}^{(k,\ell )})   \prod_{i=1}^m\tilde{V}_{\alpha_i,g }(z_i) \rangle_{\hat\C,U_t,g}. \; \; \; (8)\\
%\end{align*}
%Contraction 6 is 
%\begin{align*}
% &  2Q \mu \gamma  \E[    \partial_{u_k}^2 X_g(u_k)  \partial_{u_\ell} X_g(u_\ell)    ] \E[   \partial_{u_\ell} X_g(u_\ell )  X_g(x)   ] = -\frac{ Q\mu \gamma}{(u_\ell-x) (u_k-u_{\ell})^3}-  \frac{Q \mu \gamma}{2 (u_k-u_{\ell})^3} \partial_{u_\ell} \omega (u_\ell),\\
&C_{6}(\mathbf{u} ,{\bf z}):=-  \mu Q\gamma  \sum_{\ell=1}^{k-1}\int_{\C \setminus U_t}\frac{ 1}{(u_\ell-x)(u_k-u_\ell)^3 }  \langle   T(\mathbf{u}^{(k,\ell )})   \tilde{V}_{\gamma,g}(x)\prod_{i=1}^m \tilde{V}_{\alpha_i,g }(z_i) \rangle_{\hat\C,U_t,g} \,\dd x \\
& \qquad \qquad\quad  -  \sum_{\ell=1}^{k-1} \frac{Q \mu \gamma}{2 (u_k-u_{\ell})^3} \partial_{u_\ell} \omega (u_\ell) \int_{\C \setminus U_t}  \langle   T_g(\mathbf{u}^{(k,\ell )})   \tilde{V}_{\gamma,g}(x)\prod_{i=1}^m \tilde{V}_{\alpha_i,g }(z_i) \rangle_{\hat\C,U_t,g} \,\dd x \; \; \; (9)\\
%\end{align*}
%Contraction 7 is 
%\begin{align*}
%&  2Q   \E[    \partial_{u_k} X_g(u_k)  \partial_{u_{\ell'}}^2 X_g(u_{\ell'})    ] \E[  \partial_{u_k} X_g(u_k)  \partial_{u_\ell} X_g(u_\ell)  ] = \frac{ Q }{(u_k-u_\ell)^2 (u_k-u_{\ell'})^3}\\
\end{align*}
\begin{align*}
& C_7(\mathbf{u} ,{\bf z}):=  \sum_{\ell\not=\ell'=1}^{k-1}\frac{ Q}{(u_k-u_\ell)^2(u_k-u_{\ell'})^3}  \langle \partial_{u_\ell }\Phi_g(u_{\ell})  
T_g(\mathbf{u}^{(k,\ell,\ell')})   \prod_{i=1}^m \tilde{V}_{\alpha_i,g }(z_i) \rangle_{\hat\C,U_t,g} \\
%\end{align*}
%Contraction 8 is 
%\begin{align*}
 %&  -4 \mu \gamma  \E[    \partial_{u_k} X_g(u_k)  \partial_{u_\ell} X_g(u_\ell)    ] \E[   \partial_{u_k} X_g(u_k)  X_g(x)   ] = - \frac{\mu \gamma}{(u_k-u_\ell)^2(u_k-x)} - \frac{\mu \gamma}{2(u_k-u_\ell)^2} \partial_{u_k} \omega (u_k) \\
& C_{8}(\mathbf{u} ,{\bf z}):=-   \mu  \gamma  \sum_{\ell=1}^{k-1}\int_{\C \setminus U_t} \frac{ 1}{(u_k-x)(u_k-u_\ell)^2 }  \langle  \partial_{u_\ell}\Phi_g(u_{\ell})   T_g(\mathbf{u}^{(k,\ell )})   \tilde{V}_{\gamma,g}(x)\prod_{i=1}^m \tilde{V}_{\alpha_i,g }(z_i) \rangle_{\hat\C,U_t,g} \,\dd x \\
 &\qquad\qquad  -   \mu  \gamma \partial_{u_k} \omega (u_k)  \sum_{\ell=1}^{k-1}  \frac{ 1}{2(u_k-u_\ell)^2 }     \int_{\C \setminus U_t}   \langle  \partial_{u_\ell}\Phi_g(u_{\ell})   T_g(\mathbf{u}^{(k,\ell )})   \tilde{V}_{\gamma,g}(x)\prod_{i=1}^m \tilde{V}_{\alpha_i }(z_i) \rangle_{\hat\C,U_t,g} \,\dd x
\; \; \; (10)   \\
%\end{align*}
%Contraction 9 is 
%\begin{align*}
%&  4  \alpha_p  \E[    \partial_{u_k} X_g(u_k)  \partial_{u_\ell} X_g(u_\ell)    ] \E[   \partial_{u_k} X_g(u_k)  X_g(z_p)   ] =  \frac{\alpha_p}{(u_k-u_\ell)^2(x_k-z_p)} + \frac{\alpha_p}{2(u_k-u_\ell)^2} \partial_{u_k} \omega (u_k), \\
& C_9(\mathbf{u} ,{\bf z}):=     \sum_{\ell=1}^{k-1} \sum_{p=1}^{m}\frac{ \alpha_p}{(u_k-u_\ell)^2(u_k-z_p)}  \langle \partial_{u_\ell}\Phi_g (u_{\ell})       T_g(\mathbf{u}^{(k,\ell)})   \prod_{i=1}^m \tilde{V}_{\alpha_i,g }(z_i) \rangle_{\hat\C,U_t,g} \\
&  \qquad\qquad \quad  +     \partial_{u_k} \omega (u_k)  \sum_{\ell=1}^{k-1} \sum_{p=1}^{m} \frac{ \alpha_p}{2(u_k-u_\ell)^2 }       \langle  \partial_{u_\ell} \Phi_g (u_{\ell})   T_g(\mathbf{u}^{(k,\ell )})   \prod_{i=1}^m \tilde{V}_{\alpha_i,g }(z_i) \rangle_{\hat\C,U_t,g}. \; \; \; (11)\\
%\end{align*}
%Contraction 10 is 
%\begin{align*}
%&  - \alpha_p \alpha_{p'} \E[    \partial_{u_k} X_g (u_k)   X_g(z_p)    ] \E[   \partial_{u_k} X_g(u_k)  X_g(z_{p'})  ] \\ 
%&\qquad  = - \frac{\alpha_p \alpha_{p'}}{4}  (\frac{1}{(u_k-z_p)} + \frac{1}{2} \partial_{u_k}\omega (x_k)  )( \frac{1}{(u_k-z_{p'})}+ \frac{1}{2} \partial_{u_k} \omega(u_k)),\\
& C_{10}(\mathbf{u} ,{\bf z})\\
& \qquad := - \frac{1}{4} \sum_{p\not =p'=1}^{m}  \alpha_p\alpha_{p'}    \Big(\frac{1}{(u_k-z_p)} + \frac{1}{2} \partial_{u_k}\omega (u_k)  \Big)\Big( \frac{1}{(u_k-z_{p'})}+ \frac{1}{2} \partial_{u_k} \omega (u_k)\Big)  \langle    T_g(\mathbf{u}^{(k )})   \prod_{i=1}^m \tilde{V}_{\alpha_i,g }(z_i) \rangle_{\hat\C,U_t,g}  . \;  (12)  \\
%\end{align*}
%Contraction 11 is
%\begin{align*}
%&  -\mu^2 \gamma^2 \E[    \partial_{u_k} X_g(u_k)   X_g(x)    ] \E[   \partial_{u_k} X_g(u_k)  X_g(x')]=  -\frac{\mu^2 \gamma^2}{4}  \Big(\frac{1}{(u_k-x)} + \frac{1}{2} \partial_{u_k}\omega (u_k)\Big)\Big( \frac{1}{(u_k-x')}+ \frac{1}{2} \partial_{u_k}\omega (u_k)\Big), \\
& C_{11}(\mathbf{u} ,{\bf z}) := -  \frac{\mu^2  \gamma^2}{4}   \int_{\C \setminus U_t}\int_{\C \setminus U_t} \Big(\frac{1}{(u_k-x)} + \frac{1}{2} \partial_{u_k}\omega (u_k)  \Big)\Big( \frac{1}{(u_k-x')}+ \frac{1}{2} \partial_{u_k}\omega (u_k)  \Big) \\
& \qquad \qquad \qquad \qquad\qquad \qquad\qquad \qquad\qquad \qquad\quad \times \langle    T_g(\mathbf{u}^{(k  )})   \tilde{V}_{\gamma,g}(x) \tilde{V}_{\gamma,g}(x')\prod_{i=1}^m \tilde{V}_{\alpha_i,g }(z_i) \rangle_{\hat\C,U_t,g} \,\dd x\,\dd x' . \; \; \; (13) \\
%\end{align*}
%Contraction 12 is 
%\begin{align*}
 %& -4  \E[   \partial_{u_k} X_g(u_k)  \partial_{u_\ell} X_g(u_\ell)  ]   \E[   \partial_{u_k} X_g(u_k)  \partial_{u_{\ell'}} X_g(u_{\ell'})  ]= -\frac{1}{(u_k-u_\ell)^2(u_k-u_{\ell'})^2} , \\
& C_{12}(\mathbf{u} ,{\bf z}):=  -  \sum_{\ell\not=\ell'=1}^{k-1}\frac{ 1}{(u_k-u_\ell)^2(u_k-u_{\ell'})^2}  \langle \partial_{u_\ell} \Phi_g(u_{\ell})    \partial_{u_{\ell'}}\Phi_g(u_{\ell'})     T_g(\mathbf{u}^{(k,\ell,\ell')})   \prod_{i=1}^m \tilde{V}_{\alpha_i,g }(z_i) \rangle_{\hat\C,U_t,g}.
\end{align*}
%Contraction 13 is
\begin{align*}
%&  2 \mu \gamma Q \E[    \partial_{u_k} X_g(u_k)  \partial^2_{u_\ell} X_g(u_\ell)    ] \E[   \partial_{u_k} X_g(u_k)  X_g(x)   ]  =  \frac{\mu \gamma Q}{( u_k-u_\ell )^3}   (  \frac{1}{(  u_k- x )}  + \frac{1}{2}  \partial_{u_k} \omega(u_k) ),  \\
& C_{13}(\mathbf{u} ,{\bf z}):=  \mu Q\gamma  \sum_{\ell=1}^{k-1}\int_{\C \setminus U_t}\frac{ 1}{(u_k-x)(u_k-u_\ell)^3 }  \langle   T(\mathbf{u}^{(k,\ell )})   V_{\gamma,g}(x)\prod_{i=1}^m V_{\alpha_i ,g}(z_i) \rangle_{\hat\C,U_t,g} \,\dd x
\\
 &\qquad\qquad\quad  + \frac{\mu Q\gamma}{2} \partial_{u_k} \omega (u_k)  \sum_{\ell=1}^{k-1} \frac{ 1}{(u_k-u_\ell)^3 } \int_{\C \setminus U_t }  \langle   T(\mathbf{u}^{(k,\ell )})   \tilde{V}_{\gamma,g}(x)\prod_{i=1}^m \tilde{V}_{\alpha_i ,g}(z_i) \rangle_{\hat\C,U_t,g} \,\dd x.
  \; \; \; (14)  \\
%\end{align*}
%Contraction 14 is 
%\begin{align*}
%&  2 \mu \gamma \alpha_p  \E[    \partial_{u_k} X_g (u_k) X_g(z_p)    ] \E[   \partial_{u_k} X_g(u_k)  X_g(x)   ]  =  \frac{\mu \gamma \alpha_p}{2}   (  \frac{1}{(  u_k- z_p )}  +\frac{1}{2}  \partial_{u_k} \omega (u_k)   )   (  \frac{1}{(  u_k- x )}  + \frac{1}{2}  \partial_{u_k} \omega (u_k)),\\
& C_{14}(\mathbf{u} ,{\bf z}):=   \frac{\mu  \gamma}{2} \sum_{p=1}^{m}\alpha_p \int_{\C \setminus U_t } \Big [ \Big(  \frac{1}{(  u_k- z_p )}  +\frac{1}{2}  \partial_{u_k} \omega (u_k)   \Big)   \Big(  \frac{1}{(  u_k- x )}  + \frac{1}{2}  \partial_{u_k} \omega (u_k)\Big)  \\
& \qquad\qquad\qquad \qquad\qquad\qquad \qquad\qquad\qquad\qquad\qquad\qquad \qquad\times  \langle    T(\mathbf{u}^{(k  )})   \tilde{V}_{\gamma,g}(x)\prod_{i=1}^m \tilde{V}_{\alpha_i,g }(z_i) \rangle_{\hat\C,U_t,g}\Big] \,\dd x .  \; \; \; (15)\\
%\end{align*}
%Contraction 15 is
%\begin{align}
%&  -2Q\alpha_p  \E[    \partial_{u_k} X_g(u_k) \partial^2_{u_\ell} X_g(u_\ell)    ] \E[   \partial_{u_k} X_g(u_k)  X_g(z_p)   ]  =  -\frac{Q \alpha_p}{(u_k-u_\ell)^3} (  \frac{1}{(  u_k- z_p )}  +\frac{1}{2}  \partial_{u_k} \omega (u_k)) ,\nonumber \\
& C_{15}(\mathbf{u} ,{\bf z}):= -\sum_{\ell=1}^{k-1} \sum_{p=1}^{m}\frac{ Q\alpha_p}{(u_k-z_p) (u_k-u_{\ell})^3}  \langle     T(\mathbf{u}^{(k,\ell )})   \prod_{i=1}^m \tilde{V}_{\alpha_i ,g}(z_i) \rangle_{\hat\C,U_t,g}\nonumber \\
&\qquad \qquad\quad  -\frac{1}{2} \partial_{u_k} \omega (u_k) \sum_{\ell=1}^{k-1} \sum_{p=1}^{m}\frac{ Q\alpha_p}{ (u_k-u_{\ell})^3}  \langle     T(\mathbf{u}^{(k,\ell )})   \prod_{i=1}^m \tilde{V}_{\alpha_i ,g}(z_i) \rangle_{\hat\C,U_t,g}  \; \; \; (16) .%\label{Cterms}
\end{align*}

\paragraph{{\bf The metric dependent terms}}  
The first step is to show that the metric dependent terms cancel. This follows by a repeated use fo the identity  \eqref{KPZtypeidentity}. Applying it with $F=T_g(\mathbf{u}^{(k,\ell )})$ to the last term in $C_2$ and (8), (9) we see they cancel:
%There are two types of metric dependent terms: sums of terms $\partial_{u_\ell} \omega (u_\ell)$ with $\ell \leq k-1$ and terms involving derivatives of $\omega (u_k)$. By identity \eqref{KPZtypeidentity} we have for all $\ell, k$ 
%\[
%\mu \gamma \int_{\C \setminus U_t}  \langle   T_g(\mathbf{u}^{(k,\ell )})   \tilde{V}_{\gamma,g}(x)\prod_{i=1}^m \tilde{V}_{\alpha_i,g }(z_i) \rangle_{\hat\C,U_t,g} \,\dd x = \Big(\sum_{p=1}^m \alpha_p -2 Q   \Big) \langle   T_g(\mathbf{u}^{(k,\ell )}) \prod_{i=1}^m \tilde{V}_{\alpha_i,g }(z_i) \rangle_{\hat\C,U_t,g}.
%\]
%Using the above identity, we conclude that 
\[
-    \sum_{\ell=1}^{k-1}\frac{Q^2}{(u_k-u_\ell)^3}  \partial_{u_{\ell}} \omega (u_{\ell}) \langle     T_g(\mathbf{u}^{(\ell,k})   \prod_{i=1}^m \tilde{V}_{\alpha_i,g }(z_i) \rangle_{\hat\C,U_t,g}+(8)+(9)=0.
\]

\vspace{0.2 cm}

\noindent
Next, consider terms proportional to $\partial_{u_k}^2\omega (u_k)$ or $(\partial_{u_k}\omega (u_k))^2$ .%We first gather all the terms of the form 
%\[
% \partial^2_{u_k} \omega (u_k)  \langle   T_g(\mathbf{u}^{(k)})   \prod_{i=1}^m \tilde{V}_{\alpha_i ,g}(z_i) \rangle_{\hat\C,U_t,g}, \quad  (\partial_{u_k} \omega (u_k))^2  \langle   T_g(\mathbf{u}^{(k)})   \prod_{i=1}^m \tilde{V}_{\alpha_i,g }(z_i) \rangle_{\hat\C,U_t,g}.
%\]
For the former we obtain
\[
(1)+(4)=  -\frac{Q^2}{2}  \partial^2_{u_k} \omega (u_k)  \langle   T(\mathbf{u}^{(k)})   \prod_{i=1}^m \tilde{V}_{\alpha_i,g }(z_i) \rangle_{\hat\C,U_t,g}.
\] 
For the latter summing these terms in %$(\partial_{u_k} \omega (u_k))^2$ of 
in $(3), (5), (12), (13), (15)$ we get
\[
- \frac{Q^2}{4} (\partial_{u_k} \omega g(u_k))^2  \langle   T(\mathbf{u}^{(k)})   \prod_{i=1}^m \tilde{V}_{\alpha_i }(z_i) \rangle_{\hat\C,U_t,g}.
\]
Combining these with %the $(\partial_{u_k} \omega (u_k))^2$ term in 
(7a) all cancel
 %\[
 %\frac{Q^2}{2} (\partial_{u_k} \omega (u_k))^2 \langle   T(\mathbf{u}^{(k)})   \prod_{i=1}^m \tilde{V}_{\alpha_i,g }(z_i) \rangle_{\hat\C,U_t,g}
 %\]
 by using the relation $\frac{1}{2} (\partial_{u} \omega (u))^2 - \partial^2_{u} \omega(u)=0 $.

 Next we turn to the terms proportional to $\frac{\partial_{u_k} \omega (u_k)}{u_k-z_p} $.  They occur  in 
 $(2)$, $(7e)$, $(12)$, $(15)$ and again, using the   identity \eqref{KPZtypeidentity},  their sum is seen to vanish.  
%  \begin{align*}
 % & (2)+(7)+(12)+(15)   \quad \quad  ( \text{terms in }  \frac{\partial_{u_k} \omega (u_k)}{u_k-z_p}  )  \\
%  &   -\frac{\partial_{u_k} \omega (u_k)}{4} \sum_{p=1}^m  \frac{\alpha_p^2}{(u_k-z_p)}   \langle   T_g(\mathbf{u}^{(k)})   \prod_{i=1}^m \tilde{V}_{\alpha_i ,g}(z_i) \rangle_{\hat\C,U_t,g}  
%    + \frac{Q}{2} \partial_{u_k} \omega (u_k)   \sum_{p=1}^n  \frac{\alpha_p}{(u_k-z_p)} \langle   T_g(\mathbf{u}^{(k)})   \prod_{i=1}^m \tilde{V}_{\alpha_i,g }(z_i) \rangle_{\hat\C,U_t,g} \\
% &    \quad - \frac{1}{4} \partial_{u_k} \omega (u_k)  \sum_{p \not = p'}  \frac{\alpha_p \alpha_{p'}}{u_k-z_p}  \langle   T_g(\mathbf{u}^{(k)})   \prod_{i=1}^m \tilde{V}_{\alpha_i }(z_i) \rangle_{\hat\C,U_t,g} 
 %+ \frac{\mu \gamma}{4}  (\sum_{p=1}^n \frac{\alpha_p}{u_k-z_p}) \int_{\C \setminus U_t}   \langle   T_g(\mathbf{u}^{(k)}) \tilde{V}_{\gamma,g} (x)  \prod_{i=1}^m \tilde{V}_{\alpha_i,g }(z_i) \rangle_{\hat\C,U_t,g} \,\dd x \\
 % & = 0. 
 % \end{align*}
  % Next, we turn to terms of the form 
 % $\partial_{u_k}  \omega (u_k) \int_{\C \setminus U_t}  \frac{1}{u_k-x}  \langle   T_g(\mathbf{u}^{(k)}) \tilde{V}_{\gamma,g} (x)  \prod_{i=1}^m \tilde{V}_{\alpha_i,g }(z_i) \rangle_{\hat\C,U_t,g} \,\dd x$.  
 
Next we gather terms from  $(6)$, $(7d)$, $(13)$, $(15)$ involving $V_{\gamma,g}$ insertions obtaining
  \begin{align*}
%  & (6)+(7)+(13)+(15)      \quad \quad  ( \text{terms in }  \int_{\C \setminus U_t}  \frac{1}{u_k-x}  \langle   T_g(\mathbf{u}^{(k)}) \tilde{V}_{\gamma,g} (x)  \prod_{i=1}^m \tilde{V}_{\alpha_i,g }(z_i) \rangle_{\hat\C,U_t,g} \,\dd x  )  \\
  %& =  
 & (\frac{\mu \gamma^2}{4}-\frac{Q\mu \gamma}{2} + \frac{\mu \gamma}{4} \sum_{p=1}^m \alpha_p )
  \partial_{u_k}  \omega (u_k)  \int_{\C \setminus U_t} \frac{1}{u_k-x}\langle   T(\mathbf{u}^{(k)})   \tilde{V}_{\gamma,g}(x)\prod_{i=1}^m \tilde{V}_{\alpha_i,g }(z_i) \rangle_{\hat\C,U_t,g}\,\dd x  \\
 %&  \quad  
%-\frac{Q\mu \gamma}{2} \partial_{u_k} \omega (u_k)  \int_{\C \setminus U_t} \frac{1}{u_k-x} \langle   T(\mathbf{u}^{(k)})   \tilde{V}_\gamma(x)\prod_{i=1}^m \tilde{V}_{\alpha_i,g }(z_i) \rangle_{\hat\C,U_t,g}\,\dd x  \\
 & \quad -\frac{\mu^2 \gamma^2}{4} \partial_{u_k}  \omega (u_k)   \int_{\C \setminus U_t}  \int_{ \C \setminus U_t} \frac{1}{u_k-x} \langle   T(\mathbf{u}^{(k)})   \tilde{V}_{\gamma,g}(x) \tilde{V}_{\gamma,g}(x') \prod_{i=1}^m \tilde{V}_{\alpha_i,g }(z_i) \rangle_{\hat\C,U_t,g}\,\dd x \,\dd x' %  \\
 % & \quad + \frac{\mu \gamma}{4} \left ( \sum_{p=1}^m \alpha_p   \right ) \partial_{u_k}  \omega (u_k)    \int_{\C \setminus U_t} \frac{1}{u_k-x} \langle   T(\mathbf{u}^{(k)})   \tilde{V}_{\gamma,g}(x)\prod_{i=1}^m \tilde{V}_{\alpha_i,g }(z_i) \rangle_{\hat\C,U_t,g}\,\dd x  
  \end{align*}
  and these sum to zero using \eqref{KPZtypeidentity} on the $\dd x' $ integral (note that we use \eqref{KPZtypeidentity} with $\gamma+\sum_{p=1}^m \alpha_p$ instead of $\sum_{p=1}^m \alpha_p$).
  
The remaining terms proportional to   $ \partial_{u_k} \omega (u_k) $ are now  $(7b)+(7c)+(10)+(11)+(14)+(16)$ which again sum to zero.
  
\vskip 2mm

\noindent{\bf The $N$-term.}
 %%%%%%%%%%%%%%%% 
We will next rewrite the $N$-contribution to make it cancel with some $C$-terms. For this, we regularise the vertex insertions  (beside the $\tilde{V}_{\gamma,g}$ insertion, we also regularise the $\tilde{V}_{\alpha_i,g}$'s for later need)  
in $N(\mathbf{u},\mathbf{z})$ given by \eqref{Nterm}, and perform the integration by parts (Green formula) in the $x$ integral  to get  
\begin{align*} 
 %T_7
 N(\mathbf{u},\mathbf{z}) =& - \mu  \lim_{\epsilon\to 0} \int_{\C \setminus U_t} \partial_x\frac{1}{u_k-x }\langle   T_g(\mathbf{u}^{(k)})   \tilde{V}_{\gamma,g,\epsilon}(x)\prod_{i=1}^m \tilde{V}_{\alpha_i,g,\epsilon }(z_i) \rangle_{\hat\C,U_t,g} \,\dd x\\
=& B(\mathbf{u},\mathbf{z}) + \mu \gamma  \lim_{\epsilon\to 0} \int_{\C \setminus U_t}  \frac{1}{u_k-x } \langle   T_g(\mathbf{u}^{(k)})   \partial_x\Phi_{g,\epsilon}(x)\tilde{V}_{\gamma,g,\epsilon}(x)\prod_{i=1}^m \tilde{V}_{\alpha_i,g,\epsilon}(z_i) \rangle_{\hat\C,U_t,g} \, \dd x\\=: &B(\mathbf{u},\mathbf{z}) + \tilde N(\mathbf{u},\mathbf{z}),
  \end{align*} 
where  we used the relation
%\begin{equation*}
$\partial_x \tilde{V}_{\gamma,\epsilon}(x) = \gamma \partial_x   \Phi_{g, \epsilon}(x)  \tilde{V}_{\gamma,\epsilon}(x)$
%\end{equation*}
and denoted with $ B(\mathbf{u},\mathbf{z})$ the $\epsilon\to 0$ limit of the boundary term appearing in the Green formula:
 \begin{equation}\label{BT}
B(\mathbf{u},\mathbf{z})=\frac{i}{2}\oint_{\partial U_t }\frac{1}{  u_k-  x}\langle   T(\mathbf{x}^{(k)})    \tilde{V}_{\gamma,g }(x)\prod_{i=1}^m \tilde{V}_{\alpha_i ,g }(z_i) \rangle_{\hat\C,U_t,g}  \dd\bar x.
\end{equation}
In $\tilde N(\mathbf{u},\mathbf{z})$ we integrate by parts the $ \partial_x\Phi_{g,\epsilon}(x)$ and end up with 
  \begin{align}\nonumber
\tilde N(\mathbf{x},\mathbf{z})
=&  - \mu Q \gamma\sum_{\ell=1}^{k-1}  \int_{\C \setminus U_t}   \frac{1}{(u_k-x)(x-u_\ell)^3 } \langle   T_g(\mathbf{u}^{(\ell,k)})   \tilde{V}_{\gamma,g}(x)\prod_{i=1}^m \tilde{V}_{\alpha_i,g }(z_i) \rangle_{\hat\C,U_t,g} \,\dd x
\\
& +\mu  \gamma\sum_{\ell=1}^{k-1}  \int_{\C \setminus U_t}   \frac{1}{(u_k-x)(x-u_\ell)^2 } \langle \partial_{u_\ell} \Phi_g (u_\ell)  T_g(\mathbf{u}^{(\ell,k)})   \tilde{V}_{\gamma,g}(x)\prod_{i=1}^m \tilde{V}_{\alpha_i,g }(z_i) \rangle_{\hat\C,U_t,g} \,\dd x \nonumber
\\
&  +\frac{\mu^2  \gamma^2  }{2}  \int_{\C \setminus U_t}  \int_{\C \setminus U_t}   \frac{1}{(u_k-x) (x-x') } \langle  T_g(\mathbf{u}^{(k)}) \tilde{V}_{\gamma,g}(x) \tilde{V}_{\gamma,g}(x')\prod_{i=1}^m \tilde{V}_{\alpha_i,g }(z_i) \rangle_{\hat\C,U_t,g} \,\dd x \,\dd x' \nonumber
\\
&+  \mu  \gamma      \sum_{p=1}^m \int_{\C \setminus U_t}   \frac{\alpha_p}{(u_k-x)  } C_{\epsilon,0}(x,z_p) \langle  T_g(\mathbf{u}^{(k)})  \tilde{V}_{\gamma,g}(x)\prod_{i=1}^m \tilde{V}_{\alpha_i,g,\epsilon }(z_i) \rangle_{\hat\C,U_t,g}  \,\dd x +  \tilde N'(\mathbf{u},\mathbf{z})  \nonumber
\\
&=: 
C'_6(\mathbf{u},\mathbf{z})+C'_8(\mathbf{u},\mathbf{z})+ C'_{11}(\mathbf{u},\mathbf{z})+ C'_{14}(\mathbf{u},\mathbf{z}) +  \tilde N'(\mathbf{u},\mathbf{z})  \label{T4leq}
\end{align}
where again we took the  $\epsilon\to 0$ limit in the terms where it was obvious and where the extra terms are proportional to $\partial_x \omega (x)$ and again vanish due to  \eqref{KPZtypeidentity}:
\begin{align*}
 \tilde N' (\mathbf{x},\mathbf{z})'  & =  \frac{\mu \gamma Q}{2}    \int_{\C \setminus U_t}   \frac{1}{u_k-x} \partial_x \omega (x) \langle   T_g(\mathbf{u}^{(k)})   \tilde{V}_{\gamma,g}(x)\prod_{i=1}^m \tilde{V}_{\alpha_i,g }(z_i) \rangle_{\hat\C,U_t,g}  \,\dd x  \\
&  \quad +\frac{\mu^2  \gamma^2  }{4}  \int_{\C \setminus U_t}  \int_{\C \setminus U_t}   \frac{1}{(u_k-x) } \partial_x \omega (x) \langle  T_g(\mathbf{u}^{(k)}) \tilde{V}_{\gamma,g}(x)V_{\gamma,g}(x')\prod_{i=1}^m \tilde{V}_{\alpha_i,g }(z_i) \rangle_{\hat\C,U_t,g} \,\dd x \,\dd x'  \\
& \quad - \frac{\mu  \gamma}{4}       ( \sum_{p=1}^m  \alpha_p )\int_{\C \setminus U_t}   \frac{1}{(u_k-x)  } \partial_x \omega (x)  \langle  T_g(\mathbf{u}^{(k)})  \tilde{V}_{\gamma,g}(x)\prod_{i=1}^m \tilde{V}_{\alpha_i,g,\epsilon }(z_i) \rangle_{\hat\C,U_t,g}  \,\dd x \\
& \quad - \frac{\mu \gamma^2}{4} \int_{\C \setminus U_t}   \frac{1}{(u_k-x)  } \partial_x \omega (x)  \langle  T_g(\mathbf{u}^{(k)})  \tilde{V}_{\gamma,g}(x)\prod_{i=1}^m \tilde{V}_{\alpha_i,g,\epsilon }(z_i) \rangle_{\hat\C,U_t,g}  \,\dd x = 0 .
\end{align*}

 In particular the identity \eqref{T4leq} proves that the limit on the RHS, denoted by $C'_{14}(\mathbf{u},\mathbf{z})$, exists. The numbering of these terms and the ones below will be used when comparing with the $C$ terms.

\vskip 2mm

\noindent {\bf Derivatives of correlation functions.} We want to compare the expression \eqref{IPPstress} to  the derivatives of the function $\langle   T_g(\mathbf{u}^{(k)})   \prod_{i=1}^ m \tilde{V}_{\alpha_i,g }(z_i) \rangle_{\hat\C,U_t,g}$ involved in the Ward identity.  We have:
\begin{lemma}\label{deriv2}
Let
   \begin{align*}
I_\epsilon(\mathbf{u}, {\bf z}):=
\sum_{p=1}^m\frac{1}{u_k-z_p}&  \partial_{z_p}\langle   T_g(\mathbf{u}^k)  \prod_{i=1}^m \tilde{V}_{\alpha_i,g,\epsilon }(z_i) \rangle_{\hat\C,U_t,g} .% = K_\epsilon({\bf u})+L_\epsilon({\bf u})
  \end{align*}  
  Then $\lim_{\epsilon\to 0}I_\epsilon(\mathbf{u},{\bf z}):= I(\mathbf{u},{\bf z})$
exists and defines a continuous function for non coinciding $(\mathbf{z},{\bf u})\in \caO^{\rm ext}_{\C,U_t} $ satisfying, for fixed $\mathbf{u}$,
  \begin{align}\label{contourI}
  \int  I(\mathbf{u},{\bf z}) \bar{\varphi}({\bf z})\dd {\bf z}=  \int  \langle   T_g(\mathbf{u}^k)  \prod_{i=1}^m \tilde{V}_{\alpha_i,g }(z_i) \rangle_{\hat\C,U_t,g} \hat{\bf D}^*\bar{\varphi}({\bf z})\dd {\bf z}
\end{align}
for all smooth function $\varphi $ with compact support in the set $\{{\bf z}\in \hat\C^n\,|\, (\mathbf{z},{\bf u})\in \caO^{\rm ext}_{\C,U_t}\}$, and $\hat{\bf D}=\sum_{p=1}^n\frac{1}{u_k-z_p}   \partial_{z_p}$. Moreover, the derivatives in $z_1, z_2, z_3$ which appear in the definition of $I(\mathbf{u},{\bf z})$ exist in the classical sense.
\end{lemma}
\begin{proof}
We have
   \begin{align*}
I_\epsilon(\mathbf{u},{\bf z})=\sum_{p=1}^n\frac{\alpha_p}{u_k-z_p}&  \langle   T(\mathbf{u}^{(k)}) \partial_{z_p}\Phi_{g,\epsilon}(z_p)\prod_{i=1}^m \tilde{V}_{\alpha_i,g,\epsilon }(z_i) \rangle_{\hat\C,U_t,g}  = D_\epsilon(\mathbf{u}, {\bf z})+L_\epsilon(\mathbf{u}, {\bf z})
  \end{align*}  
  where we integrate by parts the $\partial_{z_p}\Phi_{g,\epsilon}(z_p)$ and $D_\epsilon({\bf u},{\bf z})$ collects the terms with an obvious $\epsilon\to 0$ limit $D({\bf u},{\bf z})$:  
    \begin{align*}
 D(\mathbf{u},{\bf z})  =&- \sum_{p=1}^m\sum_{\ell=1}^{k-1}\frac{Q\alpha_p}{(u_k-z_p)(z_p-u_\ell)^3}  \langle   T_g(\mathbf{u}^{(\ell,k)})  \prod_{i=1}^m \tilde{V}_{\alpha_i ,g}(z_i) \rangle_{\hat\C,U_t,g}  %\red{ \quad  \frac{Q\alpha_p z_p^{1-\nu_k}}{(z_p-u_\ell)^3}}  ?
   \\
&+ \sum_{p=1}^m\sum_{\ell=1}^{k-1}\frac{\alpha_p}{(u_k-z_p)(z_p-u_\ell)^2}  \langle  \partial_{z} \Phi_g (x_\ell) T_g(\mathbf{u}^{(\ell,k)})   \prod_{i=1}^m \tilde{V}_{\alpha_i,g }(z_i) \rangle_{\hat\C,U_t,g}     %  \red{ \quad  \frac{-\alpha_p z_p^{1-\nu_k}}{(z_p-u_\ell)^2}}  
    %  \red{ \quad  \frac{-\alpha_p z_p^{1-\nu_k}}{(z_p-u_\ell)^2}}  
 \\  
&- \sum_{p\not= p'=1}^m \frac{\alpha_p\alpha_{p'}}{2}\frac{1}{(u_k-z_p)(z_p-z_{p'})}  \langle  T_g(\mathbf{u}^{(k)})  \prod_{i=1}^m \tilde{V}_{\alpha_i,g }(z_i) \rangle_{\hat\C,U_t,g}
   \\
%& +\frac{\mu  \gamma  }{2}  \int_{|x|>1} \frac{1}{(u_k-x)x} \Big( (\alpha+\gamma +\sum_{p=1}^n \alpha_p-2Q)  \langle   T(\mathbf{u}^{(k)})  V_\alpha(0)V_\gamma(x)\prod_{i=1}^nV_{\alpha_i }(z_i) \rangle_t\\
  =:&D_5(\mathbf{u},{\bf z})+D_{9}(\mathbf{u},{\bf z})+D_{10}(\mathbf{u},{\bf z}) ,
  \end{align*}  
whereas 
\begin{align*}
 L_\epsilon(\mathbf{u},{\bf z}) %(=D_4(\mathbf{u}))
=-\mu\gamma\sum_{p=1}^m  %\frac{\alpha_p }{2}
\alpha_p\int_{  \C \setminus U_t}%\frac{1}{(u_k-z_p)(z_p-x)} 
\frac{1}{u_k-z_p}C_{\epsilon,0}(z_p,x)  \langle  T_g(\mathbf{u}^{(k)})  \tilde{V}_{\gamma,g}(x)\prod_{i=1}^m \tilde{V}_{\alpha_i ,g,\epsilon}(z_i) \rangle_{\hat\C,U_t,g}     \,\dd x   .
  \end{align*}  
A simple computation shows that the metric terms disappear.
Since $C_{0,0}(z_p,x) =-\frac{1}{2}\frac{1}{z_p-x}$  and since it is not clear that $\frac{1}{z_p-x} \langle  T(\mathbf{u}^{k})  \tilde{V}_{\gamma,g}(x)\prod_{i=1}^m \tilde{V}_{\alpha_i,g }(z_i) \rangle_t    $ is integrable (except for $p=1,2,3$) the $\epsilon\to 0$ limit of $L_\epsilon$ is problematic. However, we can compare it with the term ${C_{14}'}$ in \eqref{T4leq}, for which convergence was established there.  Writing 
$$\frac{1}{u_k-z_p}=\frac{1}{u_k-x}+\frac{z_p-x}{(u_k-z_p)(u_k-x)},
$$
we conclude that $ L_\epsilon$ converges:
\begin{align*}
\lim_{\epsilon\to 0} L_\epsilon(\mathbf{u},{\bf z})
 =&\frac{\mu \gamma}{2} \sum_{p=1}^n  %\frac{\alpha_p }{2}
\alpha_p  \int_{ \C \setminus U_t }
\frac{1}{(u_k-x)(z_p-x)}   \langle  T(\mathbf{x}^{(k)})  \tilde{V}_{\gamma,g}(x)\prod_{i=1}^m \tilde{V}_{\alpha_i,g,\epsilon }(z_i) \rangle_{\hat\C,U_t,g}    \,\dd x
\\
& + \frac{\mu \gamma}{2}  \int_{\C \setminus U_t}
\frac{1}{(u_k-z_p)(u_k-x)}   \langle  T(\mathbf{u}^{(k)})  \tilde{V}_{\gamma,g}(x)\prod_{i=1}^m \tilde{V}_{\alpha_i,g }(z_i) \rangle_{\hat\C,U_t,g}     \,\dd x 
\\
 =&C_{14}'(\mathbf{u},{\bf z})+C_{14}(\mathbf{u},{\bf z}).
  \end{align*}  
  We will use the identity 
  \begin{equation}\label{eq:derward}
I(\mathbf{u},{\bf z})=D_5(\mathbf{u},{\bf z})+D_{9}(\mathbf{u},{\bf z})+D_{10}(\mathbf{u},{\bf z})+C_{14}'(\mathbf{u},{\bf z})+C_{14}(\mathbf{u},{\bf z})
\end{equation}
in the following.
\end{proof}

Now we turn to the derivatives of the SET-insertions. Let 
\begin{equation}\label{Sterm}
S(\mathbf{u},{\bf z}):=\sum_{\ell=1}^{k-1}\frac{1}{u_k-u_{\ell}}\partial_{u_{\ell}}\langle   T_g(\mathbf{u}^{(k)}) \prod_{i=1}^m \tilde{V}_{\alpha_i,g }(z_i) \rangle_{\hat\C,U_t,g}.
\end{equation}
From \eqref{defSETsimple} we get
\begin{align*}%\label{}
\partial_{u_{\ell}}T_g(\mathbf{u}^{(\ell)})=Q \partial_{u_\ell}^3 X_g(u_\ell)- 2  \partial_{u_\ell} X_g(u_\ell) \partial_{u_\ell}^2 X_g(u_\ell)-Q\partial_{u_\ell} \omega (u_\ell)\partial_{u_\ell}^2 X_g(u_\ell) -Q\partial_{u_\ell}^2 \omega (u_\ell)\partial_{u_\ell} X_g(u_\ell).
\end{align*}
We collect the  metric dependent terms to
\begin{align*}
S' (\mathbf{u},{\bf z})& :=  -Q \sum_{\ell=1}^{k-1}  \frac{\partial_{u_\ell} \omega (u_\ell) }{u_k-u_\ell} \langle  \partial_{u_\ell}^2 X_g(u_\ell) T_g(\mathbf{u}^{(k ,\ell)})    \prod_{i=1}^m \tilde{V}_{\alpha_i,g }(z_i)   \rangle_{\hat\C,U_t,g} \\ & -Q   \sum_{\ell=1}^{k-1}\frac{\partial_{u_\ell}^2 \omega (u_\ell) }{u_k-u_\ell}
  \langle  \partial_{u_\ell} X_g(u_\ell) T_g(\mathbf{u}^{(k ,\ell)})    \prod_{i=1}^m \tilde{V}_{\alpha_i,g }(z_i)   \rangle_{\hat\C,U_t,g}
\end{align*}
Then we claim that
\begin{align*}
S(\mathbf{u},{\bf z})=\sum_{i}S_i(\mathbf{u},{\bf z})+S'(\mathbf{u},{\bf z})
\end{align*}
 with the terms $S_i$ given by:%\\
 %Contraction 1 is 
\begin{align*}
 %&  Q^2 \E[\partial_{u_\ell}^3 X_g(u_\ell) \partial_{u_{\ell'}}^2 X_g(u_{\ell'})   ]=  -\frac{12 Q^2}{(u_\ell- u_{\ell'})^5}, \\ 
& S_1(\mathbf{u},{\bf z}):= - \sum_{\ell\not=\ell'=1}^{k-1}\frac{12 Q^2}{(u_k-u_{\ell})(u_\ell-u_{\ell'})^5} \langle   T_g(\mathbf{u}^{(k,\ell,\ell')})  \prod_{i=1}^m \tilde{V}_{\alpha_i,g }(z_i) \rangle_{\hat\C,U_t,g}\\
%\end{align*}
%Contraction 3 is
%\begin{align*}
%& -2 Q \E[  \partial_{u_\ell}^2 X(u_\ell) \partial_{u_{\ell'}}^2 X(u_{\ell'})    ] =   \frac{-6Q}{(u_\ell-u_{\ell'})^4},  \\
& S_{3}(\mathbf{u},{\bf z}):=  -6Q  \sum_{\ell\not=\ell'=1}^{k-1}   \frac{ 1}{(u_k-u_{\ell})(u_{\ell}- u_{\ell'})^4} \langle   \partial_{u_\ell}X_g(u_{\ell}) T(\mathbf{u}^{(k ,\ell,\ell')})    \prod_{i=1}^mV_{\alpha_i,g }(z_i) \rangle_{\hat\C,U_t,g} \\
%\end{align*}
%Contraction 5 is 
%\begin{align*}
%& Q \alpha_p \E[  \partial_{u_\ell}^3 X_g(u_\ell)  X(z_p) ] =   -\frac{Q \alpha_p}{(u_\ell-z_p)^3}-\frac{Q \alpha_p}{4} \partial_{u_\ell}^3 \omega (u_\ell)  \\
& S_{5}(\mathbf{u},{\bf z}):= - \sum_{\ell=1}^{k-1} \sum_{p=1}^m    \frac{ Q\alpha_p}{(u_k-u_{\ell})(u_\ell-z_p)^3} \langle  T_g(\mathbf{u}^{(k,\ell )})   \prod_{i=1}^m \tilde{V}_{\alpha_i,g }(z_i) \rangle_{\hat\C,U_t,g},\\
& \qquad \qquad \quad - \frac{1}{4} \sum_{\ell=1}^{k-1} \sum_{p=1}^m    \frac{ Q\alpha_p}{(u_k-u_{\ell})} \partial_{u_\ell}^3 \omega (u_\ell)   \langle  T_g(\mathbf{u}^{(k,\ell )})   \prod_{i=1}^mV_{\alpha_i,g }(z_i) \rangle_{\hat\C,U_t,g}\\
%\end{align*}
%Contraction 6 is
%\begin{align*}
%& -Q\mu \gamma \E[  \partial_{u_\ell}^3 X_g(u_\ell)  X_g(x)  ] =   \frac{Q \mu \gamma }{(u_\ell-x)^3}+\frac{Q \mu \gamma }{4} \partial_{u_\ell}^3 \omega (u_\ell)  \\
& S_{6}(\mathbf{u},{\bf z}):=  Q\mu\gamma  \sum_{\ell=1}^{k-1}  \int_{\C \setminus U_t} \frac{ 1}{(u_k-u_{\ell})(u_{\ell}- x)^3} \langle  T_g(\mathbf{u}^{(k ,\ell)})   \tilde{V}_{\gamma,g}(x)\prod_{i=1}^m \tilde{V}_{\alpha_i,g }(z_i) \rangle_{\hat\C,U_t,g}\,\dd x\\
& \qquad \qquad \quad +  \frac{Q\mu\gamma}{4}  \sum_{\ell=1}^{k-1} \partial_{u_\ell}^3 \omega (u_\ell)  \int_{\C \setminus U_t  } \frac{ 1}{(u_k-u_{\ell})}  \langle  T_g(\mathbf{u}^{(k ,\ell)})   \tilde{V}_{\gamma,g}(x)\prod_{i=1}^m \tilde{V}_{\alpha_i,g }(z_i) \rangle_{\hat\C,U_t,g}\,\dd x\\
%\end{align*}
%Contraction 7 is
%\begin{align*}
%& -2 Q  \E[  \partial_{u_\ell}^3 X_g(u_\ell) \partial_{u_{\ell'}} X_g(u_{\ell'})     ] =   \frac{6Q}{ (u_\ell-u_{\ell'})},  \\
& S_7(\mathbf{u},{\bf z}) := 6Q \sum_{\ell\not=\ell'=1}^{k-1}\frac{1}{(u_k-u_{\ell})(u_\ell-u_{\ell'})^4} \langle  \partial_{u_{\ell'}}\Phi_g(u_{\ell'}) T_g(\mathbf{u}^{(k,\ell,\ell')})  \prod_{i=1}^m \tilde{V}_{\alpha_i,g }(z_i) \rangle_{\hat\C,U_t,g}\\
%\end{align*}
%Contraction 8 is
%\begin{align*}
% & 2 \mu \gamma   \E[  \partial_{u_\ell}^2 X_g(u_\ell)  X_g(x) ] =   \frac{\mu \gamma}{  (u_\ell-x)^2}  -\frac{\mu \gamma}{2} \partial_{u_\ell}^2 \omega (u_\ell), \\
& S_{8}(\mathbf{x},{\bf z}):= + \mu\gamma \sum_{\ell=1}^{k-1}  \int_{\C \setminus U_t } \frac{ 1}{(u_k-u_{\ell})(u_{\ell}- x)^2} \langle   \partial_{u_\ell} X(u_{\ell} )T_g(\mathbf{u}^{(k,\ell )})   \tilde{V}_{\gamma,g}(x)\prod_{i=1}^m \tilde{V}_{\alpha_i,g }(z_i) \rangle_{\hat\C,U_t,g}\,\dd x  \\
&\qquad \qquad - \frac{\mu\gamma}{2}  \sum_{\ell=1}^{k-1} \partial_{u_\ell}^2 \omega (u_\ell) \int_{\C \setminus U_t} \frac{ 1}{(u_k-u_{\ell})} \langle   \partial_{u_\ell} X(u_{\ell} )T_g(\mathbf{u}^{(k,\ell )})   \tilde{V}_{\gamma,g}(x)\prod_{i=1}^m \tilde{V}_{\alpha_i,g }(z_i) \rangle_{\hat\C,U_t,g}\,\dd x\\
%\end{align*}
%Contraction 9 is 
%\begin{align*}
%& -2 \alpha_p  \E[  \partial_{u_\ell}^2 X_g(u_\ell)  X_g(z_p)     ] =  - \frac{\alpha_p}{  (u_\ell-z_p)^2}  +\frac{\alpha_p}{2} \partial_{u_\ell}^2 \omega (u_\ell),\\
& S_{9}(\mathbf{u},{\bf z}):= -  \sum_{\ell=1}^{k-1}\sum_{p=1}^m  \frac{\alpha_p}{(u_k-u_{\ell})(u_{\ell}-z_p)^2} \langle  \partial_{u_\ell}X_g(u_{\ell} ) T(\mathbf{u}^{(k,\ell  )})    \prod_{i=1}^m \tilde{V}_{\alpha_i,g }(z_i) \rangle_{\hat\C,U_t,g} \\
& \qquad\qquad\quad +  \sum_{\ell=1}^{k-1}\sum_{p=1}^m  \frac{\alpha_p}{2 (u_k-u_{\ell})} \partial_{u_\ell}^2 \omega (u_\ell) \langle  \partial_{u_\ell}X_g(u_{\ell} ) T_g(\mathbf{u}^{(k,\ell  )})    \prod_{i=1}^m \tilde{V}_{\alpha_i,g }(z_i) \rangle_{\hat\C,U_t,g}\\
%\end{align*}
%Contraction 12 is 
%\begin{align*}
% & 4  \E[  \partial_{u_\ell}^2 X_g(u_\ell) \partial_{u_{\ell'}} X_g(u_{\ell'})       ] =  \frac{4}{  (u_\ell-u_{\ell'})^3},  \\
& S_{12}(\mathbf{u},{\bf z}):= 4 \sum_{\ell\not=\ell'=1}^{k-1}   \frac{ 1}{(u_k-u_{\ell})( u_\ell-u_{\ell'})^3} \langle  \partial_{u_\ell}X_g(u_{\ell} )\partial_{u_{\ell'}}\Phi_g(u_{\ell'} )T(\mathbf{u}^{(k,\ell,\ell' )})  \prod_{i=1}^m \tilde{V}_{\alpha_i,g }(z_i) \rangle_{\hat\C,U_t,g} .
\end{align*}
Again, the numbering will be used to compare with the $C$ terms. Also, to establish this formula, we regularise the SET-insertion $T(u_\ell)$, differentiate it, then use Gaussian integration by parts  and then pass to the limit as $\epsilon\to 0$. Notice that the convergence of all these terms is obvious as the variables $u_\ell$ for $\ell=1,\dots,k-1$ belong to $U_t$. The same strategy can be applied to establish that
\begin{equation}\label{Lterm}
L(\mathbf{u}, {\bf z}):=\sum_{\ell=1}^{k-1}\frac{2}{(u_k-u_{\ell})^2} \langle   T_g(\mathbf{u}^{(k)})    \prod_{i=1}^m \tilde{V}_{\alpha_i,g }(z_i) \rangle_{\hat\C,U_t,g}
\end{equation}
is the sum of a metric dependent term which is
\[
- \sum_{\ell=1}^{k-1}\frac{2 Q}{(u_k-u_{\ell})^2} \partial_{u_\ell} \omega (u_\ell) \langle \partial_{u_\ell}  X_g (u_\ell)   T(\mathbf{u}^{(k,\ell)})    \prod_{i=1}^m \tilde{V}_{\alpha_i,g }(z_i) \rangle_{\hat\C,U_t,g}
\]
and another term which can be written as  $\sum_{i=1}^{}L_i(\mathbf{u}, {\bf z})$ with the terms $L_i$'s given by:
%Contraction 1 is 
\begin{align*}
%&  Q^2 \E[\partial_{u_\ell}^2 X_g(u_\ell) \partial_{u_{\ell'}}^2 X_g(u_{\ell'})  ] = \frac{3 Q^2}{(u_\ell-u_{\ell'})^4} , \\
&  L_1(\mathbf{u}, {\bf z}) := \sum_{\ell\not=\ell'=1}^{k-1}\frac{6Q^2}{(u_k-u_\ell)^2(x_\ell-x_{\ell'})^4} \langle   T_g(\mathbf{u}^{(k,\ell,\ell')}) \prod_{i=1}^m \tilde{V}_{\alpha_i,g }(z_i) \rangle_{\hat\C,U_t,g} \\ 
%\end{align*}
%Contraction 3 is 
%\begin{align*}
%&  -Q \E[\partial_{u_\ell} X_g(u_\ell) \partial_{u_{\ell'}}^2 X_g(u_{\ell'})  ] = \frac{ Q}{(u_\ell-u_{\ell'})^3} , \\
& L_3(\mathbf{u},{\bf z}) :=  2 Q  \sum_{\ell\not=\ell'=1}^{k-1} \frac{1}{(u_k-u_\ell)^2(u_\ell-u_{\ell'})^3} \langle      \partial_{u_\ell} X_g(u_{\ell} )   T_g(\mathbf{u}^{(k,\ell ,\ell')})    \prod_{i=1}^m \tilde{V}_{\alpha_i,g }(z_i) \rangle_{\hat\C,U_t,g}   \\
%\end{align*}
%Contraction 5 is
%\begin{align*}
 %& Q \alpha_p \E[\partial_{u_\ell}^2 X_g(u_\ell)  X_g(z_p)  ] = \frac{ Q \alpha_p }{2(u_\ell-z_p)^2}  - \frac{Q \alpha_p}{4}  \partial_{u_\ell}^2 \omega (u_\ell)    \\
& L_5(\mathbf{x}, {\bf z}) :=\sum_{\ell=1}^{k-1} \sum_{p=1}^m \frac{Q\alpha_p}{(u_k-u_{\ell})^2(u_\ell-z_p)^2}   \langle      T_g(\mathbf{u}^{(k,\ell )})   \prod_{i=1}^m \tilde{V}_{\alpha_i,g }(z_i) \rangle_{\hat\C,U_t,g} , \\
&\qquad\qquad\quad -\sum_{\ell=1}^{k-1} \sum_{p=1}^m \frac{Q\alpha_p}{2 (u_k-u_{\ell})^2} \partial_{u_\ell}^2  \omega (u_\ell)   \langle      T_g(\mathbf{u}^{(k,\ell )})   \prod_{i=1}^m \tilde{V}_{\alpha_i,g }(z_i) \rangle_{\hat\C,U_t,g}\\
%\end{align*}
%Contraction 6 is 
%\begin{align*}
%& -Q \mu \gamma \E[\partial_{u_\ell}^2 X_g(u_\ell)  X_g(x)  ] = -\frac{ Q \mu \gamma }{2(u_\ell-z_p)^2}  + \frac{Q \mu \gamma}{4}  \partial_{u_\ell}^2 \omega (u_\ell) , \\
& L_{6}(\mathbf{u} ,{\bf z}):= -\mu\gamma Q\sum_{\ell=1}^{k-1} \int_{\C \setminus U_t } \frac{1}{(u_k-u_\ell)^2(x_\ell-x)^2} \langle      T_g(\mathbf{u}^{(k,\ell )})  \tilde{V}_{\gamma,g}(x)\prod_{i=1}^m \tilde{V}_{\alpha_i,g }(z_i) \rangle_{\hat\C,U_t,g}  \,\dd x  \\
& \qquad\qquad\quad +\mu\gamma Q\sum_{\ell=1}^{k-1} \int_{ \C \setminus U_t } \frac{1}{2 (u_k-u_\ell)^2}   \partial_{u_\ell}^2 \omega (u_\ell)  \langle      T_g(\mathbf{u}^{(k,\ell )})  \tilde{V}_{\gamma,g}(x)\prod_{i=1}^m \tilde{V}_{\alpha_i,g }(z_i) \rangle_{\hat\C,U_t,g}  \,\dd x \\
%\end{align*}
%Contraction 7 is
%\begin{align*}
%& -2 Q  \E[\partial_{u_\ell}^2 X(u_\ell)   \partial_{u_{\ell'}} X(u_{\ell'})     ] = -2 \frac{ Q }{(u_\ell-u_{\ell'})^3},\\
& L_7(\mathbf{u}, {\bf z}) :=-4Q \sum_{\ell\not=\ell'=1}^{k-1}\frac{1}{(u_k-u_\ell)^2(u_\ell-u_{\ell'})^3}     \langle     \partial_{u_{\ell'}}\Phi_g(u_{\ell'} )T_g(\mathbf{u}^{(k,\ell,\ell')})   \prod_{i=1}^m \tilde{V}_{\alpha_i,g }(z_i) \rangle_{\hat\C,U_t,g}\\
%\end{align*}
%Contraction 8 is 
%\begin{align*}
%& \mu \gamma  \E[\partial_{u_\ell} X_g(u_\ell)   X_g(x)     ] = - \frac{ \mu \gamma }{2 (u_\ell-x)}  -\frac{\mu \gamma}{4} \partial_{u_\ell}  \omega (u_\ell),  \\
& L_{8}(\mathbf{x} ,{\bf z}) :=-\mu\gamma   \sum_{\ell=1}^{k-1}\int_{\C \setminus U_t } \frac{1}{(u_k-u_{\ell})^2(u_{\ell}-x)} \langle      \partial_{u_\ell}X_g(u_{\ell} )    T_g(\mathbf{u}^{(k ,\ell)})   \tilde{V}_{\gamma,g}(x)\prod_{i=1}^m \tilde{V}_{\alpha_i,g }(z_i) \rangle_{\hat\C,U_t,g}  \,\dd x\\
& \qquad\qquad \quad -\frac{\mu\gamma}{2}   \sum_{\ell=1}^{k-1}  \partial_{u_\ell}  \omega (u_\ell)  \int_{\C \setminus U_t  } \frac{1}{(u_k-u_{\ell})^2} \langle      \partial_{u_\ell} X_g(u_{\ell} )    T_g(\mathbf{u}^{(k ,\ell)})   \tilde{V}_{\gamma,g}(x)\prod_{i=1}^m \tilde{V}_{\alpha_i,g }(z_i) \rangle_{\hat\C,U_t,g}  \,\dd x\\
%\end{align*}
%Contraction 9 is 
%\begin{align*}
%& -\alpha_p \E[\partial_{u_\ell} X_g(u_\ell)   X_g(z_p) ] =  \frac{ \alpha_p }{2 (u_\ell-z_p)}  +\frac{\alpha_p}{4} \partial_{u_\ell}  \omega (u_\ell)      \\
& L_{9}(\mathbf{u}, {\bf z}) :=\sum_{\ell=1}^{k-1} \sum_{p=1}^m \frac{ \alpha_p}{(u_k-u_{\ell})^2(u_\ell-z_p) }   \langle         \partial_{u_\ell}X_g(u_{\ell} )       T_g(\mathbf{u}^{(k,\ell  )})    \prod_{i=1}^m \tilde{V}_{\alpha_i,g }(z_i) \rangle_{\hat\C,U_t,g}    
\\
&\qquad\qquad\quad +  \sum_{\ell=1}^{k-1} \sum_{p=1}^m \frac{ \alpha_p}{2 (u_k-u_{\ell})^2 }  \partial_{u_\ell}  \omega (u_\ell)   \langle         \partial_{u_\ell}X(u_{\ell} )       T_g(\mathbf{u}^{(k,\ell  )})    \prod_{i=1}^m \tilde{V}_{\alpha_i,g }(z_i) \rangle_{\hat\C,U_t,g} \\
%\end{align*}
%Contraction 12 is 
%\begin{align*}
%& 2 \E[\partial_{u_\ell} X_g(u_\ell) \partial_{u_{\ell'}} X_g(u_{\ell'})      ] =  - \frac{ 1} { (u_\ell-u_{\ell'})^2} ,      \\
& L_{12}(\mathbf{u} ,{\bf z}):=- 2\sum_{\ell\not=\ell'=1}^{k-1} \frac{1}{(u_k-u_{\ell})^2(u_\ell-u_{\ell'})^2} \langle      \partial_{u_\ell}X_g(u_{\ell} )  \partial_{u_{\ell'}} \Phi_g (u_{\ell'} )   T_g(\mathbf{u}^{(k,\ell,\ell' )})   \prod_{i=1}^m \tilde{V}_{\alpha_i,g }(z_i) \rangle_{\hat\C,U_t,g}  .
\end{align*}

\noindent {\bf Ward algebra.}
A long and tedious computation (using \eqref{eqconffactor}) shows that the metric dependent terms in $S(\mathbf{u} ,{\bf z})$ and $L(\mathbf{u} ,{\bf z})$ combine in such a way that $S(\mathbf{u} ,{\bf z})= \sum_i \tilde{S}_i$ and $L(\mathbf{u} ,{\bf z})= \sum_i \tilde{L}_i$ where $\tilde{S}_i$ ($\tilde{L}_i$) is obtained from $S_i$ ($L_i$) by replacing all the $\partial_{u_\ell} X_g(u_\ell)$ which appear in the definition of $S_i$ ($L_i$) by $\partial_{u_\ell} \Phi_g(u_\ell)$ (if no terms of this form appear than $\tilde{S}_i=S_i$).  This now allows direct comparison with  $C(\mathbf{u}, \mathbf{z})+ N(\mathbf{u}, \mathbf{z})$. Now we are going to show that the appropriate combination of all these expressions combine to produce the desired identity. Let us consider the expression
\begin{align*} 
K(\mathbf{u},{\bf z}):=&\langle   T_g(\mathbf{u}) \prod_{i=1}^ m \tilde{V}_{\alpha_i,g }(z_i) \rangle_{\hat\C,U_t,g }  -I(\mathbf{u}, {\bf z})- M(\mathbf{u}, \mathbf{z})- T(\mathbf{u}, \mathbf{z})-S(\mathbf{u}, \mathbf{z})-L(\mathbf{u}, \mathbf{z}) , \\
=&C(\mathbf{u}, \mathbf{z})   + N(\mathbf{u}, \mathbf{z})    -I(\mathbf{u},{\bf z})-S(\mathbf{u},\mathbf{z})-L(\mathbf{u},\mathbf{z})  .%-J(\mathbf{u},{\bf z})
\end{align*}
 Also, we have obtained the relation
\begin{align}
N -I &=B+C'_6+C'_8+C'_{11}   -D_5-D_9-D_{10}-C_{14},\label{N-I}
\end{align}
in such a way that $K$ can be rewritten as
\begin{align*} 
K =& A_1+A_3+A_5+A_6+A_7+A_8+A_9+A_{10}+A_{11}+A_{12}+B
\end{align*}
with

 \begin{align}
A_1:=&C_1+C_2-\tilde{S}_1-L_1\label{theone}\\
A_3:=&C_3+C_4/2-\tilde{S}_3/2-\tilde{S}_7/2-\tilde{L}_3/2-\tilde{L}_7/2\label{thethree}\\
A_5:=&C_5+C_{15}-D_5-\tilde{S}_5-\tilde{L}_5\label{thefive}\\
 A_6:=& C_6+C_{13}-C'_6-\tilde{S}_6-\tilde{L}_6\label{thesix}\\
 A_7:=&C_7+C_4/2-\tilde{S}_3/2-\tilde{S}_7/2-\tilde{L}_3/2-\tilde{L}_7/2\label{theseven}\\
A_8:=&C_{8}+C'_{8}-\tilde{S}_{8}-\tilde{L}_{8} \label{theeight}\\
A_9:=& C_9-D_9-\tilde{S}_9-\tilde{L}_9 \label{thenine}\\
A_{10}:=& C_{10}-D_{10}\\
A_{11}:=&C_{11}+C'_{11}\\
A_{12}:= &C_{12}-\tilde{S}_{12}-\tilde{L}_{12}\label{fuck!!!!}.
\end{align}
Finally, we claim that all the $A_i$'s vanish.  Indeed, this is straightforward for $A_{10}, A_{11}$: it comes from the relation 
$$\frac{1}{(u_k-z_p)(u_k-z_{p'})}=\frac{1}{z_p-z_{p'}}\Big(\frac{1}{u_k-z_p}-\frac{1}{u_k-z_{p'}}\Big)$$ and a re-indexation of the double sum for $A_{10}$.

All the other terms results from  algebraic identities:  for \eqref{theone}, we use
\begin{align*}
\frac{1}{(z-u_\ell)^3 (z-u_{\ell'})^3}=&\frac{1}{(z-u_\ell)^3 (u_\ell-u_{\ell'})^3}-\frac{3}{(z-u_\ell)^2 (u_\ell-u_{\ell'})^4}+\frac{6}{(z-u_\ell) (u_\ell-u_{\ell'})^5}
\\ 
&-\frac{1}{(z-u_{\ell'})^3 (u_\ell-u_{\ell'})^3}-\frac{3}{(z-u_{\ell'})^2 (u_\ell-u_{\ell'})^4}-\frac{6}{(z-u_\ell) (u_{\ell'}-u_{\ell})^5}.
\end{align*}
For \eqref{thethree} and \eqref{theseven} we use 
\begin{align*}\frac{1}{(z-u'_{\ell'})^3(z-u'_{r'})^2}=&\frac{1}{(z-u'_{\ell'})^3(u'_{\ell'}-u'_{r'})^2}-\frac{2}{(z-u'_{\ell'})^2(u'_{\ell'}-u'_{r'})^3} 
\\
&+\frac{3}{(z-u'_{\ell'})(u'_{\ell'}-u'_{r'})^4} +\frac{1}{(z-u'_{r'})^2(u'_{r'}-u'_{\ell'})^3}- \frac{3}{(z-u'_{r'}) (u'_{r'}-u'_{\ell'})^4}.
\end{align*} 
For \eqref{thefive} and \eqref{thesix} we use
$$  \frac{1}{(u_k-u_\ell)^3(u_k-x)}= \frac{1}{(u_k-u_\ell)^3(u_\ell-x)} - \frac{1}{(u_k-u_\ell)^2(x-u_\ell)^2}+\frac{1}{(u_k-u_\ell)(u_\ell-x)^3}-\frac{1}{(u_k-x)(u_\ell-x)^3}. $$
For \eqref{theeight} and \eqref{thenine} we use
$$ \frac{1}{(u_k-x_\ell)^2(u_k-x)}= \frac{1}{(u_k-u_\ell)^2(u_\ell-x)} - \frac{1}{(u_k-u_\ell)(x-u_\ell)^2}+\frac{1}{(x-u_\ell)^2(u_k-x)} .$$
For \eqref{fuck!!!!} we use the identity
$$\frac{1}{(z-u_\ell)^2(z-x_{\ell'})^2}=\frac{1}{(z-u_\ell)^2(u_\ell-u_{\ell'})^2} -\frac{2}{(z-u_\ell) (u_\ell-u_{\ell'})^3}+\frac{1}{(u_\ell-u_{\ell'})^2(z-u_{\ell'})^2}-\frac{2}{(u_{\ell'}-u_\ell)^3(z-u_{\ell'}) }.$$
This concludes the proof of Proposition \ref{propward}.\qed

%%%%%%%%%%%%%%%%%%%%%%%%%%%%%%%%%%%%%%%%%%%%%%%%%%%%%%%%%%%%%%%%%%%%%%%%%%%%%%%%%%%%%%%%%%%%%%%%%%%%%%%

\bibliographystyle{alpha}
\bibliography{Segal}

\newcommand{\etalchar}[1]{$^{#1}$}
\begin{thebibliography}{ACSW24}

\bibitem[ACSW24]{AngSun21_CLE}
M.~Ang, G.~Cai, X.~Sun, and B.~Wu.
\newblock Integrability of conformal loop ensemble: Imaginary dozz formula and
  beyond.
\newblock {\em arXiv:2107.01788}, 2024.

\bibitem[ARS23]{AngRemySun21_FZZ}
M.~{Ang}, G.~{Remy}, and X.~{Sun}.
\newblock {FZZ formula of boundary Liouville CFT via conformal welding}.
\newblock {\em Journal of the European Math. Society}, DOI 10.4171/JEMS/1391,
  2023.

\bibitem[ARSZ23]{Ang_Remy_Sun_Zhu}
M.~Ang, G.~Remy, X.~Sun, and T.~Zhu.
\newblock Derivation of all structure constants for boundary {L}iouville {CFT}.
\newblock {\em arXiv: 2305.18266}, 2023.

\bibitem[BCT82]{Braaten_Curtright_Thorn}
E.~Braaten, T.~Curtright, and C.~Thorn.
\newblock Quantum {B}acklund transformation for the {L}iouville theory.
\newblock {\em Phys. Lett.}, 118B:115--120, 1982.

\bibitem[BFK92]{BurgheleaFK92}
D.~Burghelea, L.~Friedlander, and T.~Kappeler.
\newblock Meyer-{V}ietoris type formula for determinants of elliptic
  differential operators.
\newblock {\em J. Funct. Anal.}, 107(1):34--65, 1992.

\bibitem[BGK{\etalchar{+}}24]{BGKRV}
G.~{Baverez}, C.~{Guillarmou}, A.~{Kupiainen}, R.~{Rhodes}, and V.~{Vargas}.
\newblock {The Virasoro structure and the scattering matrix for Liouville
  conformal field theory}.
\newblock {\em Probability and Mathematical Physics}, 5(2):269--320, 2024.

\bibitem[BGKR24]{BGKR1}
G.~Baverez, C.~Guillarmou, A.~Kupiainen, and R.~Rhodes.
\newblock {Semigroup of annuli in Liouville CFT}.
\newblock {\em arXiv:2403.10914}, 2024.

\bibitem[BGKR25]{BGKR2}
G.~Baverez, C.~Guillarmou, A.~Kupiainen, and R.~Rhodes.
\newblock The conformal blocks of {L}iouville {CFT}.
\newblock {\em Preprint in preparation}, 2025.

\bibitem[BK86]{Belavin_knizhnik}
A.~Belavin and V.~Knizhnik.
\newblock Algebraic geometry and the geometry of quantum strings.
\newblock {\em Physics Letters}, 168(3):201--206, 1986.

\bibitem[Bor86]{borcherds}
R.~Borcherds.
\newblock {Vertex algebras, Kac-Moody algebras, and the Monster}.
\newblock {\em {Proceedings of the National Academy of Sciences of the United
  States of America}}, 83:3068--3071, 1986.

\bibitem[BPZ84]{BPZ84}
A.A. Belavin, A.M. Polyakov, and A.B. Zamolodchikov.
\newblock Infinite conformal symmetry in two-dimensional quantum field theory.
\newblock {\em Nuclear Physics B}, 241(2):333--380, 1984.

\bibitem[Car02]{Carron}
G.~Carron.
\newblock D\'{e}terminant relatif et la fonction {X}i.
\newblock {\em Amer. J. Math.}, 124(2):307--352, 2002.

\bibitem[{Cer}]{cercle2022}
B.~{Cercl\'{e}}.
\newblock Three-point correlation functions in the $\mathfrak{sl}_3$ toda
  theory i: Reflection coefficients.
\newblock {\em Journal of the European Math. Soc., to appear}.

\bibitem[{Cer}24]{cercle2024}
B.~{Cercl\'{e}}.
\newblock {Three-point correlation functions in the $\mathfrak{sl}_3$ Toda
  theory II: the Fateev-Litvinov formula}.
\newblock {\em Probability Theory and Related Fields}, 188:89--158, 2024.

\bibitem[CH24]{cerclehughenin2024}
B.~{Cercl\'{e}} and N.~{Huguenin}.
\newblock {Higher-spin symmetry in the $\mathfrak{sl}_3$ boundary Toda
  conformal field theory}.
\newblock {\em arXiv:2412.13874}, 2024.

\bibitem[CRV23]{TODA_2023}
B.~{Cercl\'{e}}, R.~Rhodes, and V.~Vargas.
\newblock Probabilistic construction of {Toda} {Conformal} {Field} {Theories}.
\newblock {\em Annales Henri Lebesgue}, 6:31--64, 2023.

\bibitem[CT82]{CurtrightThorn82}
T.~Curtright and C.~Thorn.
\newblock {Conformally Invariant Quantization of the Liouville Theory}.
\newblock {\em Phys. Rev. Lett.}, 48:1309, 1982.
\newblock [Erratum: Phys.Rev.Lett. 48, 1768 (1982)].

\bibitem[DDDF20]{DingDDF19_tightness}
J.~{Ding}, J.~{Dub{\'e}dat}, A.~{Dunlap}, and H.~{Falconet}.
\newblock {Tightness of Liouville first passage percolation for $\gamma \in
  (0,2)$}.
\newblock {\em Publ. math. IHES}, 132:353--403, 2020.

\bibitem[DFG{\etalchar{+}}20]{DubedatFGPS19_metric}
J.~{Dub{\'e}dat}, H.~{Falconet}, E.~{Gwynne}, J.~{Pfeffer}, and X.~{Sun}.
\newblock {Weak LQG metrics and Liouville first passage percolation}.
\newblock {\em Probab. Theory Relat. Fields}, 178:369--436, 2020.

\bibitem[Dim07]{Dimock2007}
J.~Dimock.
\newblock {Transition amplitudes and sewing properties for bosons on the
  Riemann sphere}.
\newblock {\em Journal of Mathematical Physics}, 48(5):052308, 05 2007.

\bibitem[DKRV16]{DKRV16}
F.~David, A.~Kupiainen, R.~Rhodes, and V.~Vargas.
\newblock Liouville quantum gravity on the {R}iemann sphere.
\newblock {\em Comm. Math. Phys.}, 342(3):869--907, 2016.

\bibitem[DMS21]{MatingOfTrees}
B.~{Duplantier}, J.~{Miller}, and S.~{Sheffield}.
\newblock {Liouville quantum gravity as a mating of trees}.
\newblock {\em Ast\'erisque}, 427, 2021.

\bibitem[DO94]{DornOtto94}
H.~Dorn and H.-J. Otto.
\newblock Two- and three-point functions in {L}iouville theory.
\newblock {\em Nuclear Phys. B}, 429(2):375--388, 1994.

\bibitem[DRV16]{DRV16_tori}
F.~David, R.~Rhodes, and V.~Vargas.
\newblock Liouville quantum gravity on complex tori.
\newblock {\em J. Math. Phys.}, 57(2):022302, 25, 2016.

\bibitem[EM12]{EaMa12}
C.~Earle and A.~Marden.
\newblock Holomorphic plumbing coordinates.
\newblock {\em Contemporary Mathematics}, 575:41--52, 2012.

\bibitem[FLM88]{Frenkel:1988xz}
I.~Frenkel, J.~Lepowsky, and A.~Meurman.
\newblock {\em {Vertex Operator Algebras and the monster}}, volume 134.
\newblock Academic Press, 1988.

\bibitem[FS87]{FriedanShenker87}
D.~Friedan and S.~Shenker.
\newblock The analytic geometry of two-dimensional conformal field theory.
\newblock {\em Nuclear Phys. B}, 281(3-4):509--545, 1987.

\bibitem[Gaw96]{Gawedzki96_CFT}
K.~Gawedzki.
\newblock Lectures on conformal field theory.
\newblock {\em Nucl. Phys. B}, 328:733--752, 1996.

\bibitem[GGK00]{GGK_book_2000}
I.~{Gohberg}, S.~{Goldberg}, and N.~{Krupnik}.
\newblock {\em Traces and Determinants of Linear Operators}.
\newblock Operator Theory: Advances and Applications. Birkh{\"a}user Basel,
  2000.

\bibitem[GKR23]{CILT2023}
C.~{Guillarmou}, A.~{Kupiainen}, and R.~Rhodes.
\newblock Compactified imaginary {L}iouville theory.
\newblock {\em 2310.18226}, 2023.

\bibitem[GKR24]{guillarmou2024}
C.~Guillarmou, A.~Kupiainen, and R.~Rhodes.
\newblock Review on the probabilistic construction and conformal bootstrap in
  {L}iouville theory.
\newblock {\em arXiv:2403.12780}, 2024.

\bibitem[GKR25]{GKR_WZW}
C.~{Guillarmou}, A.~{Kupiainen}, and R.~Rhodes.
\newblock {Probabilistic construction of the $\mathbb{H}^3$-Wess-Zumino-Witten
  conformal field theory and correspondence with Liouville theory}.
\newblock {\em arXiv:2502.16341}, 2025.

\bibitem[GKRV24]{GKRV20_bootstrap}
C.~{Guillarmou}, A.~{Kupiainen}, R.~{Rhodes}, and V.~{Vargas}.
\newblock {Conformal bootstrap in Liouville Theory}.
\newblock {\em Acta Mathematica}, 233(1):33--194, 2024.

\bibitem[GM21]{GM20}
E.~Gwynne and J.~Miller.
\newblock {Existence and uniqueness of the Liouville quantum gravity metric for
  \\$\gamma\in(0,2)$}.
\newblock {\em Inventiones mathematicae}, 223:213--333, 2021.

\bibitem[GN84]{GERVAIS1984125}
J.-L. Gervais and A.~Neveu.
\newblock Novel triangle relation and absence of tachyons in liouville string
  field theory.
\newblock {\em Nuclear Physics B}, 238(1):125--141, 1984.

\bibitem[GRSS23]{GRSS2}
R.~Ghosal, G.~Remy, X.~Sun, and Y.~Sun.
\newblock Analyticity and symmetry of virasoro conformal blocks via liouville
  cft.
\newblock Preprint 2023.

\bibitem[GRSS24]{GosalRemySun20}
P.~{Ghosal}, G.~{Remy}, X.~{Sun}, and Y.~{Sun}.
\newblock {Probabilistic conformal blocks for Liouville CFT on the torus}.
\newblock {\em Duke Math. J.}, 173(6):1085--1175, 2024.

\bibitem[GRV19]{GRVIHES}
C.~Guillarmou, R.~Rhodes, and V.~Vargas.
\newblock Polyakov's formulation of {$2d$} bosonic string theory.
\newblock {\em Publ. Math. Inst. Hautes \'{E}tudes Sci.}, 130:111--185, 2019.

\bibitem[GRW24]{GRW1}
C.~Guillarmou, R.~Rhodes, and B.~Wu.
\newblock {Conformal Bootstrap for surfaces with boundary in Liouville CFT.
  Part 1: Segal axioms}.
\newblock {\em arXiv:2408.13133}, 2024.

\bibitem[GRW25]{GRW2}
C.~Guillarmou, R.~Rhodes, and B.~Wu.
\newblock {Conformal Bootstrap for surfaces with boundary in Liouville CFT.
  Part 2: Spectral resolution and Ward identity}.
\newblock {\em in preparation}, 2025.

\bibitem[{Gui}24]{Bin_Gui}
B.~{Gui}.
\newblock Sewing and propagation of conformal blocks.
\newblock {\em New York Journal of Mathematics}, 30:187--230, 2024.

\bibitem[Hin10]{Hinich2010}
V.~Hinich.
\newblock Plumbing coordinates on {T}eichm{\"u}ller space: A counterexample.
\newblock {\em Israel Journal of Mathematics}, 175(1):151--156, 2010.

\bibitem[{Hua}]{Huang_blog}
Y-Z. {Huang}.
\newblock A program to construct and study conformal field theories.

\bibitem[{Hua}97]{Huang_1997}
Y.-Z. {Huang}.
\newblock {\em Two-Dimensional Conformal Geometry and Vertex Operator
  Algebras}.
\newblock Progress in Mathematics. Birkh{\"a}user Boston, MA, 1997.

\bibitem[HV10]{Hinich-Vaintrob}
V~Hinich and A.~Vaintrob.
\newblock {Augmented Teichm\"uller spaces and orbifolds}.
\newblock {\em Selecta {M}athematica}, 16:533--629, 2010.

\bibitem[HX23]{holden_sun_cardy}
Nina {Holden} and Sun {Xin}.
\newblock Convergence of uniform triangulations under the cardy embedding.
\newblock {\em Acta Mathematica}, 230(1):93--203, 2023.

\bibitem[IJS16]{PhysRevLett.116.130601}
Y.~Ikhlef, J.~Jacobsen, and H.~Saleur.
\newblock Three-point functions in $c\ensuremath{\le}1$ {L}iouville theory and
  conformal loop ensembles.
\newblock {\em Phys. Rev. Lett.}, 116:130601, Mar 2016.

\bibitem[Kah85]{Kahane85}
J-P. Kahane.
\newblock Sur le chaos multiplicatif.
\newblock {\em Ann. Sci. Math. Qu\'{e}bec}, 9(2):105--150, 1985.

\bibitem[KPZ88]{doi:10.1142/S0217732388000982}
V.. Knizhnik, A.~Polyakov, and A.~Zamolodchikov.
\newblock Fractal structure of 2d quantum gravity.
\newblock {\em Modern Physics Letters A}, 03(08):819--826, 1988.

\bibitem[Kra90]{Kra}
I.~Kra.
\newblock {Horocyclic coordinates for Riemann surfaces and moduli spaces. I.
  Teichm\"uller and Riemann spaces of Kleinian groups}.
\newblock {\em J. Amer. Math. Soc.}, 3:499--578., 1990.

\bibitem[KRV20]{KRV_DOZZ}
A.~Kupiainen, R.~Rhodes, and V.~Vargas.
\newblock Integrability of {L}iouville theory: proof of the {DOZZ}formula.
\newblock {\em Ann. of Math. (2)}, 191(1):81--166, 2020.

\bibitem[KV94]{kontsevich1994}
M.~Kontsevich and S.~Vishik.
\newblock Determinants of elliptic pseudo-differential operators.
\newblock {\em arXiv:9404046}, 1994.

\bibitem[LG13]{LeG13}
J.-F. Le~Gall.
\newblock Uniqueness and universality of the {B}rownian map.
\newblock {\em Ann. Probab.}, 41(4):2880--2960, 07 2013.

\bibitem[Lin24]{Lin}
J.~Lin.
\newblock {The Bayes Principle and Segal Axioms for $P(\varphi)_2$, with
  application to Periodic Covers}.
\newblock {\em arXiv:2403.12804}, 2024.

\bibitem[Mar87]{marden341}
A.~Marden.
\newblock {\em Geometric Complex Coordinates for Teichm{\"u}ller Space, in
  Mathematical Aspects of String Theory}, pages 341--364.
\newblock World {S}cientific, 1987.

\bibitem[Mie13]{Mie13}
G.~Miermont.
\newblock The {B}rownian map is the scaling limit of uniform random plane
  quadrangulations.
\newblock {\em Acta Math.}, 210(2):319--401, 2013.

\bibitem[MS15a]{MS15a}
J.~Miller and S.~Sheffield.
\newblock Liouville quantum gravity and the {B}rownian map {I}: The
  {QLE}(8/3,0) metric.
\newblock arXiv:1507.00719, 2015.

\bibitem[MS15b]{MS16a}
J.~Miller and S.~Sheffield.
\newblock {Liouville quantum gravity and the Brownian map II: geodesics and
  continuity of the embedding}.
\newblock arXiv:1507.00719, 2015.

\bibitem[MS16]{MS16b}
J.~Miller and S.~Sheffield.
\newblock {Liouville quantum gravity and the Brownian map III: }the conformal
  structure is determined.
\newblock arXiv:1608.05391, 2016.

\bibitem[Oik19]{Joona2019}
J.~Oikarinen.
\newblock Smoothness of correlation functions in {L}iouville conformal field
  theory.
\newblock {\em Annales Henri Poincar{\'e}}, 20(7):2377--2406, 2019.

\bibitem[OPS88]{OsgoodPS88}
B~Osgood, R~Phillips, and P~Sarnak.
\newblock Extremals of determinants of {L}aplacians.
\newblock {\em Journal of Functional Analysis}, 80(1):148--211, 1988.

\bibitem[Pic08]{Pickrell}
D.~Pickrell.
\newblock {$P(\phi)_2$ quantum field theories and Segal's axioms}.
\newblock {\em Comm. Math. Phys.}, 280:403--425, 2008.

\bibitem[Pol81]{Polyakov81}
A.~M. Polyakov.
\newblock Quantum geometry of bosonic strings.
\newblock {\em Phys. Lett. B}, 103(3):207--210, 1981.

\bibitem[Pol08]{polyakov2008quarksstrings}
A.~M. Polyakov.
\newblock From quarks to strings, 2008.

\bibitem[QC96]{QuineChoi}
J.R. Quine and J.~Choi.
\newblock {Zeta regularized products and functional determinants on spheres}.
\newblock {\em Rocky Mountain Journal of Mathematics}, 26(2):719 -- 729, 1996.

\bibitem[Rem20]{Remy20}
G.~Remy.
\newblock The {F}yodorov-{B}ouchaud formula and {L}iouville conformal field
  theory.
\newblock {\em Duke Math. J.}, 169(1):177--211, 2020.

\bibitem[{Rib}14]{Ribault14}
S.~{Ribault}.
\newblock {Conformal field theory on the plane}.
\newblock {\em arXiv:1406.4290}, 2014.

\bibitem[RS71]{Ray-Singer}
D.~Ray and I.~Singer.
\newblock R-torsion and the {L}aplacian on {R}iemannian manifolds.
\newblock {\em Advances in Mathematics}, 7(2):145--210., 1971.

\bibitem[RV14]{rhodes2014_gmcReview}
R.~Rhodes and V.~Vargas.
\newblock Gaussian multiplicative chaos and applications: a review.
\newblock {\em Probab. Surv.}, 11:315--392, 2014.

\bibitem[RZ22]{MR4483018}
G.~Remy and T.~Zhu.
\newblock Integrability of boundary {L}iouville conformal field theory.
\newblock {\em Comm. Math. Phys.}, 395(1):179--268, 2022.

\bibitem[Seg88]{Segal87}
G.~B. Segal.
\newblock The definition of conformal field theory.
\newblock In {\em Differential geometrical methods in theoretical physics
  ({C}omo, 1987)}, volume 250 of {\em NATO Adv. Sci. Inst. Ser. C Math. Phys.
  Sci.}, pages 165--171. Kluwer Acad. Publ., Dordrecht, 1988.

\bibitem[Son88]{SONODA1988}
H.~Sonoda.
\newblock Sewing conformal field theories ii.
\newblock {\em Nuclear Physics B}, 311(2):417--432, 1988.

\bibitem[SW12]{sheffield_werner_CLE}
S.~{Sheffield} and W.~{Werner}.
\newblock Conformal loop ensembles: the {M}arkovian characterization and the
  loop-soup construction.
\newblock {\em Annals of Mathematics}, 176(3):1827--1917, 2012.

\bibitem[Tes01]{Teschner_revisited}
J.~Teschner.
\newblock Liouville theory revisited.
\newblock {\em Classical and Quantum Gravity}, 18(23):R153--R222, nov 2001.

\bibitem[Tes09]{10.1007/978-90-481-2810-5_46}
J.~Teschner.
\newblock Nonrational conformal field theory.
\newblock In Vladas Sidoravi{\v{c}}ius, editor, {\em New Trends in Mathematical
  Physics}, pages 697--739, Dordrecht, 2009. Springer Netherlands.

\bibitem[Tes11]{TeschnerHitchin}
J.~Teschner.
\newblock {Quantization of the Hitchin moduli spaces, Liouville theory and the
  geometric Langlands correspondence I}.
\newblock {\em Advances in Theoretical and Mathematical Physics}, 15(2):471 --
  564, 2011.

\bibitem[TV15]{Teschner_Vartanov}
J.~Teschner and G.~Vartanov.
\newblock Supersymmetric gauge theories, quantization of $\mathcal{M}_{rm
  flat}$, and conformal field theory.
\newblock {\em Advances in Theoretical and Mathematical Physics}, 19(1):1--135,
  2015.

\bibitem[Wei87]{Wei}
W.I. Weisberger.
\newblock Conformal invariants for determinants of {L}aplacians on {R}iemann
  surfaces.
\newblock {\em Commun. Math. Phys}, 112:633--638, 1987.

\bibitem[Wol90]{10.4310/jdg/1214444322}
S.~Wolpert.
\newblock {The hyperbolic metric and the geometry of the universal curve}.
\newblock {\em Journal of Differential Geometry}, 31(2):417 -- 472, 1990.

\bibitem[Wu22]{Wu}
B.~Wu.
\newblock {Conformal Bootstrap on the Annulus in Liouville CFT}.
\newblock {\em arXiv:2203.11830}, 2022.

\bibitem[Zam05]{Zamolodchikov_05}
Al. Zamolodchikov.
\newblock Three-point function in the minimal {L}iouville gravity.
\newblock {\em Theoretical and Mathematical Physics}, 142(2):183--196, 2005.

\bibitem[ZZ96]{Zamolodchikov96}
A.~Zamolodchikov and Al. Zamolodchikov.
\newblock Conformal bootstrap in {L}iouville field theory.
\newblock {\em Nuclear Phys. B}, 477(2):577--605, 1996.

\end{thebibliography}

\end{document}